\documentclass[a4paper,10pt]{article}
\usepackage{amssymb}
\usepackage{filecontents}
\usepackage{latexsym}
\usepackage{amsmath}
\usepackage{amsthm}
\usepackage{amsfonts}
\usepackage{amssymb}
\usepackage{amsmath,amscd}
\usepackage{graphicx}
\usepackage{xcolor}
\usepackage{hyperref}

\usepackage{mathtools}

\newtheorem{Th}{Theorem}
\newtheorem{Prop}{Proposition}

\newtheorem{Lm}{Lemma}
\newtheorem{Lma}{Lemma}[section]

\newtheorem{Rm}{Remark}
\newtheorem{Ope}{Open Problem}
\newcommand{\be}{\begin{equation}}
\newcommand{\ee}{\end{equation}}
\newcommand{\bes}{\begin{equation*}}
\newcommand{\ees}{\end{equation*}}

\newcommand{\R}{\mathbb{R}}
\newcommand{\N}{\mathbb{N}}
\newcommand{\C}{\mathbb{C}}
\newcommand{\Z}{\mathbb{Z}}

\newcommand\res{\mathop{\hbox{\vrule height 7pt width .5pt depth 0pt
\vrule height .5pt width 6pt depth 0pt}}\nolimits}

\newcommand{\reset}{\setcounter{equation}{0}\setcounter{Th}{0}\setcounter{Prop}{0}\setcounter{Co}{0}
\setcounter{Lm}{0}\setcounter{Rm}{0}}

\def\Xint#1{\mathchoice
{\XXint\displaystyle\textstyle{#1}}%
{\XXint\textstyle\scriptstyle{#1}}%
{\XXint\scriptstyle\scriptscriptstyle{#1}}%
{\XXint\scriptscriptstyle\scriptscriptstyle{#1}}%
\!\int}
\def\XXint#1#2#3{{\setbox0=\hbox{$#1{#2#3}{\int}$}
\vcenter{\hbox{$#2#3$}}\kern-.5\wd0}}
\def\dashint{\Xint-}

\def\La{\Lambda}
\def\La{\Lambda}
\def\ti{\tilde}
\def\lf{\left}
\def\rg{\right}

\def\al{\alpha}
\def\la{\lambda}

\def\ep{\varepsilon}

\def\ds{\displaystyle}
\def\ov{\overline}
\def\Om{\Omega}
\def\om{\omega}
\def\p{\partial}

\def\res{\mathop{\hbox{\vrule height 7pt width .5pt 
depth 0pt\vrule height .5pt width 6pt depth 0pt}}\nolimits}
\def\B{\mathbb B}
\def\S{\mathbb S}
\def\tie{\tilde{\frak e}}

\DeclareMathOperator{\dist}{dist}

\makeatletter
\newcommand{\tpitchfork}{%
  \vbox{
    \baselineskip\z@skip
    \lineskip-.52ex
    \lineskiplimit\maxdimen
    \m@th
    \ialign{##\crcr\hidewidth\smash{$-$}\hidewidth\crcr$\pitchfork$\crcr}
  }%
}
\makeatother

\numberwithin{equation}{section}
\numberwithin{Th}{section}
\numberwithin{Prop}{section}
\numberwithin{Lm}{section}
\numberwithin{Co}{section}
\numberwithin{Rm}{section}

\begin{document}

\title{ 
Gauge Symmetry Breaking in the Asymptotic Analysis \\of  Self Dual Yang-Mills-Higgs $SU(2)$ Monopoles}
\author{ Tristan Rivi\`ere\footnote{Department of Mathematics, ETH Zentrum,
CH-8093 Z\"urich, Switzerland.}}

\maketitle

{\bf Abstract :} {\it We consider the $SU(2)$ Self-Dual Yang Mills Higgs Lagrangian in 3 dimension. By adding a ''Gauge Mass'' term to this YMH Lagrangian in the form of $L^2$ norm of the connection  we break the gauge invariance and critical points are automatically fulfilling globally the Coulomb condition. We study the so called ``large mass asymptotic'', which has the effect of ''squeezing'' the monopoles. For any unit Higgs field data at the boundary  we  prove that minimizers of this Coulomb-Yang-Mills-Higgs Functional converge to  harmonic maps into $S^2$ extending this data. This asymptotic  moreover is subject to concentration conpactness phenomena and the convergence is strong away from a 1 dimensional rectifiable closed concentration set. Then we prove that, having chosen a large enough coupling constant, the limiting minimal energy is converging towards the minimal Brezis-Coron-Lieb relaxed harmonic map energy for this boundary data. In the second part of the paper we examine a different asymptotic regime characterised by overloading monopoles.
In this regime we prove that asymptotically, the magnetic field becomes exclusively longitudinal with a $U(1)$ abelian component along the  Higgs Field while the Higgs field itself converges to a smooth absolute minimizer of a relaxation of the Faddeev-Skyrme functional of maps from $\B^3$ into $\S^2$. 
In the third part of the paper we study the behaviour of these configurations when the parameter in front of the Fadeev-Skyrme component respectively goes to zero and $+\infty$. In the first case one recovers the Brezis Coron Lieb relaxed energy  at the limit while in the second case the minimal limiting energy is converging towards the minimal Dirichlet energy of maps into ${\S}^3$ whose projection by the Hopf fibration is equal to the fixed boundary data.
}

\bigskip


\bigskip

\section{Introduction}

The XXth Hilbert Problem for the Dirichlet Energy of Sphere valued maps on the 3-dimensional euclidian ball is asking about the existence or not of a smooth critical point of the Dirichlet energy extending an arbitrary
data $\phi\ :\ \p\B^3_1(0)\ \rightarrow\ \S^n$. Such a critical point is called harmonic map. For $n=2$ and for this dimension only (since $\pi_2(\S^n)\ne 0$ if and only if $n=2$) a topological obstruction has to be taken into consideration and the degree of $\phi$ as a map from $\p \B^3_1(0)$ into $\S^2$ has to be assumed to be zero. In order to solve this Hilbert problem the most natural approach consists in minimizing the Dirichlet energy itself among maps equal to $\phi$ at the boundary. This approach has been initiated by R.Schoen and K.Uhlenbeck in the early 80's and for $n\ne 2$ they have proved that any smooth boundary data into $\S^n$ admits a smooth $\S^n$ valued minimizing harmonic extension (\cite{SU3}).
In the case $n=2$ however, it has been discovered that some smooth zero degree maps from $\p\B^3_1(0)$ into $\S^2$ admits no smooth minimizing harmonic extension. More precisely R.Hardt and F.H.Lin discovered in \cite{HL1} the following
``Lavrentiev phenomenon'' :
\be
\label{Lav}
\begin{array}{l}
\ds\exists\,\phi\in C^\infty(\p\B_1^3(0),\S^2)\quad\mbox{ s. t. }\quad\mbox{deg}\,\phi=0\qquad\\[5mm]
\ds\inf\lf\{E(u):= \frac{1}{2}\int_{\B_1(0)}|\nabla u|^2\ dx^3\ ;\ u\in C^1(\B_1(0),{\S}^2)\quad\mbox{ and }\quad u=\phi\quad\mbox{ on }\p\B^3_1(0) \rg\} \\[5mm]
\ds\quad>\min\lf\{E(u):= \frac{1}{2}\int_{\B_1(0)}|\nabla u|^2\ dx^3\ ;\ u\in W^{1,2}(\B_1(0),{\S}^2)\quad\mbox{ and }\quad u=\phi\quad\mbox{ on }\p\B^3_1(0) \rg\} 
\end{array}
\ee
where $W^{1,2}(\B_1(0),{\S}^2)$ is the natural class of minimization made of maps with weak derivatives in $L^2(\B_1(0))$. As a matter of fact, Hardt and Lin observed that in order to minimize the Dirichlet energy for some boundary data, it is preferable to allow for the existence of point singularities with non zero topological degrees. This theorem was then establishing the optimality of the  regularity established few years before by Schoen and Uhlenbeck who proved that any Dirichlet energy minimizer among maps in $W^{1,2}(\B_1(0),{\S}^2)$ equal to $\phi$ at the boundary have at most isolated singularities which do not accumulate at the boundary, hence there are at most finitely many of them\footnote{The number of these singularities is controlled by a constant times the Dirichlet Energy of $\phi$ \cite{AL}} and they all have $\pm 1$ topological degree (\cite{SU1}, \cite{SU2}).

In the same period of time, H.Brezis, J.-M. Coron and E.Lieb analysed the Lavrentiev gap (\ref{Lav}) and proved that for any map $u$ with isolated singularities
\[
\inf\lf\{\limsup_{k\rightarrow+\infty}E(v_k)\ ;\ v_k\in C^1(\B_1(0),{\S}^2)\, \quad\mbox{ and }\quad v_k\rightharpoonup u\in {\mathcal D}'(\B)\rg\}=E(u)+4\pi\,L(u)
\]
where $L(u)$ is the {\it minimal connection} of $u$ that is the minimal mass among the 1-dimensional currents bounding the topological singularities\footnote{The notion of topological singularities for maps in $W^{1,2}(\B^3,S^2)$    can be given a rigorous mathematical definition. This identifies with the distributional 3-form
\[
d(u^\ast\om_{S^2})
\]
It measures exactly the obstruction to  approximate the map $u$ by smooth maps in $C^\infty(\B^3,\S^2)$ strongly in $W^{1,2}$ (see \cite{Be1}).} of $u$ taking into account their multiplicities (see \cite{BCL}). They also gave an ``analytic'' formulation of $L(u)$ based on the fact that area minimizing curent of dimension 1 are calibrated (H. Federer \cite{Fed} section 5) fact for which they gave an independent proof.
\[
L(u):=\frac{1}{4\pi}\sup\lf\{\int_{\B_1(0)}  d\xi\wedge u^\ast\om_{\S^2}-\int_{\p\B_1(0)} \xi \,\phi^\ast\om_{\S^2}\ ;\ \|d\xi\|_\infty\le 1\rg\}
\]
where $\om_{\S^2}$ is the volume form on $\S^2$.  More or less simultaneously and independently from \cite{BCL}
M.Giaquinta, G.Modica and J.Soucek found a Geometric Measure Theory characterisation of $4\pi\, L(u)$ as being the minimal mass of the ``vertical part'' needed to complete the graph of $u$ seen as a rectifiable 3-dimensional current in $\B^3_1(0)\times \S^2$ in order to make it ''boundary free'' in its interior. With this interpretation at hand the  following so called {\it relaxed energy} was proposed  for any $u\in W^{1,2}_\phi(\B_1(0),{\S}^2)$
\[
F(u)=E(u)+4\pi\,L(u)
\]
for which the GMT interpretation gives  the lower semi-continuity in $ W^{1,2}_\phi(\B_1(0),{\S}^2)$ (see \cite{GMS-1}, \cite{GMS-2} and \cite{GMS-3}). A direct more analytic proof not using GMT was independently given by F.Bethuel, H.Brezis and J.M.Coron in \cite{BBC}.
Hence there exists a minimizer to $F$ in $W^{1,2}_\phi(\B_1(0),{\S}^2)$ and there holds
\[
\begin{array}{l}
\ds\inf\lf\{E(u):= \frac{1}{2}\int_{\B_1(0)}|\nabla u|^2\ dx^3\ ;\ u\in C^1(\B_1(0),{\S}^2)\quad\mbox{ and }\quad u=\phi\quad\mbox{ on }\p\B^3_1(0) \rg\} \\[5mm]
\ds\quad=\min\lf\{F(u):= \frac{1}{2}\int_{\B_1(0)}|\nabla u|^2\ dx^3+4\pi\,L(u)\ ;\ u\in W^{1,2}(\B_1(0),{\S}^2)\quad\mbox{ and }\quad u=\phi\quad\mbox{ on }\p\B^3_1(0) \rg\} 
\end{array}
\]
The rather ``exact'' penalisation of topological singularities obtained by adding the \underbar{null Lagrangian} $4\pi \, L(u)$ (which does not affect the Euler Lagrange Equation) to the Dirichlet energy together with the lower semi-continuity of $F$ was giving the hope that by simply minimizing $F$ one would solve the Hilbert XXth problem for the Dirichlet energy of $\S^2$ valued maps. Unfortunately the best known  regularity result at this stage for the {\it relaxed energy} is telling that the singular set of any minimizer of $F$ is at most one dimensional (\cite{GMS-2}). This is clearly not as good as  the Schoen Uhlenbeck regularity result for minimizers of $E$. In fact, in an unexpected and ironic way, the relaxed energy was the main tool used by the author  in his  PhD thesis to prove the existence of everywhere discontinuous $\S^2-$valued harmonic extension to any non constant $\phi$ (\cite{Riv95}). The question to know whether or not minimizers of the relaxed energy $F$ have topological singularities that is
\[
d(u^\ast\om_{\S^2})\qquad\mbox{ is equal to zero or not }
\]
is still an open problem (one of the problems from Brezis's ``legacy paper'' \cite{Brez}).

Some applications in the calculus of variations aiming at constructing new harmonic maps require to regularise the Dirichlet energy of sphere valued maps in order to call upon the classical deformation theory in infinite dimensional spaces. For instance the author proposed in \cite{Rivom} a minmax operation that should establish the minimality of the Hopf fibration for the 3-energy among all maps from $S^3$ into $S^2$ which are not homotopic to a constant that would answer a conjecture posed in \cite{Riv}. The most natural and handful  regularisation which is both Palais-Smale and that fulfils a nice monotonicity formula is the Ginzburg Landau regularisation 
\[
E_\ep(u):=\frac{1}{2}\int_{\B_1(0)}|\nabla u|^2+\frac{1}{2\ep^2}(1-|u|^2)^2\ dx^3\ .
\]
This regularisation is defined on the whole Hilbert Sobolev space $W^{1,2}(\B^3_1(0),{\R}^3)$ and is  releasing the constraint $u\in {\S}^2$, replacing it by a penalisation term in the Lagrangian, $(1-|u|^2)^2$  wich is the a so called ``Ginzburg-Landau'' type potential and preceded
by a ``very large'' parameter $1/\ep^2$.
The study of its critical points and the asymptotic analysis is very reminiscent to the one of stationary harmonic maps\footnote{ The asymtotic analysis of the non gauge invariant Ginzburg-Landau energy $E_\ep$ has been initiated first for maps into ${\R}^2$ by F.Bethuel, H.Brezis and F.H\'elein in \cite{BBH}. The situation in the ${\R}^2$ valued case differs significantly from the ${\R}^3$ valued problem we are considering in this work and offers different challenges. The first one is related to the fact that in the ${\R}^2$-valued case the minimizing energy might blow-up  depending on the boundary datas chosen in $H^{1/2}(\p\B^m,\S^1)$.  In the $\R^n$ case $n>2$ instead, for any boundary data $\phi\in H^{1/2}(\p\B^m,\S^{n-1})$ the minimum of $E_\ep$ is uniformly bounded.} (see \cite{LW1},\cite{LW2} and \cite{Gia} in the context of minmax). The drawback of using this energy is that at the limit nothing prevents to converge to the subclass of  harmonic maps which have topological singularities and are not  approximable by smooth maps. This regularisation does not ``see'' topological singularity.

In the present work we aim at proposing a regularisation of the Dirichlet enery which is both penalising the topological singularities and to which simultaneously Palais Smale theory of deformation in infinite dimensional spaces can be applied.
To that aim we will embed the harmonic map problem into gauge theory and instead of considering solely  the Dirichlet energy of a map $u\in W^{1,2}(\B^3_1(0),\S^2)$ to which a penalisation of the form of a Ginzburg-Landau type potential preceded by a large coefficient $\ep^{-2}$ is added we will also include a $\frak{su}(2)$ connection one form $A$. The role of this connection will be to ``overload'' each of the possible topological singularities of the limiting map $u^\infty$ by  a Self Dual $SU(2)-$Yang-Mills-Higgs monopoles. In that way, by increasing the coupling constant we will ``clearout'' every possible topological singularity for $u$ and obtain a strongly approximable map. The Lagrangian we propose is the self dual $SU(2)-$Yang-Mills-Higgs Energy to which a ``gauge breaking'' term which is going to converge to the Dirichlet energy of $u$  is added. This term is simply the $L^2$ energy of the connection. Percisely, for $\la\ge1$, $\mu\ge 1$ and $0<\ep<1$ we consider
\[
CYMH_{\ep}^{\la,\mu}(u,A):=\frac{1}{2}\int_{\B_1(0)}|A|^2 \ dx^3 +\mu\,\int_{\B_1(0)}\frac{|d_Au|^2}{\ep}+\ep\,|F_A|^2+\frac{\la}{2\ep^3}(1-|u|^2)^2\ dx^3
\]
where $u\in W^{1,2}_\phi(\B^3_1(0),{\S}^2)$ is the so called ``Higgs field'' and $A\in L^2(\Om^1(\B^3_1(0))\otimes\frak{su}(2))$ is an $L^2$ connection one form taking values into $\frak{su}(2)$ the Lie algebra of $SU(2)$. Moreover we adopt the conventional notation in gauge theory that is $d_Au$ is the covariant derivative of $u$ with respect to the connection one form $A$ given by
\[
d_Au:=du+[A,u]
\]
where $[\cdot,\cdot]$ is the usual Lie bracket operation in the Lie Algebra $\frak{su}(2)$ and $F_A$ is the curvature of $A$
\[
F_A:=dA+\frac{1}{2}\,[A\,\wedge\,A]=dA+A\wedge A\ .
\]
Observe that our Lagrangian can be decomposed in the following way
\[
CYMH_{\ep}^{\la,\mu}(u,A):=\frac{1}{2}\int_{\B_1(0)}|A|^2 \ dx^3 +\mu\,YMH_\ep^\la(u,A)
\]
where $YMH_\ep^\la(u,A)$ is the classical Gauge invariant Self Dual Yang-Mills-Higgs energy whose critical points are called $SU(2)$ Monopoles. The asymptotic analysis of $G-$Monopoles  started in the work of A.Jaffe and C.Taubes \cite{JT} in 3 real dimensions and is going through a recent intense and thorough development on higher dmensional manifolds with special holonomy : $G=SU(2)$ on Calabi Yau 3-folds (Calabi-Yau Monopoles) and $G=G_2$ on $G_2-$manifolds $G_2$ monopoles. A very special attention is given  in particular in the so called {\it Large Mass limit } $\ep\rightarrow 0$ which has the effect of ``squeezing'' monopoles and to  ``abelianise'' the connection and the structure of the bundle away from the concentration set (see in particular the works by D.Fadel, A.Nagy and G.Oliveira \cite{FNO} or even more recently the impressive  contributions by  D.Parise,  A.Pigati and D.Stern \cite{ PPS} or by Y. Li \cite{YL}). We would like to insist on the fact that our analysis below, though we are considering in the first part the large mass limit, differs substantially for the previous works du to the presence of the ``gauge breaking'' term $\int |A|^2\ dx^3$.

\medskip

The gauge transformation associated to any choice of 
map $g\in W^{1,2}(\B^3_1(0),SU(2))$  is the following mapping
\[
(u,A)\quad\longrightarrow\quad (u^g,A^g):=(g^{-1}\,u\,g, g^{-1} A\,g+g^{-1}\,dg)\ .
\]
By gauge invariance for the Yang-Mills-Higgs Lagrangian it is meant that
\[
\forall \,g\in W^{1,2}(\B^3_1(0),SU(2))\qquad YMH_\ep^\la(u^g,A^g)=YMH_\ep^\la(u,A)
\]
The first term however, the $L^2$ norm of the connection is not gauge invariant and is used to produce {\it Coulomb Gauge}. Precisely, the gauge class of $(u,A)$ being fixed and the energy  $YMH_\ep^\la(u,A)$ being fixed, critical points to
$CYMH_{\ep}^{\la,\mu}(u,A)$ under gauge changes within a fixed gauge class will automatically give a critical point of the map
\[
g\quad\longrightarrow\quad\frac{1}{2}\int_{\B^3_1(0)}\lf|g^{-1} A\,g+g^{-1}\,dg\rg|^2\ dx^3
\]
which are known to be characterised by the {\it Coulomb Condition}
\[
d^\ast A^g=0\ .
\]
 Hence looking at critical points to $CYMH$ leads naturally to a {\bf global Coulomb Gauge}\footnote{The search for global Coulomb gauge in optimal spaces is a challenging question considered in particular by the author in collaboration with M.Petrache in \cite{PeR}. } (see proposition~\ref{pr-coul} below). This is why we call the Lagrangian $CYMH_{\ep}^{\la,\mu}(u,A)$ the {\it Coulomb-Yang-Mills-Higgs} Lagrangian.In sections II  we prove that for any choice of $\phi\in H^{1/2}(\p\B_1(0),\S^2)$ there exists a minimizer of $CYMH$ under the constraint $u=\phi$ on $\p\B_1(0)$ (see proposition~\ref{pr-II.1}). Then in section III we establish the smoothness of minimizers of (see theorem~\ref{th-reg}). Our third main result in the present work is the following theorem which describe the asymptotics of the minimal energy for fixed $\la$ and $\mu$ large enough as $\ep\rightarrow 0$.
\begin{Th}  
\label{th-clearout}{\bf [Asymptotical Value of the minimal Energy]}
Let $\phi\in C^2(\p\B^3,\S^2)$ of zero degree. There exists $\La_\phi>0$ such that for any choice of $1\le\la\le \mu$ and $\la\mu\ge \La_\phi$ there holds
\[
\begin{array}{l}
\ds \lim_{\ep\rightarrow 0}\min\lf\{CYMH_\ep^{\la,\mu}(u,A)\ ;\ (u,A)\in {\mathcal E}_\phi\rg\}\\[5mm]
\ds=\min\lf\{ \frac{1}{8}\,\int_{\B_1(0)}|du|^2\ dx^3+\pi\, L_\phi(u) \  ;\  u\in W_\phi^{1,2}(\B^3,\S^2)\rg\}
\end{array}
\]
where 
\[
L_\phi(u):=\frac{1}{4\pi}\sup\lf\{ \int_{\B^3}d\xi\wedge u^\ast\om_{S^2} -\int_{\p\B^3}\xi\ \phi^\ast \om_{S^2}\ ;\ \|d\xi\|_\infty\le 1\rg\}\ .
\]
\hfill $\Box$
\end{Th}

The proof of theorem~\ref{th-clearout} is making use of the following concentration compactness theorem for sequences of minimizers which is one of the main contribution of the present work. 
\begin{Th}
\label{th-cons-comp} {\bf[Concentration compactness]}
Let $\phi$ in $C^2(\p\B_1(0),{\S}^2)$ having  zero degree, let $\ep_k\rightarrow 0$ and let $(u^k,A^k)$ be a sequence of minimizers of $CYMH_{\ep_k}^{\la,\mu}$ where $\la\ge 1$ and $\mu\ge 1$ are fixed. Then there exists a subsequence, a one dimensional rectifiable relatively closed set $K\subset \B_1(0)$ and a weakly harmonic map $u^\infty\in W^{1,2}_\phi(\B_1(0),{\S}^2)$ such that
\[
A^k\longrightarrow -\frac{1}{4}\,[u^\infty, d u^\infty]\qquad\mbox{ strongly in }L^\infty_{loc}(\B_1(0)\setminus K)
\] 
moreover
\[
u^\infty\in C^\infty(\B_1(0)\setminus K)\ ,
\]
and
\[
\begin{array}{l}
\ds \lf(|{A}^k|^2+\mu\,\frac{|d_{{A}^k} {u}^k|^2}{\ep}+\mu\,\ep\,|F_{{A^k}}|^2\ +\frac{\la\mu}{2\,\ep^3}(1-|{u}^k|^2)^2\rg)\ dx^3\\[5mm]
\ds\rightharpoonup \ \frac{1}{4}|du^\infty|^2\ dx^3+f\ {\mathcal H}^1\res K \qquad\mbox{ in }\quad{\mathcal M}_{loc}(\B^3_1(0))\ ,
\end{array}
\]
where ${\mathcal M}(\B^3_1(0))$ is the space of Radon measures in $\B^3_1(0)$ and $f\in L^1_{{\mathcal H}^1}(K,{\R}_+)$. Finally we have for any $r_1<1$
\[
{\mathcal H}^1(K\res {\B}_{r_1}(0))<+\infty\ .
\]
\hfill $\Box$
\end{Th}
\begin{Rm}
\label{rm-longconnexion}
One of the striking fact of this convergence, which was not a-priori obvious, is that the longitudinal part of the connection form is converging weakly to zero as $k\rightarrow +\infty$
\[
A^k\cdot u^k\ \longrightarrow \ 0\qquad\mbox{ strongly in }L^\infty_{loc}(\B_1(0)\setminus K)\ .
\]
As a consequence one passes from a problematic of $SU(2)$ valued map (in small scales) to $S^2-$valued map
\hfill $\Box$
\end{Rm}
\begin{Rm}
\label{rm-conccom} The asymptotic analysis of $CYMH^{\mu,\la}_\ep$ as $\ep\rightarrow 0$ (for $\mu$ and $\la$ fixed) is made particularly delicate due to the ``interference'' of two concentration compactness effects : one is related to the formation of $SU(2)$ Monopoles at a smaller and smaller scale and is $0$ dimensional, the other one is related to the formation of line concentration converging to minimal rectifiable currents connecting the topological singularities of the limiting map $u^\infty$ and which is then one dimensional. This competition between the two scales is ``materialised'' by sign distributions in the monotonicity formula which makes it a-priori impossible to use as such without further thorough analysis of the influence of each of the terms in this formula (see the comments in the next paragraphs of the introduction). Observe that these two concentration effects, zero and one dimensional are a-priori ``decorelated'' in space. The limiting location of the monopoles at points $a_i$ whose limiting charge is given by
\[
\lim_{r\rightarrow 0}\lim_{k\rightarrow +\infty}\frac{1}{4\pi}\int_{\B_r(a_i)}\mbox{Tr}\lf(F_{A^k}\wedge d_{A^k}u^k\rg)
\]
contribute to the topological singularities of $u^\infty$ but, for low values of the parameters $\mu$ and $\la$, it is not excluded that the line concentration will not connect them at all.
\hfill $\Box$
\end{Rm}
\begin{Rm}
\label{rm-thresh}
It seems that by decreasing the coupling constant $\mu$ from  $\La_\phi/\la$ to $0$ one penalises less and less the presence of $SU(2)-$monopoles and the formation of topological singularities. We then expect  that, below some critical value of $\mu$, the limiting map $u^\infty$ will be an absolute minimizer of the Dirichlet energy as given by the Schoen-Uhlenbeck analysis and no line concentration happens. In other words, the moving of the parameter $\mu$ is realising an interpolation between the minimization of the Dirichlet energy and the minimization of Brezis-Coron-Lieb relaxed energy.\hfill $\Box$
\end{Rm}
The proof of theorem~\ref{th-cons-comp} is based on a $\delta-$regularity lemma independent of $\ep$. This $\delta-$regularity result is made particularly challenging and is taking a substantial number of pages of the present work due to the {\bf absence of obvious Monotonicity Formula}. 

This absence can be explained in the following way. Consider for instance a much simpler linear framework and we look for a one form $B$ solving
\be
\label{linea}
\lf\{
\begin{array}{l}
d^\ast dB+B=0\qquad\mbox{ in }\B^3\\[5mm]
d^\ast B=0\qquad\mbox{ in }\B^3
\end{array}
\rg.
\ee
The scaling invariant terms associated respectively to the $1-$form $B$ and the $2-$form $dB$ are for any $x_0\in\B^3_1(0)$ and $r>0$
\be
\label{scalrul}
\frac{1}{r}\int_{\B_r(x_0)}|B|^2\ dx^3\qquad\mbox{ and }\qquad r\int_{\B_r(x_0)}|dB|^2\ dx^3
\ee
Thus naturally, by following the stationarity condition satisfied by solutions to (\ref{linea}) critical points of the Lagrangian $\int\ |dB|^2+|B|^2\ dx^3$ among co-closed forms is giving
\be
\label{monfail}
\begin{array}{l}
\ds-\int_{\B_r(x_0)}{|B|^2}\ dx^3+{r}\,\int_{\p \B_r(x_0)}|B|^2\ dvol_{\p \B_\rho(x_0)}\\[5mm]
\ds+\int_{\B_r(x_0)}{|dB|^2}\ dx^3+r\,\int_{\p \B_r(x_0)}{|dB|^2}\ dvol_{\p \B_\rho(x_0)}\\[5mm]
\ds=2\,{r}\int_{\p \B_r(x_0)}\frac{|B\res \p_r|^2}{\ep}\ dvol_{\p \B_r(x_0)}+2\, r\int_{\p \B_r(x_0)}{|dB\res\p_r|^2}\ dvol_{\p \B_r(x_0)}
\end{array}
\ee
The distribution of signs
\[
\begin{array}{cc}
-& +\\[3mm]
+ & +
\end{array}
\]
 in the l.h.s.of (\ref{monfail}) are faithful to the scaling rules (\ref{scalrul}) of respectively 1-forms and 2-forms and this unfortunately leads to a non conclusive formula regarding monotonicity.
 The fact that the stationary equation is not leading to a monotonicity formula nevertheless does not exclude monotonicity to hold but it has to be obtained following a different argument. In the linear case (\ref{linea}) one would compute the {\it Bochner-Weitzenb\"ock} formula
 \[
 \Delta(|dB|^2+|B|^2)=2\,|\nabla dB|^2+2|\nabla B|^2\ge 0
 \]
 from which we deduce that
 \[
 \frac{d}{dr}\lf(\frac{1}{r^3}\int_{\B_r(x_0)}|dB|^2+|B|^2\ dx^3\rg)\ge 0
 \]

 Coming back to the Coulomb-Yang-Mills-Higgs system we will have in particular the common scaling rules of the 1-forms $A$ and $d_Au$ being different from the one of the two form $F_A$. This makes the stationary argument non conclusive as described in the above discussion.
 
 In the abelian case for the self dual Yang-Mills-Higgs energy A.Pigati and D.Stern used the {\bf Bochner-Weitzenb\"ock} formula to obtain a pointwize control of the norm of the curvature by the Ginzburg-Landau potential part (see \cite{PiSt} and \cite{JT} Theorem 8.1 is chapter III for the two dimensional counterpart of this argument). In the non abelian case with a different scaling on the potential and with the Coulomb term in addition this argument does not seem to be adaptable. Our approach instead consist in using the minimality of the solution combined with an {\bf epiperimetric inequality} for special solutions to the harmonic map equation to establish monotonicity. Inspired by the terminology of minimal surface theory, the {\it epiperimetric inequality} for $S^2-$valued harmonic maps is saying that the minimal Dirichlet energy of a radial extension on a ball is larger than a strict portion of the minimal harmonic energy (see lemma~\ref{lm-ext})
 \[
 \frac{1}{r}\,\int_{\B_r(x_0)}|\nabla u|^2\ dx^3\le (1-\nu)\,\int_{\p \B_r(x_0)}|\nabla_Tu|^2\ dvol_{\p\B_r(x_0)}\ ,
 \]
 where $u$ is a smooth harmonic extension with sall energy, $\nabla_T$ is the tangential derivative and $0<\nu<1$ is a \underbar{universal constant}. This argument permit to prove a monotonicity formula for a minimizing configuration but only between the scales 1 and $\sqrt{\ep}$. Between the scales $\sqrt{\ep}$ and $\ep$ we argue quite differently. Using the monotonicity up to $r=\sqrt{\ep}$, we prove that the $L^{3/2}$ norm of the curvature is small enough in such a way that we can extract a controlled Coulomb gauge using Uhlenbeck Coulomb Gauge extraction method. An adapted use of this controlled Coulomb gauge permits to ``push'' monotonicity between $\sqrt{\ep}$ and $\ep$.

 \bigskip

In the second part of the paper we study the limit $\ep\rightarrow 0$, $\la=\la_0$ and $\mu=\tau/\ep$ where $\tau$ is an arbitrary positive fixed number. Our first main result in this asymptotic process is the following.
\begin{Th}
\label{th-harmflu}
Let $\phi\in C^2(\B^3,\S^2)$ having zero degree, $\tau\in(0,1)$, $\la\ge 1$ and let $\ep_k\rightarrow 0$. Let $(u^k,A^k)$ be a sequence of minimizers of $CYMH_{\ep^k}^{\la,\tau/\ep_k}$ in ${\mathcal E}_\phi$. Then modulo extraction of a subsequence
\[
A^k\rightarrow A^\infty\qquad\mbox{ and }\qquad u^k\rightarrow u^\infty\qquad\mbox{ strongly in }C^l_{loc}(\B^3)
\]
for any $l\in{\N}$ moreover $(u,A)$ satisfy the following conditions : There exists a smooth closed 2-form $G$ such that
\be
\label{lond}
\lf\{
\begin{array}{l}
\ds -d\ast\lf(u\wedge d u\rg)=\tau\, d\ast G\wedge du\qquad\mbox{ in }\B_1(0)\\[5mm]
\ds \tau\,d\,d^\ast G+G=\,u^\ast\om_{S^2}\qquad\mbox{ in }\B_1(0)\\[5mm]
\ds dG=0\qquad\mbox{ in }\B_1(0)\\[5mm]
\ds\iota^\ast_{\p\B_1(0)}\ast G=0
\end{array}
\rg.
\ee
and $F_{A}=-2^{-1}\,G\ u$ and
\be
\label{conscon}
A(u,\eta)=-\frac{1}{4}[u,du]+\frac{\tau}{2}\, d^\ast G\ u
\ee
Moreover, let $\eta$ given by
\[
\eta:=\tau\,d^\ast G\qquad\mbox{ in such a way that }\qquad F_A=2^{-1}\,(d\eta-u^\ast\om_{\S^2})\ u
\]
The pair $(u,\eta)$ is an absolute minimizer among smooth maps and forms of the following  Lagrangian\footnote{The Euler-Lagrange Equations associated to the Lagrangian
\[
YMP(A):=\frac{1}{2}\int_{\B^3} m^2\,|A|^2+\frac{1}{2}\,|F_{A}|^2\ dx^3
\]
where $A$ is a general connection form (unlike (\ref{conscon} where $A$ is constarined) are known under the name of {\bf Yang-Mills-Proca} Equations and were introduced by  N.G.Marchuk nd D.S.Shirokov in \cite{MaS}. The system is used to study nonlinear field configurations and solutions that differ from standard, massless Yang-Mills theories (like QCD). They are often used to model scenarios where gauge bosons acquire mass without relying strictly on the Higgs mechanism (see also the comments at the end of the introduction). }

\be
\label{mathR}
{\mathcal R}_\tau(u,\eta):=\frac{1}{8}\int_{\B^3_1(0)}|du|^2+|\eta|^2+\tau\,|d\eta-u^\ast\om_{\S^2}|^2\ dx^3=\frac{1}{2}\int_{\B^3}|A(u,\eta)|^2+\tau\,|F_{A(u,\eta)}|^2\ dx^3
\ee
under the solely constraint $u=\phi$ on $\p\B_1(0)$.

\hfill $\Box$
\end{Th}
\begin{Rm}
\label{rm-Rel}
The Lagrangian ${\mathcal R}_\tau$ is a {\bf local} aproximation of the highly {\bf non-local} Brezis-Coron-Lieb Relaxed Energy. The addition of the variable $\eta$ follows an approach inspired by   convexification methods where various 
constraints  (here the constraint of staying within strongly approximable map) in variational problems can be implemented by adding a new variable and the dependance with respect to this variable happens to be linear (this idea has been  developed  in the works of Y.Brenier in the last decades  for Euler Equation, Monge Ampere and Optimal Incompressible transport Problems (see \cite{Bren}) or by R.Strichartz for sub-riemannian geodesics (\cite{Stri})...).
\end{Rm}
\begin{Rm}
\label{rm-lond}
Observe that the equation satisfied by $G$ in (\ref{lond}) is a {\bf London} equation with {\bf magnetic current} given $u^\ast\om$  satisfied by the limiting {\bf magnetic fields} (or the abelian curvature) for the Ginzburg Landau equations in the strongly repulsive limit (see \cite{BeRi} in 2D and \cite{RivU1} in 3D) .\hfill $\Box$
\end{Rm}
\begin{Rm}
\label{rm-FS}
Observe that by taking $\eta\equiv 0$, for any smooth $u$ one has
\[
{\mathcal R}_\tau(u,0):=FS_\tau(u)=\frac{1}{8}\int_{\B^3_1(0)}|du|^2+\tau\,|u^\ast\om|^2\ dx^3\ ,
\]
which is the {\bf Fadeev-Skyrme} Lagrangian of maps from $\B^3$ into $\S^2$.
\end{Rm}
\begin{Rm}
\label{rm-tauinf}
The Lagrangian ${\mathcal R}_\tau(u,\eta)$ can be interpreted  as follows assuming $d\eta=u^\ast\om_{\S^2}$ then ${\mathcal R}_\tau(u,\eta):=\frac{1}{8}\int |du|^2+|\eta|^2$ is the Dirichlet energy of a lift of $u$ by the Hopf fibration (see Lemma~\ref{lm-hopf}).\hfill $\Box$
\end{Rm}

Finally our last main results are about the description of the  asymptotic analysis for minimizers of ${\mathcal R}_\tau$ respectively as $\tau\rightarrow 0$ and as $\tau\rightarrow +\infty$.
\begin{Th}
\label{th-asmatcalR}
Let $\phi\in C^2(\B^3,\S^2)$ having zero degree. Let $\tau_k\rightarrow 0$ and $(u^k,\eta^k)\in C^1(\B^3,\S^2)\times C^1(\Om_1(\B^3))$ be a sequence of smooth minimizers of ${\mathcal R}_{\tau_k}$ under the only constraint $u^k=\phi$. Then modulo extraction of a subsequence, there exists a closed rectifiable 1-dimensional set $K$ such that for any $l\in {\N}$
\[
u^k\rightarrow u^\infty\qquad\mbox{ in }C^l_{loc}(\B_1(0)\setminus K,\S^2)\ .
\]
Moreover $u^\infty$ is a minimizer of the Brezis-Coron-Lieb relaxed energy 
\[
F_\phi(u):=\frac{1}{2}\int_{\B_1(0)}|du|^2\ dx^3+\pi\, L_\phi(u) 
\]
among maps in $W^{1,2}(\B^3,{\S}^2)$ equal to $\phi$ at the boundary. Finally we have
\[
\lim_{k\rightarrow +\infty}{\mathcal R}_{\tau_k}(u^k,\eta^k)=\min\lf\{ \frac{1}{8}\,\int_{\B_1(0)}|du|^2\ dx^3+\pi\, L_\phi(u) \  ;\  u\in W_\phi^{1,2}(\B^3,\S^2)\rg\}\ .
\]\hfill $\Box$
\end{Th}
As $\tau\rightarrow +\infty$ the minimal energy ${\mathcal R}(u,\eta)$ instead is converging to the minimal Dirichlet energy among any ${\S}^3$ valued extension of any lift of $\phi$. 
\begin{Th}
\label{th-hopfcv}
Let $\phi\in C^2(\B^3,\S^2)$ having zero degree. Let $\tau_k\rightarrow +\infty$ and $(u^k,\eta^k)\in C^1(\B^3,\S^2)\times C^1(\Om_1(\B^3))$ be a sequence of smooth minimizers of ${\mathcal R}_{\tau_k}$ under the only constraint $u^k=\phi$. Then, modulo extraction of a subsequence $u^k$ and $\eta^k$ strongly converge in any $C^l$ norm and
\[
\lim_{k\rightarrow +\infty}{\mathcal R}_{\tau_k}(u^k,\eta^k)=\min\lf\{  \frac{1}{2}\int_{\B^3}|dg|^2\ dx^3\ ;\ \frak{h}\circ g=\phi\quad\mbox{ on }\p\B^3_1(0)\rg\}
\]
where $\frak{h}$ is the {\bf Hopf Fibration} from $\S^3$ into $\S^2$.\hfill $\Box$
\end{Th}

\noindent{\bf Open Questions} In the course of the paper we are meeting several open problems that we are now listing.

\begin{Ope}
\label{open-crit} Extend the above asymptotic analysis for minimizers to general critical points either for $CYMH$ or for ${\mathcal R}_\tau$. In particular, for stationary critical points of $CYMH^{\la,\mu}_\ep$ as $\ep\rightarrow 0$ and $\mu$ large enough (depending on $\phi$) it could be useful to extend the $\delta-$regularity theorem~\ref{lm-deltareg} we are proving for minimizers to {\bf general stationary critical points}. Similarly extend the $\delta-$regularity theorem~\ref{th-delregt} for smooth (or even stationary) critical points of ${\mathcal R}_\tau$. This would probably enable to implement the scheme proposed in \cite{Riv2} for proving the optimality of the Hopf Map from $\S^3$ into $\S^2$ for the 3-energy as conjectured by the author in \cite{Riv}. While $CYMH$ itself is not satisfying the requirements for applying Palais-Smale deformation theory, the  following \underbar{linear} perturbation 
\[
\frac{1}{2}\int_{\B_1(0)}|A|^2+\sigma^2\,|\nabla A|^2 \ dx^3 +\mu\,YMH_\ep^\la(u,A)
\]
where we are adding to $CYMH$ the Dirichlet energy of the connection preceded by the viscous parameter $\sigma$ fulfils the requirements for calling upon this theory and the viscosity method introduced by the author originally in \cite{RivV} could be applied to produce solutions to the minmax problems proposed in \cite{Riv2} with Morse Index control and where the topological singularities are penalised (unlike what is happening with the Ginzburg-Landau perturbation).
  \hfill $\Box$
\end{Ope}

\begin{Ope}
\label{open-gamlim} Instead of considering either minimizers or critical points it would be interesting to study the {\bf Gamma-limits} of the Coulomb-Yang-Mills-Higgs Lagrangian or the ${\mathcal R}_\tau$ Energy in the various regimes we are considering in the present work.\hfill $\Box$
\end{Ope}

\begin{Ope}
\label{open-force}
It could be interesting to study the asymptotics of the minimizers of $CYMH_\ep^{1,1}$ energy to which an ``external Field'' is applied. More precisely a fixed real valued 2-form $H_{ext}$ being fixed one could add to $CYMH$ the Gauge invariant term
\[
\int_{\B^3}|F_A\cdot u-H_{ext}|^2\ dx^3
\]
It could have the effect of forcing the level set of $u/|u|$ (the vortex filaments), that is $u^{-1}(p)$ for $p\in{\S^2}$ to ``adopt'' the shape imposed by $H$. In the Abelian case and in the strongly repulsive regime, a regime which differs from the self-dual regime\footnote{ See \cite{RivB} for comments on these different regimes. } we are considering here, there has been an intensive study of the effect of imposed external magnetic fields since the breakthrough works by S.Serfaty and  E.Sandier and S.Serfaty (\cite{Ser1}, \cite{Ser2}, \cite{Ser3}, \cite{SaSe})
\end{Ope}

\begin{Ope} 
\label{open-regrel} Does a careful asymptotic of ${\mathcal R}_\tau$ minimizers giving any improvement in the bound of the singular set of the Minimizers of the Brezis-Coron-Lieb Relaxed Energy obtained in this way ? In particular can the limit $u^\infty$ in
 theorem~\ref{th-asmatcalR} have topological singularities or is it approximable by smooth maps ? Studying carefully sequences of smooth minimizers to ${\mathcal R}_\tau$ as $\tau\rightarrow 0$ is giving a specific and concrete strategy for addressing the  list of open problems posed by H.Brezis in \cite{Brez} section 4.  \hfill $\Box$
\end{Ope}

\begin{Ope}
\label{open-appFS}  Another asymptotical analysis could be interesting to study : it consists in penalising the longitudinal part of the connection and consider
\[
CYMH_{\ep,\sigma}^{\la,\mu}(u,A):=\frac{1}{2}\int_{\B_1(0)}|A|^2+\frac{1}{\sigma^2}|A\cdot u|^2 \ dx^3 +\mu\,YMH_\ep^\la(u,A)
\]
Considering first the limit $\ep\rightarrow 0$ and $\mu=1/\ep$ where $\sigma$ is fixed in the first place, following word by word the arguments of section IV, one obtains a smooth pair $(u,\eta)$ minimizing
\[
{\mathcal R}^\sigma(u,\eta):=\frac{1}{8}\int_{\B^3_1(0)}|du|^2+\lf(1+\frac{1}{\sigma^2}\rg)\,|\eta|^2+|d\eta-u^\ast\om_{\S^2}|^2\ dx^3\ .
\]
Then it is expected that by passing to the limit $\sigma\rightarrow 0$ one should obtain a minimizer of the relaxed Faddeev-Skyrme energy
\[
\overline{FS}(u):=\inf\lf\{\frac{1}{8}\int_{\B^3_1(0)}|du|^2+\frac{1}{4}\,|du\wedge du|^2\ dx^3\ ;\, u\in C^1(\B^3,\S^2)\ ,\ u=\phi\quad\mbox{ on }\p\B^3_1(0)\rg\}\ .
\]
\hfill $\Box$
\end{Ope}
In \cite{Est} M.Esteban has introduced the class (P) of maps $u\in W^{1,2}(\B^3,\S^3)$ such that there exists $u_k\in C^1(\B^3,S^3)$ such that
\[
\nabla u_k\quad\longrightarrow\quad\nabla u \qquad\mbox{ and }\qquad u_k^\ast\om \quad\longrightarrow \quad  u^\ast\om\qquad\mbox{ strongly in }L^2(\B^3) \
\]
for any choice of smooth $2-$forms $\om$ in ${\R}^4$. While the characterisation of this class (P) remains up to now a largely open problem, the asymptotic analysis of the present work is motivating a slightly ``milder'' question to solve.
\begin{Ope}
\label{open-appFS-2}
Let $u\in W^{1,2}(\B^3,\S^2)$ such that $u^\ast\om_{\S^2}\in L^2(\Om_2(\B^3))$ what is the obstruction for the existence of $u^k\in C^1(\B^3,\S^2)$ such that
\[
u^k\ \rightharpoonup u\qquad\mbox{ weakly in }W^{1,2}(\B^3,\S^2)
\]
and
\[
\limsup_{k\rightarrow+\infty}\int_{\B^3_1(0)}|du^k|^2+\frac{1}{4}\,|du^k\wedge du^k|^2\ dx^3<+\infty\ ?
\]
The weak limit of the minimizers of ${\mathcal R}^\sigma$ are expected to be in this class.
\hfill $\Box$
\end{Ope}
A more accessible question but related to the previous one is the following
\begin{Ope}
\label{ope-lift}
Let $u\in W^{1,2}(\B_1^3(0),\S^2)$ such that there exists $\eta\in L^2(\Om_1(\B_1^3(0))$ with
\[
d\eta=u^\ast\om
\]
Does there exists $u_k\in C^\infty(\B^3_1,\S^2)$ and $\eta_k\in C^\infty(\Om_1(\B^3))$ such that $d\eta_k=u_k^\ast\om$
\[
u_k\ \longrightarrow u\quad\mbox{ strongly in } W^{1,2}(\B^3(0),\S^2)\qquad\mbox{ and }\qquad\eta_k\longrightarrow \eta\quad\mbox{ strongly in } L^2(\Om_1(\B_1^3(0))\ .
\]
\hfill $\Box$
\end{Ope}
The above question is motivated by lemma~\ref{lm-hopf} below which itself leads naturally to the following open problem.
\begin{Ope}
\label{ope-lift0}
Let $u\in W^{1,2}(\B^3,\S^2)$ such that there exists $\eta\in L^2(\Om_1(\B_1^3(0))$ with
\[
d\eta=u^\ast\om
\]
Does there exists a lift $\check{u}\in W^{1,2}(\B^3_1(0),\S^3)$ such that $u$ is the composition of $\check{u}$ with the Hopf fibration ? 
\hfill $\Box$
\end{Ope}
Observe that in two dimension, for any map $u\in W^{1,2}({\B^2},\S^2)$, there exists $\eta\in L^2(\Om_1(\B^2))$ such that $d\eta=u^\ast\om_{\S^2}$ (\cite{Hel} Theorem 5.2.1). However no  control of the $L^2$ norm of $\eta$ in term of the $L^2$ norm of $du$ has to be expected\footnote{Nevertheless, obviously, for any $u\in W^{1,2}({\B^2},\S^2)$ there exists $\eta\in L^{2,\infty}(\Om_1(\B^2))$ such that
\[
\|\eta\|_{L^{2,\infty}(\B^2)}\le C\, \|d u\|_{L^2(\B^2)}\ .
\]
} unless it is below some threshold (The abelian counterpart in 2-dimension of Proposition III.1  \cite{Riv-Tia} gives a counter-example). 
\bigskip

\noindent{\bf Organisation of the paper.} { In section II we introduce the notation and establish some fundamental computations regarding the $CYMH$ Lagrangian. Section III is devoted to the proof of the regularity of $CYMH$ minimizers. In section IV we perform the analysis $\ep\rightarrow 0$ for $\mu$ and $\la$ fixed
(i.e. proof of theorem~\ref{th-clearout} and theorem~\ref{th-cons-comp}). In section V we study the asymptotic of $CYMH$ minimizers in the limit $\ep\rightarrow 0$ for $\mu=\tau/\ep$ and  both $\tau$ and $\la$ are fixed (proof of theorem~\ref{th-harmflu}). In section VI we study the asymptotic of the  Lagrangian ${\mathcal R}_\tau$ as $\tau\rightarrow 0$ (proof of theorem~\ref{th-asmatcalR}) and $\tau\rightarrow+\infty$ (proof of theorem~\ref{th-hopfcv}). Finally the appendix is devoted to the construction of smooth harmonic map extensions under small energy assumptions using a continuity method ``\`a la Uhlenbeck''.}

\bigskip

\noindent{\bf General comments} { Gauge symmetry breaking appears in physics in the so called {\bf Higgs Mechanism} for the standard model in the study of fluctuations around vacuum (when the vacuum expectation is going to infinity which consequently breaks the symmetry). The study of this asymptotic is generating a {\it mass term} for the field $A$ which mathematically materialises as a contribution
\[
\frac{1}{2}\,\int m^2\, |A|^2\ dx^3
\]
to the Lagrangian. This contribution  is interpreted in physics as  a massive Boson and the coefficient $m$ is the mass of this Gauge Field. In the expression (\ref{mathR}) for the ${\mathcal R}_\tau$ asymptotic Lagrangian the mass corresponds to $1/\sqrt{\tau}$.
We then prove in theorem~\ref{th-asmatcalR} that in the infinite Gauge Field mass limit ($\tau \rightarrow 0$) the Lagrangian converges to the Brezis-Coron-Lieb relaxed energy while in the zero mass limit it converges to the minimal Dirichlet energy among all Hopf-Fibration lifts of the Higgs field.

Finally let just mention that in Abelian Gauge Theory, V.A. Rubakov and A.N. Tavkhelidze introduced a non gauge invariant Chern Simon Theory (i.e. Gauge Anomaly) involving $L^2$ norm of the gauge Field $A$ in the static energy functional \cite{RuTa} (see also \cite{ScSh}).
}
\section{Preliminaries}
\reset
\subsection{The Coulomb Yang-Mills Higgs Lagrangian}
\subsubsection{ Fundamental Identities : Jacobi, Leibnitz, Bianchi...}
We shall be using classical differential forms notations. We denote $\Om^k(\B^3)$ the smooth $k$-forms on $\B^3$ and by $\Om_0^k(\B^3)$ the compactly supported smooth $k$-forms on $\B^3$. For any classical Banach function space $E$ we denote $E(\Om^k(\B^3))$ the space of $k$-forms with coefficients in $E$.
If $V$ is a finite dimensional vector space we denote by $E(\Om^k(\B^3)\otimes V)$ the space of $k$-forms with coefficients in $E$ and taking values in $V$.
\[
\lf\{
\begin{array}{l}
\ds\forall v\in L^2(\B^3)\qquad d v:=\sum_{l=1}^3\p_{x_l}v\ d{x_l}\\[5mm]
\ds\forall B\in L^2(\Om^1(\B^3))\qquad d B:=\sum_{l,m=1}^3\p_{x_l}B_m\ d{x_l}\wedge dx_m\\[5mm]
\ds\forall F\in L^2(\Om^2(\B^3)\qquad d F:=\sum_{k=1}^3\sum_{l<m=1}^3\p_{x_k}F_{lm}\ dx_k\wedge d{x_l}\wedge dx_m
\end{array}
\rg.
\]
We recall the Hodge operator (using ${\Z}_3$ indexation notation so that  $i+2=i-1$)
\[
\ast\, dx_i=dx_{i+1}\wedge dx_{i-1}\ ,\quad\ast \,dx_{i+1}\wedge dx_{i-1}=dx_i\quad\ast 1=dx_1\wedge dx_2\wedge dx_3=dx^3
\]
so that
\[
\ast\ast=id_{\Om^k(\B^3)}\quad\mbox{ on }\quad\Om^k(\B^3)\qquad\mbox{ for }k=0,1,2,3
\]
We denote by $\cdot$ the standard scalar product on $k$-forms on $\B^3$. The co-differential operator $d^\ast$ on $\Om^k(\B^3)$ for $k\ge 1$ is defined by
\[
\forall\, \al\in \Om_0^{k-1}(\B^3)\ ,\ \forall\, \beta\in \Om_0^k(\B^3)\qquad\int_{\B^3}d\al\cdot\beta\ dx^3=\int_{\B^3}\al\cdot d^\ast\beta\ dx^3
\]
This gives
\[
d^\ast=(-1)^k\,\ast\,d\ast
\]
Indeed
\[
\int_{\B^3}\al\cdot d^\ast\beta\ dx^3=(-1)^k\int_{\B^3}\al\wedge d\ast\beta=-\int_{\B^3}d(\al\wedge \ast\beta)+\int_{\B^3}d\al\wedge \ast\beta=\int_{\B^3}d\al\cdot\beta\ dx^3\ .
\]
We have in particular
\[
\begin{array}{l}
\ds\forall B\in L^2(\Om^1(\B^3))\qquad d^\ast B:=-\sum_{l=1}^3\p_{x_l}B_l\\[5mm]
\ds\forall F\in L^2(\Om^2(\B^3))\qquad d^\ast F:=\sum_{j=1}^3 \ast d\lf( F_{j+1,j-1}\ dx_j   \rg)\\[5mm]
\ds\qquad=\sum_{j=1}^3 \p_{x_{j+1}}F_{j+1,j-1}\, \ast dx_{j+1}\wedge dx_{j}+\p_{x_{j-1}}F_{j+1,j-1}\, \ast dx_{j-1}\wedge dx_{j}\\[5mm]
\ds\qquad=-\sum_{l,m=1}^3 \p_{x_l}F_{lm}\ dx_m
\end{array}
\]
Let $A$ be an $L^2$ connection 1-form on ${\mathbb B}^3$ where $\B^3$ is denoting the unit ball in ${\R}^3$.
$$A=\sum_{l=1}^3A_l \,dx_l\quad\mbox{ where }\quad A_l\in \frak{su}(2)\simeq \Im m(\mathbb H)$$
Following the notation we just introduced, the space of $L^q$ k-forms on $\B^3$ taking values into $\frak{su}(2)$ is denoted $$L^q(\Omega^k(\B^3)\otimes \frak{su}(2))\ .$$  On $\frak{su}(2)\simeq \Im m(\mathbb H)$ we define the bracket operation\footnote{While considering a vector $x\in \frak{su}(2)$ we sometimes use the boldface notation$\bf{x}$ to stress the fact that we are using its quaternionic counterpart in $ \Im m(\mathbb H)$ and the multiplication properties attached to it.}
\[
[\bf{i},\bf{j}]= \bf{i}\,\bf{j}-\bf{j}\,\bf{i}=2\,\bf{k}\quad,\quad[\bf{j},\bf{k}]= \bf{j}\,\bf{k}-\bf{k}\,\bf{j}=2\,\bf{i}\quad\mbox{ and }\quad [\bf{k},\bf{i}]= \bf{k}\,\bf{i}-\bf{i}\,\bf{k}=2\,\bf{j}\ .
\]
Let $g\in SU(2)$ and $(u,v)\in \frak{su}(2)$. We denote $u^g:=g^{-1}\,u\,g$ and $v^g:=g^{-1}\,v\,g$ and we have
\[
[u^g,v^g]=[g^{-1}\,u\,g,g^{-1}\,u\,g]={\bf g}^{-1}{\bf u}{\bf g}\, {\bf g}^{-1}{\bf v}{\bf g}-{\bf g}^{-1}{\bf u}{\bf g}\,{\bf g}^{-1}{\bf v}{\bf g}=g^{-1}[u,v]g=[u,v]^g
\]
The scalar product in $ \frak{su}(2)$ is given by $y\cdot z=\Re\lf({\bf y}^\ast {\bf z}\rg)=-\Re\lf({\bf y} {\bf z}\rg)=-\Re\lf({\bf z} {\bf y}\rg)$ where we identify elements in $\frak{su}(2)$ with the corresponding imaginary quaternions (recall that for imaginary quaternions there holds ${\bf y}^\ast=-{\bf y}$). We also sometimes denote this scalar product $<y,z>$. We have in particular
\be
\label{bra-sca}
\begin{array}{l}
\ds<a,[b,c]>=\Re\lf({\bf a}^\ast ({\bf b}\,{\bf c}- {\bf c}\,{\bf b}\rg)=-\Re\lf({\bf a}({\bf b}\,{\bf c}- {\bf c}\,{\bf b})\rg)\\[5mm]
\ds=\Re\lf({\bf a}^\ast {\bf c}^\ast\,{\bf b}-{\bf c}^\ast\,{\bf a}^\ast {\bf b}\rg)=\Re\lf(({\bf c}\,{\bf a}-{\bf a}\,{\bf c})^\ast\,{\bf b}\rg)\\[5mm]
\ds=<[c,a],b>=\cdots=<[a,b],c>
\end{array}
\ee
We take this occasion to recall the {\it Jacobi identity} we are going to use at several occasions : 
\be
\label{jacob}
\forall\, a,b,c\in \frak{su}(2)\qquad [a,[b,c]]+[c,[a,b]]+[b,[c,a]]=0\ .
\ee
For any $u\in W^{1,4/3}(\B^3,\frak{su}(2))$ we denote respectively
\[
d_A u=\sum_{l=1}^3\lf(\p_{x_l}u+[A_l,u]\rg)\ d{x_l}\quad\mbox{ and }\quad\nabla_A u=\sum_{l=1}^3\lf(\p_{x_l}u+[A_l,u]\rg)\ \p_{x_l}\
\]
the covariant derivative of $u$ with respect to $A$ (either as a 1-form or as  a vector-field). We shall denote
\[
(\nabla_A)_lu:=\p_{x_l}u+[A_l,u]
\]
with this notation we then have $\nabla_A u=\sum_{l=1}^3(\nabla_A)_lu\ \p_{x_l}$.

\medskip
There will be two important identities we shall which are {\it Leibnitz Rules} for covariant derivatives with respect to scalar products and Lie Brackets. For any pair $(u,v)$ in $L^4\cap W^{1,4/3}(\B^3,\frak{su}(2))$ and $A\in L^2(\Omega_1(\B^3)\otimes \frak{su}(2))$, thanks to (\ref{bra-sca})
\[
\begin{array}{l}
\ds\p_{x_l}<u,v>=<\p_{x_l}u,v>+<u,\p_{x_l}v>\\[5mm]
\ds\quad=<\p_{x_l}u+[A_l,u],v>+<u,\p_{x_l}v+[A_l,v]>-<[A_l,u],v>-<u,[A_l,v]>\\[5mm]
\ds\quad=<\p_{x_l}u+[A_l,u],v>+<u,\p_{x_l}v+[A_l,v]>\ .
\end{array}
\]
Hence
\be
\label{leib-sc}
\nabla<u,v>=<\nabla_Au,v>+<u,\nabla_Av>\ .
\ee
Using Jacobi identity we have also
\be
\label{com-cov}
\begin{array}{rl}
\ds(\nabla_A)_n[v,w]&=\p_{x_n}([v,w])+[A_n,[v,w]]\\[5mm]
\ds\quad&=[\p_{x_n}v,w]+[v,\p_{x_n}w]+[[A_n,v],w]+[v,[A_n,w]]\\[5mm]
\ds\quad&=[(\nabla_A)_nv,w]+[v,(\nabla_A)_nw]
\end{array}
\ee
The curvature of $A$ on $\B^3$ is the following $\frak{su}(2)$ valued 2-form
\[
\forall\, X,Y\in T_{x}\B^3\qquad F_A(X,Y):=dA(X,Y)+[A(X),A(Y)]
\]
This implies in particular
\[
F_A=\sum_{l<m} (\p_{x_l}A_m-\p_{x_m}A_l+[A_l,A_m])\ dx_l\wedge dx_m
\]
For any pair of $1$-forms $(A,B)$ in $\Omega_1(\B^3)\otimes \frak{su}(2)$ we define their bracket which is an element in $\Omega_2\B^3\otimes \frak{su}(2)$ as follows
\[
\forall\, X,Y\in T_{x}\B^3\qquad [A\,\wedge\, B](X,Y):=[A(X), B(Y)]-[A(Y),B(X)]
\]
With this notation we then have\footnote{The product $\frac{1}{2}[A\wedge A]$ is often denoted $A\wedge A$ in the literature. We shall adopt this notation as well.}
\[
F_A=dA+\frac{1}{2}[A\wedge A]
\]
Observe that we have using Jacobi identity for any $v\in L^4\cap W^{1,4/3}(\B^3,\frak{su}(2))$ and $A\in L^2(\Omega_1(\B^3)\otimes \frak{su}(2))$ and $l,m\in\{1,2,3\}$
\[
\begin{array}{l}
\ds(\nabla_A)_l((\nabla_A)_m v)=\p_{x_l}(\p_{x_m} v+[A_m,v])+[A_l,\p_{x_m} v+[A_m,v]]\\[5mm]
\ds=\p_{x_m}(\p_{x_l} v+[A_l,v])+[A_m,\p_{x_l} v+[A_l,v]]\\[5mm]
\ds+\p_{x_l}([A_m,v])-\p_{x_m}([A_l,v])+[A_l,\p_{x_m} v+[A_m,v]]-[A_m,\p_{x_l} v+[A_l,v]]\\[5mm]
\ds=(\nabla_A)_m((\nabla_A)_l v)+[\p_{x_l}A_m-\p_{x_m}A_l+[A_l,A_m], v]=(\nabla_A)_m((\nabla_A)_l v)+[F_{lm},v]
\end{array}
\]
Hence we have
\be
\label{cov-ex}
\forall\,l,m\in\{1,2,3\}\qquad(\nabla_A)_l((\nabla_A)_m v)-(\nabla_A)_m((\nabla_A)_l v)=[F_{lm},v]
\ee
We have also the following useful identity
\be
\label{curcovder}
\begin{array}{rl}
\ds F_{lm}&\ds=\p_{x_l}A_m-\p_{x_m}A_l+[A_l,A_m]=   (\nabla_A)_lA_m-[A_l,A_m]- (\nabla_A)_mA_l+[A_m,A_l]+[A_l,A_m]\\[5mm]
&\ds= (\nabla_A)_l A_m-(\nabla_A)_m A_l+ [A_m,A_l]
\end{array}
\ee

\medskip

We denote by $d_A$ the covariant exterior derivative on $0,1,2$ forms as follows\footnote{It is important to observe that due to the $1/2$ factor we have
\[
d_AA=dA+[A\wedge A]\ne F_A=dA+\frac{1}{2}[A\wedge A]\ .
\]
}
\[
\lf\{
\begin{array}{l}
\ds\forall v\in L^2(\B^3, \frak{su}(2))\qquad d_A v:=\sum_{l=1}^3\lf(\p_{x_l}v+[A_l,v]\rg)\ d{x_l}\\[5mm]
\ds\forall B\in L^2(\Om^1(\B^3)\otimes \frak{su}(2))\qquad d_A B:=\sum_{l,m=1}^3\lf(\p_{x_l}B_m+[A_l,B_m]\rg)\ d{x_l}\wedge dx_m\\[5mm]
\ds\forall F\in L^2(\Om^2(\B^3)\otimes \frak{su}(2))\qquad d_A F:=\sum_{k=1}^3\sum_{l<m=1}^3\lf(\p_{x_k}F_{lm}+[A_k,B_{lm}]\rg)\ dx_k\wedge d{x_l}\wedge dx_m
\end{array}
\rg.
\]
We have in particular using again {\it Jacobi identity}
\[
\begin{array}{l}
\ds (d_AF_A)_{123}=\p_{x_1}(F_A)_{23}+[A_1,(F_A)_{23}]+\p_{x_2}(F_A)_{31}+[A_2,(F_A)_{31}]+\p_{x_3}(F_A)_{12}+[A_3,(F_A)_{12}]\\[5mm]
\ds =\p_{x_1}\lf(\p_{x_2}A_3-\p_{x_3}A_2+[A_2,A_3]\rg)+\p_{x_2}\lf(\p_{x_3}A_1-\p_{x_1}A_3+[A_3,A_1]\rg)\\[5mm]
\ds\ +\p_{x_3}\lf(\p_{x_1}A_2-\p_{x_2}A_1+[A_1,A_2]\rg)+[A_1,\p_{x_2}A_3-\p_{x_3}A_2+[A_2,A_3]]\\[5mm]
+[A_2,\p_{x_3}A_1-\p_{x_1}A_3+[A_3,A_1]]+[A_3, \p_{x_1}A_2-\p_{x_2}A_1+[A_1,A_2]]
\ds\ =...=0
\end{array}
\]
Hence we have
\be
\label{bian}
d_AF_A=0\ ,
\ee
which is the famous {\it Bianchi identity}.

\medskip 

Let $u$ be a map from $\B^3$ into $\frak{su}(2)$. We have
\[
\begin{array}{l}
\ds d_A(d_Au)=d_A(\sum_{l=1}^3\p_{x_l} u+[A_l,u]\ dx_l)=\sum_{l,m=1}^3\lf(\p_{x_m}(\p_{x_l} u+[A_l,u])+[A_m,\p_{x_l} u+[A_l,u]]\rg)\ dx_m\wedge dx_l\\[5mm]
\ds=\sum_{l,m=1}^3\lf([\p_{x_m}A_l,u]+[A_l,\p_{x_m}u]+[A_m,\p_{x_l} u]+[A_m,[A_l,u]]\rg)\ dx_m\wedge dx_l\\[5mm]
\ds=\sum_{l<m}[\p_{x_m}A_l-\p_{x_l}A_m,u]\ dx_m\wedge dx_l+\sum_{l<m}[A_m,[A_l,u]]-[A_l,[A_m,u]]\ dx_m\wedge dx_l\
\end{array}
\]
Observe that, thanks to Jacobi
\[
[A_m,[A_l,u]]-[A_l,[A_m,u]]=[A_m,[A_l,u]]+[A_l,[u,A_m]]=-[u,[A_m,A_l]]
\]
This finally gives that the curvature is the defect for $d_A\circ d_A$ to be zero
\be
\label{dAsqu}
\begin{array}{l}
\ds d_A(d_Au)=[F_A,u]
\end{array}
\ee

\medskip

Finally we define the  covariant exterior co-differential with respect to the connection 1-form $A$ on $\Om^k(\B^3\otimes \frak{su}(2))$ for $k\ge 1$ by
\[
\forall\, \eta\in \Om_0^{k-1}(\B^3\otimes \frak{su}(2)))\ ,\ \forall\, \zeta\in \Om_0^k(\B^3\otimes \frak{su}(2)))\qquad\int_{\B^3}d_A\eta\cdot\zeta\ dx^3=\int_{\B^3}\eta\cdot d_A^\ast\zeta\ dx^3
\]
This gives first for $k=1$ for any choice of $v\in C_0^\infty(\B^3,\frak{su}(2))$ and $B\in \Om_0^1(\B^3)\otimes\frak{su}(2)$
\[
\begin{array}{l}
\ds\int_{\B^3}v\cdot d_A^\ast B\ dx^3=\sum_{l=1}^3\int_{\B^3}\lf(\p_{x_l}v+[A_l,v]\rg)\cdot B_l\ dx^3\\[5mm]
\ds=\sum_{l=1}^3\int_{\B^3}\lf(\p_{x_l}v+[A_l,v]\rg)\cdot B_l\ dx^3\\[5mm]
\ds=\sum_{l=1}^3\int_{\B^3}v\cdot (-\p_{x_l}B_l -[A_l,B_l] )\ dx^3
\end{array}
\]
This gives first for $k=2$ for any choice of $B\in \Om^1_0(\B^3)\otimes\frak{su}(2)$ and $F\in \Om^2_0(\B^3)\otimes\frak{su}(2)$ 
\[
\begin{array}{l}
\ds\int_{\B^3}B\cdot d_A^\ast F\ dx^3= \int_{\B^3}d_AB\cdot  F\ dx^3=\int_{\B^3} \sum_{l,m=1}^3\lf(\p_{x_l}B_m+[A_l,B_m]\rg) \cdot F_{lm} \ dx^3\\[5mm]
\ds=\int_{\B^3} \sum_{l,m=1}^3B_m\cdot \lf(-\p_{x_l}F_{lm}-[A_l,F_{lm}]\rg)  \ dx^3
\end{array}
\]
Finally for $k=3$ we have for any choice of $F\in \Om^2_0(\B^3)\otimes\frak{su}(2)$ and $\om=v\, dx_1\wedge dx_2\wedge dx_3\in \Om^3_0(\B^3)\otimes\frak{su}(2)$
\[
\begin{array}{l}
\ds\int_{\B^3}F\cdot d_A^\ast \om\ dx^3= \int_{\B^3}d_AF\cdot  \om\ dx^3\\[5mm]
\ds =\int_{\B^3}\lf(\p_{x_1}F_{23}+[A_1,F_{23}]+\p_{x_2}F_{31}+[A_2,F_{31}]+\p_{x_3}F_{12}+[A_3,F_{12}]\rg)\cdot v\ dx^3\\[5mm]
\ds=\int_{\B^3} F_{23}\cdot(-\p_{x_1}v-[A_1,v])+F_{31}\cdot(-\p_{x_2}v - [ A_2,v])+F_{12}\cdot(-\p_{x_3}v-[A_3,v]  )\ dx^3
\end{array}
\]
To summarise we have for any connection 1-form $A \in L^2(\Om^1(\B^3)\otimes \frak{su}(2))$
\[
\lf\{
\begin{array}{l}
\ds\forall B\in L^2(\Om^1(\B^3)\otimes \frak{su}(2))\qquad d^\ast_A B:=\sum_{l,m=1}^3\lf(-\p_{x_l}B_l-[A_l,B_l]\rg)\\[5mm]
\ds\forall F\in L^2(\Om^2(\B^3)\otimes \frak{su}(2))\qquad d^\ast_A F:=\sum_{l,m=1}^3\lf(-\p_{x_l}F_{lm}-[A_l,F_{lm}]\rg)\ dx_m\\[5mm]
\ds\forall \om=\ast\ v\in L^2(\Om^3(\B^3)\otimes \frak{su}(2))\qquad d^\ast_A \om:=\sum_{j=1}^3\lf(-\p_{x_j}v-[A_j,v]\rg)\ dx_{j+1}\wedge dx_{j-1}
\end{array}
\rg.
\]
\subsubsection{Longitudinal and Transverse Components} A unit vector $e\in \S^2\subset\frak{su}(2)$  being given then for any element $B\in \frak{su}(2)$ we define the Longitudinal and transverse projections of $B$ along $e$ in the following way
\[
B^L:=B\cdot e\ e\qquad\mbox{ and }\qquad B^T:=B-B\cdot e\ e\ .
\]
Using the Lie Bracket we observe that
\be
\label{LoTra}
B^T=-\frac{1}{4}[e,[e,B]]\ .
\ee
\subsubsection{Gauge Transformations}
For any choice of $A\in L^q(\Omega_1(\B^3)\otimes \frak{su}(2))$ and $g\in W^{1,q}(\B^3,SU(2))$ we adopt the following notation
for  the change of gauge
\[
A^g:=g^{-1}dg+g^{-1}\, A\,g\ .
\]
We say that $A$ and $A^g$ are Gauge equivalent.
Let $A\in L^2(\Omega_1\B^3\otimes \frak{su}(2))$ and $g\in W^{1,2}(\B^3,SU(2))$ a classical computation  gives
\[
\begin{array}{l}
\ds F_{A^g}=d\lf(g^{-1}dg+g^{-1}\, A\,g\rg)+\frac{1}{2}[g^{-1}dg+g^{-1}\, A\,g,g^{-1}dg+g^{-1}\, A\,g]\\[5mm]

\ds=dg^{-1}\wedge dg+dg^{-1}g\wedge g^{-1}Ag+g^{-1}\, dA\,g-g^{-1}\,A\,g\wedge g^{-1}dg+g^{-1}\,dg\wedge g^{-1}\,dg\\[5mm]
\ds+\frac{1}{2}[g^{-1}dg\wedge g^{-1}\, A\,g]+\frac{1}{2}[g^{-1}\, A\,g\wedge g^{-1}dg]+\frac{1}{2}\,g^{-1}[A,A]\,g

\end{array}
\]
Observe that
\[
dg^{-1}\wedge dg=dg^{-1}g\wedge g^{-1}dg=-g^{-1}dg\wedge g^{-1}dg
\]
we have also
\[
\begin{array}{l}
\ds dg^{-1}g\wedge g^{-1}Ag(X,Y)-g^{-1}\,A\,g\wedge g^{-1}dg(X,Y)\\[5mm]
\ds=-g^{-1}dg(X)\, g^{-1}A(Y)g+g^{-1}dg(Y)\,  g^{-1}A(X)g\\[5mm]
\ds -g^{-1}\,A(X)\,g\, g^{-1}dg(Y)+g^{-1}\,A(Y)\,g\, g^{-1}dg(X)
\end{array}
\]
and
\[
\begin{array}{l}
\ds\frac{1}{2}[g^{-1}dg\wedge g^{-1}\, A\,g](X,Y)+\frac{1}{2}[g^{-1}\, A\,g\wedge g^{-1}dg](X,Y)\\[5mm]
\ds=\frac{1}{2}[g^{-1}dg(X), g^{-1}\, A(Y)\,g]-\frac{1}{2}[g^{-1}dg(Y), g^{-1}\, A(X)\,g]\\[5mm]
\ds+\frac{1}{2}[g^{-1}\, A(X)\,g, g^{-1}dg(Y)]-\frac{1}{2}[g^{-1}\, A(Y)\,g, g^{-1}dg(X)]\\[5mm]
\ds=g^{-1}dg(X) g^{-1}\, A(Y)\,g-g^{-1}\, A(Y)\,g\,g^{-1}dg(X)+g^{-1}\, A(X)\,g\, g^{-1}dg(Y)-g^{-1}dg(Y)\,g^{-1}\, A(X)\,g
\end{array}
\]
Hence finally we have
\be
\label{gau-cu}
F_{A^g}=g^{-1}\,F_A \,g\ .
\ee
For $u\in W^{1,4/3}(\B^3,\mathfrak{ su}(2))$ and $g\in W^{1,2}(\B^3,SU(2))$ we denote $u^g=g^{-1}\,u\,g$ and we have
\[
\begin{array}{l}
\ds d_{A^g}u^g=\sum_{l=1}^3(\p_{x_l}(g^{-1}u g)+[(A^g)_l,g^{-1}u g])\ dx_l\\[5mm]
\ds=\sum_{l=1}^3\p_{x_l}g^{-1}\,ug+g^{-1}\,\p_{x_l}u\,g+g^{-1}\,u\,\p_{x_l}g+[g^{-1}\p_{x_l}g+g^{-1}A_l g,g^{-1}u g]\ dx_l\\[5mm]
\ds=\sum_{l=1}^3\lf(-[g^{-1}\p_{x_l}g,u^g]+[g^{-1}\p_{x_l}g,u^g]+ g^{-1}(\p_{x_l}u+[A_l,u])\,g\rg)\ dx_l=g^{-1}\,d_Au\,g
\end{array}
\]
Let $F\in L^2(\Om^2(\B^3)\otimes \frak{su}(2))$, we have
\[
\begin{array}{l}
\ds d^\ast_{A^g}(g^{-1}\,F\,g)=\sum_{l,m=1}^3\lf(-\p_{x_l}(g^{-1}\,F_{lm}\,g)-[(A^g)_l,g^{-1}\,F_{lm}\,g]\rg)\ dx_m\\[5mm]
\ds =\sum_{l,m=1}^3\lf(-\p_{x_l}g^{-1}g\,g^{-1}\,F_{lm}\,g-g^{-1}\,\p_{x_l}F_{lm}\,g - g^{-1}\,F_{lm}\,g\, g^{-1}\p_{x_l}g\rg.\\[5mm]
\ds\qquad\lf.-g^{-1}[A_l,F_{lm}]\,g-[g^{-1}\,\p_{x_l}g,g^{-1}\,F_{lm}\,g]\rg)\ dx_m\\[5mm]
\ds=\sum_{l,m=1}^3g^{-1}\lf(-\p_{x_l}F_{lm}-[A_l,F_{lm}]\rg)\,g\ dx_m\ .
\end{array}
\]
Hence
\be
\label{gau-ch}
d^\ast_{A^g}(g^{-1}\,F\,g)=g^{-1}\,d^\ast_AF\,g\ .
\ee
\subsubsection{Uhlenbeck's Coulomb Gauge Extraction}
We shall consider the following space of Sobolev connections for $p\ge 3/2$.
\[
\mathfrak{ a}^p_{\B^3}:=\lf\{
\begin{array}{c}
\ds A\in L^2(\Omega_1\B^3\otimes \frak{su}(2))\ ;\ F_A \in L^p(\Omega^2(\B^3)\otimes \frak{su}(2))\\[5mm]
\ds \forall x_0\in \B^3\ \exists\,r>0\ \mbox{ and }\ g\in W^{1,2}(\B_r(x_0),SU(2))\ \mbox{ s. t. }A^g\in W^{1,p}(\B_r(x_0))
\end{array} \rg\}
\]
We recall an important result of K.Uhlenbeck analytical Gauge theory 
\begin{Th}
\label{th-Uhl}\cite{Uh} Let $A\in \mathfrak{ a}^p_{\B^3}$ for $p\ge 3/2$. There exists $\delta_{p}>0$  and $C_p>0$ such that if
\be
\label{II.0}
\|F_{A}\|_{L^{3/2}(\B^3)}<\delta_p
\ee
then there exists $g\in W^{1,2}(\B^3,SU(2))$ such that
\begin{itemize}
\item[i)]
 \be
 \label{II.1}
 \|A^g\|_{W^{1,p}(\B^3)}\le\, C_p\ \|F_{A}\|_{L^{p}(\B^3)}
\ee
\item[ii)]
\be
\label{II.2}
d^\ast A^g=\sum_{l=1}^3\p_{x_l}A^g_l=0\qquad\mbox{ in }{\mathcal D}'(\B^3)
\ee
\item[iii]
\be
\label{II.3}
A^g(\p_r)=\sum_{l=1}^3 x_l\,A^g_l=0\qquad\mbox{ in }{\mathcal D}'(\p\B^3)
\ee
\end{itemize}
where $A_l^g$ denotes the $l-$th coordinate of $A^g$ in the  canonical basis $(dx_l)_{l=1,2,3}$ :
\[
A^g=\sum_{l=1}^3 A_l^g\ dx_l\ .
\]
\hfill $\Box$
\end{Th}
The condition (\ref{II.2}) for the gauge $A^g$ is called the {\it Coulomb Condition}. In \cite{Riv-Tia} the author proved that the same result holds replacing (\ref{II.0}) by the weaker hypothesis
\[
\|F_{A}\|_{L^{(3/2,\infty)}(\B^3)}<\delta_p\ ,
\]
where $L^{(3/2,\infty)}(\B^3)$ is the weak $L^{3/2}$ norm on $\B^3$. We have
\[
\|F\|_{L^{(3/2,\infty)}(\B^3)}\simeq \sup_{t>0}t\,\lf|\lf\{ x\in \B^3\ ;\ |f(x)|>t\rg\}\rg|^{2/3}
\]
where $|\cdot|$ denotes the Lebesgue measure on $\B^3$.

\subsubsection{The $YMH$ and $CYMH$ Functionals}

We denote
\be
\label{matE}
{\mathcal E}:=\lf\{ 
\begin{array}{c}
\ds (u,A)\in W^{1,3/2}(\B^3,\frak{su}(2))\times \mathfrak{ a}^2_{\B^3}\ \\[5mm]
\mbox{s. t. }d_Au\in L^2(\Omega_1(\B^3)\otimes \frak{su}(2))\
\end{array} \rg\}
\ee
Observe that by Kato inequality one has $|\nabla|u||\le |\nabla_Au|$. Since for any element $(u,A)\in \mathcal E$ we have $\nabla_Au\in L^2({\B}^3)$, by Sobolev inequality there holds $u\in L^6(\B^3_1(0),{\R}^3)$ and then, since $A\in L^2(\Om_1(\B^3)\otimes\frak{su}(2))$ we deduce that $du\in L^{3/2}(\Om_1(\B^3_1(0))\otimes\frak{su}(2))$.

\medskip

For $(\ep,\la,\mu)\in ({\R}_+)^3$ we introduce the ``Coulomb Yang-Mills Higgs'' Lagrangian  
\be
\label{II.4}
CYMH^{\la,\mu}_{\ep}(u,A):=\frac{1}{2}\int_{\B^3}|A|^2\ dx^3+\mu \, YMH^{\la}_{\ep}(u,A)
\ee
and $YMH^{\la}_{\ep}(u,A)$ is the classical Yang-Mills-Higgs Lagrangian in the {\it self-dual scaling} introduced initially in \cite{JT}
\be
\label{II.5}
YMH^{\la}_{\ep}(u,A)=\frac{1}{2}\,\int_{B^3}\frac{|d_A u|^2}{\ep}+\ep\,|F_A|^2\ dx^3+\frac{\la}{2\,\ep^3}(1-|u|^2)^2\ dx^3
\ee
\subsubsection{Density of Smooth Higgs Fields/Connections Forms in ${\mathcal E}$.}
\begin{Prop}
\label{lm-dens}
Let $(u,A)$ be an arbitrary element of ${\mathcal E}$. There exists a sequence of smooth pairs $(u_k,A_k)\in {\mathcal E}$ such that
\be
\label{dens-uA}
\lf\{
\begin{array}{l}
\ds A_k\ \longrightarrow\ A\qquad\mbox{ strongly in }L^2(\Om^1(\B^3)\otimes \frak{su}(2))\\[5mm]
\ds d_{A_k}u_k\ \longrightarrow\ d_{A}u\qquad\mbox{ strongly in }L^{3/2}(\Om^1(\B^3)\otimes \frak{su}(2))\\[5mm]
\ds F_{A_k}\ \longrightarrow\ F_A\qquad\mbox{ strongly in }L^2(\Om^2(\B^3)\otimes \frak{su}(2))\\[5mm]
\ds u_k\ \longrightarrow\ u\qquad\mbox{strongly in }L^6(\B^3,\frak{su}(2))
\end{array}
\rg.
\ee
\hfill$\Box$
\end{Prop}
\noindent{\bf Proof of Proposition~\ref{lm-dens}.} Let $3/2<p<2$ to be fixed later on.
\[
\int_{\B_r(x_0)}|F_A|^{p}\ dx^3\le C\, r^{3 -3p/2}\  \lf(\int_{\B_r(x_0)}|F_A|^{2}\ dx^3\rg)^{p/2}
\]
Then for any $\delta<\delta_{p}$ ( where $\delta_{p}>0$ is given by theorem~\ref{th-Uhl})there exist $r_\delta$ such that
\[
\sup_{x\in \ov{\B}_1^3(0)}\int_{\B_{r_\delta}(x)}|F_A|^{p}\ dx^3\le \delta
\]
We consider a good covering\footnote{A good covering refers to the property according to which intersections of elements of the covering are homeomorphic to balls.} of $ \ov{\B}_1^3(0)$ by ball $(\B_{r_\delta}(x_i))_{i\in \{1\cdots Q \}}$ such that each point of $\ov{\B}_1^3(0)$ is covered at most by a universal number $N$ of balls $(\B_{r_\delta}(x))_{i\in \{1\cdots Q\}}$ where $N$ is independent of $\delta$ and $Q$, . We consider on each ball $\B_{r_\delta}(x_i)$ a controlled Coulomb gauge  $A^{g_i}:=g_i^{-1}\,dg_i+g_i^{-1}\,A\,g_i$ in $W^{1,p}(\B_{r_\delta}(x)$ given by  theorem~\ref{th-Uhl}. The family $(g_i )_{i\in \{1\cdots Q\}}$ generates a family of transition functions $g_{ij}:=g_{i}^{-1}\,g_j\in W^{2,p}(\B_{r_\delta}(x_i)\cap \B_{r_\delta}(x_j))$ satisfying obviously the cocycle condition
\[
\forall\,i,j,l\in\{1\cdots Q\}\qquad g_{ij}\,g_{jl}=g_{il}\qquad\mbox{on }\B_{r_\delta}(x_i)\cap \B_{r_\delta}(x_j)\cap \B_{r_\delta}(x_l)
\]
Using \cite{Sil} theorem 1, up to a slight shrinking of the covering by replacing $r_\delta$ by a strictly smaller radius $r^1_\delta$, there exists $(g^\ep_{ij})_{ij\in\{1\cdots Q\}})$ such that
\[
g^\ep_{ij}\in C^\infty(\B_{r^1_\delta}(x_i)\cap \B_{r^1_\delta}(x_j),SU(2))\qquad\mbox{ moreover }\qquad g^\ep_{ij}\rightarrow g_{ij}\qquad\mbox{ in } W^{2,p}(\B_{r_\delta}(x_i)\cap \B_{r_\delta}(x_j),SU(2))\ .
\]
Using now proposition 3.2 of \cite{Uh} for $\ep>0$ small enough there exists $h^\ep_i  \in W^{2,p}(\B_{r_\delta}(x_i),SU(2))$ such that
\[
g^\ep_{ij}=(h^{\ep}_i)^{-1}\, g_{ij}\,h^{\ep}_j=(h^{\ep}_i)^{-1}\, g_{i}^{-1}\,g_j\,h^{\ep}_j\qquad\mbox{ in }B_{r^1_\delta}(x_i)\cap \B_{r^1_\delta}(x_j)
\]
We fix now $\ep$. Let $k_i^\ep:= \,g_i\,h^{\ep}_i$ on $B_{r^1_\delta}(x_i)$ then we have that 
\[
A^{k_i^\ep}:=(k_i^\ep)^{-1}\,dk_i^\ep+(k_i^\ep)^{-1}\,A\, k_i^\ep\in W^{1,p}(\Om_1(B_{r^1_\delta}(x_i))\otimes \frak{su}(2))
\]
moreover the transition function $k_{ij}^\ep:=(k_i^\ep)^{-1}\,k_j^\ep$ are smooth. Since the bundle these transition functions are defining has $\B^3_1(0)$ as a base it is trivial and there exists
$f_i^\ep \in C^\infty(B_{r^1_\delta}(x_i), SU(2))$ such that 
\[
k_{ij}^\ep= (f_i^\ep)^{-1}\,f_j^\ep
\]
Let $\xi^\ep_i:=k_i^\ep\, (f_i^\ep)^{-1}$ there holds $\xi^\ep_i=\xi_j^\ep$ on $B_{r^1_\delta}(x_i)\cap \B_{r^1_\delta}(x_j)$ moreover, denoting simply $\xi^\ep:=\xi^\ep_i$ on $B_{r^1_\delta}(x_i)$ there holds
\[
A^{\xi^\ep}=(A^{k_i^\ep})^{(f_i^\ep)^{-1}}\in W^{1,p}(\Om_1(B_{r^1_\delta}(x_i)\otimes\frak{su}(2))
\]
We choose $12/7<p<2$ so that $A^{\xi^\ep}\wedge A^{\xi^\ep}\in L^2(\Om_2(\B^3_1(0))\otimes\frak{su}(2))$ and we deduce $dA^{\xi^\ep}\in L^2(\Om_2(\B^3_1(0))\otimes\frak{su}(2))$. Let $\chi\in C^\infty_c(\R^3)$ such that $\int_{\R^3}\chi(x)\ dx^3=1$. Denote $\chi_t:=t^{-n}\,\chi_t$. There holds
\[
\chi_t\star A^{\xi^\ep}:=\sum_{i=1}^3\chi_t\star A_i^{\xi^\ep}\ dx_i\quad\longrightarrow\quad A^{\xi^\ep}\quad\mbox{ in }W^{1,p}(\Om_1(\B^3_1(0))\otimes\frak{su}(2))
\] 
and 
\[
d\lf(\chi_t\star A^{\xi^\ep}\rg)\quad\longrightarrow\quad dA^{\xi^\ep}\quad\mbox{ in } L^2(\Om_2(\B^3_1(0))\otimes\frak{su}(2))
\]
Thus
\be
\label{appcur}
F_{\chi_t\star A^{\xi^\ep}}\quad\longrightarrow \quad F_{A^{\xi^\ep}}=(\xi^\ep)^{-1}\,F_A\, \xi^\ep\quad\mbox{ in } L^2(\Om_2(\B^3_1(0))\otimes\frak{su}(2))
\ee
Next, since $\xi^\ep\in W^{1,2}(\B^3_1(0), SU(2))$ and since $\pi_2(SU(2))=0$ using the main result of \cite{Bet} there exists $\xi_s\in C^{\infty}(\B^3_1(0), SU(2))$ such that
\[
\xi_s\quad\longrightarrow\quad \xi^\ep\mbox{ in } W^{1,2}(\B^3_1(0), SU(2))\qquad\mbox{ as }s\rightarrow 0\ .
\]
We consider then
\[
A_{t,s}:=\lf(\chi_t\star A^{\xi^\ep}\rg)^{\xi_s^{-1}}\quad .
\]
Using dominated convergence we have
\be
\label{appcur-2}
F_{A_{t,s}}\quad\longrightarrow \quad \xi^{\ep}\, F_{\chi_t\star A^{\xi^\ep}}\,(\xi^\ep)^{-1}\quad\mbox{ in } L^2(\Om_2(\B^3_1(0))\otimes\frak{su}(2))\quad\mbox{ as }s\rightarrow 0\ .
\ee
We have also
\be
\label{appcur-3}
A_{t,s}\quad\longrightarrow\quad \lf(\chi_t\star A^{\xi^\ep}\rg)^{(\xi^\ep)^{-1}}\quad\mbox{ in } L^2(\Om_1(\B^3_1(0))\otimes\frak{su}(2))\quad\mbox{ as }s\rightarrow 0\ .
\ee
Using a diagonal argument we then construct $A_k\in C^\infty(\Om_1(\B^3_1(0))\otimes\frak{su}(2))$ satisfying 
\[
\lf\{
\begin{array}{l}
\ds A_k\ \longrightarrow\ A\qquad\mbox{ strongly in }L^2(\Om^1(\B^3)\otimes \frak{su}(2))\\[5mm]
\ds F_{A_k}\ \longrightarrow\ F_A\qquad\mbox{ strongly in }L^2(\Om^2(\B^3)\otimes \frak{su}(2))
\end{array}
\rg.
\]
Regarding $u$ we have that $d_Au\in L^2$ together with $A\in L^2$ is implying $du\in L^{3/2}$. We have then  $u\star\chi_t\ \longrightarrow\ u$ in $W^{1,3/2}\cap L^6$ and this concludes the proof of proposition~\ref{lm-dens}.\hfill $\Box$

\subsection{ Existence of a minimizer}
We have the following proposition
\begin{Prop}
\label{pr-II.1}
Let $\phi\in L^\infty\cap H^{1/2}(\B^3,\frak{su}(2))$ and $(\ep,\la,\mu)\in ({\R}^\ast_+)^3$ then there exists a minimizer of $CYMH^{\la,\mu}_{\ep}(u,A)$ in ${\mathcal E}$ among maps $u$ such that
\[
u=\phi\qquad\mbox{ on }\ \p \B^3
\]
\hfill $\Box$
\end{Prop}
\noindent{\bf Proof of Proposition~\ref{pr-II.1}.} We denote by ${\mathcal E}_\phi$ the subspace of $(u,A)\in {\mathcal E}$ such that $u=\phi$ on $\p \B^3$. Let $(u^k,A^k)$ be a minimizing sequence. Modulo extraction of a subsequence one can assume that the following convergences hold weakly in $L^2(\B^3)$
\[
\lf\{
\begin{array}{l}
\ds A_k\ \rightharpoonup \ A_\infty\\[5mm]
\ds d_{A_k}u_k\ \rightharpoonup\ B_\infty\\[5mm]
\ds F_{A_k}\ \rightharpoonup\ F_\infty
\end{array}
\rg.
\]
Observe that, since $u_k$ is uniformly bounded in $L^4(\B^3)$ there holds
\[
\limsup_{k\rightarrow +\infty}\|du_k\|_{L^{4/3}(\B^3)}\le \limsup \|d_{A_k}u_k\|_{L^{4/3}(\B^3)}+\|A_k\|_{L^{2}(\B^3)}\, \|u_k\|_{L^4(\B^3)}<+\infty
\]
Hence, modulo extraction of a subsequence, we have
\[
u_k\ \rightharpoonup\ u_\infty\ \mbox{ weakly in }W^{1,4/3}(\B^3)\ .
\] 
Because of the continuous embedding $W^{1,4/3}(\B^3)\ \hookrightarrow\ W^{1/4,4/3}(\p B^3)$ there holds
\be
\label{II.6}
u_\infty=\phi\quad\mbox{ on }\p B^3\ .
\ee
Using Rellich Kondrachov we have that
\[
u_k\ \longrightarrow\ u_\infty \qquad\mbox{ strongly in }L^p(\B^3)\qquad 1\le p<12/5\ .
\]
Using the fact that $u_k$ is uniformly bounded in $L^4(\B^3)$ we deduce
\[
u_k\ \longrightarrow\ u_\infty \qquad\mbox{ strongly in }L^p(\B^3)\qquad 1\le p<4\ .
\]
Hence, ``weak times strong'' convergences passes to the limit and we get
\[
[A_k,u_k]\ \rightharpoonup\ [A_\infty,u_\infty]\quad\mbox{ in }{\mathcal D}'(\B^3)
\]
Since $[A_k,u_k]$ is uniformly bounded in $L^{4/3}(\B^3)$ the previous convergence is a weak convergence in $L^{4/3}(\B^3)$. We deduce
\[
d_{A_k}u_k\ \rightharpoonup\ B_\infty=d_{A_\infty}u_\infty\qquad\mbox{ weakly in }L^2(\B^3)\ .
\]
Observe that for any $x_0\in \B^3$ and $r>0$ such that $\B_r(x_0)\subset \B^3$
\[
\int_{\B_r(x_0)}|F_{A_k}|^{3/2}\ dx^3\le C\ r^{3/4}\, \lf(\int_{B_r(x_0)} |F_{A_k}|^{2}\rg)^{3/4}
\]
Hence there exists $r_\delta>0$ such that
\[
\sup_{k\in {\N}}\sup_{x_0\in \B^3, r<r_\delta}\|F_{A_k}\|_{L^{3/2}(\B_r(x_0))}\le \delta
\]
For any $x_0\in \B^3$ we consider $r<r_\delta$ such that $\B_r(x_0)\subset \B^3$ and we apply Uhlenbeck Coulomb Gauge extraction theorem~\ref{th-Uhl} and we obtain the existence of a gauge change $g_k\in W^{1,2}(\B_r(x_0),SU(2))$ such that
\[
\limsup_{k\rightarrow+\infty}\|A_k^g\|_{W^{1,2}(\B_r(x_0))}<+\infty
\]
This implies in particular
\[
\limsup_{k\rightarrow +\infty}\|dg_k\|_{L^2(\B_r(x_0))}\le \limsup_{k\rightarrow+\infty}\|A_k^g\|_{L^2(\B_r(x_0))}+\|A_k\|_{L^2(\B_r(x_0))}<+\infty\ .
\]
Thus, modulo extraction of a subsequence (that depends on $x_0$ obviously) there holds
\[
g_k\rightharpoonup g_\infty\qquad\mbox{ weakly in }W^{1,2}(\B_r(x_0))\ ,
\]
and, thanks to Rellich Konderachov we have that $g_\infty\in SU(2)$ almost everywhere and $g_\infty\in W^{1,2}(\B_r(x_0),SU(2))$. We also have
\[
A^{g_k}_k=g_k^{-1}\,A_k\, g_k+g_k^{-1}\,dg_k\ \rightharpoonup\ g_\infty^{-1}\,A_\infty\, g_\infty+g_\infty^{-1}\,dg_\infty=A_\infty^{g_\infty}\qquad\mbox{ weakly in }W^{1,2}(\B_r(x_0))
\]
We first deduce that $A_\infty$ is locally gauge equivalent to $W^{1,2}$ connection form . We have also
This implies that, since the embedding $W^{1,2}(\B_r(x_0))\ \hookrightarrow\ L^4(\B_r(x_0))$ is compact in 3 dimensions, using again Rellich Kondrachov 
\[
F_{A_k}=g_k\, F_{A_k^{g_k}}\, g_k^{-1}=g_k\,\lf(dA_k^{g_k}+A^{g_k}_k\wedge  A^{g_k}_k \rg)\,g_k^{-1}\ \rightharpoonup\ g_\infty\,\lf(dA_\infty^{g_\infty}+A^{g_\infty}_\infty\wedge  A^{g_\infty}_\infty \rg)\,g_\infty^{-1}=F_{A_\infty}
\]
weakly in $L^2(\B_r(x_0))$. Thus we deduce  $F_{A_\infty}=F\in L^2(\Omega^2(\B^3)\otimes \frak{su}(2))$. Together with the fact that  $A_\infty$ is everywhere locally gauge equivalent to a $W^{1,2}$ connection form, we, using also (\ref{II.6}) we  deduce that
\[
(u_\infty,A_\infty)\in{\mathcal E}_\phi
\]
Observe that, using Rellich Kondrachov we have $u_k\ \rightarrow \ u_\infty$ strongly in $L^p(\B^3)$.
Using the lower-semicontinuity of the norms we deduce
\[
CYMH^{\la,\mu}_\ep(u_\infty,A_\infty)\le \lim_{k\rightarrow +\infty}CYMH^{\la,\mu}_\ep(u_k,A_k)=\inf_{(u,A)\in{\mathcal E}_\phi}CYMH^{\la,\mu}_\ep(u,A)
\]
Hence $(u_\infty,A_\infty)$ is a minimiser of the Coulomb Yang-Mills Higgs functional among $u$ such that $u=\phi$ on $\p B^3$ and proposition~\ref{pr-II.1} is proved.\hfill $\Box$

\medskip

\subsection{Uniform Bound on the minimal value of $CYMH^{\la,\mu}_\ep(u,A)$}

\begin{Lm}
\label{lm-bd}
Let $\phi\in C^1(\p\B^3_1(0), {\S}^2)$ such that deg$(\phi)=0$. Then  for any $\la>0$, $0<\ep<1$ and $\mu\, \ep< M$ there exists  $C_{\phi,M}<+\infty$ such that
\be
\label{bd}
\inf_{(u,A)\in {\mathcal E}_\phi} CYMH^{\la,\mu}_\ep(u,A)\le C_{\phi,M}
\ee
\hfill $\Box$
\end{Lm}
\noindent{\bf Proof of Lemma~\ref{lm-bd}.} Since $\phi$ has degree zero from $\p\B^3_1(0)$ into $\S^2$ it admits a $C^1$ extension $u$ from $\B^3_1(0)$ into $\S^2$. Let
\[
A:=-4^{-1}\,[u,du]\qquad\mbox{in }\B^3_1(0)\ .
\]
Observe that
\[
[A,u]=-4^{-1}\,[[u,du],u]=4^{-1}\,[u,[u,du]]=-du\ .
\]
Hence $d_Au=0$. Next we compute
\[
\begin{array}{l}
\ds dA=-\,2^{-1}\,d\lf(  (u^{\bf i}\, du^{\bf j}-u^{\bf j}\, du^{\bf i})\, {\bf k}+\,(u^{\bf j}\, du^{\bf k}-u^{\bf k}\, du^{\bf j})\, {\bf i}+\,(u^{\bf k}\, du^{\bf i}-u^{\bf i}\, du^{\bf k})\, {\bf j}  \rg)\\[5mm]
\ds=du^{\bf i}\wedge du^{\bf j}\, {\bf k}+du^{\bf j}\wedge du^{\bf k}\, {\bf i}+du^{\bf k}\wedge du^{\bf i}\, {\bf j} 
\end{array}
\]
and
\[
\begin{array}{l}
\ds 2^{-1}\,[A\wedge A]=2^{-2}\,(u^{\bf i}\, du^{\bf j}-u^{\bf j}\, du^{\bf i})\wedge(u^{\bf j}\, du^{\bf k}-u^{\bf k}\, du^{\bf j})\, [{\bf k},{\bf i}]\\[5mm]
\ds+2^{-2}\,(u^{\bf i}\, du^{\bf j}-u^{\bf j}\, du^{\bf i})\wedge(u^{\bf k}\, du^{\bf i}-u^{\bf i}\, du^{\bf k})\,[{\bf k},{\bf j}]\\[5mm]
\ds+2^{-2}\,((u^{\bf j}\, du^{\bf k}-u^{\bf k}\, du^{\bf j})\wedge(u^{\bf k}\, du^{\bf i}-u^{\bf i}\, du^{\bf k})\, [{\bf i},{\bf j}]\\[5mm]
\ds=2^{-1}\lf(u^{\bf i}\, du^{\bf j}\wedge u^{\bf j}\, du^{\bf k}+u^{\bf j}\, du^{\bf k}\wedge u^{\bf j}\, du^{\bf i}+u^{\bf j}\, du^{\bf i}\wedge u^{\bf k}\, du^{\bf j}\rg)\,{\bf j}\\[5mm]
\ds-2^{-1}\,\lf(  u^{\bf i}\, du^{\bf j}\wedge u^{\bf k}\, du^{\bf i}+ u^{\bf i}\, du^{\bf k}\wedge  u^{\bf i}\, du^{\bf j}+u^{\bf j}\, du^{\bf i}\wedge u^{\bf i}\, du^{\bf k}\rg)   {\bf i}\\[5mm]
\ds +2^{-1}\,\lf( u^{\bf j}\, du^{\bf k}\wedge u^{\bf k}\, du^{\bf i}+ u^{\bf k}\, du^{\bf i}\wedge u^{\bf k}\, du^{\bf j}+ u^{\bf k}\, du^{\bf j} \wedge u^{\bf i}\, du^{\bf k}\rg)\ {\bf k}
\end{array}
\]
Using $u^{\bf i}\, du^{\bf i}+u^{\bf j}\, du^{\bf j}+u^{\bf k}\, du^{\bf k}=0$ we have respectively
\[
\lf\{
\begin{array}{l}
\ds u^{\bf i}\, du^{\bf j}\wedge u^{\bf j}\, du^{\bf k}=-(u^{\bf i})^2\,du^{\bf i}\wedge du^{\bf k}\\[5mm]
\ds u^{\bf j}\, du^{\bf i}\wedge u^{\bf k}\, du^{\bf j}=-(u^{\bf k})^2\,du^{\bf i}\wedge du^{\bf k}\\[5mm]
\ds  u^{\bf i}\, du^{\bf j}\wedge u^{\bf k}\, du^{\bf i}=-(u^{\bf k})^2\,du^{\bf j}\wedge d u^{\bf k}\\[5mm]
\ds u^{\bf j}\, du^{\bf i}\wedge u^{\bf i}\, du^{\bf k}=-(u^{\bf j})^2\,du^{\bf j}\wedge d u^{\bf k}\\[5mm]
\ds u^{\bf j}\, du^{\bf k}\wedge u^{\bf k}\, du^{\bf i}=-(u^{\bf j})^2\,\,du^{\bf j}\wedge d u^{\bf i}\\[5mm]
\ds u^{\bf k}\, du^{\bf j} \wedge u^{\bf i}\, du^{\bf k}=-(u^{\bf i})^2\,du^{\bf j}\wedge d u^{\bf i}
\end{array}
\rg.
\]
Thus
\[
2^{-1}\,[A\wedge A]=2^{-1}\, du^{\bf k}\wedge du^{\bf i}\ \,{\bf j}+2^{-1}\,du^{\bf j}\wedge du^{\bf k}\ \,{\bf i}+2^{-1}\,du^{\bf i}\wedge du^{\bf j}\ \,{\bf k}
\]
This gives
\[
F_A=dA+2^{-1}\,[A\wedge A]=\frac{3}{2}\,\lf(du^{\bf i}\wedge du^{\bf j}\, {\bf k}+du^{\bf j}\wedge du^{\bf k}\, {\bf i}+du^{\bf k}\wedge du^{\bf i}\, {\bf j} \rg)
\]
Observe that
\[
\begin{array}{rl}
\ds [du\wedge du]&=[ \lf(du^{\bf i}\,{\bf i}+du^{\bf j}\,{\bf j}+du^{\bf k}\,{\bf k}\rg),\lf(du^{\bf i}\,{\bf i}+du^{\bf j}\,{\bf j}+du^{\bf k}\,{\bf k}\rg)]\\[5mm]
&\ds=4\, du^{\bf i}\wedge du^{\bf j}\ \,{\bf k}+4\,du^{\bf k}\wedge du^{\bf i}\ \,{\bf j}+4\, \,du^{\bf j}\wedge du^{\bf k}\ \,{\bf i}
\end{array}
\]
This finally gives
\be
\label{curv}
F_A=\frac{3}{8}\,[du\wedge du]
\ee
Hence for this choice of $(u,A)$ we have
\be
\label{fadeev}
CYMH^{\la,\mu}_\ep(u,A)=\frac{1}{8}\int_{\B^3}|du|^2+\frac{9}{16}\,\mu\,\ep\, |[du\wedge du]|^2\ dx^3\ .
\ee
This concludes the proof of the Lemma~\ref{lm-bd}.\hfill $\Box$

\begin{Rm} Observe that for this choice of $(u,A)$ there holds
\[
u\cdot F_A=\frac{3}{2}\,\lf(u^{\bf k}\,du^{\bf i}\wedge du^{\bf j}+u^{\bf i}\,du^{\bf j}\wedge du^{\bf k}+u^{\bf j}\,du^{\bf k}\wedge du^{\bf i}\,  \rg)=\frac{3}{2}\,u^\ast\om
\]
We have moreover
\[
\begin{array}{l}
[u,F_A]=0
\end{array}
\]
Observe  moreover that under these constraints the Lagrangian $CYMH^{\la,\mu}_\ep(u,A)$ given by (\ref{fadeev}) coincides with the one of the {\bf Fadeev Model} .
\hfill $\Box$
\end{Rm}

\subsection{The Coulomb Yang-Mills Higgs System}
Consider a variation $(u_t,A_t):=(u+tv,A+ta)$ where $v\in C^\infty_0(B^3,{\R}^3)$ and $a\in C^\infty_0(\Om^1(\B^3)\otimes \frak{su}(2))$ then we have
\[
\begin{array}{l}
\ds\lf.\frac{d}{dt}CYMH(u_t,A_t)\rg|_{t=0}=\int_{\B^3}\, A\cdot a+\frac{\mu}{\ep}d_Au\cdot\frac{d}{dt}(d_{A_t}u_t)+\mu\,\ep\,F_A\cdot \frac{d}{dt}(F_{A_t})\\[5mm]
\ds\qquad-\frac{\mu\,\la}{\ep^3}\int_{\B^3}\, v\cdot u\, (1-|u|^2)\ dx^3
\end{array}
\]
We have
\[
\frac{d}{dt}(d_{A_t}u_t)=dv+[A,v]+[a,u]\ 
\]
and
\[
\frac{d}{dt}(F_{A_t})(X,Y)=da(X,Y)+[a(X),A(Y)]+[A(X),a(Y)]
\]
If we denote
\[
[A\wedge a](X,Y):=[a(X),A(Y)]+[A(X),a(Y)]=[A(X),a(Y)]-[A(Y),a(X)]
\]
we then have
\[
\frac{d}{dt}(F_{A_t})=da+[A\wedge a]=:d_{A}a
\]
Thus
\[
\begin{array}{l}
\ds\lf.\frac{d}{dt}CYMH(u_t,A_t)\rg|_{t=0}=\int_{B^3}\, A\cdot a-\frac{\mu}{\ep} \mbox{div}(\nabla_Au)\cdot v+\frac{\mu}{\ep}<\nabla_Au,[A,v]>\ dx^3\\[5mm]
\ds+\int_{B^3}\,\frac{\mu}{\ep}<\nabla_Au,[a,u]>+\ep\,\mu\, F_A\cdot( d a+[A\wedge a])\ dx^3-\int_{B^3}\frac{\mu\,\lambda}{\ep^3}\, v\cdot u\, (1-|u|^2)\ dx^3
\end{array}
\]
We have
\[
<\nabla_Au,[A,v]>=-<[A\res \nabla_Au],v>\qquad\mbox{ where }\qquad [A\res \nabla_Au]:=\sum_{l=1}^3[A_l,\p_{x_l}u+[A_l\,u]]
\]
and 
\[
<d_Au,[a,u]>=<[u,d_Au], a>\qquad
\]
and
\[
\begin{array}{l}
\ds<F_A;[A\wedge a]>=\sum_{l<m} <F_{lm},[A_l,a_m]-[A_m,a_l]>\\[5mm]
\ds=\sum_{l<m} <[A_m,F_{lm}],a_l>-<[A_l,F_{lm}],a_m>=\sum_{l\ne m}<[A_m,F_{lm}],a_l>
\end{array}
\]
We denote
\[
[A\res F_A]:= \sum_{l\ne m}[A_m,F_{lm}]\ dx_l
\]
Finally we obtain
\[
\begin{array}{l}
\ds\lf.\frac{d}{dt}YMH(u_t,A_t)\rg|_{t=0}=\int_{B^3}\,A\cdot a-\frac{\mu}{\ep} \mbox{div}(\nabla_Au)\cdot v-\frac{\mu}{\ep}<[A,d_Au],v>\ dx^3\\[5mm]
\ds+\int_{B^3}\,\frac{\mu}{\ep} <[u,d_Au],a>  +\mu\,\ep\, (d^\ast F_A+[A\res F_A])\cdot a+\frac{\la\,\mu}{\ep^3}\, u\cdot v\,(1-|u|^2)\ dx^3
\end{array}
\]
Hence $(u,A)$ is a critical point of $CYMH$ if and only if 
\be
\label{EL}
\lf\{
\begin{array}{l}
\ds - \mbox{div}(\nabla_Au)-[A\res\nabla_Au]=\,\frac{\la}{\ep^2}\, u\,(1-|u|^2)\\[5mm]
\ds\ep d^\ast_A F_A+\frac{1}{\ep}[u,d_Au]+\frac{A}{\mu}=0
\end{array}
\rg.
\ee
written using exterior covariant differentiation the Coulomb Yang Mills Higgs System becomes
\be
\label{CYMHS}
\lf\{
\begin{array}{l}
\ds d^\ast_Ad_Au=\,\frac{\la}{\ep^2}\, u\,(1-|u|^2)\\[5mm]
\ds\ep \,d^\ast_AF_A+\frac{1}{\ep}[u,d_Au]+\frac{A}{\mu}=0
\end{array}
\rg.
\ee
We can rewrite this system using coordinates in the following form
\be
\label{CYMHS-co}
\lf\{
\begin{array}{l}
\ds -\sum_{l=1}^3(\nabla_A)_l(\nabla_A)_lu=\,\frac{\la}{\ep^2}\, u\,(1-|u|^2)\\[5mm]
\ds\forall m=1,2,3\qquad \mu^{-1}\,A_m=\ep \,\sum_{l=1}^3(\nabla_A)_lF_{lm}+\ep^{-1}\, [(\nabla_A)_mu,u]
\end{array}
\rg.
\ee
While considering critical points of $CYMH^{\la,\mu}_\ep$ under the constraints $u=\phi$ and $\p\B_1(0)$ the Euler Lagrange equation becomes
\be
\label{CYMHS-bdy}
\lf\{
\begin{array}{l}
\ds d^\ast_Ad_Au=\,\frac{\la}{\ep^2}\, u\,(1-|u|^2)\qquad\mbox{ in }\B_1(0)\\[5mm]
\ds\ep \,d^\ast_AF_A+\frac{1}{\ep}[u,d_Au]+\frac{A}{\mu}=0\qquad\mbox{ in }\B_1(0)\\[5mm]
\ds u=\phi\qquad\mbox{ on }\p \B_1(0)\\[5mm]
\ds\iota_{\p \B_1(0)}^\ast \ast F_A=0
\end{array}
\rg.
\ee
\subsection{$CYMH\Longrightarrow$ Coulomb}
The following proposition justifies the denomination ``Coulomb Yang Mill Higgs''.
\begin{Prop}
\label{pr-coul}
Let $(u,A)\in {\mathcal E}$ be a solution to the Coulomb Yang Mills Higgs System then the following condition holds
\be
\label{coul}
d^\ast A=0\qquad\mbox{ in }\B^3\ .
\ee
\hfill $\Box$
\end{Prop}
\noindent{\bf Proof of proposition~\ref{pr-coul}} We apply $d^\ast$ to the second equation of (\ref{CYMHS}). Observe that
\[
d^\ast_AA=-\sum_{l=1}^3\p_{x_l}A_l+[A_l,A_l]=-\sum_{l=1}^3\p_{x_l}A_l=d^\ast A\ .
\]This gives
\[
\mu^{-1}\,{d^\ast A}=-\ep\, d^\ast_Ad^\ast_AF_A-\frac{1}{\ep}d^\ast_A[u,d_Au]
\]
We have in one hand using {\it Jacobi identity} and the first equation in the CYMS system (\ref{CYMHS})
\[
\begin{array}{rl}
\ds-\frac{1}{\ep}d^\ast_A[u,d_Au]&\ds=\frac{1}{\ep}\sum_{l=1}^3\p_{x_l}[u,(\nabla_A)_lu]+[A_l,[u,(\nabla_A)_lu]]\\[5mm]
\ds&\ds=\frac{1}{\ep}\sum_{l=1}^3[\p_{x_l}u,(\nabla_A)_lu]+[u, \p_{x_l}((\nabla_A)_lu)]+[u,[A_l,(\nabla_A)_lu]]+[(\nabla_A)_lu,[u,A_l]]\\[5mm]
\ds&\ds=-\frac{1}{\ep}[u,d^\ast_Ad_Au]+\frac{1}{\ep}\sum_{l=1}^3[(\nabla_A)_lu,(\nabla_A)_lu]=0
\end{array}
\]
We have using (\ref{cov-ex})
\[
\begin{array}{l}
\ds d^\ast_Ad^\ast_AF_A=-\,d^\ast_A\lf(\sum_{l,m=1}^3(\nabla_A)_lF_{lm}\ dx_m\rg)\\[5mm]
\ds=\sum_{l,m=1}^3(\nabla_A)_m(\nabla_A)_l F_{lm}=\sum_{l<m}\lf((\nabla_A)_m(\nabla_A)_l -(\nabla_A)_l(\nabla_A)_m\rg)F_{lm}\\[5mm]
\ds=\sum_{l<m}[F_{ml},F_{lm}]=0
\end{array}
\]
Combining the previous identities we get (\ref{coul}) and proposition~\ref{pr-coul} is proved. \hfill $\Box$
\subsection{The Monotonicity Formula}
Recall that {\it stationary critical points} of a Lagrangian are critical points which are also critical for variations in the domain. This is the case in particular for minimizers. The following result holds
\begin{Lm}
\label{lm-mono}
Let $(u,A)$ be a stationary critical point of $CYMH^{\la,\mu}_\ep$ then for any $x_0\in \B^3$ and $r>0$ such that $\B_r(x_0)\subset\B^3$ there holds
\be
\label{III.1}
\begin{array}{l}
\ds\frac{d}{dr}\lf(   \frac{1}{2r}\int_{\B_r(x_0)}|A|^2+\mu\,\frac{|\nabla_Au|^2}{\ep}+\frac{3\mu\la}{2\ep^3}(1-|u|^2)^2\ dx^3\rg)\\[5mm]
\ds\quad+\frac{1}{2\,r^2}\int_{\B_r(x_0)}\,\ep\,\mu\,|F_A|^2\ dx^3 +\frac{1}{2\,r}\,\int_{\p \B_r(x_0)}\ep\,\mu\,{|F_A|^2}\ dx^3\\[5mm]
\ds=\frac{1}{r}\,\int_{\p \B_r(x_0)}|A_r|^2+\mu\frac{|(\nabla_Au)_r|^2}{\ep}+\frac{\mu\,\la}{2\ep^3}(1-|u|^2)^2+\mu\,\ep\,|F_A\res\p_r|^2\ dvol_{\p \B_r(x_0)}
\end{array}
\ee
where $(\nabla_Au)_r:=d_Au(\p_r)$ and $F_A\res\p_r$ is the contraction between the curvature and the radial vector-field $\p_r$. This contraction is given by
\[
F_A\res\p_r:=\sum_{l<m} F_{lm}\, \lf(\frac{(x-x_0)_l}{|x-x_0|}\ dx_m-\frac{(x-x_0)_m}{|x-x_0|}\ dx_l\rg)=\sum_{l,m=1}^3 F_{lm}\, \frac{(x-x_0)_l}{|x-x_0|}\ dx_m\ .
\]
\end{Lm}
\begin{Rm}\label{rm-mon}
Observe that the formula (\ref{III.1}) is equivalent to the following monotonicity
\be
\label{trumon}
\begin{array}{l}
\ds\frac{d}{dr}\lf(   \frac{1}{2r}\int_{\B_r(x_0)}|A|^2+\mu\,\frac{|\nabla_Au|^2}{\ep}+\frac{3\la\mu}{2\ep^3}(1-|u|^2)^2+\lf(\frac{2r}{|x-x_0|}-1\rg)\mu\,\ep\,|F_A|^2\ dx^3 \rg)\\[5mm]
\ds=\frac{1}{r}\,\int_{\p \B_r(x_0)}|A_r|^2+\mu\frac{|(\nabla_Au)_r|^2}{\ep}+\frac{\la\,\mu}{2\ep^3}(1-|u|^2)^2+\,\ep\,\mu\,|F_A\res\p_r|^2\ dvol_{\p \B_r(x_0)}
\end{array}
\ee
\end{Rm}
\noindent{\bf Proof of Lemma~\ref{lm-mono}.} Since there is no possible confusion on the position of the various parameters, we present the proof for $\ep=\mu=\la=1$. Let $X\in C^\infty_0(\B^3,{\R}^3)$ and $\phi_t(x):=x+t\,X$. Assume first $(u,A)$ is smooth which shall give more flexibility in the computations. Then we will conclude using an approximation argument. We have respectively
\[
\frac{d}{dt}\lf(\phi_t^\ast A\rg)=\frac{d}{dt}\lf(\phi_t^\ast \lf(\sum_{l=1}^3 A_l\ dx_l\rg)\rg)=\sum_{k=1}^3\sum_{l=1}^3\frac{\p A_l}{\p x_k}\ X_k\ dx_l+\sum_{k=1}^3\sum_{l=1}^3A_l\frac{\p X_l}{\p x_k}\ dx_k
\]
moreover
\[
\begin{array}{l}
\ds\frac{d}{dt}\lf(\phi_t^\ast d_Au\rg)=\frac{d}{dt}\lf(\phi_t^\ast \lf(\sum_{l=1}^3\p_{x_l}u+[A_l,u]) dx_l\rg)\rg)\\[5mm]
\ds=\sum_{k=1}^3\sum_{l=1}^3\frac{\p}{\p x_k}\lf(\p_{x_l}u+[A_l,u]\rg)\ X_k\ dx_l+\sum_{k=1}^3\sum_{l=1}^3(\p_{x_l}u+[A_l,u]) \frac{\p X_l}{\p x_k}\ dx_k
\end{array}
\]
and
\[
\begin{array}{l}
\ds\frac{d}{dt}\lf(\phi_t^\ast F_A\rg)=\frac{d}{dt}\lf(\phi_t^\ast\sum_{l<m} (\p_{x_l}A_m-\p_{x_m}A_l+[A_l,A_m])\ dx_l\wedge dx_m\rg)\\[5mm]
\ds=\sum_{k=1}^3\sum_{l<m} \frac{\p}{\p x_k}(\p_{x_l}A_m-\p_{x_m}A_l+[A_l,A_m])\ X_k\ dx_l\wedge dx_m\\[5mm]
\ds\quad+\sum_{k=1}^3\sum_{l<m} F_{lm}\, \lf(\frac{\p X_l}{\p x_k}\ dx_k\wedge dx_m+\frac{\p X_m}{\p x_k}\ dx_l\wedge dx_k\rg)
\end{array}
\]
This gives in particular respectively
\[
\begin{array}{l}
\ds\lf.\frac{1}{2}\frac{d}{dt}\int_{B^3}|\phi_t^\ast A|^2\ dx^3\rg|_{t=0}=\frac{1}{2}\int_{\B^3} \sum_{k=1}^3\frac{\p}{\p x_k}|A|^2\, X_k\ dx^3+\int_{B^3}\sum_{k=1}^3\sum_{l=1}^3 \frac{\p X_l}{\p x_k}\ A_l\cdot A_k\ dx^3

\end{array}
\]
moreover
\[
\begin{array}{l}
\ds\lf.\frac{1}{2}\frac{d}{dt}\int_{\B^3}|\phi_t^\ast d_Au|^2\ dx^3\rg|_{t=0}\\[5mm]
\ds=\frac{1}{2}\int_{\B^3} \sum_{k=1}^3\frac{\p}{\p x_k}|\nabla_A u|^2\, X_k\ dx^3+\int_{\B^3}\sum_{k=1}^3\sum_{l=1}^3 \frac{\p X_l}{\p x_k}\ (\p_{x_l}u+[A_l,u])\cdot(\p_{x_k}u+[A_k,u])\ dx^3

\end{array}
\]
and
\[
\begin{array}{l}
\ds\lf.\frac{1}{2}\frac{d}{dt}\int_{\B^3}|\phi_t^\ast F_A|^2\ dx^3\rg|_{t=0}=\int_{B^3}\sum_{k=1}^3\sum_{l<m} \frac{\p F_{lm}}{\p x_k}\cdot F_{lm}\ dx^3+\int_{\B^3}\sum_{k=1}^3\sum_{l<m} F_{lm}\cdot F_{km}\ \frac{\p X_l}{\p x_k}\ dx^3\\[5mm]
\ds+\int_{\B^3}\sum_{k=1}^3\sum_{l<m} F_{lm}\cdot F_{lk}\ \frac{\p X_m}{\p x_k}\ dx^3\ .
\end{array}
\]
We have also
\[
\begin{array}{l}
\ds\lf.\frac{1}{4}\frac{d}{dt}\int_{B^3}\phi_t^\ast\lf((1-|u|^2)^2\rg)\ dx^3\rg|_{t=0}=\lf.\frac{1}{4}\frac{d}{dt}\int_{B^3}(1-|u|^2)^2\circ\phi_t\ dx^3\rg|_{t=0}\\[5mm]
\ds\quad=\frac{1}{4}\int_{B^3}\frac{\p(1-|u|^2)^2}{\p x_k}\,X_k\ dx^3

\end{array}
\]
Thus, denoting  $(u_t,A_t):=\phi_t^\ast (u,A)$ we have
\[
\begin{array}{l}
\ds \lf.\frac{d}{dt}CYMH(u_t,A_t)\rg|_{t=0}=\frac{1}{2}\int_{\B^3} \sum_{k=1}^3\frac{\p}{\p x_k}|A|^2\, X_k\ dx^3+\int_{\B^3}\sum_{k=1}^3\sum_{l=1}^3 \frac{\p X_l}{\p x_k}\ A_l\cdot A_k\ dx^3\\[5mm]
\ds+\frac{1}{2}\int_{\B^3} \sum_{k=1}^3\frac{\p}{\p x_k}|\nabla_A u|^2\, X_k\ dx^3+\int_{\B^3}\sum_{k=1}^3\sum_{l=1}^3 \frac{\p X_l}{\p x_k}\ (\p_{x_l}u+[A_l,u])\cdot(\p_{x_k}u+[A_k,u])\ dx^3\\[5mm]
\ds+\frac{1}{2}\int_{\B^3}\frac{\p}{\p x_k}|F_A|^2\ X_k\ dx^3+\int_{\B^3}\sum_{k=1}^3\sum_{l<m} F_{lm}\cdot F_{km}\, \frac{\p X_l}{\p x_k}+F_{lm}\cdot F_{lk}\, \frac{\p X_m}{\p x_k}\ \ dx^3\\[5mm]
\ds+\frac{1}{4}\int_{\B^3}\sum_{k=1}^3\frac{\p(1-|u|^2)^2}{\p x_k}\,X_k\ dx^3
\end{array}
\]
To simplify the notations we take $x_0=0$  and we choose $X^\delta= \sum_{k=1}^3 x_k \eta_\delta(|x|)$ where $\eta_\delta$ is a mollification of the characteristic function of $[-r,r]$. We have in particular as $\delta\rightarrow 0$
\[
\begin{array}{l}
\ds\frac{1}{2}\int_{\B^3} \sum_{k=1}^3\frac{\p}{\p x_k}|A|^2\, X^\delta_k\ dx^3+\frac{1}{2\ep}\int_{\B^3} \sum_{k=1}^3\frac{\p}{\p x_k}|\nabla_A u|^2\, X^\delta_k\ dx^3+\frac{1}{2}\int_{\B^3}\frac{\p}{\p x_k}|F_A|^2\ X^\delta_k\ dx^3\\[5mm]
\ds\rightarrow\frac{1}{2}\int_{\B_r}\sum_{k=1}^3x_k\,\lf(\frac{\p}{\p x_k}|A|^2+\frac{1}{\ep}\frac{\p}{\p x_k}|\nabla_A u|^2+\ep \, \frac{\p}{\p x_k}|F_A|^2\rg)
\end{array}
\]
We have
\[
\begin{array}{l}
\ds\int_{\B^3}\sum_{k=1}^3\sum_{l=1}^3 \frac{\p X^\delta_l}{\p x_k}\ A_l\cdot A_k\ dx^3=\int_{\B^3}\eta_\delta(|x|)\ |A|^2\ dx^3\\[5mm]
\ds+\int_{\B^3}\sum_{k=1}^3\sum_{l=1}^3 x_l \frac{x_k}{|x|}\,(\eta_\delta)'(|x|)\ A_l\cdot\ A_k\ dx^3\\[5mm]
\ds\longrightarrow \int_{\B_r}|A|^2\ dx^3-r\,\int_{\p \B_r}\lf|A_r\rg|^2\ dvol_{\p B_r}
\end{array}
\]
moreover
\[
\begin{array}{l}
\ds \int_{\B^3}\sum_{k=1}^3\sum_{l=1}^3 \frac{\p X_l}{\p x_k}\ (\p_{x_l}u+[A_l,u])\cdot(\p_{x_k}u+[A_k,u])\ dx^3=\int_{\B^3} \eta_\delta(|x|)\ |\nabla_A u|^2\ dx^3\\[5mm]
\ds+\int_{\B^3}\sum_{k=1}^3\sum_{l=1}^3 x_l \frac{x_k}{|x|}\, \,(\eta_\delta)'(|x|)\ (\p_{x_l}u+[A_l,u])\cdot(\p_{x_k}u+[A_k,u])\ dx^3\\[5mm]
\ds\longrightarrow \,\int_{\B_r}|\nabla_A u|^2\ dx^3-r\,\int_{\p \B_r}\lf|(\nabla_A u)_r\rg|^2\ dvol_{\p \B_r}\ .
\end{array}
\]
Finally
\[
\begin{array}{l}
\ds\int_{\B^3}\sum_{k=1}^3\sum_{l<m} F_{lm}\cdot F_{km}\, \frac{\p X_l}{\p x_k}+F_{lm}\cdot F_{lk}\, \frac{\p X_m}{\p x_k}\ \ dx^3\\[5mm]
\ds={2}\int_{\B^3}\eta_\delta(|x|)\ |F_A|^2\ dx^3+\int_{\B^3}\sum_{k=1}^3\sum_{l<m} F_{lm}\cdot F_{km}\,  x_l \frac{x_k}{|x|}\, \,\eta_\delta'(|x|)\ dx^3 \\[5mm]
\ds+\int_{\B^3}\sum_{k=1}^3\sum_{l<m} \,  F_{lm}\cdot F_{lk}\, x_m \frac{x_k}{|x|}\, \,\eta_\delta'(|x|)\ dx^3\\[5mm]
\ds={2}\int_{\B^3}\eta_\delta(|x|)\ |F_A|^2\ dx^3+\int_{\B^3}\sum_{k=1}^3\sum_{l<m} F_{lm}\cdot F_{km}\,  x_l \frac{x_k}{|x|}\, \,\eta_\delta'(|x|)\ dx^3\\[5mm]
\ds+\int_{\B^3}\sum_{k=1}^3\sum_{m<l} \,  F_{ml}\cdot F_{mk}\, x_l \frac{x_k}{|x|}\, \,\eta_\delta'(|x|)\ dx^3\\[5mm]
\ds={2}\int_{\B^3}\eta_\delta(|x|)\ |F_A|^2\ dx^3+r\,\int_{\B^3}\eta_\delta'(|x|)\ \sum_{m=1}^3 \lf|\sum_{l=1}^3F_{ml}\,\frac{x_l}{|x|}\rg|^2\ dx^3
\end{array}
\]
Thus
\[
\begin{array}{l}
\ds\int_{\B^3}\sum_{k=1}^3\sum_{l<m} F_{lm}\cdot F_{km}\, \frac{\p X_l}{\p x_k}+F_{lm}\cdot F_{lk}\, \frac{\p X_m}{\p x_k}\ \ dx^3\\[5mm]
\ds\longrightarrow \ {2}\int_{\B_r}|F_A|^2\ dx^3-\, r\,\int_{\B_r}\lf|F_A\res\p_r\rg|^2\ dx^3
\end{array}
\]
We have
\[
\begin{array}{l}
\ds\frac{1}{4}\int_{\B^3}\sum_{k=1}^3\frac{\p(1-|u|^2)^2}{\p x_k}\,X_k\ dx^3=\frac{1}{4}\int_{\B^3}\sum_{k=1}^3\frac{\p(1-|u|^2)^2}{\p x_k}\,  x_k \eta_\delta(|x|)\ dx^3\\[5mm]
\ds=-\frac{3}{4}\int_{\B^3}  \eta_\delta(|x|)\ (1-|u|^2)^2\ dx^3 -\frac{1}{4}\int_{\B^3}\sum_{k=1}^3\frac{x_k^2}{|x|} \,(\eta_\delta)'(|x|)\,(1-|u|^2)^2\ dx^3
\end{array}
\]
Finally we obtain for any smooth pair $(u,A)$ 
\[
\begin{array}{l}
\ds \lf.\frac{d}{dt}CYMH(u_t,A_t)\rg|_{t=0}=\frac{1}{2}\int_{\B_r}\sum_{k=1}^3x_k\,\lf(\frac{\p}{\p x_k}|A|^2+\frac{\p}{\p x_k}|\nabla_A u|^2+\frac{\p}{\p x_k}|F_A|^2\rg)\\[5mm]
\ds+\int_{\B_r}|A|^2\ dx^3-r\,\int_{\p \B_r}\lf|A_r\rg|^2\ dvol_{\p \B_r}+\int_{\B_r}|\nabla_A u|^2\ dx^3-r\,\int_{\p \B_r}\lf|(\nabla_A u)_r\rg|^2\ dvol_{\p \B_r}\\[5mm]
\ds+{2}\int_{\B_r}|F_A|^2\ dx^3-r\,  \int_{\B_r}\lf|F_A\res\p_r\rg|^2\ dvol_{\p B_r}\\[5mm]
\ds-\frac{3}{4}\int_{\B_r}   (1-|u|^2)^2\ dx^3         +\frac{r}{4}\int_{\p\B_r}(1-|u|^2)^2\ dvol_{\p B_r}
\end{array}
\]
Let $B\in L^2(\Om^1(\B^3)\otimes\frak{su}(2))$ observe that for the sequence of diffeomorphisms above, the map 
\[
\begin{array}{l}
\ds t\rightarrow \ \int_{\B^3}|\phi_t^\ast B|^2\ dx^3=\sum_{k=1}^3\int_{\B^3}\lf|\sum_{l=1}^3B_l(\phi_t(x))\ \p_{x_k}\phi_t^l\rg|^2\ dx^3 \\[5mm]
\ds=\sum_{k=1}^3\int_{\B^3}\lf|\sum_{l=1}^3B_l(y)\ \p_{x_k}\phi_t^l\circ\phi_t^{-1}\rg|^2\ \mbox{det}(\nabla \phi_t^{-1})\ dy^3 
\end{array}
\]
is obviously $C^1$ and the derivative at 0 is a continuous function of $B$ in $L^2(\Om^1(\B^3)\otimes\frak{su}(2))$. The same hold for any choice of $F\in  L^2(\Om^1(\B^3)\otimes\frak{su}(2))$, the map
\[
\begin{array}{l}
\ds t\longrightarrow \ \int_{\B^3}|\phi_t^\ast F|^2\ dx^3
\end{array}
\]
is obviously $C^1$ and the derivative at 0 is a continuous function of $F$ in $L^2(\Om^2(\B^3)\otimes\frak{su}(2))$. Hence using the density result lemma~\ref{lm-dens}.

For an arbitrary stationary point there holds for $(u_t,A_t):=\phi_t^\ast (u,A)$ integration by part is finally giving
\be
\label{premon}
\begin{array}{l}
\ds-\frac{1}{2}\int_{\B_r}|A|^2+|\nabla_A u|^2\ dx^3+\frac{r}{2}\int_{\p \B_r}|A|^2+|\nabla_A u|^2\ dvol_{\p \B_r}\\[5mm]
\ds+\frac{1}{2}\int_{\B_r}|F_A|^2\ dx^3+\frac{r}{2}\int_{\p \B_r}|F_A|^2\ dvol_{\p \B_r}\\[5mm]
\ds-\frac{3}{4}\int_{\B_r}   (1-|u|^2)^2\ dx^3         +r\,\frac{1}{4}\int_{\p\B_r}(1-|u|^2)^2\ dvol_{\p \B_r}\\[5mm]
\ds=r\,\int_{\p \B_r}\lf|A_r\rg|^2\ dvol_{\p B_r}+r\,\int_{\p \B_r}\lf|(\nabla_A u)_r\rg|^2\ dvol_{\p \B_r}+r\, \int_{\B_r}\lf|F_A\res\p_r\rg|^2\ dvol_{\p \B_r}
\end{array}
\ee
This gives (\ref{III.1}) and lemma~\ref{lm-mono} is proved.\hfill $\Box$.
\subsection{Bochner -Weitzenb\"ock Identities}
\subsubsection{Bochner -Weitzenb\"ock for the complete CYMH energy densitiy}
\begin{Prop}{\bf[Bochner -Weitzenb\"ock for CYMH]}
\label{pr-bw}
Let $(u,A)$ be a smooth critical point of 
\[
CYMH(u,A):=\frac{1}{2}\int_{B^3}\lf(|A|^2\ dx^3+\frac{\mu}{\ep}|d_A u|^2+\mu\,\ep\,|F_A|^2+\frac{\la\,\mu}{2\ep^3}(1-|u|^2)^2\rg)\ dx^3
\]
then the following identities hold
\be
\label{bw-1}
\begin{array}{l}
\ds\,\Delta \frac{|\nabla_A u|^2}{2\,\ep}=\frac{|(\nabla_A)(\nabla_A u)|^2}{\ep}+\frac{1}{\ep^3} |[u,(\nabla_A)u]|^2\\[5mm]
\ds+ \frac{2}{\ep}\sum_{l,m=1}^3[(\nabla_A)_lu,(\nabla_A)_mu]\cdot F_{lm}+\frac{1}{\mu\,\ep^2}\,\sum_{m=1}^3[u,(\nabla_A)_mu]\cdot A_m\\[5mm]
\ds-\,\frac{\la}{\ep^3}\, |\nabla_A u|^2\,(1-|u|^2)+\frac{\,\la}{2\ep^3}\,|\nabla|u|^2|^2\ ,
\end{array}
\ee
and
\be
\label{bw-2}
\begin{array}{l}
\ds2^{-1}\,\Delta (\ep|F_A|^2)=\ep\,|(\nabla_A) F|^2+\frac{1}{\ep}|[F,u]|^2+\frac{2}{\mu}\,|F_A|^2\\[5mm]
\ds+\frac{2}{\ep}\,\sum_{l,m=1}^3[(\nabla_A)_lu,(\nabla_A)_mu]\cdot F_{lm}+2\,\ep\,\sum_{l,m,n=1}^3[F_{nl},F_{nm}]\cdot F_{lm}\\[5mm]
\ds+\frac{1}{\mu}\sum_{l,m=1}^3[A_l,A_m]\cdot F_{lm}\ ,
\end{array}
\ee
\hfill $\Box$
\end{Prop}
\noindent{\bf Proof of Proposition~\ref{pr-bw}.}
\[
\begin{array}{l}
\ds 2^{-1}\,\p_{x_l}|\nabla_A u|^2=\nabla_A u\cdot (\nabla_A)_l(\nabla_A u)
\end{array}
\]
and
\[
2^{-1}\,\Delta |\nabla_A u|^2=\sum_{l=1}^3(\nabla_A)_l(\nabla_A u)\cdot (\nabla_A)_l(\nabla_A u)+\nabla_A u\cdot \sum_{l=1}^3(\nabla_A)_l((\nabla_A)_l(\nabla_A u))
\]
Recall the identity (using Jacobi identity)
\[
\begin{array}{l}
\ds(\nabla_A)_l((\nabla_A)_m v)=(\nabla_A)_m((\nabla_A)_l v)+[F_{lm},v]
\end{array}
\]
Thus
\[
\begin{array}{l}
\ds \sum_{l=1}^3(\nabla_A)_l((\nabla_A)_l(\nabla_A u))=\sum_{l,m=1}^3(\nabla_A)_l((\nabla_A)_l((\nabla_A)_m u))\ \p_{x_m}\\[5mm]
\ds=\sum_{l,m=1}^3\lf((\nabla_A)_l((\nabla_A)_m((\nabla_A)_l u))+(\nabla_A)_l([F_{lm},u])\rg)\, \p_{x_m}\\[5mm]
\ds=\sum_{l,m=1}^3\lf((\nabla_A)_m((\nabla_A)_l((\nabla_A)_l u))+(\nabla_A)_l([F_{lm},u])+[F_{lm},(\nabla_A)_l u]\rg)\, \p_{x_m}\\[5mm]
\end{array}
\]
We have using again Jacobi
\[
\begin{array}{l}
\ds(\nabla_A)_l[F_{lm},u]=[(\nabla_A)_l F_{lm},u]+[F_{lm},(\nabla_A)_lu]
\end{array}
\]
Hence
\be
\label{X.1}
\begin{array}{l}
\ds  \sum_{l=1}^3(\nabla_A)_l((\nabla_A)_l(\nabla_A u))= (\nabla_A)\lf(\sum_{l,m=1}^3\lf((\nabla_A)_l((\nabla_A)_l u))\rg)\rg)+\\[5mm]
\ds+\sum_{l,m=1}^3\,[(\nabla_A)_l F_{lm},(\nabla_A)_lu]+2\,[F_{lm},(\nabla_A)_lu]\, \p_{x_m}
\end{array}
\ee
This implies
\[
\begin{array}{l}
\ds\nabla_A u\cdot \sum_{l=1}^3(\nabla_A)_l((\nabla_A)_l(\nabla_A u))=\sum_{l,m=1}^3(\nabla_A)_mu\cdot (\nabla_A)_l((\nabla_A)_l((\nabla_A)_m u))\\[5mm]
\ds=\sum_{l,m=1}^3(\nabla_A)_mu\cdot (\nabla_A)_l((\nabla_A)_m((\nabla_A)_l u))+(\nabla_A)_mu\cdot (\nabla_A)_l([F_{lm},u])\\[5mm]
\ds=\sum_{l,m=1}^3(\nabla_A)_mu\cdot (\nabla_A)_m((\nabla_A)_l((\nabla_A)_l u))+2\, \sum_{l,m=1}^3(\nabla_A)_mu\cdot [F_{lm}, (\nabla_A)_l u] \\[5mm]
 \ds +\sum_{l,m=1}^3(\nabla_A)_mu\cdot [(\nabla_A)_lF_{lm},u])\
\end{array}
\]
Assuming now that $(u,A)$ is a critical point of CYMH, there holds for every $m=1,2,3$
\[
\ds\mu\,\sum_{l=1}^3(\nabla_A)_lF_{lm}=\mu\,\sum_{l=1}^3\p_{x_l}F_{lm}+[A_l,F_{lm}]=\frac{\mu}{\ep^2}[u,(\nabla_A)_mu]+\frac{1}{\ep}A_m
\]
Thus
\[
\begin{array}{l}
\ds 2^{-1}\,\Delta |\nabla_A u|^2=\sum_{l=1}^3(\nabla_A)_l(\nabla_A u)\cdot (\nabla_A)_l(\nabla_A u)+\frac{1}{\ep^2} \sum_{l,m=1}^3|[u,(\nabla_A)_mu]|^2\\[5mm]
\ds+\sum_{l,m=1}^3(\nabla_A)_mu\cdot (\nabla_A)_m((\nabla_A)_l((\nabla_A)_l u))+ 2\sum_{l,m=1}^3(\nabla_A)_mu\cdot [F_{lm}, (\nabla_A)_l u] \\[5mm]
\ds+\frac{1}{\mu\,\ep}\,\sum_{l,m=1}^3(\nabla_A)_mu\cdot [A_m,u]
\end{array}
\]
The Euler Lagrange equation is implying
\[
\sum_{l=1}^3(\nabla_A)_l((\nabla_A)_l u)=-\,\frac{\la}{\ep^2}\, u\, (1-|u|^2)
\]
This gives
\[
\begin{array}{l}
\ds\sum_{l,m=1}^3(\nabla_A)_mu\cdot (\nabla_A)_m((\nabla_A)_l((\nabla_A)_l u))=-\la\,\sum_{m=1}^3(\nabla_A)_mu\cdot (\nabla_A)_m( u\, (1-|u|^2))\\[5mm]
\ds=-\,\frac{\la}{\ep^2}\, |\nabla_A u|^2\,(1-|u|^2)+\frac{\la}{2\ep^2}\, \sum_{m=1}^3|\nabla|u|^2|^2
\end{array}
\]
Hence
\[
\begin{array}{l}
\ds 2^{-1}\,\Delta |\nabla_A u|^2=\sum_{l=1}^3(\nabla_A)_l(\nabla_A u)\cdot (\nabla_A)_l(\nabla_A u)+\frac{1}{\ep^2} \sum_{l,m=1}^3|[u,(\nabla_A)_mu]|^2\\[5mm]
\ds+ 2\sum_{l,m=1}^3(\nabla_A)_mu\cdot [F_{lm}, (\nabla_A)_l u] +\frac{1}{\mu\,\ep}\,\sum_{l,m=1}^3[u,(\nabla_A)_mu]\cdot A_m\\[5mm]
\ds-\,\frac{\la}{\ep^2}\, |\nabla_A u|^2\,(1-|u|^2)+\frac{\la}{2\ep^2}\, \sum_{m=1}^3|\nabla|u|^2|^2
\end{array}
\]
Now we compute $2^{-1}\,\Delta |F_A|^2$. We have
\[
2^{-1}\,\p_{x_n}|F_A|^2=2^{-1}\sum_{l,m=1}^3\p_{x_n}(F_{lm}\cdot F_{lm})= \sum_{l,m=1}^3(\nabla_A)_n F_{lm}\cdot F_{lm}
\]
Hence
\[
2^{-1}\,\Delta |F_A|^2=\sum_{l,m,n=1}^3(\nabla_A)_n F_{lm}\cdot (\nabla_A)_n F_{lm}+\sum_{l,m,n=1}^3(\nabla_A)_n(\nabla_A)_n F_{lm}\cdot F_{lm}
\]
We have
\[
F_{lm}=(\nabla_A)_l A_m-(\nabla_A)_m A_l+[A_m,A_l]
\]
Recall
\[
\lf\{
\begin{array}{l}
\ds(\nabla A)_n(\nabla A)_l A_m=(\nabla A)_l(\nabla A)_n A_m+[F_{nl},A_m]\\[5mm]
\ds-(\nabla A)_n(\nabla A)_m A_l=-(\nabla A)_m(\nabla A)_n A_l-[F_{nm},A_l]
\end{array}
\rg.
\]
Hence
\[
\begin{array}{l}
\ds(\nabla_A)_n(\nabla_A)_n F_{lm}=(\nabla_A)_n(\nabla_{A})_l(\nabla_A)_n A_m+(\nabla_A)_n([F_{nl},A_m])\\[5mm]
\ds-(\nabla_A)_n(\nabla_{A})_m(\nabla_A)_n A_l-(\nabla_A)_n([F_{nm},A_l])+(\nabla_A)_n(\nabla_A)_n ([A_m,A_l])
\end{array}
\]
We write respectively
\[
\lf\{
\begin{array}{l}
\ds (\nabla_{A})_l(\nabla_A)_n A_m=(\nabla_{A})_lF_{nm}+(\nabla_{A})_l(\nabla_A)_m A_n+(\nabla_{A})_l([A_n,A_m]\\[5mm]
\ds-(\nabla_{A})_m(\nabla_A)_n A_l=-(\nabla_{A})_mF_{nl}-(\nabla_{A})_m(\nabla_A)_l A_n-(\nabla_{A})_m([A_n,A_l]
\end{array}
\rg.
\]
Thus
\[
\begin{array}{l}
\ds(\nabla_A)_n(\nabla_A)_n F_{lm}=(\nabla_A)_n (\nabla_{A})_l F_{nm}+(\nabla_A)_n(\nabla_{A})_l(\nabla_A)_m A_n+(\nabla_A)_n(\nabla_{A})_l([A_n,A_m]\\[5mm]
-(\nabla_A)_n (\nabla_{A})_m F_{nl}-(\nabla_A)_n(\nabla_{A})_m(\nabla_A)_l A_n-(\nabla_A)_n(\nabla_{A})_m([A_n,A_l]\\[5mm]
+(\nabla_A)_n([F_{nl},A_m])-(\nabla_A)_n([F_{nm},A_l])+(\nabla_A)_n(\nabla_A)_n ([A_m,A_l])\\[5mm]
\end{array}
\]
We have respectively
\[
\lf\{
\begin{array}{l}
\ds(\nabla_A)_n (\nabla_{A})_l F_{nm}=(\nabla_A)_l (\nabla_{A})_n F_{nm}+[F_{nl},F_{nm}]\\[5mm]
\ds-(\nabla_A)_n (\nabla_{A})_m F_{nl}=-(\nabla_A)_m (\nabla_{A})_n F_{nl}-[F_{nm},F_{nl}]\\[5mm]
\ds(\nabla_A)_n(\nabla_{A})_l(\nabla_A)_m A_n-(\nabla_A)_n(\nabla_{A})_m(\nabla_A)_l A_n=(\nabla_A)_n([F_{lm},A_n])
\end{array}
\rg.
\]
Thus
\[
\begin{array}{l}
\ds\ds(\nabla_A)_n(\nabla_A)_n F_{lm}=(\nabla_A)_l (\nabla_{A})_n F_{nm}-(\nabla_A)_m (\nabla_{A})_n F_{nl}\\[5mm]
\ds+2\,[F_{nl},F_{nm}]+(\nabla_A)_n([F_{lm},A_n])+(\nabla_A)_n([F_{nl},A_m])-(\nabla_A)_n([F_{nm},A_l])\\[5mm]
\ds+(\nabla_A)_n(\nabla_{A})_l([A_n,A_m])-(\nabla_A)_n(\nabla_{A})_m([A_n,A_l])+(\nabla_A)_n(\nabla_A)_n ([A_m,A_l])
\end{array}
\]
Recall that using Jacobi identity we have for any $w\in{\frak su}(2)$
\[
\begin{array}{l}
\ds(\nabla_A)_n[v,w]=[(\nabla_A)_nv,w]+[v,(\nabla_A)_nw]
\end{array}
\]
Thus
\[
\begin{array}{l}
\ds(\nabla_A)_n(\nabla_A)_n F_{lm}=(\nabla_A)_l (\nabla_{A})_n F_{nm}-(\nabla_A)_m (\nabla_{A})_n F_{nl}\\[5mm]
\ds+2\,[F_{nl},F_{nm}]+[(\nabla_A)_nF_{lm},A_n]+[F_{lm},(\nabla_A)_nA_n]\\[5mm]
\ds+[(\nabla_A)_nF_{nl},A_m]+[F_{nl},(\nabla_A)_nA_m]\\[5mm]
\ds-[(\nabla_A)_nF_{nm},A_l]-[F_{nm},(\nabla_A)_nA_l]\\[5mm]
\ds+(\nabla_A)_n([(\nabla_{A})_lA_n,A_m])+(\nabla_A)_n([A_n,(\nabla_{A})_lA_m])\\[5mm]
\ds-(\nabla_A)_n([(\nabla_{A})_mA_n,A_l])-(\nabla_A)_n([A_n,(\nabla_{A})_mA_l])\\[5mm]
+(\nabla_A)_n[(\nabla_A)_nA_m,A_l])+(\nabla_A)_n[A_m,(\nabla_A)_nA_l])
\end{array}
\]
Recall from the Euler Lagrange equation that
\[
\lf\{
\begin{array}{l}
\ds\mu\sum_{n=1}^3 (\nabla_{A})_n F_{nm}=\frac{\mu}{\ep^2}[u,(\nabla_A)_mu]+\frac{1}{\ep}A_m\\[5mm]
\ds\mu\sum_{n=1}^3 (\nabla_{A})_n F_{nl}=\frac{\mu}{\ep^2}[u,(\nabla_A)_lu]+\frac{1}{\ep}A_l
\end{array}
\rg.
\]
Hence
\[
\begin{array}{l}
\ds\sum_{n=1}^3 \lf((\nabla_A)_l (\nabla_{A})_n F_{nm}-(\nabla_A)_m (\nabla_{A})_n F_{nl}\rg)\cdot F_{lm}\\[5mm]
\ds=(\nabla_A)_l \lf(\frac{1}{\ep^2}[u,(\nabla_A)_mu]+\frac{1}{\mu\ep}A_m\rg)\cdot F_{lm}-(\nabla_A)_m \lf(\frac{1}{\ep^2}[u,(\nabla_A)_lu]+\frac{1}{\mu\ep}A_l\rg)\cdot F_{lm}\\[5mm]
\ds=\frac{2}{\ep^2}\,[(\nabla_A)_lu,(\nabla_A)_mu]\cdot F_{lm}+\frac{1}{\ep^2}[u, (\nabla_A)_l(\nabla_A)_mu-(\nabla_A)_m(\nabla_A)_lu]\cdot F_{lm}\\[5mm]
\ds+\frac{1}{\mu\ep} \lf( (\nabla_A)_l A_m- (\nabla_A)_m A_l  \rg)\cdot F_{lm}\\[5mm]
\ds=\frac{2}{\ep^2}\,[(\nabla_A)_lu,(\nabla_A)_mu]\cdot F_{lm}+\frac{1}{\ep^2}|[F_{lm},u]|^2+\frac{1}{\mu\ep}|F_{lm}|^2+\frac{1}{\mu\ep}[A_l,A_m]\cdot F_{lm}\ .
\end{array}
\]
Moreover
\[
\begin{array}{l}
\ds\sum_{l,m,n=1}^3[(\nabla_A)_nF_{nl},A_m]\cdot F_{lm}-[(\nabla_A)_nF_{nm},A_l]\cdot F_{lm}\\[5mm]
\ds=\frac{2}{\ep^2}\,\sum_{l,m,n=1}^3[[u,(\nabla_A)_lu],A_m]\cdot F_{lm}+\frac{2}{\mu\ep}\sum_{l,m,n=1}^3[A_l,A_m]\cdot F_{lm}\ .
\end{array}
\]
We have
\[
\sum_{l,m,n=1}^3[F_{lm},(\nabla_A)_nA_n]\cdot F_{lm}=\sum_{l,m,n=1}^3[F_{lm},F_{lm}]\cdot (\nabla_A)_nA_n=0\ .
\]
We have also
\[
\begin{array}{l}
\ds\sum_{l,m,n=1}^3[F_{nl},(\nabla_A)_nA_m]\cdot F_{lm}-[F_{nm},(\nabla_A)_nA_l]\cdot F_{lm}=2\sum_{l,m,n=1}^3[F_{nl},(\nabla_A)_nA_m]\cdot F_{lm}\ .
\end{array}
\]
Moreover
\[
\begin{array}{l}
\ds\sum_{l,m,n=1}^3(\nabla_A)_n([(\nabla_{A})_lA_n,A_m])\cdot F_{lm}+(\nabla_A)_n([A_n,(\nabla_{A})_lA_m])\cdot F_{lm}\\[5mm]
\ds-\sum_{l,m,n=1}^3(\nabla_A)_n([(\nabla_{A})_mA_n,A_l])\cdot F_{lm}-\sum_{l,m,n=1}^3(\nabla_A)_n([A_n,(\nabla_{A})_mA_l])\cdot F_{lm}\\[5mm]
\ds+\sum_{l,m,n=1}^3(\nabla_A)_n[(\nabla_A)_nA_m,A_l])\cdot F_{lm}+\sum_{l,m,n=1}^3(\nabla_A)_n[A_m,(\nabla_A)_nA_l])\cdot F_{lm}\\[5mm]
\ds=2\,\sum_{l,m,n=1}^3(\nabla_A)_n([(\nabla_{A})_lA_n,A_m])\cdot F_{lm}+(\nabla_A)_n([A_n,(\nabla_{A})_lA_m])\cdot F_{lm}\\[5mm]
\ds+2\,\sum_{l,m,n=1}^3(\nabla_A)_n[(\nabla_A)_nA_m,A_l])\cdot F_{lm}
\end{array}
\]
We have
\[
\begin{array}{l}
\ds2\sum_{l,m,n=1}^3(\nabla_A)_n([A_n,(\nabla_{A})_lA_m])\cdot F_{lm}=\sum_{l,m,n=1}^3(\nabla_A)_n([A_n,(\nabla_{A})_lA_m-(\nabla_{A})_mA_l])\cdot F_{lm}\\[5mm]\ds=\sum_{l,m,n=1}^3(\nabla_A)_n([A_n,F_{lm}])\cdot F_{lm}+\sum_{l,m,n=1}^3(\nabla_A)_n([A_n,[A_l,A_m]])\cdot F_{lm}\\[5mm]
\ds=\sum_{l,m,n=1}^3[A_n,(\nabla_A)_nF_{lm}]\cdot F_{lm}+\sum_{l,m,n=1}^3(\nabla_A)_n([A_n,[A_l,A_m]])\cdot F_{lm}
\end{array}
\]
Hence
\[
\begin{array}{l}
\ds\sum_{l,m,n=1}^3(\nabla_A)_n(\nabla_A)_n F_{lm}\cdot F_{lm}\\[5mm]
\ds=\frac{2}{\ep^2}\,\sum_{l,m=1}^3[(\nabla_A)_lu,(\nabla_A)_mu]\cdot F_{lm}+\frac{1}{\ep^2}\sum_{l,m=1}^3|[F_{lm},u]|^2+\frac{1}{\mu\ep}|F_{lm}|^2+\frac{1}{\mu\ep}\sum_{l,m=1}^3[A_l,A_m]\cdot F_{lm}\\[5mm]
\ds+2\,\sum_{l,m,n=1}^3[F_{nl},F_{nm}]\cdot F_{lm}+\sum_{l,m,n=1}^3[(\nabla_A)_nF_{lm},A_n]\cdot F_{lm}\\[5mm]
\ds+\frac{2}{\ep^2}\,\sum_{l,m,n=1}^3[[u,(\nabla_A)_lu],A_m]\cdot F_{lm}+\frac{2}{\mu\ep}\sum_{l,m,n=1}^3[A_l,A_m]\cdot F_{lm}+2\,\sum_{l,m,n=1}^3[F_{nl},(\nabla_A)_nA_m]\cdot F_{lm}\\[5mm]
\ds+2\,\sum_{l,m,n=1}^3(\nabla_A)_n([(\nabla_{A})_lA_n,A_m])\cdot F_{lm}+\sum_{l,m,n=1}^3[A_n,(\nabla_A)_nF_{lm}]\cdot F_{lm}\\[5mm]
\ds+\sum_{l,m,n=1}^3(\nabla_A)_n([A_n,[A_l,A_m]])\cdot F_{lm}+2\,\sum_{l,m,n=1}^3(\nabla_A)_n[(\nabla_A)_nA_m,A_l])\cdot F_{lm}
\end{array}
\]
Which gives
\[
\begin{array}{l}
\ds\sum_{l,m,n=1}^3(\nabla_A)_n(\nabla_A)_n F_{lm}\cdot F_{lm}\\[5mm]
\ds=\frac{2}{\ep^2}\,\sum_{l,m=1}^3[(\nabla_A)_lu,(\nabla_A)_mu]\cdot F_{lm}+\frac{1}{\ep^2}\sum_{l,m=1}^3|[F_{lm},u]|^2+\frac{1}{\mu\ep}|F_{lm}|^2+\frac{3}{\mu\ep}\sum_{l,m=1}^3[A_l,A_m]\cdot F_{lm}\\[5mm]
\ds+2\,\sum_{l,m,n=1}^3[F_{nl},F_{nm}]\cdot F_{lm}\\[5mm]
\ds+\frac{2}{\ep^2}\,\sum_{l,m,n=1}^3[[u,(\nabla_A)_lu],A_m]\cdot F_{lm}+2\,\sum_{l,m,n=1}^3[F_{nl},(\nabla_A)_nA_m]\cdot F_{lm}\\[5mm]
\ds+\sum_{l,m,n=1}^3(\nabla_A)_n([A_n,[A_l,A_m]])\cdot F_{lm}+2\,\sum_{l,m,n=1}^3(\nabla_A)_n[(\nabla_A)_nA_m-(\nabla_{A})_mA_n,A_l])\cdot F_{lm}
\end{array}
\]
We have
\[
\begin{array}{l}
\ds 2\,\sum_{l,m,n=1}^3(\nabla_A)_n[(\nabla_A)_nA_m-(\nabla_{A})_mA_n,A_l])\cdot F_{lm}\\[5mm]
\ds= 2\,\sum_{l,m,n=1}^3(\nabla_A)_n[F_{nm},A_l]\cdot F_{lm}+ 2\,\sum_{l,m,n=1}^3(\nabla_A)_n[[A_n,A_m],A_l]\cdot F_{lm}\\[5mm]
\ds= 2\,\sum_{l,m,n=1}^3(\nabla_A)_n[F_{nm},A_l]\cdot F_{lm}\\[5mm]
\ds+\sum_{l,m,n=1}^3(\nabla_A)_n[[A_n,A_m],A_l]\cdot F_{lm}-\sum_{l,m,n=1}^3(\nabla_A)_n[[A_n,A_l],A_m]\cdot F_{lm}
\end{array}
\]
Observe that
\[
\begin{array}{l}
\ds\sum_{l,m,n=1}^3(\nabla_A)_n([A_n,[A_l,A_m]])\cdot F_{lm}+\sum_{l,m,n=1}^3(\nabla_A)_n[[A_n,A_m],A_l]\cdot F_{lm}\\[5mm]
\ds-\sum_{l,m,n=1}^3(\nabla_A)_n[[A_n,A_l],A_m]\cdot F_{lm}\\[5mm]
\ds=\sum_{l,m,n=1}^3(\nabla_A)_n\lf([[A_m,A_l],A_n]+[[A_n,A_m],A_l]+[[A_l,A_n],A_m]\rg)\cdot F_{lm}=0
\end{array}
\]
Hence
\[
\begin{array}{l}
\ds\sum_{l,m,n=1}^3(\nabla_A)_n(\nabla_A)_n F_{lm}\cdot F_{lm}\\[5mm]
\ds=\frac{2}{\ep^2}\,\sum_{l,m=1}^3[(\nabla_A)_lu,(\nabla_A)_mu]\cdot F_{lm}+\frac{1}{\ep^2}\sum_{l,m=1}^3|[F_{lm},u]|^2+\frac{1}{\mu\ep}|F_{lm}|^2+\frac{3}{\mu\ep}\sum_{l,m=1}^3[A_l,A_m]\cdot F_{lm}\\[5mm]
\ds+2\,\sum_{l,m,n=1}^3[F_{nl},F_{nm}]\cdot F_{lm}\\[5mm]
\ds+\frac{2}{\ep^2}\,\sum_{l,m,n=1}^3[[u,(\nabla_A)_lu],A_m]\cdot F_{lm}+2\,\sum_{l,m,n=1}^3[F_{nl},(\nabla_A)_nA_m]\cdot F_{lm}\\[5mm]
\ds+2\,\sum_{l,m,n=1}^3(\nabla_A)_n[F_{nm},A_l]\cdot F_{lm}\
\end{array}
\]
We write
\[
\begin{array}{l}
\ds2\,\sum_{l,m,n=1}^3[F_{nl},(\nabla_A)_nA_m]\cdot F_{lm}+2\,\sum_{l,m,n=1}^3(\nabla_A)_n[F_{nm},A_l]\cdot F_{lm}\\[5mm]
\ds=+2\,\sum_{l,m,n=1}^3[F_{nl},(\nabla_A)_nA_m]\cdot F_{lm}+2\,\sum_{l,m,n=1}^3[(\nabla_A)_nF_{nm},A_l]\cdot F_{lm}+2\,\sum_{l,m,n=1}^3[F_{nl},(\nabla_A)_nA_m]\cdot F_{ml}\\[5mm]
\ds=\frac{2}{\ep^2}\,\sum_{l,m}^3[[u,(\nabla_A)_mu],A_l]\cdot F_{lm}+\frac{2}{\mu\ep}\,\sum_{l,m}^3[A_m,A_l]\cdot F_{lm}
\end{array}
\]
Thus finally
\[
\begin{array}{l}
\ds\sum_{l,m,n=1}^3(\nabla_A)_n(\nabla_A)_n F_{lm}\cdot F_{lm}\\[5mm]
\ds=\frac{2}{\ep^2}\,\sum_{l,m=1}^3[(\nabla_A)_lu,(\nabla_A)_mu]\cdot F_{lm}+\frac{1}{\ep^2}\sum_{l,m=1}^3|[F_{lm},u]|^2+\frac{1}{\mu\ep}|F_{lm}|^2+\frac{1}{\mu\ep}\sum_{l,m=1}^3[A_l,A_m]\cdot F_{lm}\\[5mm]
\ds+2\,\sum_{l,m,n=1}^3[F_{nl},F_{nm}]\cdot F_{lm}\\[5mm]
\end{array}
\]
This implies proposition~\ref{pr-bw}.\hfill $\Box$
\subsubsection{Bochner -Weitzenb\"ock for the transversal part of the Curvature}
One of the major observation in \cite{JT} is that the transversal part of curvature of the non abelian $SU(2)-YMH$ Fields has a shorter range of propagation and enjoys an exponential decrease.
This is due to the following identities where the Bessel potentials instead of Riesz potentials are involved. We now give the $CYMH$ counterpart of this main result in \cite{JT}.
\begin{Prop}
\label{pr-bwt}
Let $(u,A)$ be a smooth critical point of $CYMH_\ep^{\la,1}$ and assume $1\ge |u|^2\ge 1-\delta$ and  $(1+\la)\,\delta\le1/2$ then we have
\be
\label{bwt}
\begin{array}{l}
\ds -\mbox{div}\lf(|u|^{-4}\,\nabla |[u,F_A]|\rg)+\frac{1}{2\ep^2}|[u,F_A]|\le 8\, |u|^{-6}\, |\nabla_A\,u|^2\,|F_A|+4\, |u|^{-4}\,|\nabla_Au|\,|\nabla(F_A\cdot u)|\\[5mm]
\ds\qquad+4\,\ep^{-2}|u|^{-4}\,|\nabla_A u|^2+2\,\mu^{-1}\,\ep^{-1}\,|u|^{-4}\,|A|^2+4\,|u|^{-4}\,|F_A|^2\ .
\end{array}
\ee
\hfill $\Box$
\end{Prop}
\noindent{\bf Proof of Proposition~\ref{pr-bwt}.} We use ${\mathbb Z}_3$ indexation and we write
\[
F_j:=F_{j+1, j-1}
\]
We have
\[
\begin{array}{l}
\ds\Delta |[u,F_A]|^2=\sum_{l=1}^3\p^2_{x_l^2}|[u,F_A]|^2=2\,\sum_{l,j=1}^3\p_{x_l}\lf([u,F_j]\cdot(\nabla_{A})_l[u,F_j]\rg)\\[5mm]
\ds=2\, \sum_{l,j=1}^3|(\nabla_{A})_l[u,F_j]|^2+2\,\sum_{l,j=1}^3[u,F_j]\cdot(\nabla_A)_l(\nabla_{A})_l[u,F_j]
\end{array}
\]
We have
\[
\begin{array}{l}
\ds\sum_{l=1}^3(\nabla_A)_l(\nabla_{A})_l[u,F_j]=\sum_{l=1}^3[(\nabla_A)_l(\nabla_{A})_lu,F_j]+2\, \sum_{l=1}^3[(\nabla_{A})_lu,(\nabla_A)_lF_j]+\sum_{l=1}^3[u,(\nabla_A)_l(\nabla_{A})_lF_j]\\[5mm]
\ds=-\frac{\la}{\ep^2}\,(1-|u|^2)\,[u,F_j]+2\, \sum_{l=1}^3[(\nabla_{A})_lu,(\nabla_A)_lF_j]+\sum_{l=1}^3[u,(\nabla_A)_l(\nabla_{A})_lF_j]
\end{array}
\]
We have 
\[
[u,(\nabla_A)_l(\nabla_{A})_lF_j]\cdot [u,F_j]=|u|^2\, (\nabla_A)_l(\nabla_{A})_lF_j\cdot F_j-u\cdot (\nabla_A)_l(\nabla_{A})_lF_j\ (u\cdot F_j)
\]
We have seen
\[
\begin{array}{l}
\ds\sum_{l,j=1}^3(\nabla_A)_l(\nabla_A)_l F_j\cdot F_j\\[5mm]
\ds=\frac{1}{\ep^2}\,\sum_{l,m=1}^3[(\nabla_A)_{j+1}u,(\nabla_A)_{j-1}u]\cdot F_{j}+\frac{1}{\ep^2}\sum_{j=1}^3|[F_{j},u]|^2+\frac{1}{\mu\,\ep}|F_{j}|^2+\frac{1}{2\,\mu\ep}\sum_{l,m=1}^3[A_{j+1},A_{j-1}]\cdot F_{j}\\[5mm]
\ds+\,\frac{1}{2}\sum_{j,n=1}^3[F_{n,j+1},F_{n,j-1}]\cdot F_{j}\\[5mm]
\end{array}
\]
We have also
\[
\begin{array}{l}
\ds-\,u\cdot (\nabla_A)_l(\nabla_{A})_lF_j\ (u\cdot F_j)=-(u\cdot F_j)\,\Delta(u\cdot F_j)+2\,(\nabla_A)_lu\cdot(\nabla_A)_lF_j\ (u\cdot F_j)\\[5mm]
\ds+\, (\nabla_A)_l(\nabla_A)_lu\cdot F_j\ (u\cdot F_j)
\end{array}
\]
We have on a 2-form $B$
\[
(d^\ast_A d_A+d_Ad^\ast_A B)_{lm}=-\sum_{n=1}^3(\nabla_A)_n(\nabla_A)_n B_{lm}+\sum_{n=1}^3[F_{nl},B_{nm}]-[F_{nm},B_{nl}] 
\]
Hence, since
\[
d^\ast_AF_A=-\frac{[u, d_Au]}{\ep^2}- \frac{A}{\mu\ep}
\]
we have
\[
-d_A\lf(\frac{[u, d_Au]}{\ep^2}\rg)- \frac{d_AA}{\mu\ep}=(d^\ast_A d_A+d_Ad^\ast_A )F_A=\nabla_A^\ast\nabla_A F+2\, [F_{A}\res F_A]
\]
which gives
\[
\lf[u,\sum_{l=1}^3(\nabla_A)_l(\nabla_A)_l F_j\rg]=\ep^{-2}\,[u,d_A\lf( [u, d_Au]\rg)]_j+\mu^{-1}\,\ep^{-1}\,[u,d_AA]_j+2\,[u, [F_{j-1},F_{j+1}]]
\]
Observe
\[
\begin{array}{l}
\ds[u,d_A\lf( [u, d_Au]\rg)]=\lf[u, d_A\lf(\sum_{i=1}^3 [u,(\nabla_A)_iu]\ dx_i\rg)\rg]\\[5mm]
\ds=\lf[u, \sum_{i,j=1}^3 [(\nabla_A)_ju,(\nabla_A)_iu]\rg]\ dx_j\wedge dx_i+\sum_{i<j}[u,[u,(\nabla_A)_j(\nabla_A)_iu-(\nabla_A)_i(\nabla_A)_j]]\ dx_j\wedge dx_i\\[5mm]
\ds=\lf[u, \sum_{i,j=1}^3 [(\nabla_A)_ju,(\nabla_A)_iu]\rg]\ dx_j\wedge dx_i+ [u,[u,[F_A,u]]]
\end{array}
\]
and
\[
d_AA=d_A\lf(\sum_{m=1}^3 A_m\, dx_m\rg)=\sum_{l,m=1}^3(\p_{x_l}A_m+[A_l,A_m])\ dx_l\wedge dx_m=F_A+A\wedge A
\]
Hence
\[
\begin{array}{l}
\ds\lf[u,\sum_{l=1}^3(\nabla_A)_l(\nabla_A)_l F_j\rg]=2\,\ep^{-2}\, \lf[u, [(\nabla_A)_{j+1}u,(\nabla_A)_{j-1}u]\rg]\\[5mm]
\ds+\ep^{-2}\,[u,[u,[F_j,u]]]+\mu^{-1}\,\ep^{-1}\,[u, F_j]+\mu^{-1}\ep^{-1}\,[u,[A_{j+1},A_{j-1}]]+2\,[u, [F_{j-1},F_{j+1}]]
\end{array}
\]
This gives finally
\[
\begin{array}{l}
\ds\Delta |[u,F_A]|^2=2\, \sum_{l,j=1}^3|(\nabla_{A})_l[u,F_j]|^2+2\,\sum_{l,j=1}^3[u,F_j]\cdot(\nabla_A)_l(\nabla_{A})_l[u,F_j]\\[5mm]
\ds=2\, \sum_{l,j=1}^3|(\nabla_{A})_l[u,F_j]|^2-2\,\frac{\la}{\ep^2}\,(1-|u|^2)\,|[u,F_j]|^2+4\, \sum_{l,j=1}^3[(\nabla_{A})_lu,(\nabla_A)_lF_j]\cdot[u,F_j]\\[5mm]
\ds+4\,\ep^{-2}\, \sum_{j=1}^3\lf[u, [(\nabla_A)_{j+1}u,(\nabla_A)_{j-1}u]\rg]\cdot[u,F_j]+2\,\ep^{-2}\,[u,[u,[F_j,u]]]\cdot[u,F_j]\\[5mm]
\ds+2\,\mu^{-1}\,\ep^{-1}\,|[u,F_A]|^2+2\,\mu^{-1}\,\ep^{-1}\, \sum_{j=1}^3[u,[A_{j+1},A_{j-1}]]\cdot[u,F_j]+4\,\sum_{j=1}^3[u, [F_{j-1},F_{j+1}]]\cdot[u,F_j]
\end{array}
\]
Observe that we have
\[
4\, [(\nabla_{A})_lu,(\nabla_A)_lF_j]=4\,|u|^{-2}\,u\cdot\p_{x_l}u \ [u, (\nabla_A)_lF_j]  -4\,[[u,[u,(\nabla_{A})_lu]],(\nabla_A)_lF_j]
\]
Hence
\[
\begin{array}{l}
\ds 4\, [(\nabla_{A})_lu,(\nabla_A)_lF_j]\cdot[u,F_j]=4\,|u|^{-2}\,u\cdot\p_{x_l}u \ [u, (\nabla_A)_lF_j] \cdot[u,F_j]\\[5mm]
\ds -4\,[[u,[u,(\nabla_{A})_lu]],(\nabla_A)_lF_j]\cdot[u,F_j]
\end{array}
\]
We have in one hand
\[
\begin{array}{l}
\ds4\,|u|^{-2}\,u\cdot\p_{x_l}u \ [u, (\nabla_A)_lF_j] \cdot[u,F_j]\\[5mm]
\ds=|u|^{-2}\,\p_{x_l}|u|^2\,\p_{x_l}|[u,F_j]|^2-2\, |u|^{-2}\,\p_{x_l}|u|^2\,[(\nabla_A)_lu,F_j]\cdot[u,F_j]
\end{array}
\]
Thus
\[
\begin{array}{l}
\ds 4\, [(\nabla_{A})_lu,(\nabla_A)_lF_j]\cdot[u,F_j]=|u|^{-2}\,\p_{x_l}|u|^2\,\p_{x_l}|[u,F_j]|^2-2\, |u|^{-2}\,\p_{x_l}|u|^2\,[(\nabla_A)_lu,F_j]\cdot[u,F_j]\\[5mm]
\ds  -4\,[[u,[u,(\nabla_{A})_lu]],(\nabla_A)_lF_j]\cdot[u,F_j]
\end{array}
\]
This gives finally
\[
\begin{array}{l}
\ds\Delta |[u,F_A]|^2=2\, \sum_{l,j=1}^3|(\nabla_{A})_l[u,F_j]|^2+2\,\frac{1}{\ep^2}\,[|u|^2-\la(1-|u|^2)]\,|[u,F_j]|^2\\[5mm]
\ds +|u|^{-2}\,\p_{x_l}|u|^2\,\p_{x_l}|[u,F_j]|^2-2\, |u|^{-2}\,\p_{x_l}|u|^2\,[(\nabla_A)_lu,F_j]\cdot[u,F_j]\\[5mm]
\ds  -4\,[[u,[u,(\nabla_{A})_lu]],(\nabla_A)_lF_j]\cdot[u,F_j]\\[5mm]
\ds+4\,\ep^{-2}\, \sum_{j=1}^3\lf[u, [(\nabla_A)_{j+1}u,(\nabla_A)_{j-1}u]\rg]\cdot[u,F_j]\\[5mm]
\ds+2\,\mu^{-1}\,\ep^{-1}\,|[u,F_A]|^2+2\,\mu^{-1}\,\ep^{-1}\, \sum_{j=1}^3[u,[A_{j+1},A_{j-1}]]\cdot[u,F_j]\\[5mm]
\ds+4\,\sum_{j=1}^3[u, [F_{j-1},F_{j+1}]]\cdot[u,F_j]
\end{array}
\]
 We also write $((\nabla_{A})_lu)^T$ the component of $(\nabla_{A})_lu$ which is transverse to $u$ and we have
\[
\begin{array}{l}
\ds -4\,[[u,[u,(\nabla_{A})_lu]],(\nabla_A)_lF_j]\cdot[u,F_j]=4\,[((\nabla_{A})_lu)^T,u]\cdot[u,F_j]\,u\cdot(\nabla_A)_lF_j\\[5mm]
\ds-4\,\,|u|^{-2}\,[((\nabla_{A})_lu)^T,[u,[u,(\nabla_{A})_lF_j]]]\cdot [u,F_j]\\[5mm]
\ds=4\,[(\nabla_{A})_lu,u]\cdot[u,F_j]\,\p_{x_l}\lf(u\cdot F_j\rg)-4\,[(\nabla_{A})_lu,u]\cdot[u,F_j]\,(\nabla_A)_lu\cdot F_j\\[5mm]
\ds-4\,\,|u|^{-2}\,[((\nabla_{A})_lu)^T,[u,[u,(\nabla_{A})_lF_j]]]\cdot [u,F_j]\
\end{array}
\]
Observe that both $((\nabla_{A})_lu)^T$ and $[u,[u,(\nabla_{A})_lF_j]]$ are in the plane perpendicular to $u$. Hence their bracket is parallel to $u$ and we deduce
\[
\begin{array}{l}
\ds-4\,\,|u|^{-2}\,[((\nabla_{A})_lu)^T,[u,[u,(\nabla_{A})_lF_j]]]\cdot [u,F_j]=0
\end{array}
\]
This gives finally
\[
\begin{array}{l}
\ds -4\,[[u,[u,(\nabla_{A})_lu]],(\nabla_A)_lF_j]\cdot[u,F_j]\\[5mm]
\ds=4\,[(\nabla_{A})_lu,u]\cdot[u,F_j]\,\p_{x_l}\lf(u\cdot F_j\rg)-4\,[(\nabla_{A})_lu,u]\cdot[u,F_j]\,(\nabla_A)_lu\cdot F_j
\end{array}
\]
We deduce
\[
\begin{array}{l}
\ds\Delta |[u,F_A]|^2=2\, \sum_{l,j=1}^3|(\nabla_{A})_l[u,F_j]|^2+2\,\frac{1}{\ep^2}\,[|u|^2-\la(1-|u|^2)]\,|[u,F_j]|^2\\[5mm]
\ds +|u|^{-2}\,\p_{x_l}|u|^2\,\p_{x_l}|[u,F_j]|^2-2\, |u|^{-2}\,\p_{x_l}|u|^2\,[(\nabla_A)_lu,F_j]\cdot[u,F_j]\\[5mm]
\ds+4\,[(\nabla_{A})_lu,u]\cdot[u,F_j]\,\p_{x_l}\lf(u\cdot F_j\rg)-4\,[(\nabla_{A})_lu,u]\cdot[u,F_j]\,(\nabla_A)_lu\cdot F_j\\[5mm]
\ds+4\,\ep^{-2}\, \sum_{j=1}^3\lf[u, [(\nabla_A)_{j+1}u,(\nabla_A)_{j-1}u]\rg]\cdot[u,F_j]\\[5mm]
\ds+2\,\mu^{-1}\,\ep^{-1}\,|[u,F_A]|^2+2\,\mu^{-1}\,\ep^{-1}\, \sum_{j=1}^3[u,[A_{j+1},A_{j-1}]]\cdot[u,F_j]\\[5mm]
\ds+4\,\sum_{j=1}^3[u, [F_{j-1},F_{j+1}]]\cdot[u,F_j]
\end{array}
\]
We assume $1-|u|^2\le \delta$ and we choose $\delta>0$ such that $[|u|^2-\la(1-|u|^2)]\ge (1-\delta-\la\delta)>2^{-1}$ and we observe that the following general identity holds for any  function $v$ such that $v^2\in C^2$
\[
\begin{array}{l}
\ds\Delta\sqrt{v^2+t^2}=\mbox{div}\lf[ \frac{v}{\sqrt{v^2+t^2}}\nabla v\rg]= \frac{v\,\Delta v}{\sqrt{v^2+t^2}}+\frac{|\nabla v|^2}{\sqrt{v^2+t^2}}-\frac{v^2\,|\nabla v|^2}{(v^2+t^2)^{3/2}}\\[5mm]
\ds=\frac{v\,\Delta v}{\sqrt{v^2+t^2}}+\frac{v^2\,|\nabla v|^2}{(v^2+t^2)^{3/2}}\ge\frac{1}{2}\frac{\Delta v^2-2\,|\nabla v|^2}{\sqrt{v^2+t^2}}
\end{array}
\]
We apply this identity to  $v=|[u,F_A]|$ to deduce after letting $t\rightarrow 0$ we obtain
\[
\begin{array}{l}
\ds\Delta |[u,F_A]|-4\,|u|^{-1}\,\p_{x_l}|u|\,\p_{x_l}|[u,F_A]|-\frac{1}{2\ep^2}|[u,F_A]|\ge-8\, |u|^{-2}\, |\nabla_A\,u|^2\,|F_A|\\[5mm]
\ds -4\, |\nabla_Au|\,|\nabla(F_A\cdot u)|-4\,\ep^{-2}|\nabla_A u|^2-2\,\mu^{-1}\,\ep^{-1}\,|A|^2-4\,|F_A|^2
\end{array}
\]
and this concludes the proof of proposition~\ref{pr-bwt}.\hfill $\Box$

\section{The Regularity of $CYMH$ Minimizers} 
\reset
In this part we are going to work with fixed parameters and prove the regularity of minimizers to the $CYMH$ Functional. Hence, to simplify the notations we  take all the parameters $\ep,\mu,\la$ equal to one. Precisely we prove the following result.
\begin{Th}
\label{th-reg}
Let $(u,A)\in {\mathcal E}$ be a minimizer of
\[
CYMH(u,A)=\frac{1}{2}\int_{\B^3}|A|^2+|\nabla_A u|^2+|F_A|^2\ +\frac{1}{2}(1-|u|^2)^2\ dx^3
\]
under the condition $u=\phi$ on $\p B^3$ where $\phi\in H^{1/2}(\p\B_1(0),{\S}^2$ then $u$ and $A$ are $C^\infty$ in the interior of $\B^3$.
\end{Th}
Before proving theorem~\ref{th-reg} we establish the following lemma.
\begin{Lm}
\label{lm-L6}
Let  $(u,A)$ be a stationary critical point of $CYMH$ on $\B^3$ then $d_Au$ and $F_A$ are in $L^6_{loc}(\B^3)$ and $u\in L^\infty_{loc}(\B^3)$.
\hfill $\Box$
\end{Lm}
\noindent{\bf Proof of lemma~\ref{lm-L6}}. Let $x_0\in \B^3$ and choose $r>0$ such that
\[
\|F_A\|_{L^{3/2}(\B_r(x_0))}\le \delta_2
\]
where $\delta_2$ is given by (\ref{II.0}). Hence there exists $g\in W^{1,2}(\B_r,SU(2))$ such that (\ref{II.1}),(\ref{II.2}), (\ref{II.3}) holds for $p=2$. In particular we have
\[
\|A^g\|_{L^{(3,3/2)}(\B_r(x_0))}+\|\nabla A^g\|_{L^{3/2}(\B_r(x_0))}\le C\ \|F_A\|_{L^{3/2}(\B_r(x_0))}\le C\ \delta_2
\]
where $L^{(3,3/2)}(\B_r(x_0))$ is the Lorentz space in which $W^{1,3/2}(\B_r(x_0))$ critically embeds (see \cite{Riv-Func}) and
\[
\|A^g\|_{L^{(6,2)}(\B_r(x_0))}+\|\nabla A^g\|_{L^2(\B_r(x_0))}\le C\ \|F_A\|_{L^2(\B_r(x_0))}
\]
Combining Bianchi identity together with the gauge identity formula (\ref{gau-ch}) second equation in the $CYMH$ system \ref{CYMHS} we obtain
\[
\lf\{
\begin{array}{l}
\ds d_{A^g}F_{A_g}=0\\[5mm]
\ds g\, d^\ast_{A^g}F_{A^g}\,g^{-1}=- [u,d_Au]-A\ .
\end{array}
\rg.
\]
This gives
\[
\lf\{
\begin{array}{l}
\ds d F_{A^g}=-[A^g\wedge  F_{A_g}]\qquad\mbox{ in }\B_r(x_0)\\[5mm]
\ds d^\ast F_{A^g}=-[A^g\res  F_{A_g}]-g\,[u,d_Au]\,g^{-1}-g\,A\, g^{-1}\qquad\mbox{ in }\B_r(x_0)\ .
\end{array}
\rg.
\]
This gives
\[
\begin{array}{l}
\ds-\Delta F_{A^g}=d\,d^\ast F_{A^g}+d^\ast\,d F_{A^g}\\[5mm]
\ds =-d^\ast\lf([A^g\wedge  F_{A_g}]\rg)-d([A^g\res  F_{A_g}])-d\lf( [u,d_{A^g}u^g]+A\rg)
\end{array}
\]
For any $p>3/2$
\[
\|[A^g\res  F]\|_{L^{3p/(p+3)}(\B_r)}+\|[A^g\wedge F]\|_{L^{3p/(p+3)}(\B_r)}\le  C\,\|A^g\|_{L^3({\B}_r)}\ \|F\|_{L^p(\B_r)}\le C\, \delta_2\  \|F\|_{L^p(\B_r)}\ .
\]
In particular for $p=6$ we have
\[
\|[A^g\res  F]\|_{L^{2}(\B_r)}+\|[A^g\wedge F]\|_{L^{2}(\B_r)}\le C\, \delta_2\  \|F\|_{L^6(\B_r)}
\]
For any choice of $B\in L^2(\Om^1(\B_{r/2})\otimes \frak{su}(2))$, classical result in elliptic theory give that, for $\delta_2$ chosen small enough any $L^2$ solution $F$ of
\[
-\Delta F=-d^\ast\lf([A^g\wedge  F]\rg)-d([A^g\res  F])-dB
\]
satisfies
\[
\|F\|_{L^6((\B_{r/2}(x_0))}+\sum_{k,l,m=1}^3\|\p_{x_k} F_{lm}\|_{L^2(\B_{r/2}(x_0))}\le C\, \|B\|_{L^2(\B_{r}(x_0))}+C\,r^{-1}\,\|F\|_{L^2(\B_{r}(x_0))}\ .
\]
Hence we have that $F_{A^g}\in W^{1,2}(\Om^2(\B_{r/2})\otimes \frak{su}(2))$ and in particular $F_{A^g}\in L^6(\Om^2(\B_{r/2})\otimes \frak{su}(2))$. Moreover the following inequality holds
\be
\label{curv-l6}
\|F_{A}\|_{L^6((\B_{r/2}(x_0))}\le \, C\, \lf(\|A\|_{L^2((\B_{r}(x_0))}+\|d_Au\|_{L^2((\B_{r}(x_0))}+C\,r^{-1}\,\|F\|_{L^2(\B_{r}(x_0))}\rg)
\ee

\medskip

Concerning now the covariant derivative of $u^g$ with respect to $A^g$ we have using the fact that $d^\ast A^g=0$ 
\[
\begin{array}{l}
\ds\sum_{l=1}^3(\nabla_{A^g})_l(\nabla_{A^g})_lu^g=\sum_{l=1}^3\p_{x_l}(\p_{x_l}u^g+[A^g_l,u])+[A^g_l, \p_{x_l}u^g+[A^g_l,u^g]]\\[5mm]
\ds=\Delta u^g+2\,\sum_{l=1}^3[A_l^g,\p_{x_l}u^g]+[A_l^g[A_l^g,u^g]]
\end{array}
\]
Hence
\be
\label{Hig-EL}
\Delta u^g+2\,\sum_{l=1}^3[A_l^g,\p_{x_l}u^g]+[A_l^g[A_l^g,u^g]]=-u^g\,(1-|u^g|^2)
\ee
Since $u^g\in L^4(\B_r,\frak{su}(2))$,  and $A^g\in L^6(\Om^1(\B_r)\otimes\,\frak{su}(2))$, there holds $[A^g,u^g]\in L^{12/5}(\Om^1(\B_r)\otimes\frak{su}(2))$. Combining this fact with
$d_{A^g}u^g\in L^2(\Om^1(\B_r)\otimes\frak{su}(2))$ we obtain that $u^g\in W^{1,2}(\B_r,\frak{su}(2))$. We have also using Kato inequality
\[
|\nabla|u||\le |\nabla_Au|
\]
Thus
\be
\label{L6u}
\|u\|_{L^6({\B}_r(x_0))}\le \, \|\nabla_A u\|_{L^2(\B_r(x_0))}+\, r^{-1/4}\, \|u\|_{L^4(\B_r(x_0))}\ .
\ee
We have the a-priori estimate for any $v\in W^{1,p}(\B_r,\frak{su}(2))$ and $p>3/2$
\[
\|[A_l^g,\p_{x_l}v]\|_{L^{{3p}/(p+3)}(\B_r)}\le\,\|A^g\|_{L^3(\B_r) }\ \|\p_{x_l} v\|_{L^p(\B_r)}\le C\,\delta_2 \ \|\p_{x_l} v\|_{L^p(\B_r)}
\]
and for $3>p>3/2$
\[
\|[A_l^g[A_l^g,v]]\|_{L^{{3p}/(p+3)}(\B_r)}\le\,\|A^g\|^2_{L^3(\B_r) }\ \| v\|_{L^{{3p}/{(3-p)}}(\B_r)}\le C\,\delta_2^2\ \| v\|_{W^{1,p}(\B_r)}\ ,
\]
and  we have for any $p=3$, using Lorentz spaces estimates we have respectively
\[
\|[A_l^g,\p_{x_l}v]\|_{L^{(3/2,1)}(\B_r)}\le\,\|A^g\|_{L^{(3,3/2)}(\B_r) }\ \|\p_{x_l} v\|_{L^3(\B_r)}\le C\,\delta_2 \ \|\p_{x_l} v\|_{L^3(\B_r)}\ ,
\]
and 
\[
\|[A_l^g[A_l^g,v]]\|_{L^{(3/2,1)}(\B_r)}\le\,\|A^g\|^2_{L^{(3,2)}(\B_r) }\ \|v\|_\infty\le C\,\delta_2^2\ \| v\|_{W^{1,(3,1)}(\B_r)}\ .
\]
Hence, classical elliptic estimates give that $W^{1,2}$ solutions to the system 
\[
\Delta v+2\,\sum_{l=1}^3[A_l^g,\p_{x_l}v]+[A_l^g[A_l^g,v]]=f\ ,
\]
where $f\in L^{(3/2,1)}(\B_r\otimes \frak{su}(2))$ are in $L^\infty_{loc}(\B_r)$ and in particular satisfy
\[
\|v\|_{L^\infty(\B_{3r/4})}\le C\ \|f\|_{L^{(3/2,1)}(\B_r))}+C\,r^{-1/2}\,\|v\|_{L^6(\B_r)}\ .
\]
Hence in particular $u^g\in L^\infty(\B_{3r/4})$ and there holds
\be
\label{linftyug}
\begin{array}{l}
\ds\|u\|_{L^\infty(\B_{3r/4}(x_0))}\le C\, r^{1/2}\, \||u|^3\|_{L^2({\B}_r(x_0))}+C\, r^{-1/2}\, \|u\|_{L^6({\B}_r(x_0))}\\[5mm]
\ds\le C \,  r^{1/2}\, \lf( \|\nabla_A u\|^3_{L^2(\B_r(x_0))}+\, r^{-3/4}\, \|u\|^3_{L^4(\B_r(x_0))}\rg)\\[5mm]
\ds\qquad+C\, r^{-1/2}\, \lf( \|\nabla_A u\|_{L^2(\B_r(x_0))}+\, r^{-1/4}\, \|u\|_{L^4(\B_r(x_0))}\rg)
\end{array}
\ee
This gives that $[A_l^g[A_l^g,u^g]]\in L^2(\B_{3r/4})$ and bootstrapping further in the system (\ref{Hig-EL}), assuming $\delta_2$ is sufficiently small, one obtains that
\[
\begin{array}{l}
\ds \|\nabla u^g\|_{L^{(6,2)}(\B_{r/2}(x_0))} +\|\nabla^2 u^g\|_{L^2(\B_{r/2}(x_0))}\\[5mm]
\ds\quad\le C\, \|u^g\|_{L^\infty(\B_{3r/4}(x_0))}\, \|A^g\|^2_{L^4(\B_{3r/4}(x_0))}+ \|u^g\|_{L^\infty(\B_{3r/4}(x_0))}^3+C\,r^{-1}\, \|\nabla u^g\|_{L^{2}(\B_{r}(x_0))} \ .
\end{array}
\]
This concludes the proof of lemma~\ref{lm-L6}.\hfill $\Box$

\medskip

We shall then use the following Lemma
\begin{Lm}
\label{lm-smal}
Let  $(u,A)$ be a stationary critical point of $CYMH$ on $\B^3$. Let $r_1<1$ and let $\delta>0$. There exists $r_0>0$  depending only on 
$\|d_Au\|_{L^6(B_{r_1}(0))}$, $\|F_A\|_{L^6(B_{r_1}(0))}$ and $\|u\|_{L^\infty(B_{r_1}(0))}$  such that  for any $x_0\in \B_{r_1}(0)$ and any $r<r_0$ for which $\B_{r}(x_0)\subset \B_{r_1}(0)$  assume
\[
\frac{1}{r}\int_{\B_r(x_0)}|A|^2\ dx^3<\delta, 
\]
then, for any $y\in B_{r/2}(x_0)$ and any $\rho<r/4$
\[
\frac{1}{\rho}\int_{\B_\rho(y)}|A|^2\ dx^3<5\,\delta, 
\]
where $C$ is universal.\hfill $\Box$
\end{Lm}
\noindent{\bf Proof of lemma~\ref{lm-smal}}
 Let $B_{r}(y)\subset \B_{r_1}(0)$. There holds
\[
\frac{1}{r}\int_{ B_r(y)}{|\nabla_Au|^2}\ dx^3\le \frac{1}{r}\ (4/3\,\pi\, r^3)^{2/3}\ \|\nabla_Au\|_{L^6(B_{r_1}(0))}^2\le (4/3)^{2/3}\, {r}\, \ \|\nabla_Au\|_{L^6(B_{r_1}(0))}^2
\]
moreover
\[
\frac{1}{r}\int_{ B_r(y)}\frac{3}{2}(1-|u|^2)^2\ dx^3\le \frac{4\,\pi}{2}\, r^2\, \lf(1+\|u\|^2_{L^\infty(B_{r_1}(0))}\rg)^2
\]
and
\[
\int_{B_r}\lf(\frac{2}{|x-y|}-\frac{1}{r}\rg)\,|F_A|^2\ dx^3\le 2\,\lf(\frac{8\pi}{3}\rg)^{2/3}\,r\, \|F_A\|^2_{L^6(B_{r_1}(0))}\
\]
where $C>0$ in the  inequalities is universal. We choose $r_0>0$ such that
\be
\label{r0ch}
\lf(\frac{4}{3}\rg)^{2/3}  {r_0}\,\|\nabla_Au\|_{L^6(B_{r_1}(0))}^2+\frac{4\,\pi}{2}\, r_0^2\,\lf(1+\|u\|^2_{L^\infty(B_{r_1}(0))}\rg)^2+2\,\lf(\frac{8\pi}{3}\rg)^{2/3}r_0\, \|F_A\|^2_{L^6(B_{r_1}(0))}\le \delta
\ee
The monotonicity formula (\ref{III.1}) is implying that for any $y\in \B_1(0)$, $0<r<1-|y|$
\be
\label{mon-i}
\frac{d}{dr}\lf(\frac{1}{r}\int_{B_{r}(y)}|A|^2+{|\nabla_Au|^2}+\frac{3}{2}(1-|u|^2)^2+\lf(2\,\frac{r}{|x|}-1\rg)\,|F_A|^2\ \ dx^3\rg)\ge 0
\ee
Let  $\B_{r}(x_0)\subset \B_{r_1}(0)$ and $r<r_0$. Assume
\[
\frac{1}{r}\int_{\B_r(x_0)}|A|^2\ dx^3<\delta, 
\]
Then for any $y\in B_{r/2}(x_0)$ there holds
\[
\frac{4}{r}\int_{\B_{r/4}(y)}|A|^2(x)\ dx^3\le\, 4\delta
\]
This implies that, thanks to the fact tha $r/4<r_0$ and that $r_0$ has been chosen such that (\ref{r0ch}) holds true, we have
\[
\frac{4}{r}\int_{B_{r/4}(y)}|A|^2+{|\nabla_Au|^2}+\frac{3}{2}(1-|u|^2)^2+\lf(\frac{r}{2\,|x|}-1\rg)\,|F_A|^2\ \ dx^3\le 5\,\delta
\]
Using the monotonicity (\ref{mon-i}) we obtain in particular
\[
\forall y\in B_{r/2}(x_0)\qquad\forall\rho<\frac{r}{4} \qquad\frac{1}{\rho}\int_{B_\rho(y)}|A|^2\ dx^3<5\,\delta\ .
\]
This concludes the proof of lemma~\ref{lm-smal}.\hfill $\Box$

\medskip

We shall now prove the following lemma
\begin{Th}
\label{th-delreg}
Let  $(u,A)$ be a stationary critical point of $CYMH$ on $\B^3$. Let $r_1<1$. There exists $\delta>0$ and  $r_0>0$  depending only on 
$\|d_Au\|_{L^6(B_{r_1}(0))}$, $\|F_A\|_{L^6(B_{r_1}(0))}$ and $\|u\|_{L^\infty(B_{r_1}(0))}$ such that  for any $x_0\in \B_{r_1}(0)$ and any $r<r_0$ for which $\B_{r}(x_0)\subset \B_{r_1}(0)$  assume
\[
\frac{1}{r}\int_{\B_r(x_0)}|A|^2\ dx^3<\delta, 
\]
then $(u,A)\in C^\infty(B_{r/2}(x_0))$ .\hfill $\Box$
\end{Th}
\noindent{\bf Proof of Theorem~\ref{th-delreg}}  Let $0<\delta<\delta_2$ to be fixed later and where $\delta_2>0$ is given by (\ref{II.0}). We consider $r_0>0$ given by lemma~\ref{lm-smal} for this $\delta$. Let $x_0\in \B_{r_1}(0)$ and let $r<r_0$ such that $\B_{r}(x_0)\subset \B_{r_1}(0)$ and
\[
\frac{1}{r}\int_{\B_r(x_0)}|A|^2\ dx^3<\delta\ . 
\]
Thanks to the previous lemma we have
\be
\label{small}
\forall x\in B_{r/2}(x_0)\qquad\forall\rho< \frac{r}{4}\qquad\frac{1}{\rho}\int_{\B_\rho(x_0)}|A|^2\ dx^3<5\,\delta\
\ee
We have also
\[
\|F_A\|_{L^{3/2}(B_r(x_0))}\le \delta<\delta_2
\]
and then, thanks to proposition~\ref{pr-coul} and theorem~\ref{th-Uhl} we have
\[
d^\ast A=0\qquad\mbox{ and }\qquad \exists\, g\in W^{1,2}(B_r(x_0),SU(2))\qquad \mbox{s. t. }\quad A=g^{-1}dg+g^{-1}d^\ast\xi \,g
\]
and
\be
\label{reg-coul}
\|d^\ast\xi\|_{L^{(3,{3/2})}(\B_r(x_0))}+\|\nabla(d^\ast\xi)\|_{L^{3/2}(\B_r(x_0))}\le C\,\|F_A\|_{L^{3/2}(\B_r(x_0))}\le C\, \delta\ .
\ee
Moreover
\be
\label{L6xi}
\|d^\ast\xi\|_{L^2(\B_r(x_0))}+\|\nabla(d^\ast\xi)\|_{L^{2}(\B_r(x_0))}\le C\,\|F_A\|_{L^{2}(\B_r(x_0))}\ .
\ee
Thus from (\ref{reg-coul}) in particular, there holds
\be
\label{mo-g}
\begin{array}{l}
\ds\forall x\in \B_{r/2}(x_0)\quad\forall\rho<\frac{r}{4} \qquad\frac{1}{\rho}\int_{\B_\rho(x_0)}|dg|^2\ dx^3\le \frac{2}{\rho}\int_{\B_\rho(x_0)}|A|^2+|d^\ast\xi|^2\ dx^3\\[5mm]
\ds\qquad\le 10\,\delta+ \rho^{-1}\,\lf(\frac{4\pi}{3}\,\rho^3\rg)^{1/3}\,\|d^\ast\xi\|^2_{L^{3}(\B_r(x_0))}\le C\,\delta\ ,
\end{array}
\ee
where we used that $\rho<1$ and $C>0$ is universal. Observe that this inequality implies thanks to Poincar\'e inequality
\[
\forall x\in \B_{r/2}(x_0)\quad\forall\rho<\frac{r}{4} \qquad \dashint_{\B_\rho(x)}\lf| g-\dashint_{\B_\rho(x)} g\ dx^3  \rg|^2\ dx^3\le \frac{C}{\rho}\int_{\B_\rho(x_0)}|dg|^2\ dx^3\le C\delta
\]
from which we deduce
\be
\label{bmo}
\|g\|^2_{BMO(\B_{r/2}(x_0))}\le C\,\sup_{\{x\in \B_{r/2}(x_0)\ ;\ \rho<r/4\}}\ \dashint_{\B_\rho(x)}\lf| g-\dashint_{\B_\rho(x)} g\ dx^3  \rg|^2\ dx^3\le C\,\delta\ .
\ee

\medskip

We consider the following minimization problem among closed 2-forms $\zeta\in W^{1,2}(\wedge^2 B_r(x_0),\frak{su}(2))$
\be
\label{min}
\inf\lf\{\int_{B_r(x_0)}\lf|g^{-1}dg -d^\ast\zeta\rg|^2\ dx^3\ ;\  \mbox{s. t. }\ d\zeta=0 \quad\mbox{and}\quad\iota_{\p B_r(x_0)}^\ast\ast\zeta=0\rg\}
\ee
There exists a minimizer and it satisfies
\[
\forall\psi\,\in W^{1,2}(\wedge^2 B_r(x_0),\frak{su}(2))\quad\mbox{s.t.}\quad\iota_{\p B_r(x_0)}^\ast\ast\psi=0\quad \int_{B_r(x_0)}g^{-1}dg -d^\ast\zeta\wedge d\ast\, \psi=0\ 
\]
Indeed, let $\beta_k:=\ast \zeta_k$ be a minimizing sequence we have
\[
\limsup_{k\rightarrow +\infty}\int_{B_r(x_0)}|d\beta_k|^2\ dx^3<+\infty\ ,
 \]
moreover, there holds
\[
\begin{array}{l}
\ds|d \beta_k|^2=|d \beta_k|^2+|d^\ast\beta_k|^2=\sum_{i<j}|\p_{x_i}\beta^j_k-\p_{x_j}\beta^i_k|^2+\lf|\sum_{i=1}^3\p_{x_i} \beta_k^i\rg|^2\\[5mm]
\ds=\sum_{i\ne j}|\p_{x_i}\beta^j_k|^2+\sum_{i=1}^3|\p_{x_i}\beta_k^i|^2-2\,\sum_{i<j} \p_{x_i}\beta^j_k\cdot\p_{x_j}\beta^i_k+\sum_{i\ne j}\p_{x_i}\beta^i_k\cdot\p_{x_j}\beta^j_k\\[5mm]
\ds=|\nabla \beta_k|^2+2\,\sum_{i<j}\p_{x_i}\beta^i_k\cdot\p_{x_j}\beta^j_k-\p_{x_i}\beta^j_k\cdot\p_{x_j}\beta^i_k\\[5mm]
\ds=|\nabla \beta_k|^2+\sum_{i,j}\p_{x_i}(\beta^i_k\cdot\p_{x_j}\beta^j_k)-\p_{x_j}(\beta_k^i\cdot\p_{x_i}\beta^j_k)
\end{array}
\]
This implies
\[
\begin{array}{l}
\ds\int_{\B_r(x_0)}|d\beta_k|^2\ dx^3=\int_{\B_r(x_0)}|d\beta_k|^2+|d^\ast\beta_k|^2\, dx^3\\[5mm]
\ds=\int_{\B_r(x_0)}|\nabla \beta_k|^2\ dx^3+\int_{\p \B_r(x_0)}\beta_k(\nu)\cdot d^\ast\beta_k\ dvol_{\p \B_r(x_0)}+\sum_{i,j=1}^3\int_{\p \B_r(x_0)}\beta_k^i\,\nu_j\,\p_{x_i}\beta_k^j\ dvol_{\p \B_r(x_0)}\\[5mm]
\ds=\int_{\B_r(x_0)}|\nabla \beta_k|^2\ dx^3+\sum_{i,j=1}^3\int_{\p \B_r(x_0)}\beta_k^i\,\nu_j\,\p_{x_i}\beta_k^j\ dvol_{\p \B_r(x_0)}
\end{array}
\]
Since $\iota_{\p \B_r(x_0)}^\ast \beta_k=0$ there holds $\beta_k\wedge dr=0$, thus 
\[
\forall\, i,j=1...3\qquad\beta_k^i\,\nu_j-\beta_k^j\,\nu_i=0\ .
\]
Thus
\[
\begin{array}{l}
\ds\sum_{i,j=1}^3\int_{\p \B_r(x_0)}\beta_k^i\,\nu_j\,\p_{x_i}\beta_k^j\ dvol_{\p \B_r(x_0)}=\frac{1}{2}\int_{\p \B_r(x_0)}\p_\nu|\beta_k|^2\ dvol_{\p \B_r(x_0)}\ .
\end{array}
\]
Then
\[
\begin{array}{l}
\ds\int_{\B_r(x_0)}|d\beta_k|^2\ dx^3=\int_{\B_r(x_0)}|d\beta_k|^2+|d^\ast\beta_k|^2\, dx^3\\[5mm]
\ds=\int_{\B_r(x_0)}|\nabla \beta_k|^2\ dx^3+\frac{1}{2}\int_{\B_r(x_0)}\Delta|\beta_k|^2\ dx^3\\[5mm]
\ds=\int_{\B_r(x_0)}|\nabla \beta_k|^2\ dx^3+\int_{\B_r(x_0)}\beta_k\cdot\Delta\beta_k\ dx^3+\int_{\B_r(x_0)}\nabla\beta_k\cdot\nabla\beta_k\ dx^3\\[5mm]
\ds=2\,\int_{\B_r(x_0)}|\nabla \beta_k|^2\ dx^3-\int_{\B_r(x_0)}\beta_k\cdot d^\ast d\beta_k\ dx^3\\[5mm]
\ds=2\,\int_{\B_r(x_0)}|\nabla \beta_k|^2\ dx^3-\int_{\B_r(x_0)}\beta_k\dot{\wedge} d\ast d\beta_k\ dx^3\\[5mm]
\ds=2\,\int_{\B_r(x_0)}|\nabla \beta_k|^2\ dx^3+\int_{\B_r(x_0)}d\lf(\beta_k\dot{\wedge} \ast d\beta_k\rg)-\int_{\B_r(x_0)}d\beta_k\dot{\wedge} \ast d\beta_k
\end{array}
\]
 We finally get the {\it Gaffney type identity}
 \be
 \label{gaffney}
 \int_{\B_r(x_0)}|d\beta_k|^2\ dx^3=\int_{\B_r(x_0)}|\nabla \beta_k|^2\ dx^3\ .
 \ee
 Hence we can extract a sub-sequence that we keep denoting $\beta_k$ such that $\beta_k\rightharpoonup \beta_\infty$ weakly in $W^{1,2}(\B_r(x_0),\frak{su}(2))$ and obviously the hypothesis $\iota_{\p \B_r(x_0)}^\ast \beta_k=0$ and $d^\ast\beta_k=0$
 both pass to the limit, so that, by lower semi-continuity of the $L^2-$norm of $g^{-1}dg -\ast\,d\beta_k$ we obtain that $\zeta:=\ast\beta_\infty$ is a minimizer of (\ref{min}).

Moreover, because of the previous Gaffney identity there holds
\[
\int_{\B_r(x_0)}\sum_{i,j,k=1}^3|\p_{x_k}\zeta_{ij}|^2\ dx^3\le C\, \int_{B_r(x_0)}|dg|^2\ dx^3
\]
We deduce also that for every $\mbox{s. t. }\ d\psi=0 \quad\mbox{and}\quad\iota_{\p \B_r(x_0)}^\ast\ast\psi=0$
\[
\int_{\B_r(x_0)}d\lf(g^{-1}dg -d^\ast\zeta\rg)\wedge \ast\, \psi=0
\]
Or in other words we have
\[
\begin{array}{c}
\ds\forall\ \beta\in W^{1,2}(\Om^1(\B_r(x_0)),\frak{su}(2))\qquad\mbox{ s. t. }\quad d^\ast\beta=0\qquad\mbox{ and }\quad\iota_{\p \B_r(x_0)}^\ast\beta=0\\[5mm]

\ds\int_{B_r(x_0)}d\lf(g^{-1}dg -d^\ast\zeta\rg)\wedge \beta=0
\end{array}
\]
Let $\phi=\sum_{i=1}^3\phi_i\ dx_i\in C^\infty_0(\Om^1(\B_r(x_0)),\frak{su}(2))$. We consider the function $\varphi$ satisfying
\[
\lf\{
\begin{array}{l}
\ds d^\ast d\varphi=-\Delta \varphi=d^\ast\phi\quad\mbox{ in }\B_r(x_0)\\[5mm]
\ds\varphi=0\qquad\mbox{ on }\p \B_r(x_0)
\end{array}
\rg.
\]
Hence there exists $\beta\in C^\infty(B_r(x_0))$ such that $d^\ast\beta=0$ and $\phi=d\varphi+\beta$. Since both $\varphi\circ \iota_{\p \B_r(x_0)}$ and $\iota_{\p \B_r(x_0)}^\ast\phi$ are equal to zero we deduce that
$\iota_{\p \B_r(x_0)}^\ast\beta=0$. Then we write
\[
\int_{\B_r(x_0)}d\lf(g^{-1}dg -d^\ast\zeta\rg)\wedge \phi=\int_{\B_r(x_0)}d\lf(g^{-1}dg -d^\ast\zeta\rg)\wedge d\varphi+\int_{\B_r(x_0)}d\lf(g^{-1}dg -d^\ast\zeta\rg)\wedge \beta=0
\]
Hence
\[
d\lf(g^{-1}dg -d^\ast\zeta\rg)=0\qquad\mbox{ in } {\mathcal D}'(\B_r(x_0))\ .
\]
Let $\sigma \in W^{1,2}(\B_r(x_0))$ such that
\be
\label{hodge}
g^{-1}dg =d\sigma+d^\ast\zeta
\ee
The pair $(\sigma,\zeta)$ such that $d\zeta=0$ and $\iota_{\p \B_r(x_0)}^\ast\ast\zeta=0$ is obviously unique after fixing $$\int_{\B_{r/2}}\sigma\ dx^3=0\ ,$$ and there holds
\be
\label{est-prel-zeta-sig}
\int_{\B_r(x_0)}\sum_{i,j=1}^3|\nabla\zeta_{ij}|^2+|\nabla\sigma|^2\ dx^3\le C\, \int_{\B_r(x_0)}|dg|^2\ dx^3\ .
\ee
The Coulomb condition $d^\ast A=0$ that is  
\[
0=d^\ast A=d^\ast(g^{-1}dg+g^{-1}d^\ast\xi \,g)=d^\ast(d\sigma+d^\ast\zeta+g^{-1}d^\ast\xi \,g)=d^\ast(d\sigma+g^{-1}d^\ast\xi \,g)
\] 
is giving
\be
\label{eq-sigma}
\begin{array}{l}
\ds\Delta\sigma=d^\ast(g^{-1}\,d^\ast\xi\,g)=-\sum_{k,l=1}^3\p_{x_k}\lf(g^{-1}\p_{x_l}\xi_{lk}\,g\rg)\\[5mm]
\ds\qquad=-\sum_{k,l=1}^3\p_{x_k}\lf(g^{-1}\rg)\,g\, g^{-1}\,\p_{x_l}\xi_{lk}\,g-\sum_{k,l=1}^3 g^{-1}\p_{x_l}\xi_{lk}\,g\,g^{-1}\p_{x_k}g\\[5mm]
\ds\qquad=\sum_{k,l=1}^3[g^{-1}\p_{x_k}g,g^{-1}\,\p_{x_l}\xi_{lk}\,g]=-[g^{-1}dg\res g^{-1}\,d^\ast\xi\, g]
\end{array}
\ee
Because of (\ref{reg-coul}) and (\ref{est-prel-zeta-sig})
\[
\begin{array}{l}
\ds\|\sigma\|_{W^{2,(3/2,1)}(\B_{r/2}(x_0))}\le C\ \|dg\|_{L^2(\B_r(x_0))}\,\|d^\ast\xi\|_{L^{(6,2)}(\B_r(x_0))}+ C\,\|d\sigma\|_{L^2(\B_r(x_0))}\\[5mm]
\ds\qquad\le C\,\|dg\|_{L^2(\B_r(x_0))}\,\|F_A\|_{L^2(\B_r(x_0))}+C\,\|dg\|_{L^2(\B_r(x_0))}\,
\end{array}
\]
This implies in particular
\[
\int_{\B_{r/2}(x_0))}|\nabla \sigma|^3\ dx^3\le C\, \delta^3\quad\mbox{ and }\quad\|\sigma\|_{L^\infty(\B_{r/2}(x_0))}\le C\,\sqrt{\delta}
\]
Hence
\[
\begin{array}{l}
\ds\forall x\in \B_{r/2}(x_0)\quad\forall\rho<r/4 \\[5mm]
\ds\frac{1}{\rho}\int_{\B_\rho(x)}|d^\ast\zeta|^2\ dx^3\le \frac{2}{\rho}\int_{\B_\rho(x)}|d g|^2\ dx^3+ \frac{2}{\rho}\int_{\B_\rho(x)}|d \sigma|^2\ dx^3\le C\, \delta +C\,\|d\sigma\|_{L^3(\B_{r/2}(x_0))}^2\le C\,\delta\ .
\end{array}
\]
We have
\[
dg=g\, d\sigma+g\,d^\ast\zeta\ .
\]
Hence
\be
\label{eq-g}
\lf\{
\begin{array}{l}
\ds-\Delta g=d^\ast dg=d^\ast(g\,d\sigma)+dg\cdot d^\ast\zeta\\[5mm]
\ds -\Delta\sigma=[g^{-1}dg\res g^{-1}d^\ast\xi g]\ .
\end{array}
\rg.
\ee
On a ball $\B_\rho(x)\subset \B_{r/2}(x_0)$ we introduce $\ti{g}$ such that
\be
\label{ti-g}
\lf\{
\begin{array}{l}
\ds-\Delta \ti{g}=d^\ast dg=d^\ast\lf(g\,d\sigma+g\, d^\ast\zeta\rg)\qquad\mbox{ in }\B_\rho(x)\\[5mm]
\ds \ti{g}=0\qquad\mbox{ on }\p \B_\rho(x)\ .
\end{array}
\rg.
\ee
Since obviously by classical elliptic estimate $\ti{g}\in W^{1,2}_0(\B_\rho(x))$ there holds
\[
\|d\ti{g}\|_{L^2(\B_\rho(x))}=\sup\lf\{\int_{\B_\rho(x))}d\ti{g}\cdot d\varphi \ dx^3\ ;\ \|d\varphi\|_{L^2(\B_\rho(x))}\le 1\quad \varphi\in W^{1,2}_0(\B_\rho(x))  \rg\}
\]
Multiplying (\ref{ti-g}) by $\varphi\in W^{1,2}_0(B_\rho(x)) $ we have
\[
\begin{array}{l}
\ds\int_{\B_\rho(x)}d\ti{g}\cdot d\varphi \ dx^3=\int_{\B_\rho(x))} {g}\,d\sigma\cdot d\varphi\ dx^3+\int_{\B_\rho(x))}{g} \, d\ast\zeta\wedge d\varphi\\[5mm]
\ds\quad \le\,\|d\sigma\|_{L^2(\B_\rho(x))} +\|g\|_{BMO(\B_{r/2}(x_0))}\, \|d\ast\zeta\|_{L^2(\B_\rho(x))} \\[5mm]
\ds\quad\le \,\|d\sigma\|_{L^2(B_\rho(x))} +\sqrt{\delta}\, (\|dg\|_{L^2(B_\rho(x))}+\|d\sigma\|_{L^2(B_\rho(x))})
\end{array}
\]
where we used (\ref{bmo}).

Let $t\in(0,1)$ such that $\B_{\rho/t}(x_0)\subset \B_{r/2}(x_0)$. We also introduce $\ti{\sigma}$ such that
\[
\lf\{
\begin{array}{l}
\ds-\Delta \ti{\sigma}=[g^{-1}dg\res g^{-1}d^\ast\xi g]\qquad\mbox{ in }\B_{\rho/t}(x)\\[5mm]
\ds \ti{\sigma}=0\qquad\mbox{ on }\p \B_{\rho/t}(x)
\end{array}
\rg.
\]
Since obviously $\sigma-\ti{\sigma}$ satisfies
\[
\lf\{
\begin{array}{l}
\ds \Delta(\sigma-\ti{\sigma})=0\qquad\mbox{ in }\B_{\rho/t}(x)\\[5mm]
\ds\sigma-\ti{\sigma}=\sigma\qquad\mbox{ on }\p \B_{\rho/t}(x)\ .
\end{array}
\rg.
\]
We obtain by the maximum principle
\[
\|\ti{\sigma}\|_{L^\infty(B_{\rho/t}(x))}\le 2\sqrt{\delta}\ .
\]
Mutiplying  by $\ti{\sigma}$ and integrating over $\B_{\rho/t}(x)$ is giving
\[
\begin{array}{rl}
\ds\int_{\B_{\rho/t}(x))}|d\ti{\sigma}|^2\ dx^3&\ds=\int_{\B_{\rho/t}(x))} \ti{\sigma}\cdot [g^{-1}dg\res g^{-1}d^\ast\xi g] \ dx^3\\[5mm]
 &\ds\le 2\,\sqrt{\delta}\, \|dg\|_{L^2(\B_{\rho/t}(x))}\,\|d^\ast\zeta\|_{L^2(\B_{\rho/t}(x))}
\end{array}
\]
Since $\sigma-\ti{\sigma}$ is harmonic we have using monotonicity formula for harmonic functions and the Dirichlet Principle
\[
\ds\frac{1}{\rho^3}\int_{\B_\rho(x)}|d(\sigma-\ti{\sigma})|^2\ dx^3\le \frac{t^3}{\rho^3}\int_{\B_{\rho/t}(x)}|d(\sigma-\ti{\sigma})|^2\ dx^3\le \frac{t^3}{\rho^3}\int_{\B_{\rho/t}(x))}|d\sigma|^2\ dx^3\ .
\]
Combining the two previous inequalities is giving
\[
\begin{array}{l}
\ds\|d\sigma\|_{L^2(\B_\rho(x))}\le \|d(\sigma-\ti{\sigma})\|_{L^2(\B_\rho(x))}+\|d\ti{\sigma}\|_{L^2(\B_\rho(x))}\\[5mm]
\ds\le\sqrt{t^3\,\int_{\B_{\rho/t}(x)}|d\sigma|^2\ dx^3}+\sqrt{2\,\sqrt{\delta}\, \|dg\|_{L^2(\B_{\rho/t}(x))}\,\|d^\ast\zeta\|_{L^2(\B_{\rho/t}(x))}}\\[5mm]
\ds\le \sqrt{t^3\,\int_{B_{\rho/t}(x)}|d\sigma|^2\ dx^3}+\sqrt{2\,\sqrt{\delta}}\, \sqrt{\|dg\|^2_{L^2(\B_{\rho/t}(x))}\,+  \|dg\|_{L^2(\B_{\rho/t}(x))}\, \|d\sigma\|_{L^2(\B_{\rho/t}(x))}}\\[5mm]
\ds\le (t^{3/2}+\sqrt{2}\,\delta^{1/4})\|d\sigma\|_{L^2(B_{\rho/t}(x))}+2\,\delta^{1/4}\,\|dg\|_{L^2(B_{\rho/t}(x))}
\end{array}
\]
Since the coordinates of $g-\ti{g}$ are harmonic functions we have
\[
\frac{1}{\rho^3}\int_{B_{t\rho}(x))}|d(g-\ti{g})|^2\ dx^3\le \frac{t^3}{\rho^3}\int_{B_{\rho}(x))}|d(g-\ti{g})|^2\ dx^3\le \frac{t^3}{\rho^3}\int_{B_{\rho}(x))}|dg|^2\ dx^3
\]
Hence
\[
\begin{array}{l}
\ds\|dg\|_{L^2(\B_{t\rho}(x))}\le \|d(g-\ti{g})\|_{L^2(\B_{t\rho}(x))}+\|d\ti{g}\|_{L^2(\B_\rho(x))}\\[5mm]
\ds\le t^{3/2}\,\|dg\|_{L^2(\B_{\rho}(x))} +(1+\sqrt{\delta})\|d\sigma\|_{L^2(\B_\rho(x))}+\sqrt{\delta}\, \|dg\|_{L^2(\B_\rho(x))}\\[5mm]
\ds\le (t^{3/2}+\sqrt{\delta})\,\|dg\|_{L^2(\B_{\rho}(x))} +(1+\sqrt{\delta})(t^{3/2}+\sqrt{2\,\delta})\|d\sigma\|_{L^2(\B_{\rho/t}(x))}\\[5mm]
\ds+2\,\sqrt{\delta}\,\|dg\|_{L^2(\B_{\rho/t}(x))}
\end{array}
\]
We have then two functions $\phi(s):=\sqrt{s^{-1}}\|dg\|_{L^2(B_{s}(x))}$ and $\psi(s):=\sqrt{s^{-1}}\|d\sigma\|_{L^2(B_{s}(x))}$
\[
\lf\{
\begin{array}{l}
\ds\psi(\rho)\le (t+\sqrt{2\sqrt{\delta}\,t^{-1}})\, \psi(\rho/t)+2\,\sqrt{\sqrt{\delta}\,t^{-1}}\,\phi(\rho/t)\\[5mm]
\ds\phi(\rho)\le (t+\sqrt{\delta\,t^{-1}})\, \phi(\rho/t)+(1+\sqrt{\delta})(t^{1/2}+t^{-1}\sqrt{2\,\delta})\, \psi(\rho/t^2) +2\,\sqrt{\delta}\,t^{-1}\phi(\rho/t^2)
\end{array}
\rg.
\]
Denoting $a_j:=\phi(r\, t^j)+\psi(r\, t^j)$, choosing $t$ sufficiently small and then $\delta$ sufficiently small (independent of $j$ and $x\in B_{r}(x_0)$) we have
\[
a_{j+1}\le 4^{-1}\,a_j+8^{-1}\,a_{j-1}
\]
This gives
\[
a_{j+1}+2^{-1} a_j\le \frac{3}{4}\,(a_j+{6}^{-1}\, a_{j-1})\le \frac{3}{4}\,(a_j+ 2^{-1} a_{j-1})
\]
Hence we deduce the existence of $1>\al>0$ such that the following Morrey estimate holds
\[
\sup_{x\in B_{r/2}(x_0)\,;\,0<\rho<r/4}\rho^{-1-\alpha}\int_{B_{\rho}(x)}|dg|^2 dx^3<+\infty
\]
Recall that
\[
A=g^{-1}\,dg+g^{-1}\,d^\ast\xi\,g
\]
where in particular $d^\ast\xi\in L^6(B_r(x_0))$ (thanks to \ref{L6xi}). Thus
\[
\begin{array}{l}
\ds\sup_{\{x\in B_{r/2}(x_0)\,;\,0<\rho<r/4\}}\rho^{-2}\int_{B_{\rho}(x)}|d^\ast\xi|^2 dx^3\\[5mm]
\ds\qquad\le C\,\sup_{\{x\in B_{r/2}(x_0)\,;\,0<\rho<r/4\}} \rho^{-2}\,(\rho^3)^{2/3}\|d\ast\xi\|^2_{L^6(B_\rho(x))}\le C\, \|d\ast\xi\|^2_{L^6(B_r(x_0))}
\end{array}
\]
Hence
\[
\sup_{\{x\in B_{r/2}(x_0)\,;\,0<\rho<r/4\}}\rho^{-1-\alpha}\int_{B_{\rho}(x)}|A|^2 dx^3<+\infty\ .
\]
Since $F_{A}\in L^6(B_r(x_0))$, using (\ref{II.1}) for $p=6$ is giving that $d^\ast\xi\in W^{1,6}(B_r(x_0))$. Hence in particular $d^\ast\xi\in L^\infty$. From (\ref{eq-g}) we have $\sigma\in W^{2,2}(B_{r/2}(x_0))$.
Injecting this information in the first identity of (\ref{eq-g}) together with the fact that
\[
\sup_{\{x\in B_{r/2}(x_0)\,;\,0<\rho<r/4\}}\rho^{-1-\alpha}\int_{B_{\rho}(x)}|d^\ast\zeta|^2 dx^3<+\infty\ .
\]
is giving using Adams estimates \cite{Ada} that $\nabla g\in L^p_{loc}(B_{r/2}(x_0))$ for some $p>2$ and hence $A\in L^p_{loc}(B_{r/2}(x_0))$ for some $p>2$. Injecting this information in the r.h.s. of the second equality in the following system
\[
\lf\{
\begin{array}{l}
\ds d F_{A^g}=-[A^g\wedge  F_{A_g}]\qquad\mbox{ in }\B_r(x_0)\\[5mm]
\ds d^\ast F_{A^g}=-[A^g\res  F_{A_g}]-g\,[u,d_Au]\,g^{-1}-g\,A\, g^{-1}\qquad\mbox{ in }\B_r(x_0)\ .
\end{array}
\rg.
\]
is giving, bearing in mind that $[u,d_Au]\in L^6(B_{r_1}(0))$, that $F_{A^g}\in W^{1,p}_{loc}(B_{r/2}(x_0))$ for some $p>2$. The result follows by standard bootstrapping argument
and this concludes the proof of theorem~\ref{th-delreg}  \hfill $\Box$

\medskip

\noindent{\bf Proof of theorem~\ref{th-reg}} Let $(u,A)$ be a minimizer of $CYMH$ for its boundary data. Let $r_1<1$. Thanks to lemma~\ref{lm-L6} we know that $F_A$ and $d_Au$ are in $L^6(B_{r_1}(0))$. Assume there exists $x_0\in B_1(0)$ such that
\be
\label{hyp-d}
\lim_{\rho\rightarrow 0}\frac{1}{\rho}\int_{B_\rho(x_0)}|A|^2\ dx^3>0\ .
\ee
We are going to contradict this assumption and then by applying theorem~\ref{th-delreg} we will have proved theorem~\ref{th-reg}. Under the hypothesis (\ref{hyp-d}), because of the previous lemma we must have
\[
\lim_{\rho\rightarrow 0}\frac{1}{\rho}\int_{B_\rho(x_0)}|A|^2\ dx^3>\delta\ .
\]
We proceed to the local extraction of a controlled Coulomb gauge on a sufficiently small ball centred at $x_0$ so that (\ref{II.0}) is satisfied for $p=6$. Thus there exists a gauge change $g$ such that
\[
d^\ast\xi=A^g=g^{-1}\,dg+g^{-1}\, A\,g
\]
Since $F_{A}\in L^6(B_{r_1}(0))$, we have thanks to (\ref{II.1}) that $A^g\in W^{1,6}$ which implies $d^\ast\xi\in L^\infty$. Hence
\[
\lim_{\rho\rightarrow 0}\frac{1}{\rho}\int_{B_\rho(x_0)}|g^{-1}\,dg|^2\ dx^3=\lim_{\rho\rightarrow 0}\frac{1}{\rho}\int_{B_\rho(x_0)}|A|^2\ dx^3>\delta
\]
Let
\[
g_\rho(y):=g(\rho\,y+x_0)\ .
\]
We have that
\[
\lim_{\rho\rightarrow 0}\int_{B_1(0)}|dg_\rho|^2\ dy^3=\lim_{\rho\rightarrow 0}\frac{1}{\rho}\int_{B_\rho(x_0)}|g^{-1}\,dg|^2\ dx^3<+\infty
\]
We extract a subsequence $\rho_k\rightarrow 0$ such that $$g_k(y):=g_{\rho_k}(y)\rightharpoonup g_\infty(y)$$. 
We also denote
\[
d^\ast \xi_k:= g_k^{-1}\,dg_k+g_k^{-1}\, D^\ast_kA\,g_k\qquad\mbox{ where }\quad D_k(y):=\rho_k\,y+x_0\ .
\]
The Coulomb condition $d^\ast A=0$ (Proposition~\ref{pr-coul}) is giving
\[
0=-d^\ast(dg_k\,g_k^{-1})+d^\ast(g_k\,d^\ast\xi_k\,g_k^{-1})\ .
\]
We have
\[
\int_{B_1(0)}|d^\ast \xi_{k}|^2\ dy^3=\frac{1}{\rho_k}\int_{B_{\rho_k}(x_0)}|d^\ast\xi|^2\ dx^3\le \frac{4\pi}{3}\rho^2_k\ \|d^\ast\xi\|_\infty\longrightarrow 0
\]
Hence
\[
d^\ast(dg_\infty\,g_\infty^{-1})=0\ .
\]
Recall that a map $h$ in $W^{1,2}(B_1(0),SU(2))$ is a critical point of the Dirichlet energy 
\[
\int_{B_1(0)}|dh|^2\ dx^3
\]
among maps taking values in the Lie group $SU(2)$  if and only if, for any $w\in C^\infty_0(B_1,\frak{su}(2))$ denoting $h_t:=\exp(t\,w)\,h$ there holds
\[
\lf.\frac{d}{dt}\int_{B_1(0)}|dh_t|^2\ dx^3\rg|_{t=0}=0
\]
which is equivalent to
\[
0=\int_{B_1(0)}<d(w\,h),dh> dx^3=\int_{B_1(0)}<w\,dh,dh>+<dw\,h,dh> dx^3
\]
Observe that, since $w\in \frak{su}(2)$, for any $X\in M_2({\mathbb C})$, there holds
\[
\mbox{ tr}\lf(X\,(\ov{w\,X})^t\rg)=<w\,X,X>=\mbox{ tr}\lf(w\,X\,\ov{X}^t\rg)
\]
but
\[
\mbox{ tr}\lf(X\,(\ov{w\,X})^t\rg)=-\mbox{ tr}\lf(X\,\ov{X}^t\,w\rg)=-\mbox{ tr}\lf(w\,X\,\ov{X}^t\rg)\ .
\]
Hence $<w\,dh,dh>=0$ and $h$ is harmonic if and only if
\[
\begin{array}{l}
\ds \forall w\in C^\infty_0(B_1,\frak{su}(2))\qquad 0=\int_{B_1(0)}<dw\,h,dh> dx^3\\[5mm]
\ds\quad=\int_{B_1(0)}<dw,dh\,h^{-1}> dx^3=\int_{B_1(0)}\lf<w,d^\ast\lf(dh\,h^{-1}\rg)\rg> dx^3
\end{array}
\]
Hence the Euler Lagrange equation is $d^\ast\lf(dh\,h^{-1}\rg)=0$ and we deduce that $g_\infty$ is weakly harmonic.
From the monotonicity formula (\ref{III.1}) we have
\[
\int_{B_{\rho}(0)}\frac{|A\res\p_r|^2}{|x-x_0|}\ dx^3<+\infty
\]
This gives obviously
\[
\int_{B_{\rho_k}(x_0)}\frac{|\p_rg(x)|^2}{|x-x_0|}\ dx^3<+\infty
\]
which also implies
\[
\lim_{k\rightarrow +\infty}\int_{B_{1}(0)}\frac{|\p_rg_k(y)|^2}{|y|}\ dy^3=0
\]
Thus
\[
\p_{r}g_\infty=0
\]
and there exists $v\,:\, S^2\,\rightarrow\, S^3$ harmonic such that $g_\infty(y)=v(y/|y|)$. This implies that $g_\infty$ is smooth outside the origin. Let $y_0\in B_1(0)\setminus\{ 0\}$ and let $s>0$ to be chosen later. We can find $\sigma_k\in [s/2,s]$ such that
\[
\int_{\p B_{\sigma_k}(y_0)}|\nabla g_k|^2\ dy^3\le s^{-1}\, \int_{ B_{s}(y_0)}|\nabla g_k|^2\ dy^3
\]
Let $\la\in (0,1/2)$ to be fixed later. Using Luckhaus lemma as stated in \cite{Mos}. There exists an ``interpolation map'' $\varphi\in W^{1,2}(B_{\sigma_k}(y_0)\setminus B_{(1-\la)\,\sigma_k}(y_0),M_2({\C}))$ such that
\[
\varphi=g_k\qquad\mbox{ on }\p B_{\sigma_k}(y_0)\qquad\mbox{ and }\qquad\varphi=g_\infty\qquad\mbox{ on }\p B_{(1-\la)\,\sigma_k}(y_0)
\]
moreover
\[
\begin{array}{l}
\ds\frac{1}{\sigma_k}\int_{B_{\sigma_k}(y_0)\setminus B_{(1-\la)\sigma_k}(y_0)}|\nabla\varphi|^2\ dy^3\le C\, \la\, \int_{\p B_{\sigma_k}(y_0)}|\nabla g_k|^2\ dvol_{\p B_{\sigma_k}(y_0)}\\[5mm]
\ds\quad+C\, \la\, \int_{\p B_{(1-\la)\,\sigma_k}(y_0)}|\nabla g_\infty|^2\ dvol_{\p B_{(1-\la)\,\sigma_k}(y_0)}\\[5mm]
\ds\quad +\frac{C}{\la\,\sigma_k^2}\,\int_{\p B_{\sigma_k}(y_0)}|g_k(y)-g_\infty((1-\la)\,y)|^2\ dvol_{\p B_{\sigma_k}(y_0)}\ .
\end{array}
\]
In addition 
\[
\begin{array}{l}
\ds\|\mbox{dist}(\varphi(y), SU(2))\|_{L^\infty(B_{\sigma_k}(y_0)\setminus B_{(1-\la)\sigma_k}(y_0))}\\[5mm]
\ds\quad\le \frac{C}{\la\sigma_k}\, \sqrt{\int_{\p B_{\sigma_k}(y_0)}|\nabla g_k|^2+|\nabla(g_\infty((1-\la)\,y)|^2\ dvol_{\p B_{\sigma_k}(y_0)}}\\[5mm]
\ds\times\sqrt{ \int_{\p B_{\sigma_k}(y_0)}|g_k(y)-g_\infty((1-\la)\,y)|^2\ dvol_{\p B_{\sigma_k}(y_0)}}\\[5mm]
\ds+\frac{C}{\la^2\sigma_k^2}\, \int_{\p B_{\sigma_k}(y_0)}|g_k(y)-g_\infty((1-\la)\,y)|^2\ dvol_{\p B_{\sigma_k}(y_0)}
\end{array}
\]
We first choose $s$ such that $0\notin B_s(y_0)$
\[
\frac{8}{s}\int_{B_s(y_0)}|\nabla g_\infty|^2\ dy^3\le C\,s^2\, \|\nabla g_\infty\|^2_{L^\infty(B_s(y_0))}\le \delta\ .
\]
Then we choose $\la\in(0,1/2)$ small enough such that
\[
\begin{array}{l}
\ds C\, \la\, \int_{\p B_{\sigma_k}(y_0)}|\nabla g_k|^2\ dvol_{\p B_{\sigma_k}(y_0)}+C\, \la\, \int_{\p B_{(1-\la)\,\sigma_k}(y_0)}|\nabla g_\infty|^2\ dvol_{\p B_{(1-\la)\,\sigma_k}(y_0)}\\[5mm]
\ds \le\,C\, \frac{\la}{s}\, \int_{B_{s}(y_0)}|\nabla g_k|^2\ dy^3+C\, \frac{\la}{s}\, \int_{B_{s}(y_0)}|\nabla g_\infty|^2\ dy^3\le \frac{\delta}{8}
\end{array}
\]
and such that
\[
\|g_\infty((1-\la)y)-g(y)\|_{L^\infty(\p B_{\sigma_k}(y_0))}\le \frac{\delta^2}{64}
\]
Finally, by making $k\rightarrow +\infty$, since $g_k\rightarrow g_\infty$ strongly in $L^2(B_1(0))$, we ensure for $k$ large enough
\[
\frac{1}{\sigma_k}\int_{B_{\sigma_k}(y_0)\setminus B_{(1-\la)\sigma_k}(y_0)}|\nabla\varphi|^2\ dy^3\le\frac{\delta}{4}
\]
and
\[
\|\mbox{dist}(\varphi(y), SU(2))\|_{L^\infty(B_{\sigma_k}(y_0)\setminus B_{(1-\la)\sigma_k}(y_0))}\le \frac{\delta}{4}\ .
\]
For $\delta$ small enough there is a unique orthogonal projection of $\varphi$ onto $SU(2)$ and we denote
\[
\begin{array}{l}
\ds g_{k,\infty}:=\pi_{SU(2)}(\varphi)\quad\mbox{ in } B_{\sigma_k}(y_0)\setminus B_{(1-\la)\sigma_k}(y_0))\\[5mm]
\ds g_{k,\infty}:=g_\infty\quad\mbox{ in }  B_{(1-\la)\sigma_k}(y_0))
\end{array}
\]
in such a way that
\be
\label{dens-sma}
\frac{1}{\sigma_k}\int_{B_{\sigma_k}(y_0)}|\nabla g_{k,\infty}|^2\ dy^3\le\frac{\delta}{4}
\ee
We then introduce $\ti{g}_{k,\infty}(x):=g(x)$ in $B_1(0)\setminus B_{\rho_k\,\sigma_k}(\rho_k\,y_0+x_0)$ and $\ti{g}_{k,\infty}(x):=g_{k,\infty}(\frac{x-x_0}{\rho_k})$ in $B_{\rho_k\,\sigma_k}(\rho_k\,y_0+x_0)$
Finally we denote $h_k:=\,g\,g_{k,\infty}^{-1}$ and, $(u^{h_k},A^{h_k})$ is the gauge equivalent pair to $(u,A)$ given by
\[
u^{h_k}:=h_{k}^{-1}u\,h_k\qquad\mbox{ and }\quad A^{h_k}:=h_k^{-1} A\,h_k+h_k^{-1} dh_k
\]
By the minimality of $(u,A)$, since in the energy density only the non gauge invariant part of the energy that is $\int |A|^2$ is impacted by this gauge change we have
\[
CYMH(u,A)\le CYMH(u^{h_k},A^{h_k})\ .
\]
This gives
\[
\int_{B_{\rho_k\,\sigma_k}(\rho_k\,y_0+x_0)}|g\,dg^{-1}+g\,d^\ast\xi\,g|^2\ dx^3\le \int_{B_{\rho_k\,\sigma_k}(\rho_k\,y_0+x_0)}|\ti{g}_{k,\infty}\,d\ti{g}_{k,\infty}^{-1}+\ti{g}_{k,\infty}\,d^\ast\xi\,\ti{g}^{-1}_{k,\infty}|^2\ dx^3
\]
This implies
\[
\begin{array}{l}
\ds\|dg\|_{L^2(B_{\rho_k\,\sigma_k}(\rho_k\,y_0+x_0))}\le \|d\ti{g}_{k,\infty}\|_{L^2(B_{\rho_k\,\sigma_k}(\rho_k\,y_0+x_0))}+2\, \|d^\ast\xi\|_{L^2(B_{\rho_k\,\sigma_k}(\rho_k\,y_0+x_0))}\\[5mm]
\ds\qquad\le \|d\ti{g}_{k,\infty}\|_{L^2(B_{\rho_k\,\sigma_k}(\rho_k\,y_0+x_0))}+2\,(\rho_k\,\sigma_k)^{3/2}\, \|d^\ast\xi\|_{L^\infty(B_r(x_0))}
\end{array}
\]
We deduce from this inequality combined with (\ref{dens-sma}), for $k$ large enough, 
\[
\frac{1}{\rho_k\,\sigma_k}\int_{B_{\rho_k\,\sigma_k}(\rho_k\,y_0+x_0))}|A|^2\ dx^3\le\delta
\]
Thanks to the $\delta$ regularity we have
\[
\|\nabla A\|_{L^\infty(B_{\rho_k\,\sigma_k/2}(\rho_k\,y_0+x_0))}\le \frac{C}{(\rho_k\,\sigma_k)^2}
\]
This gives
\[
\|\nabla^2_y\,g_k\|_{L^\infty(B_{s/4}(y_0))}\le C\,s^{-2}
\]
and we deduce that $g_k$ is converging strongly to $g_\infty$ in $W^{1,2}_{loc}(B_1(0)\setminus\{0\})$. Since there exists $C>0$ such that
\[
\int_{B_s(0)}|\nabla g_k|^2\ dy^3\le C\ s\ ,
\]
we deduce
\[
g_k\longrightarrow g_\infty\qquad\mbox{strongly in}\ W^{1,2}(B_1(0))\ .
\]
We claim that $g_\infty$ is energy minimizing for its boundary data in $W^{1,2}(B_1(0),SU(2))$. Indeed, assume there is a ball $B_r(y_0)\subset B_1(0)$ and $h_\infty\ :\ B_r(y_0)\rightarrow SU(2)$
such that 
\[
h_\infty=g_\infty \quad\mbox{ on }\p B_r(y_0)\qquad\mbox{ and }\qquad \int_{B_r(y_0)}|dh_\infty|^2(y)\ dy^3<\int_{B_r(y_0)}|dg_\infty|^2(y)\ dy^3
\]
then we use again Luckhaus Lemma but this time between $g_k$ and $h_\infty$ in $B_r(y_0)$ and we obtain a contradiction with the minimality of $(u,A)$. Now using proposition 1.2 of \cite{SU3} we obtain that $g_\infty$ is constant which is a contradiction. Hence $(u,A)$ is smooth in a neighbourhood of $x_0$ and this concludes the proof of  theorem~\ref{th-reg}.
\section{ The Asymptotic when $\la=\la_0$, $\mu=\mu_0$ are fixed and $\ep\rightarrow 0$ }
\reset
In this section we study the asymptotic behaviour of a sequence of minimizers of $(u_k,A_k)$ $CYMH^{\la,\mu}_{\ep_k}$ in $\mathcal E$ under the boundary condition $u=\phi$ where $\phi\in H^{1/2}(\p \B^3,{\S}^2)$ assuming  $\la=\la_0\ge 1$, $\mu=\mu_0\ge 1$ and $\ep_k\rightarrow 0$.

To simplify notations we shall often skip to write explicitly the index $k$. 
\subsection{$L^\infty$ estimate of $\nabla_Au$ and the control of the number of Monopoles.}

We first prove the following Lemma
\begin{Lm}
\label{lm-linfty-b}
Let $\ep\in (0,1)$ and $(u,A)$ be a critical point of $CYMH^{\la,\mu}_\ep$ such that $$CYMH^{\la,\mu}_\ep(u,A)=M\ ,$$ and $u=\phi$ at $\p \B^3$ where $\phi\in C^2(\p\B^3,\S^2)$. Then, for $\ep>0$ small enough there exists $C_{M,\phi}$ depending only on $M$ and the $C^2$ norm of $\phi$ and not on $\la\ge 1$ and $\mu\ge 1$ such that
\be
\label{linf}
\|\nabla_Au\|_{L^\infty(\B_1(0))}\le  \frac{C(M,\phi)}{\ep}\ \lf(1+\sqrt{\frac{\la}{\mu}}\rg)\ .
\ee
\hfill $\Box$
\end{Lm}
\noindent{\bf Proof of lemma~\ref{lm-linfty-b}}  We present the proof in the case when $\B_r(x_0)\subset \B^3$. A similar proof can be adapted to the case where $x_0\in \p\B^3$ using elliptic estimates up to the boundary and using the strong assumptions on $\phi$ (i.e. $\phi\|_{C^2(\p\B^3)}<+\infty$)
\[
\begin{array}{l}
\ds\int_{\B_r(x_0)}|F_A|^{3/2}\ dx^3\le C\, (r^3)^{1/4}\ \lf(\int_{B_r(x_0)}|F_A|^{2}\ dx^3\rg)^{3/4}\\[5mm]
\ds\le \frac{C}{\mu^{3/4}}\,\lf(\frac{r}{\ep}\rg)^{3/4}\ \lf(\mu\,\int_{B_r(x_0)}\ep\,|F_A|^{2}\ dx^3\rg)^{3/4}\le C\ \frac{M^{3/4}}{\mu^{3/4}}\,\lf(\frac{r}{\ep}\rg)^{3/4}\le  C\ {M^{3/4}}\,\lf(\frac{r}{\ep}\rg)^{3/4}
\end{array}
\]
Hence there exists $\al\in (0,1)$ independent of $\ep$ but only on $M=CYMH_\ep^{\la,\mu}(u,A)$ such that
\[
\lf(\int_{\B_{\al\ep}(x_0)}|F_A|^{3/2}\ dx^3\rg)^{2/3}<\min\{\delta_{3/2},\delta_2\}:=\delta\ .
\]
where $\delta_{3/2}$ and $\delta_2>0$ are given by (\ref{II.0}). We extract a Coulomb Gauge $A^g:=g^{-1}dg+g^{-1}\,A\, g$ on $\B_{\al\ep}(x_0)$. We apply (\ref{II.1}) for $p=3/2$ and $p=2$ and we get
\[
\lf\{
\begin{array}{l}
\|A^g\|_{L^{(3,3/2)}(\B_{\ep\al}(x_0))}+\|\nabla(A^g)\|_{L^{3/2}(\B_{\ep\al}(x_0))}\le C\ \|F_{A}\|_{L^{3/2}(\B_{\ep\al}(x_0))}\le C\, \delta\\[5mm]
\ds\|A^g\|_{L^{(6,2)}(\B_{\ep\al}(x_0))}+\|\nabla(A^g)\|_{L^2(\B_{\ep\al}(x_0))}\le C\ \|F_{A}\|_{L^2(\B_{\ep\al}(x_0))}\le C \sqrt{\frac{M}{\mu}}\ \ep^{-1/2}\le C\,\sqrt{\frac{M}{\ep}}\ .
\end{array}
\rg.
\]
We deduce in particular thanks to the generalised Littlewood inequality
\[
\begin{array}{l}
\ds\|A^g\|^2_{L^{(3,1)}(B_{\ep\al}(x_0))}\le \|A^g\|_{L^2(B_{\ep\al}(x_0))}\, \|A^g\|_{L^{(6,2)}(B_{\ep\al}(x_0))}\\[5mm]
\ds\quad\le\, C\, \sqrt{\ep\,\al} \  \|A^g\|_{L^3(B_{\ep\al}(x_0))}\, \|A^g\|_{L^{(6,2)}(B_{\ep\al}(x_0))}\le C_M\, \sqrt{\al}\, \delta^{2/3}
\end{array}
\]
Observe that $F_{A^g}$ satisfies in $B_{\ep\al}(x_0)$
\[
\lf\{
\begin{array}{l}
\ds d F_{A^g}=-A^g\wedge F_{A^g}\\[5mm]
\ds d^\ast F_{A^g}=A^g\res F_{A^g}-\, g^{-1}\frac{[u,\nabla_A u]}{\ep^2}\,g-\frac{1}{\mu\,\ep}\,g^{-1}\,A\,g\ .
\end{array}
\rg.
\]
This gives
\[
-\,\Delta A^g+d^\ast\lf(  A^g\wedge A^g  \rg)-A^g\res dA^g-A^g\res(A^g\wedge A^g)=-\, \frac{[u^g,\nabla_{A^g} u^g]}{\ep^2}-\frac{1}{\ep}\,g^{-1}\,A\,g
\]
Let $B\in W^{1,2}(\wedge^1B_{\ep\al}(x_0),\frak{su}(2))$ and $ D\in L^2(\wedge^2B_{\ep\al}(x_0),\frak{su}(2))$ solving
\[
-\,\Delta B+d^\ast\lf(  A^g\wedge B  \rg)-A^g\res dB-A^g\res(A^g\wedge B)= D
\]
Using the fact that $\|A^g\|_{L^{(3,1)}(B_{\ep\al}(x_0))}$ is small enough, classical elliptic estimates give
\[
\|B\|_{L^\infty(B_{2^{-1}\,\ep\al}(x_0))}+\|\nabla B\|_{L^{(3,1)}(B_{2^{-1}\,\ep\al}(x_0))}\le C\, \|D\|_{L^{(3/2,1)}(B_{\ep\al}(x_0))}+C\, (\ep\al)^{-1}\,\|B\|_{L^3(B_{\ep\al}(x_0))}
\]
Hence
\[
\begin{array}{l}
\ds\|A^g\|_{L^\infty(B_{2^{-1}\,\ep\al}(x_0))}+\|\nabla A^g\|_{ L^{(3,1)}(B_{2^{-1}\,\ep\al}(x_0)) }\\[5mm]
\ds\qquad\le C\, \ep^{-2}\,\|\nabla_Au\|_{L^{(3/2,1)}(B_{\ep\al}(x_0))}+ C\, \ep^{-1}\,\|A\|_{L^{(3/2,1)}(B_{\ep\al}(x_0))}+C_M\, (\ep\al)^{-1}\,\delta^{2/3}\ .
\end{array}
\]
We have using H\"older inequalities in Lorentz spaces
\[
\|\nabla_Au\|_{L^{(3/2,1)}(B_{\ep\al}(x_0))}\le \sqrt{\al\ep}\ \|\nabla_Au\|_{L^{2}(B_{\ep\al}(x_0))}
\]
and
\[
\|A\|_{L^{(3/2,1)}(B_{\ep\al}(x_0))}\le \sqrt{\al\ep}\ \|A\|_{L^{2}(B_{\ep\al}(x_0))}\ .
\]
Thus
\[
\begin{array}{l}
\ds\|A^g\|_{L^\infty(B_{2^{-1}\,\ep\al}(x_0))}+\|\nabla A^g\|_{ L^{(3,1)}(B_{2^{-1}\,\ep\al}(x_0)) }\\[5mm]
\ds\quad\le C_\al\, \ep^{-3/2}\ \|\nabla_Au\|_{L^{2}(B_{\ep\al}(x_0))}+C_\al\, \ep^{-1/2}\ \|A\|_{L^{2}(B_{\ep\al}(x_0))}+C_M\, \ep^{-1}\\[5mm]
\ds\quad\le C_M\, \ep^{-1}
\end{array}
\]
where $C_M>0$ depends only on  $M=CYMH^{\la,\mu}_\ep(u,A)$.

\medskip

The Euler Lagrange equation (\ref{CYMHS-co})
 is implying
\[
-2^{-1}\Delta(1-|u|^2)+\,\la\,\frac{|u|^2}{\ep^2}\, (1-|u|^2)=|\nabla_Au|^2\ge 0 .
\]
Since we are assuming $|u|\equiv 1$ on $\p \B_1(0)$, thanks to the maximum principle we have
\be
\label{linfty-module}
\forall x\in \B_1(0)\qquad|u|(x)\le 1\ .
\ee
We have also
\[
\begin{array}{l}
\ds\sum_{l=1}^3(\nabla_A^g)_l((\nabla_A^g)_l u^g)=\sum_{l=1}^3\p_{x_l}(\p_{x_l}u^g+[A^g_l,u])+[A^g_l, \p_{x_l}u^g+[A^g_l,u^g]]\\[5mm]
\ds=\Delta u^g+2\,\sum_{l=1}^3[A_l^g,\p_{x_l}u^g]+[A_l^g[A_l^g,u]] 
\end{array}
\]
thus
\be
\label{eq-ug}
\Delta u^g=-2\,[A^g;\nabla u^g]-[A^g[A^g,u^g]]-\frac{\la}{\ep^2}u^g\,(1-|u|^2)\ .
\ee
We deduce using classical Calderon Zygmund theory
\[
\begin{array}{l}
\ds\|\Delta u^g\|_{L^2\lf(B_{4^{-1}\,\ep\al}(x_0)\rg)}\le C\ \|A^g\|_{L^\infty(B_{2^{-1}\,\ep\al}(x_0))}\,\|\nabla u^g\|_{L^2\lf(B_{2^{-1}\,\ep\al}(x_0)\rg)}\\[5mm]
\ds\quad+C\,\|A^g\|^2_{L^\infty\lf(B_{2^{-1}\,\ep\al}(x_0)\rg)} \, \ep^{3/2}+C\,\frac{\la}{\ep^{2}}\, \lf( \int_{B_{2^{-1}\,\ep\al}(x_0)}(1-|u|^2)^2\ dx^3  \rg)^{1/2}+C\, \ep^{-1}\,\|\nabla u^g\|_{L^2\lf(B_{2^{-1}\,\ep\al}(x_0)\rg)}
\end{array}
\]
Then we bound
\[
\|\nabla u^g\|_{L^2\lf(B_{2^{-1}\,\ep\al}(x_0)\rg)}\le\lf[\|\nabla_Au\|_{L^2\lf(B_{2^{-1}\,\ep\al}(x_0)\rg)}+\|A^g\|_{L^2\lf(B_{2^{-1}\,\ep\al}(x_0)\rg)}\rg]\
\]
We have
\[
\|A^g\|_{L^2\lf(B_{2^{-1}\,\ep\al}(x_0)\rg)}\le C\,\sqrt{\ep}\, \|A^g\|_{L^3\lf(B_{2^{-1}\,\ep\al}(x_0)\rg)}\le C\,\sqrt{\ep}\ \delta^{1/3}
\]
We deduce
\[
\|\nabla u^g\|_{L^2\lf(B_{2^{-1}\,\ep\al}(x_0)\rg)}\le\, C\, \sqrt{\frac{\ep\,M}{\mu}}+C\,\sqrt{\ep}\ \delta^{1/3}\ .
\]
Hence
\[
\|\Delta u^g\|_{L^{2}\lf(B_{2^{-1}\,\ep\al}(x_0)\rg)}\le C_M\ \ep^{-1/2}+C\,\sqrt{\frac{\la}{\mu}}\ \ep^{-1/2}\,\sqrt{M}\ .
\]
This implies
\[
\|\Delta u^g\|_{L^{(3/2,1)}\lf(B_{2^{-1}\,\ep\al}(x_0)\rg)}\le C\,\sqrt{\ep}\, \|\Delta u^g\|_{L^{2}\lf(B_{2^{-1}\,\ep\al}(x_0)\rg)}\le C_M\,\lf(1+\sqrt{\frac{\la}{\mu}}\rg)\ .
\]
This gives
\[
\begin{array}{l}
\ds\|\nabla u^g\|_{L^{(3,1)}\lf(B_{4^{-1}\,\ep\al}(x_0)\rg)}\le\, C\,\|\Delta u^g\|_{L^{(3/2,1)}\lf(B_{2^{-1}\,\ep\al}(x_0)\rg)}+C\, \ep^{-1/2}\,\|\nabla u^g\|_{L^2\lf(B_{2^{-1}\,\ep\al}(x_0)\rg)}\\[5mm]
\ds\quad\le C_M\,\lf(1+\sqrt{\frac{\la}{\mu}}\rg)\ .
\end{array}
\]
Injecting this estimate in (\ref{eq-ug}) is giving
\[
\begin{array}{l}
\ds\|\Delta u^g\|_{L^{(3,1)}\lf(B_{4^{-1}\,\ep\al}(x_0)\rg)}\le C\ \|A^g\|_{L^\infty(B_{4^{-1}\,\ep\al}(x_0))}\,\|\nabla u^g\|_{L^{(3,1)}\lf(B_{4^{-1}\,\ep\al}(x_0)\rg)}\\[5mm]
\ds\quad+C\,\ep\,\|A^g\|^2_{L^\infty\lf(B_{4^{-1}\,\ep\al}(x_0)\rg)} +C\,\ep^{-1}\, 
\end{array}
\]
Thus
\[
\|\Delta u^g\|_{L^{(3,1)}\lf(B_{4^{-1}\,\ep\al}(x_0)\rg)}\le C_M\,\lf(1+\sqrt{\frac{\la}{\mu}}\rg)\ \ep^{-1}\ .
\]
Calderon Zygmund estimates implies
\[
\begin{array}{l}
\|\nabla^2  u^g\|_{L^{(3,1)}\lf(B_{8^{-1}\,\ep\al}(x_0)\rg)}\le C\ \|\Delta u^g\|_{L^{(3,1)}\lf(B_{4^{-1}\,\ep\al}(x_0)\rg)}+C\,\ep^{-1}\,\|\nabla u^g\|_{L^{(3,1)}\lf(B_{4^{-1}\,\ep\al}(x_0)\rg)}\\[5mm]
\ds\quad\le\,  C_M\,\lf(1+\sqrt{\frac{\la}{\mu}}\rg)\ \ep^{-1} \ .
\end{array}
\]
Hence
\[
\begin{array}{l}
\ds\|\nabla  u^g\|_{L^{\infty}\lf(B_{4^{-2}\,\ep\al}(x_0)\rg)}\le C\, \|\nabla^2  u^g\|_{L^{(3,1)}\lf(B_{8^{-1}\,\ep\al}(x_0)\rg)}+C\, \ep^{-1}\, \|\nabla u^g\|_{L^{(3,1)}\lf(B_{4^{-1}\,\ep\al}(x_0)\rg)}\\[5mm]
\ds\quad\le  C_M\,\lf(1+\sqrt{\frac{\la}{\mu}}\rg)\ \ep^{-1} \ .
\end{array}
\]
Since 
\[
\|A^g\|_{L^\infty(B_{2^{-1}\,\ep\al}(x_0))}\le C_M\,\ep^{-1}
\]
We deduce, using the fact that $|u|\le 1$
\[
\|\nabla_{A}u\|_{L^\infty(B_{2^{-1}\,\ep\al}(x_0))}=\|\nabla_{A^g}u^g\|_{L^\infty(B_{2^{-1}\,\ep\al}(x_0))}\le C\ \lf(1+\sqrt{\frac{\la}{\mu}}\rg)\ \ep^{-1} \ .
\]
This holds for any $x_0\in B_{r_1}(0)$ and $C$ is independent of $x_0$ and $\ep$ but only on $M$. Thus we deduce (\ref{linf}) in the $\al\ep$ interior of $\B_1(0)$ and taking then $x_0\in\p \B^3(0)$ we can implement the same argument
on the quasi half ball $B_{\la\ep}(x_0)\cap\B^3_1(0)$ taking into account the $C^2$ norm control of $u=\phi$ on $\p\B^3_1(0)\cap \B_{\al\ep}(x_0)$ and lemma~\ref{lm-linfty-b} is proved. \hfill $\Box$

\medskip

\begin{Lm}
\label{lm-zero-set}
Let $\ep\in (0,1)$ and $(u,A)$ a  critical point of $CYMH^{\la,\mu}_\ep$ such that $\la\ge 1$, $\mu\ge 1$ and $CYMH^{\la,\mu}_\ep(u,A)=M$.   Assume moreover that $u=\phi$ on $\p\B^3_1(0)$ where $\phi\in C^2(\p\B^3_1(0),\S^2)$ .  Then for any $t\in (0,1)$ there exists a number $Q\in {\N}$ satisfying
\be
\label{bd-Q}
Q\le \frac{C_{M,\phi}}{\la\mu}\  t^{-5}\ \lf(1+\sqrt{\frac{\la}{\mu}}\rg)^3\ .
\ee
where $C_{M,\phi}$ only depends on $M$ and the $C^2$ norm of $\phi$. Then there exists $Q$ balls of radius $\ep>0$  $(\B_\ep(x_i))_{i=1\cdots Q}$  
\be
\label{cov}
\lf\{x\in \B_{r_1}(0)\ ;\ |u|(x)<1-t\rg\}\subset \bigcup_{i=1}^{Q}\B_\ep(x_i)\ .
\ee
\hfill $\Box$
\end{Lm}
\begin{Rm}
\label{bdyassu}
In order to have Lemma~\ref{lm-linfty-b} and Lemma~\ref{lm-zero-set} to be true, we can weaken the assumption on the boundary value of $u$ by simply taking $|u|\le 1$ on $\p\B_1(0)$ and the same estimates, respectively (\ref{linf}) and (\ref{bd-Q}) and the inclusion (\ref{cov}) hold by restricting to a strictly smaller ball $\B_{r_1}(0)$ instead of $\B_1(0)$ for any $r_1<1$ and having the constant $C_{M,\phi}>$ being replaced by a constant $C_{M,r_1}>0$ depending only on $M$ and $r_1<1$.\hfill $\Box$
\end{Rm}
\noindent{\bf Proof of Lemma~\ref{lm-zero-set}.}
Using Kato inequality and the previous lemma
\[
\|\nabla|u|\|_{L^\infty(\B_{r_1}(0))}\le \|\nabla_Au\|_{L^\infty(\B_{r_1}(0))}\le \ep^{-1}\, C(M,r_1)\ \lf(1+\sqrt{\frac{\la}{\mu}}\rg)\ .
\]
Assume 
\[
|u(x_0)|<1-t
\]
then 
\[
\forall x\in B_{\tau\,\ep}(x_0)\qquad |u(x)|<1-t/2\ .
\]
where
\[
\tau:=\frac{t}{2\,C(M,r_1)\,\lf(1+\sqrt{\frac{\la}{\mu}}\rg)}
\]
This implies
\[
\int_{B_{\tau\ep}(x_0)}\frac{\la\,\mu}{\ep^3}(1-|u|^2)^2\ dx^3\ge \frac{3\pi}{4}\,\la\,\mu\, \ep^{-3}(\tau\ep)^3\ \frac{t^2}{4}\ \lf(\frac{3}{4}\rg)^2=\pi\, \la\,\mu\,\lf(\frac{3}{4}\rg)^3\ \frac{t^2}{4}\ \lf(  \frac{t}{2\,C(M,r_1)\,\lf(1+\sqrt{\frac{\la}{\mu}}\rg)} \rg)^3\ .
\]
We extract a Besicovitch covering $(B_{\tau \ep}(x))_{i\in I}$ from 
\[
(B_{\tau \ep}(x))_{x\in E}\qquad\mbox{ where }\qquad E_t:=\lf\{ x\in \B_{r_1}(0)\ ;\ |u|(x)<1-t  \rg\}
\]
so that every point in $E_t$ is covered by at most $N$ balls. Then we have
\be
\label{cardmon}
\begin{array}{l}
\ds\mbox{Card}(I)\ \pi \la\,\mu\, \frac{t^2}{4}\ \lf(\frac{3}{4}\rg)^3\,\lf(  \frac{t}{2\,C(M,r_1)\,\lf(1+\sqrt{\frac{\la}{\mu}}\rg)} \rg)^3\le \sum_{i\in I}\int_{B_{\tau\ep}(x_i)}\frac{\la\,\mu}{\ep^3}(1-|u|^2)^2\ dx^3\\[5mm]
\ds\qquad\le \int_{\B_1(0)}\sum_{i\in I}{\mathbf 1}_{B_{\tau\ep}(x_i)}(x)\ \frac{\la\,\mu}{\ep^3}(1-|u|^2)^2\ dx^3\le N\, M\ .
\end{array}
\ee
This implies Lemma~\ref{lm-zero-set}.\hfill $\Box$

\medskip

The Lemma~\ref{lm-zero-set} is telling that the number of monopoles being formed  by a critical configuration $(u,A)$ of uniformly bounded energy is uniformly bounded.



\subsection{The Energy Density Control away from the Monopoles and  above the $\sqrt{\ep}$ Scale.}

This section is devoted to the proof of the following lemma which is going to be the first important step  in our proof of the $\delta-$regularity independent of $\ep$ for minimizers. The goal of the lemma is to compensate the absence of obvious monotonicity formula for radii in the range $r\in(\sqrt{\ep},1)$.
\begin{Lm}
\label{lm-sca} Let $0<r_1<1$. There exists $\delta_0>0$ and $R>1$ independent of $\ep\in (0,1)$ but depending only on $r_1$ such that for any minimizer $(u,A)$ of $CYMH^{1,1}_\ep$ and any $\B_r(x_0)\subset\B_{r_1}(0)$ and $r>R\,\sqrt{\ep}$ if
\be
\label{hypdens}
\frac{1}{2\,r}\int_{\B_r(x_0)} |{A}|^2+\frac{|d_{{A}} {u}|^2}{\ep}+\ep\,|F_{{A}}|^2\ dx^3+\frac{1}{2\,\ep^3}(1-|{u}|^2)^2\ dx^3<\delta
\ee
then, for any $R\sqrt{\ep}<\rho<r$ there holds
\be
\label{dens}
\frac{1}{2\,\rho}\int_{\B_\rho(x_0)} |{A}|^2+\frac{|d_{{A}} {u}|^2}{\ep}+\ep\,|F_{{A}}|^2\ dx^3+\frac{1}{2\,\ep^3}(1-|{u}|^2)^2\ dx^3<2\,\delta\ .
\ee
\hfill $\Box$
\end{Lm}
\noindent{\bf Proof of Lemma~\ref{lm-sca}.}
Let $r_1<1$, $r>R\,\sqrt{\ep}$ and $x_0\subset \B_{r_1}(0)$ such that $B_r(x_0)\subset \B_{r_1}(0)$ where $R$ is going to be chosen at the end of the proof of the lemma. Assume
\be
\label{small}
\int_{\p \B_r(x_0)} |A|^2+\frac{|d_A u|^2}{\ep}+\ep\,|F_A|^2\ dx^3+\frac{1}{2\,\ep^3}(1-|u|^2)^2\ dx^3\le\delta\ .
\ee
The goal under this assumption is to find a competitor of lowest possible energy inside the ball $\B_r(x_0)$. A first attempt would be to consider the radial extensions :
\[
(\ti{u},\ti{A}):=(u\circ D_{r,x_0},D_{r,x_0}^\ast A)\qquad\mbox{ where }D_{r,x_0}(x):=r\,\frac{x-x_0}{|x-x_0|}+x_0\ .
\]
We would then have respectively 
\[
|\ti{A}|^2=\lf(\frac{r}{|x-x_0|}\rg)^2|\iota_{\p B_r}^\ast A|^2(D_{r,x_0}(x))\quad\mbox{ and }\quad |d_{\ti{A}}\ti{u}|^2=\lf(\frac{r}{|x-x_0|}\rg)^2|\iota_{\p B_r}^\ast d_Au|^2(D_{r,x_0}(x))\ .
\]
This implies
\[
\int_{\B_r(x_0)}|\ti{A}|^2\ dx^3=r\,\int_{\p \B_r(x_0)}|\iota_{\p \B_r}^\ast A|^2\ dvol_{\p \B_r(x_0)}
\]
 and 
 \[
 \int_{\B_r(x_0)}\frac{|d_{\ti{A}}\ti{u}|^2}{\ep}\ dx^3=r\,\int_{\p \B_r(x_0)}\frac{|\iota_{\p B_r}^\ast d_Au|^2}{\ep}\ dvol_{\p \B_r(x_0)}
\]
We have moreover
\[
\begin{array}{l}
\ds\int_{\B_r(x_0)} \frac{1}{2\,\ep^3}(1-|\ti{u}|^2)^2\ dx^3=\frac{1}{2\,\ep^3}\int_0^r\,d\rho\, \lf(\frac{\rho}{r}\rg)^2\,\int_{\p \B_r(x_0)} (1-|{u}|^2)^2\, dvol_{\p \B_r(x_0)}\\[5mm]
\ds=\frac{r}{3}\, \int_{\p \B_r(x_0)} \frac{1}{2\,\ep^3}(1-|{u}|^2)^2\, dvol_{\p \B_r(x_0)}
\end{array}
\]
The problem is coming from the curvature. Indeed we have
\[
|F_{\ti{A}}|^2=\lf(\frac{r}{|x-x_0|}\rg)^4|\iota_{\p B_r}^\ast F_A|^2(D_{r,x_0}(x))
\]
which implies
\[
\int_{\B_r(x_0)}\ep\,|F_{\ti{A}}|^2\ dx^3=\int_0^r d\rho\, \frac{r^2}{\rho^2}\,\int_{\p \B_r(x_0)}|\iota_{\p B_r}^\ast F_A|^2\ dvol_{\p \B_r(x_0)}=+\infty
\]
Hence the radial extension is offering no suitable strategy and we have to choose another alternative.

\medskip

On $\p \B_r(x_0)$ we denote $\hat{u}:=u/|u|$. We extend $\hat{u}$ by the harmonic extension given by lemma~\ref{lm-ext} and that we denote $\check{u}$. We have then in particular
\be
\label{IV.7}
\|\nabla\check{u}\|_{L^3(\B_r(x_0))}\le C\, \|\nabla\hat{u}\|_{L^2(\p\B_r(x_0))}
\ee
Since $|\nabla_A u|\le\,\ep^{-1}\, CM$, assume there exists $x_1\in \p \B_r$
such that $1-|u(x_1)|>t$ then on $\B_{\tau\ep}(x_1)\cap \p \B_r(x_0)$
\[
1-|u(x)|>t/2
\]
where $\tau:=\frac{t}{2\,C(M,r_1)}$. This implies
\[
\int_{\B_{\tau\ep}(x_1)\cap \p \B_r(x_0)}(1-|u|^2)^2\ dvol_{\p \B_r(x_0)}\ge \pi \tau^2\ep^2\, \frac{t^2}{4}=\pi \,\frac{t^4}{16\, C^2(M,r_1)}\ep^2
\]
We deduce from (\ref{small})
\[
t\le C\, \ep^{1/4}\ .
\]
Hence
\be
\label{linft}
\|1-|u|\|_{L^\infty( \p \B_r(x_0))}\le C\, \ep^{1/4}\ .
\ee
Let $\frak{h}$ be the solution of
\be
\label{frakh}
\lf\{
\begin{array}{l}
\ds-\Delta\frak{h}=0\qquad\mbox{ in }\B_r(x_0)\\[5mm]
\ds\frak{h}=|u|\qquad\mbox{ on }\p\B_r(x_0)\ .
\end{array}
\rg.
\ee
 By the maximum principle we have $0\le\frak{h}\le 1$ on $\B_r(x_0)$ moreover
\[
\|1-\frak{h}\|_{L^\infty(\B_r(x_0))}\le C\, \ep^{1/4}\ .
\]
We have 
\[
\begin{array}{l}
\ds\Delta(|\frak{h}|^2-1)^2=2\,\mbox{div}\lf((|\frak{h}|^2-1)\,\nabla|\frak{h}|^2   \rg)=2\,|\nabla|\frak{h}|^2|^2+2\,(|\frak{h}|^2-1)\ \Delta|\frak{h}|^2\\[5mm]
\ds\quad=8\,|\frak{h}|^2\,|\nabla\frak{h}|^2+4\,(|\frak{h}|^2-1)\ |\nabla\frak{h}|^2=(12\,\frak{h}^2-4)\, |\nabla\frak{h}|^2\ge 0
\end{array}
\]
Denote $\frak{h}_{x_0,r}(y):=\frak{h}(r\,y+x_0)$. We have
\[
\begin{array}{l}
\ds\frac{d}{dr}\int_{\p \B_r(x_0)}(|\frak{h}|^2-1)^2\ dvol_{\p \B_r(x_0)}=\frac{d}{dr}\lf( r^2\,\int_{\p \B_1(0)}(|\frak{h}_{x_0,r}|^2-1)^2\ dvol_{\p \B_1(0)}  \rg)\\[5mm]
\ds =2\, r^{-1}\,\int_{\p \B_r(x_0)}(|\frak{h}|^2-1)^2\ dvol_{\p \B_r(x_0)} + r^2\,\int_{\p \B_1(0)}\frac{\p}{\p r}(|\frak{h}_{x_0,r}|^2-1)^2\ dvol_{\p \B_1(0)}  \\[5mm]
\ds=2\, r^{-1}\,\int_{\p \B_r(x_0)}(|\frak{h}|^2-1)^2\ dvol_{\p \B_r(x_0)} + r^2\,\int_{\B_1(0)}\Delta(|\frak{h}_{x_0,r}|^2-1)^2\ dx^3\\[5mm]
\ds\ge 2\, r^{-1}\,\int_{\p \B_r(x_0)}(|\frak{h}|^2-1)^2\ dvol_{\p \B_r(x_0)} 
\end{array}
\]
This gives
\[
\frac{d}{d\rho}\lf[\frac{1}{\rho^2}\int_{\p \B_\rho(x_0)}(1-|\frak{h}|^2)^2\ dvol_{\p \B_\rho(x_0)}\rg]\ge 0
\]
We can deduce
\be
\label{IV.8}  
\begin{array}{l}
\ds\int_0^rd\rho\int_{\p \B_\rho(x_0)}(1-|\frak{h}|^2)^2\ dvol_{\p \B_\rho(x_0)}\le \int_0^r\frac{\rho^2}{r^2}\,d\rho \int_{\p \B_r(x_0)}(1-|u|^2)^2\ dvol_{\p B_r(x_0)}\\[5mm]
\ds\qquad\le \frac{r}{3}\int_{\p \B_r(x_0)}(1-|u|^2)^2\ dvol_{\p B_r(x_0)}
\end{array}
\ee
We then choose
\[
\ti{u}:= \frak{h}\ \check{u}\qquad\mbox{ in }\B_r(x_0)\ .
\]
This gives in particular
\be
\label{IV.9}
\int_{ \B_r(x_0)}\frac{1}{2\ep^3}(1-|\ti{u}|^2)^2\  dx^3\le\frac{r}{3}\int_{\p \B_r(x_0)}\frac{1}{2\ep^3}(1-|u|^2)^2\ dvol_{\p B_r(x_0)}\ .
\ee
Using the stationarity of $\frak{h}$ with respect to the Dirichlet energy one obtains the following identity 
\be
\label{dirmod}
\begin{array}{l}
\ds\int_{\B_r(x_0)}|\nabla\frak{h}|^2\ dx^3+ r\,\int_{\p \B_r(x_0)}|\p_r\frak{h}|^2\ dvol_{\p \B_r(x_0)}=  r\,\int_{\p \B_r(x_0)}|\nabla_T\frak{h}||^2\ dvol_{\p \B_r(x_0)}\\[5mm]
\ds\quad=r\,\int_{\p \B_r(x_0)}|\nabla_T|{u}||^2\ dvol_{\p \B_r(x_0)}
\end{array}
\ee
Let $\ov{\frak{h}}$ be the average of $\frak{h}$ (that is $|u|$) on $\p\B_r(x_0)$. Multiplying (\ref{frakh}) by $\frak{h}-\ov{\frak{h}}$ and integrating by parts is giving
\be
\label{poin}
\begin{array}{l}
\ds\int_{\B_r(x_0)}|\nabla\frak{h}|^2\ dx^3=\int_{\p\B_r(x_0)}(\frak{h}-\ov{\frak{h}})\,\p_r\frak{h}\ dvol_{\p\B_r(x_0)}\\[5mm]
\ds\le \|\frak{h}-\ov{\frak{h}}\|_{L^2(\p\B_r(x_0))}\ \|\p_r\frak{h}\|_{L^2(\p\B_r(x_0))}\le C\, r\ \|\nabla_T\frak{h}\|_{L^2(\p\B_r(x_0))}\ \|\p_r\frak{h}\|_{L^2(\p\B_r(x_0))}
\end{array}
\ee
where we have used Poincar\'e inequality in the last step. Combining (\ref{dirmod}) and (\ref{poin}) we obtain
\[
\begin{array}{l}
\ds r\,\int_{\p \B_r(x_0)}|\nabla_T|{u}||^2\ dvol_{\p \B_r(x_0)}-r\,\int_{\p \B_r(x_0)}|\p_r\frak{h}|^2\ dvol_{\p \B_r(x_0)}\le  C\, r\ \|\nabla_T\frak{h}\|_{L^2(\p\B_r(x_0))}\ \|\p_r\frak{h}\|_{L^2(\p\B_r(x_0))}\\[5mm]
\ds\le  2^{-1}\,r\,\int_{\p \B_r(x_0)}|\nabla_T|{u}||^2\ dvol_{\p \B_r(x_0)}+ 2^{-1}\, C^2\ r\,\int_{\p \B_r(x_0)}|\p_r\frak{h}|^2\ dvol_{\p \B_r(x_0)}\
\end{array}
\]
This implies
\[
r\,\int_{\p \B_r(x_0)}|\nabla_T|{u}||^2\ dvol_{\p \B_r(x_0)}\le (C^2+2)\ r\,\int_{\p \B_r(x_0)}|\p_r\frak{h}|^2\ dvol_{\p \B_r(x_0)}
\]
Injecting this inequality in (\ref{dirmod}) we obtain
\[
(C^2+2)\,\int_{\B_r(x_0)}|\nabla\frak{h}|^2\ dx^3+ r\,\int_{\p \B_r(x_0)}|\nabla_T|{u}||^2\ dvol_{\p \B_r(x_0)} \le (C^2+2)\, r\,\int_{\p \B_r(x_0)}|\nabla_T|{u}||^2\ dvol_{\p \B_r(x_0)}
\]
Choosing $1-\nu=(C^2+1)/(C^2+2)$ we have established\footnote{This inequality (\ref{epifrakh}) and the existence of $\nu\in(0,1)$ is known in the minimal surface theory under the name of ``epiperimetric inequality'' (see \cite{Rivep})  It measure the shift between the gain of area (here Dirichlet energy) while replacing a radial extension by an area minimizing (resp. Dirichlet energy minimizing) extension. This is the same phenomenon as in \ref{A-04}. }
\be
\label{epifrakh}
\int_{\B_r(x_0)}|\nabla\frak{h}|^2\ dx^3\le (1-\nu)\  r\,\int_{\p \B_r(x_0)}|\nabla_T|{u}||^2\ dvol_{\p \B_r(x_0)}\ .
\ee
On $\p \B_r(x_0)$ we write using (\ref{LoTra}) 
\[
\iota_{\p \B_r(x_0)}^\ast A=-\frac{1}{4}\lf[\hat{u}\lf[\hat{u},\iota_{\p B_r(x_0)}^\ast A\rg]\rg]+\hat{u}\cdot\iota_{\p B_r(x_0)}^\ast A\ \hat{u}
\]
Then we further decompose
\[
\iota_{\p \B_r(x_0)}^\ast A=-\frac{1}{4}\lf[\hat{u},\iota_{\p B_r(x_0)}^\ast d\hat{u}\rg]+\frac{1}{4}\lf[\hat{u},\iota_{\p B_r(x_0)}^\ast d_A\hat{u}\rg]+\hat{u}\cdot\iota_{\p B_r(x_0)}^\ast A\ \hat{u}
\]
Observe moreover that
\[
\frac{1}{4}\lf[\hat{u},\iota_{\p B_r(x_0)}^\ast d_A\hat{u}\rg]=\frac{1}{4|u|}\lf[\hat{u},\iota_{\p B_r(x_0)}^\ast d_A{u}\rg]
\]
We shall extend $\iota_{\p B_r(x_0)}^\ast A$ inside $B_r(x_0)$ by a one form $\check{A}$ defined in the following way
\[
\check{A}:=-\frac{1}{4}\lf[\check{u},d\check{u}\rg]+{\eta}^T+d^\ast G\ \check{u}
\]
where 
\begin{itemize}
\item[i)] $\check{u}$ is the harmonic map extension of $\hat{u}$ given by Lemma~\ref{lm-ext}

\item[ii)] $G$ is the real valued two form solving
\be
\label{sG}
\lf\{
\begin{array}{l}
\ds  dd^\ast G+\frac{G}{\ep}=\frac{1}{8}\,\check{u}\cdot [d\check{u}\wedge d\check{u}]+d\tie\qquad\mbox{ in }\B_r(x_0)\\[5mm]
\ds d G=0\qquad\mbox{ in }\B_r(x_0)\\[5mm]
\ds \iota_{\p \B_r(x_0)}^\ast d^\ast G= \iota_{\p \B_r(x_0)}^\ast  A\cdot\hat{u}\ .
\end{array}
\rg.
\ee
where $\tie$ is going to be a small error term.
\item[iii)]  ${\eta}$ is the one  $\frak{su}(2)$-valued form solving
\be
\label{seta}
\lf\{
\begin{array}{l}
\ds d^\ast\,d\eta+\frac{\eta}{\ep^2}=0\qquad\mbox{ in }\B_r(x_0)\\[5mm]
\ds d^\ast\eta=0\qquad\mbox{ in }\B_r(x_0)\\[5mm]
\ds\iota_{\p B_\rho(x_0)}^\ast \eta=- \frac{1}{4}[\hat{u},\iota_{\p \B_r(x_0)}^\ast d_A\hat{u}]\ .
\end{array}
\rg.
\ee
and $\eta^T$ is the transversal component of $\eta$ w.r.t. $\check{u}$ and given by
\[
\eta^T:=-\frac{1}{4}\,[\check{u},[\check{u},\eta]]
\]
\end{itemize}

We will first prove the existence of $G$ and estimates for $G$ and then we will prove the existence of $\eta$ and estimates for $\eta$. 

\medskip

In the sequel we will adopt the following notation for $\rho>0$
\[
Y(\rho):=CYMH({u},{A},\B_\rho(x_0))=\frac{1}{2}\int_{\B_\rho(x_0)} |{A}|^2+\frac{|d_{{A}} {u}|^2}{\ep}+\ep\,|F_{{A}}|^2\ dx^3+\frac{1}{2\,\ep^3}(1-|{u}|^2)^2\ dx^3
\]
Hence in particular we have
\[
Y'(r)=\frac{1}{2}\int_{\p\B_r(x_0)} |A|^2+\frac{|d_{{A}} {u}|^2}{\ep}+\ep\,|F_{{A}}|^2\ dx^3+\frac{1}{2\,\ep^3}(1-|{u}|^2)^2\ dvol_{\p\B_r(x_0)}
\]

\medskip

\noindent{\bf Existence and Estimates for $G$}

\medskip

We choose $G=\ast K+\ast L$, where $K$ and $L$ are defined as follows.

\medskip

\noindent{\bf Choice and Estimates for $K$}
 We first choose ${K}$ one form minimizing
\[
\inf\lf\{  \int_{\B_r(x_0)}|d K|^2+\frac{|K|^2}{\ep}-2\,\int_{\B_r(x_0)}dK\wedge \xi\ ;\quad d^\ast K=0\rg\}
\]
where  we now define $\xi$. First we choose $\ti{\frak{f}}$ such that 
\[
d\ti{\frak{f}}=\frac{1}{8}\,\check{u}\cdot [d\check{u}\wedge d\check{u}]
\]
We now precise a special choice of $\ti{\frak{f}}$. Since
\[
\int_{\p \B_r(x_0)}|\nabla\hat{u}|^2\ dvol_{\p \B_r(x_0)}\le \delta
\]
there exists $(e_1,e_2)$ orthonormal frame such that
\[
\int_{\p \B_r(x_0)}|\nabla e_i|^2\ dvol_{\p \B_\rho(x_0)}\le C\, \int_{\p \B_r(x_0)}|\nabla\hat{u}|^2\ dvol_{\p \B_r(x_0)}
\]
moreover $\hat{u}=e_1\wedge e_2$ (\cite{Hel}). This implies
\[
de_1\wedge de_2=\hat{u}^\ast\om
\]
We have the existence of $v\in W^{1,2}(\p \B_r(x_0))$ such that 
\[
\Delta_{S^2}v=\hat{u}^\ast\om=de_1\,\dot{\wedge}\, de_2
\]
which gives
\[
\int_{\p \B_\rho(x_0)}|\nabla v|^2\ dvol_{\p B_\rho(x_0)}\le C\, \lf(\int_{\p B_\rho(x_0)}|\nabla\hat{u}|^2\ dvol_{\p B_\rho(x_0)}\rg)^2
\]
We take $\iota_{\p \B_r(x_0)}^\ast\frak{f}:=(\ast)\, dv$. Then we extend $\frak{f}$ as follows, we consider the solution
\[
\lf\{
\begin{array}{l}
\ds d\ti{\frak{f}}=\frac{1}{8}\,\check{u}\cdot [d\check{u}\wedge d\check{u}]\\[5mm]
\ds d^\ast\ti{\frak{f}}=0\\[5mm]
\iota_{\p \B_r(x_0)}^\ast\ti{\frak{f}} =(\ast)\, dv
\end{array}
\rg.
\] 
Classical Calderon Zygmund theory is implying
\be
\label{estxi}
\begin{array}{l}
\ds\|\ti{\frak{f}}\|_{L^3(\B_r(x_0))}+\|\nabla\ti{\frak{f}}\|_{L^{3/2}(\B_r(x_0))}\le C\, \|d\check{u}\|^2_{L^3(\B_r(x_0))}+C\, \|dv\|_{L^2(\p\B_r(x_0))}\\[5mm]
\ds\qquad\le C\,\int_{\p B_r(x_0)}|\nabla_T\hat{u}|^2\ dvol_{\p \B_r(x_0)}\ .
\end{array}
\ee
Recall that on $\p\B_r(x_0)$ there holds
\[
A= -\frac{1}{4}[\hat{u},d\hat{u}]+\frac{1}{4}\lf[\hat{u}, d_A\hat{u}\rg]+A\cdot\hat{u}\,\hat{u}
\]
this gives
\[
\begin{array}{l}
\ds d(A\cdot\hat{u})=d_AA\cdot \hat{u}+A\cdot d_A\hat{u}=(dA+A\wedge A)\cdot\hat{u}+A\cdot d_A\hat{u}\\[5mm]
\ds= F_A\cdot\hat{u}+\frac{1}{2}\,[A\wedge A]\cdot\hat{u}+A^T\,\dot{\wedge}\,d_A\hat{u}
\end{array}
\]
We have
\[
\begin{array}{l}
\ds\frac{1}{2}\,[A\wedge A]\cdot\hat{u}=\frac{1}{32}\,\hat{u}\cdot[[\hat{u},d\hat{u}],[\hat{u},d\hat{u}]]+\frac{1}{32}\,\hat{u}\cdot[[\hat{u},d_A\hat{u}],[\hat{u},d_A\hat{u}]]-\frac{1}{16}\,\hat{u}\cdot[[\hat{u},d\hat{u}],[\hat{u},d_A\hat{u}]]
\end{array}
\]
Observe that
\[
[[{\bf i},{\bf j}],[{\bf i},{\bf k}]]=-\,[2\,{\bf k},2\,{\bf j}]=4\,[{\bf j},{\bf k}]
\]
Hence we have
\[
\frac{1}{32}\,\hat{u}\cdot[[\hat{u}\,d\hat{u}]\,\wedge\,[\hat{u},d\hat{u}]]=\frac{1}{8}\,\hat{u}\cdot[d\hat{u}\wedge d\hat{u}]
\]
Hence 
\[
\begin{array}{l}
\ds d(A\cdot\hat{u})= F_A\cdot\hat{u}+\frac{1}{8}\,\hat{u}\cdot[d\hat{u}\wedge d\hat{u}]+\frac{1}{32}\,\hat{u}\cdot[[\hat{u},d_A\hat{u}],[\hat{u},d_A\hat{u}]]\\[5mm]
\ds\quad-\frac{1}{16}\,\hat{u}\cdot[[\hat{u},d\hat{u}],[\hat{u},d_A\hat{u}]]+A^T\,\dot{\wedge}\,d_A\hat{u}
\end{array}
\]

We then introduce $\frak{e}$ to be a 1-form solution to
\[
\begin{array}{l}
\ds d\frak{e}:=\frac{1}{32}\,\hat{u}\cdot[[\hat{u},d_A\hat{u}]\,\wedge\,[\hat{u},d_A\hat{u}]]-\frac{1}{16}\,\hat{u}\cdot[[\hat{u},d\hat{u}]\,\wedge\,[\hat{u},d_A\hat{u}]]\\[5mm]
 \ds\quad-\frac{1}{4}[\hat{u},[\hat{u},A]]\dot{\wedge}\, d_A\hat{u}+\dashint_{\p \B_r(x_0)}F_A\cdot\hat{u}\ dvol_{\p \B_r(x_0)}\qquad\mbox{ on }\p\B_r(x_0)
 \end{array}
\]
in such a way that
\be
\label{tie}
d(A\cdot\hat{u}-\frak{f}-\frak{e})=F_A\cdot\hat{u}-\dashint_{\p \B_r(x_0)}F_A\cdot\hat{u}\ dvol_{\p \B_r(x_0)}\qquad\mbox{ on }\p \B_r(x_0)
\ee
We choose $\frak{e}$ such that $d^{(\ast)}\frak{e}=0$ on $\p \B_r(x_0)$. We denote
\[
(F_A\cdot\hat{u})_{x_0,r}:=\dashint_{\p \B_r(x_0)}F_A\cdot\hat{u}\ dvol_{\p \B_r(x_0)}=\frac{1}{4\pi r^2}\,\int_{\B_r(x_0)} d(F_A\cdot \hat{u})=\frac{1}{4\pi r^2}\,\int_{\B_r(x_0)} F_A\,\dot{\wedge}\, d_A\hat{u}\ .
\]
Hence
\[
\lf|(F_A\cdot\hat{u})_{x_0,r}\rg|\le \frac{1}{4\pi r^2}\, \|\sqrt{\ep}\,F_A\|_{L^2(\B_r(x_0))}\ \|\sqrt{\ep}^{-1}d_Au\|_{L^2(\B_r(x_0))}\ .
\]
We have also
\[
\begin{array}{l}
\ds (F_A\cdot\hat{u})_{x_0,r}:=\dashint_{\p \B_r(x_0)}F_A\cdot\hat{u}\ dvol_{\p \B_r(x_0)}=\frac{1}{4\pi\,r^2}\int_{\p \B_r(x_0)}\hat{u}\cdot dA+\frac{1}{2}\hat{u}\cdot [A\wedge A]\\[5mm]
\ds\quad=-\frac{1}{4\pi\,r^2}\int_{\p \B_r(x_0)}d\hat{u}\,\dot{\wedge}\,A+\frac{1}{8\pi\,r^2}\int_{\p \B_r(x_0)}\hat{u}\cdot [A\wedge A]
\end{array}
\]
Thus
\[
| (F_A\cdot\hat{u})_{x_0,r}|\le C\,r^{-2}\,Y'(r)
\]
We have
\be
\label{L3frake}
\begin{array}{l}
\ds\|\frak{e}\|_{L^{3}(\p\B_r(x_0))}\le C\ \lf(\|d\frak{e}\|_{L^{6/5}(\p\B_r(x_0))}+\|d^{(\ast)}\frak{e}\|_{L^{6/5}(\p\B_r(x_0))}\rg)=C\ \|d\frak{e}\|_{L^{6/5}(\p\B_r(x_0))}\\[5mm]
\ds\quad\le C\, \lf(\int_{\p\B_r(x_0)}|d_Au|^{12/5}+|A|^{6/5}\,|d_Au|^{6/5}\ dvol_{\p\B_r(x_0)}\rg)^{5/6}+ C\,\lf(r^2 \,r^{-12/5}\, (Y'(r))^{6/5}\rg)^{5/6}\\[5mm]
\ds\quad\le \lf(\int_{\p\B_r(x_0)}\ep^{3/5} \frac{|d_Au|^{2}}{\ep}\ dvol_{\p\B_r(x_0)}\rg)^{5/6}+\|A\|_{L^2(\p \B_r(x_0))}\, \|d_Au\|_{L^3(\p \B_r(x_0))} +C\,r^{-1/3}\,Y'(r)\\[5mm]
\quad\le C\, \sqrt{\ep}\,(Y'(r))^{5/6}+(Y'(r))^{1/2+1/3}+C\,r^{-1/3}\,Y'(r)
\end{array}
\ee
Classical elliptic estimate are giving 
\be
\label{frake}
\begin{array}{l}
\ds\|\frak{e}\|_{L^4(\p\B_r(x_0))}+\|\nabla\frak{e}\|_{L^{4/3}(\p \B_r(x_0))}\\[5mm]
\ds\le C\, \||d_Au|\,|d\hat{u}|\|_{L^{4/3}(\p \B_r(x_0))}+C\,\||d_Au|^2\|_{L^{4/3}(\p \B_r(x_0))}+C\,|(F_A\cdot\hat{u})_{x_0,r}|\, r^{3/2}\\[5mm]
\ds\le C\, \ep^{-1/4}\,\lf(\int_{\p \B_r(x_0)}|\nabla_T\hat{u}|^2\ dvol_{\p\B_r(x_0)}\rg)^{1/2}\, \lf(\int_{\p \B_r(x_0)}\ep^{-1}\,|\nabla_A{u}|^2\ dvol_{\p\B_r(x_0)}\rg)^{1/4}\\[5mm]
\ds+C\,\ep^{1/4}\,\lf(\int_{\p \B_r(x_0)}\ep^{-1}\,|\nabla_A{u}|^2\ dvol_{\p\B_r(x_0)}\rg)^{3/4}+C\,r^{-1/2}\,Y'(r)\\[5mm]
\ds\le C\, \ep^{-1/4}\,(Y'(r))^{3/4}+ C\,r^{-1/2}\,Y'(r)
\end{array}
\ee
Then we extend $\frak{e}$ by solving the following problem
\[
\lf\{
\begin{array}{l}
\ds d^\ast d\tie=0\qquad\mbox{ in }\B_r(x_0)\\[5mm]
\ds d^\ast\tie=0\qquad\mbox{ in }\B_r(x_0)\\[5mm]
\iota_{\p B_r(x_0)}^\ast \tie=\frak{e}\ .
\end{array}
\rg.
\]
Classical elliptic estimates is giving
\be
\label{tie}
r^{-1}\,\|\tie\|_{L^2(\B_r(x_0))}+\|d\tie\|_{L^2(\B_r(x_0))}\le C\, \|\frak{e}\|_{L^4(\p\B_r(x_0))}+C\,\|\nabla\frak{e}\|_{L^{4/3}(\p \B_r(x_0))}
\ee
Then we take
\[
\xi:=\ti{\frak{f}}+\tie
\]
We have then in particular
\[
d\xi=\frac{1}{8}\,\check{u}\cdot [d\check{u}\wedge d\check{u}]+d\tie\ .
\]
We have
\be
\label{estxi}
\begin{array}{l}
\ds\|\iota^\ast_{\p\B_r(x_0)}\xi\|_{L^2(\p\B_r(x_0))}\le \|\iota^\ast_{\p\B_r(x_0)}\frak{f}\|_{L^2(\p\B_r(x_0))}+\|\iota^\ast_{\p\B_r(x_0)}\frak{e}\|_{L^2(\p\B_r(x_0))}\\[5mm]
\ds\quad\le C\,\int_{\p\B_r(x_0)}|\nabla\hat{u}|^2\ dvol_{\p\B_r(x_0)}+ C\, r^{1/3}\, \|\iota^\ast_{\p\B_r(x_0)}\frak{e}\|_{L^3(\p\B_r(x_0))}\\[5mm]
\ds\quad\le C\, Y'(r)+r^{1/3}\,(Y'(r))^{5/6}
\end{array}
\ee
Combining (\ref{estxi}), (\ref{frake}) and (\ref{tie}) we have
\be
\label{L2xi}
\|\xi\|_{L^2(\B_r(x_0))}\le C\, r\, \ep^{-1/4}\,(Y'(r))^{3/4}+ C\,r^{1/2}\,Y'(r)+C r^{1/3}\,Y'(r)
\ee

We consider a minimizing sequence of
\[
\begin{array}{l}
\ds+\infty>\limsup_{k\rightarrow +\infty}\int_{\B_r(x_0)}|d {K}_k|^2+\frac{|{K}_k|^2}{\ep}\ dx^3-2\,d{K}_k\wedge \xi\\[5mm]
\end{array}
\]
Since $\check{u}\in W^{1,3}(\B_r(x_0))$, we have that  $\frac{1}{8}\,\ast\,\check{u}\cdot [d\check{u}\wedge d\check{u}]\in L^{3/2}(\B_r(x_0))$ we can choose  $\xi\in W^{1,/3/2}(\B_r(x_0))\ \hookrightarrow\ L^{3}(\B_r(x_0))$ and using Cauchy Schwartz there holds
\[
\limsup_{k\rightarrow +\infty}\int_{\B_r(x_0)}|d {K}_k|^2+|d^\ast{K}_k|^2+{|{K}_k|^2}<+\infty
\]
Hence modulo extraction of a subsequence we have  $K_k\rightharpoonup K_\infty$ and $dK_k\rightharpoonup dK_\infty$ as well as $0=d^\ast K_k\rightharpoonup d^\ast K_\infty=0$ weakly in $L^2$ which ensures the existence of a minimizer. For every $\Phi\in C^\infty(\B_r(x_0))$ such that $d^\ast\Phi=0$ we have
\be
\label{Kinf}
\int_{\B_r(x_0)}d {K}_\infty\cdot d\Phi+\frac{1}{\ep}K_\infty\cdot \Phi \ dx^3-\,d\Phi\wedge \xi =0
\ee
Let $\Psi\in C^\infty(\B_r(x_0))$ arbitrary. We proceed to the following unique Hodge decomposition of $\Psi$
\[
\Psi=d\varphi+d^\ast\zeta
\]
where $\varphi=0$ on $\p B_r(x_0)$. Since obviously  $\varphi$ and $d^\ast\zeta$  are smooth and $d^\ast d^\ast\zeta=0$ there holds
\[
\begin{array}{l}
\ds\int_{\B_r(x_0)}d {K}_\infty\cdot d\Phi+\frac{1}{\ep}K_\infty\cdot \Phi \ dx^3-d\Phi\wedge \xi =\int_{\B_r(x_0)}\frac{1}{\ep}K_\infty\cdot d\varphi\ dx^3\\[5mm]
\ds=\int_{\B_r(x_0)}\frac{1}{\ep}\,d\varphi\wedge\ast K_\infty=\int_{\B_r(x_0)}\frac{1}{\ep}\,d\lf(\varphi\wedge\ast K_\infty\rg)-\int_{\B_r(x_0)}\frac{1}{\ep}\,\varphi\wedge d\ast K_\infty=0
\end{array}
\]
Taking first $\Phi\in C^\infty_0(B_\rho(x_0))$ we obtain
\[
 d^\ast d {K}_\infty+\frac{{K}_\infty}{\ep}=\frac{1}{8}\,\ast\check{u}\cdot [d\check{u}\wedge d\check{u}]\qquad\mbox{ in }\B_r(x_0)
\]
Taking then for $\Phi\in C^\infty(\B_r(x_0))$ arbitrary we obtain
\[
\int_{\p \B_r(x_0)}\Phi\,\dot{\wedge}\,\lf(\ast dK_\infty- \xi\rg)=0
\]
 Then  finally $K:=K_\infty$ satisfies
\[
\lf\{
\begin{array}{l}
\ds  d^\ast d {K}+\frac{{K}}{\ep}=\frac{1}{8}\,\ast\check{u}\cdot [d\check{u}\wedge d\check{u}]+\ast d\ti{\frak{e}}=\ast d\xi\qquad\mbox{ in }\B_r(x_0)\\[5mm]
\ds d^\ast {K}=0\qquad\mbox{ in }\B_r(x_0)\\[5mm]
\ds \iota_{\p B_\rho(x_0)}^\ast \ast d{K}=\iota_{\p B_\rho(x_0)}^\ast \xi\qquad\mbox{ on }\p \B_r(x_0)
\end{array}
\rg.
\]
Multiplying the equation by $K$ and integrating on $\B_r(x_0)$ is giving
\[
\begin{array}{l}
\ds 0=\int_{\B_r(x_0)}K\cdot d^\ast d {K}+\frac{|{K}|^2}{\ep}\ dx^3-\int_{\B_r(x_0)}K\cdot \ast d\xi\ dx^3\\[5mm]
\ds=\int_{\B_r(x_0)}K\wedge d \lf(\ast dK-\xi\rg)+\int_{\B_r(x_0)}\frac{|{K}|^2}{\ep}\ dx^3
\end{array}
\]
Integrating by parts we deduce
\be
\label{estK}
\int_{\B_r(x_0)}|dK|^2+\frac{|K|^2}{\ep}\le \frac{1}{2\ep}\int_{\B_r(x_0)}|K|^2\ dx^3+\frac{\ep}{4}\int_{\B_r(x_0)}|\xi|^2\ dx^3
\ee
We deduce using (\ref{L2xi})
\be
\label{estK}
\begin{array}{l}
\ds\int_{\B_r(x_0)}|dK|^2+\frac{|K|^2}{2\ep}\le  C\, r^2\, \ep^{1/2}\,(Y'(r))^{3/2}+ C \ep\,r^{2/3}\,(Y'(r))^2
\end{array}
\ee
\medskip

\noindent{\bf Choice and Estimates for $L$}

\medskip

Now we choose $L$ to be a minimizer of the following problem
\be
\label{minl}
\inf\lf\{\int_{\B_r(x_0)}|d L|^2+\frac{|L|^2}{\ep}-2\,\int_{\p \B_r(x_0)} L\,\dot{\wedge} (\,A\cdot\hat{u}-\xi)\  ;\quad d^\ast L=0\rg\}\ .
\ee
The existence of $L$ is a consequence of the following facts. Keeping in mind that $(\hat{u},A)$  is smooth in $\B_r(x_0)$ we write
\[
\begin{array}{l}
\ds\int_{\p \B_r(x_0)} L\,\dot{\wedge} (\,A\cdot\hat{u}-\xi)=\int_{\p \B_r(x_0)} d\lf(L\,\dot{\wedge} (\,A\cdot\hat{u}-\xi)\rg)\\[5mm]
\ds=\int_{\p \B_r(x_0)} dL\,\dot{\wedge} (\,A\cdot\hat{u}-\xi)-\int_{\p \B_r(x_0)} L\,\dot{\wedge}\, d(\,A\cdot\hat{u}-\xi)
\end{array}
\]
Thus using Cauchy Schwartz we have the bound
\[
\begin{array}{l}
\ds\lf| \int_{\p \B_r(x_0)} L\,\dot{\wedge} (\,A\cdot\hat{u}-\xi) \rg|\\[5mm]
\ds\quad\le \|dL\|_{L^2(\B_r(x_0))}\, \|A\cdot\hat{u}-\xi\|_{L^2(\B_r(x_0))}+ \|L\|_{L^2(\B_r(x_0))}\, \|d(A\cdot\hat{u}-\xi)\|_{L^2(\B_r(x_0))}
\end{array}
\]
Hence, for any minimizing sequence $L_k$ we have that both $L_k$ and $dL_k$ are uniformly bounded in $L^2(\B_r(x_0))$ and we can assume, modulo extraction of a subsequence that
$L_k\rightharpoonup L_\infty$ and $dL_k\rightharpoonup dL_\infty$ as well as $0=d^\ast L_k\rightharpoonup d^\ast L_\infty=0$ weakly in $L^2$ which ensures the existence of a minimizer using the lower semi-continuity of the $L^2$ norm and the fact that
\[
\begin{array}{l}
\ds\int_{\p \B_r(x_0)} L_k\,\dot{\wedge} (\,A\cdot\hat{u}-\xi)\\[5mm]
\ds=\int_{\p \B_r(x_0)} dL_k\,\dot{\wedge} (\,A\cdot\hat{u}-\xi)-\int_{\p \B_r(x_0)} L_k\,\dot{\wedge}\, d(\,A\cdot\hat{u}-\xi)\\[5mm]
\ds\rightarrow \int_{\p \B_r(x_0)} dL_\infty\,\dot{\wedge} (\,A\cdot\hat{u}-\xi)-\int_{\p \B_r(x_0)} L_\infty\,\dot{\wedge}\, d(\,A\cdot\hat{u}-\xi)\\[5mm]
\ds=\int_{\p \B_r(x_0)} L_\infty\,\dot{\wedge} (\,A\cdot\hat{u}-\xi)
\end{array}
\]
For every $\Phi\in C^\infty(\B_r(x_0))$ such that $d^\ast\Phi=0$ we have
\be
\label{Kinf}
\int_{\B_r(x_0)}d {L}_\infty\cdot d\Phi+\frac{1}{\ep}L_\infty\cdot \Phi \ dx^3-\int_{\p \B_r(x_0)}\,\Phi\dot{\wedge} (\,A\cdot\hat{u}-\xi) =0
\ee
Let $\Psi\in C^\infty(\B_r(x_0))$ arbitrary. We proceed to the following unique Hodge decomposition of $\Psi$
\[
\Psi=d\varphi+d^\ast\zeta
\]
where $\varphi=0$ on $\p B_r(x_0)$. Since obviously  $\varphi$ and $d^\ast\zeta$  are smooth and $d^\ast d^\ast\zeta=0$ there holds
\[
\begin{array}{l}
\ds\int_{\B_r(x_0)}d {L}_\infty\cdot d\Psi+\frac{1}{\ep}L_\infty\cdot \Psi \ dx^3-\int_{\p \B_r(x_0)}\,\Psi\dot{\wedge} (\,A\cdot\hat{u}-\xi) =\\[5mm]
\ds=\int_{\B_r(x_0)}\frac{1}{\ep}L_\infty\cdot d\varphi\ dx^3=\int_{\B_r(x_0)}\frac{1}{\ep}d\varphi\wedge \ast L_\infty\\[5mm]
\ds=\int_{\B_r(x_0)}\frac{1}{\ep}d\lf(\varphi\wedge \ast L_\infty\rg)-\int_{\B_r(x_0)}\frac{1}{\ep}\varphi\wedge d\ast L_\infty=0
\end{array}
\]
where we used respectively that $d^\ast L_\infty=0$ and $\varphi=0$ on $\p\B_r(x_0)$. Taking first $\Psi\in C^\infty_0(\B_r(x_0))$ we deduce
\[
d^\ast d L_\infty+\frac{L_\infty}{\ep}=0\qquad\mbox{ in }{\mathcal D}'(\B_r(x_0))\ .
\]
Then taking $\Psi$ arbitrary we deduce
\[
\int_{\p\B_r(x_0)} \Psi\wedge \iota_{\p\B_r(x_0)}^\ast( \ast d L_\infty -A\cdot\hat{u}+\xi)=0
\]
Taking $L:=L_\infty$ we finally summarise the identities above as follows
\[
\lf\{
\begin{array}{l}
\ds  d^\ast d L+\frac{L}{\ep}=0\qquad\mbox{ in }B_\rho(x_0)\\[5mm]
\ds d^\ast L=0\qquad\mbox{ in }B_\rho(x_0)\\[5mm]
\ds \iota_{\p B_\rho(x_0)}^\ast \ast dL= \iota_{\p B_\rho(x_0)}^\ast  (A\cdot\hat{u}-\xi)\qquad\mbox{ on }\p B_\rho(x_0)
\end{array}
\rg.
\]
We consider a perturbation $\phi_t=x+t\,X$ where $X\in C^\infty_0(\B_r(x_0))$. 
We multiply the equation verified by $L$ inside $\B_r(x_0)$ by $\lf.\frac{d}{dt}\phi_t^\ast L\rg|_{t=0}$ and we integrate by parts, taking into consideration that this perturbation is compactly supported in $\B_r(x_0)$. This gives
\[
\begin{array}{l}
\ds 0=\int_{\B_r(x_0)} d\lf.\frac{d}{dt}\phi_t^\ast L\rg|_{t=0}\cdot dL\ dx^3+\frac{1}{\ep}\int_{\B_r(x_0)} \lf.\frac{d}{dt}\phi_t^\ast L\rg|_{t=0}\cdot L\ dx^3\\[5mm]
\ds=\int_{\B_r(x_0)} \lf.\frac{d}{dt}\phi_t^\ast d L\rg|_{t=0}\cdot dL\ dx^3+\frac{1}{\ep}\int_{\B_r(x_0)} \lf.\frac{d}{dt}\phi_t^\ast L\rg|_{t=0}\cdot L\ dx^3
\end{array}
\]
We have
\[
\lf.\frac{d}{dt}\phi_t^\ast L\rg|_{t=0}=\sum_{k=1}^3\sum_{l=1}^3X_k\,\p_{x_k}L_{l}\ dx_l+\sum_{k=1}^3\sum_{l=1}^3 L_l\ \p_{x_k}X_l\ dx_k\
\]
For convenience we choose coordinates so that $x_0$ becomes the origin and we take $X^\delta=\sum_{k=1}^3 x_k\, \eta_\delta(|x|)$ where $\eta_\delta$ is a smoothing of the characteristic function of $\B_\rho(0)$ where $\rho<r$. Hence for any $l=1,2,3$ we have
\[
\begin{array}{l}
\ds\sum_{k=1}^3 \p_{x_k}X^\delta_l\ dx_k=  \eta_\delta(|x|)\ dx_l+\eta_\delta'(|x|) \sum_{k=1}^3x_l\,\frac{x_k}{|x|} dx_k
\end{array}
\]
This gives
\[
\begin{array}{l}
\ds\int_{\B_r(0)} \lf.\frac{d}{dt}\phi_t^\ast L\rg|_{t=0}\cdot L\ dx^3=\frac{1}{2}\int_{\B_r(0)}\sum_{k=1}^3x_k\,\eta_\delta(|x|)\,\p_{x_k}|L|^2\ dx^3\\[5mm]
\ds+\int_{\B_r(0)}\eta_\delta(|x|)\ |L|^2\ dx^3+\frac{1}{|x|}\int_{\B_r(0)}\eta_\delta'(|x|) \lf|\sum_{k=1}^3 x_k\, L_k\rg|^2\ dx^3
\end{array}
\]
Letting $\delta$ go to zero we obtain
\[
\begin{array}{l}
\ds\lim_{\delta\rightarrow 0}\int_{\B_r(0)} \lf.\frac{d}{dt}\phi_t^\ast L\rg|_{t=0}\cdot L\ dx^3=\frac{1}{2}\int_{\B_\rho(0)}\sum_{k=1}^3x_k\,\p_{x_k}|L|^2\ dx^3\\[5mm]
\ds+\int_{\B_\rho(0)}\ |L|^2\ dx^3- r\,\int_{\p \B_\rho(0)}|L(\p_r)|^2\ dvol_{\p \B_\rho(0)}
\end{array}
\]
For the $2-$form $H=dL$ the computations go as follows
\[
\lf.\frac{d}{dt}\phi_t^\ast H\rg|_{t=0}=\sum_{k=1}^3\sum_{l<m}X_k\,\p_{x_k}H_{lm}\ dx_l\,\wedge\, dx_m+\sum_{k=1}^3\sum_{l<m} H_{lm}\ \lf(\p_{x_k}X_l\ dx_k\wedge d x_m +\p_{x_k}X_m dx_l\wedge d x_k\rg)
\]
Moreover, for the same choice of $X^\delta$ we have
\[
\begin{array}{l}
\ds\sum_{k=1}^3 \lf(\p_{x_k}X^\delta_l\ dx_k\wedge d x_m +\p_{x_k}X^\delta_m dx_l\wedge d x_k\rg)= 2\, \eta_\delta(|x|)\, dx_l\wedge dx_m\\[5mm]
\ds\quad+\eta_\delta'(|x|) \sum_{k=1}^3x_l\,\frac{x_k}{|x|} dx_k\wedge d x_m+x_m\,\frac{x_k}{|x|} dx_l\wedge d x_k
\end{array}
\]
This gives
\[
\begin{array}{l}
\ds\int_{\B_r(0)}\lf.\frac{d}{dt}\phi_t^\ast H\rg|_{t=0}\cdot H\ dx^3=\frac{1}{2}\sum_{k=1}^3\int_{\B_r(0)}\eta_\delta(|x|)\, x_k\,\p_{x_k}|H|^2\ dx^3+2\,\int_{\B_r(0)}\, \eta_\delta(|x|)\, |H|^2\, \ dx^3\\[5mm]
\ds+\sum_{k=1}^3\int_{\B_r(0)}\sum_{l<m} \eta_\delta'(|x|)\ \lf(x_l\,\frac{x_k}{|x|} H_{lm}\ H_{km}+H_{lm}\, H_{lk}\ x_m\,\frac{x_k}{|x|} \rg)\ dx^3
\end{array}
\]
Letting $\delta$ go to zero we obtain
\[
\begin{array}{l}
\ds\lim_{\delta\rightarrow 0}\int_{\B_r(0)} \lf.\frac{d}{dt}\phi_t^\ast dL\rg|_{t=0}\cdot dL\ dx^3=\frac{1}{2}\sum_{k=1}^3\int_{\B_\rho(0)}\, x_k\,\p_{x_k}|dL|^2\ dx^3+2\,\int_{\B_\rho(0)}\,  |dL|^2\, \ dx^3\\[5mm]
\ds\quad-\,\rho\,\int_{\p\B_\rho(0)}\,  |dL\res \p_r|\ dvol_{\p \B_r(x_0)}
\end{array}
\]
Hence we have by taking $\rho\rightarrow r$
\[
\begin{array}{l}
\ds\frac{1}{2\ep}\int_{\B_r(0)}\sum_{k=1}^3x_k\,\p_{x_k}|L|^2\ dx^3+\frac{1}{\ep}\,\int_{\B_r(0)}\ |L|^2\ dx^3- r\,\frac{1}{\ep}\,\int_{\p \B_r(0)}|L(\p_r)|^2\ dvol_{\p \B_r(0)}\\[5mm]
\ds+\frac{1}{2}\sum_{k=1}^3\int_{\B_r(0)}\, x_k\,\p_{x_k}|dL|^2\ dx^3+2\,\int_{\B_r(0)}\,  |dL|^2\, \ dx^3-\,r\,\int_{\p\B_r(0)}\,  |dL\res \p_r|\ dvol_{\p \B_r(x_0)}=0
\end{array}
\]
This implies
\[
\begin{array}{l}
\ds-\frac{1}{2}\int_{\B_r(x_0)}\frac{|L|^2}{\ep}\ dx^3+\frac{r}{2}\int_{\p \B_r(x_0)}\frac{|L|^2}{\ep}\ dvol_{\p \B_\rho(x_0)}\\[5mm]
\ds+\frac{1}{2}\int_{\B_r(x_0)}{|dL|^2}\ dx^3+\frac{r}{2}\int_{\p \B_r(x_0)}{|dL|^2}\ dvol_{\p \B_\rho(x_0)}\\[5mm]
\ds={r}\int_{\p \B_r(x_0)}\frac{|L(\p_r)|^2}{\ep}\ dvol_{\p \B_r(x_0)}+ r\int_{\p \B_r(x_0)}{|dL\res\p_r|^2}\ dvol_{\p \B_r(x_0)}
\end{array}
\]
then
\be
\label{monL}
\begin{array}{l}
\ds-\frac{1}{2}\int_{\B_r(x_0)}\frac{|L|^2}{\ep}\ dx^3+\frac{r}{2}\int_{\p \B_r(x_0)}\frac{|\iota^\ast_{\p \B_r(x_0)}L|^2}{\ep}\ dvol_{\p \B_r(x_0)}\\[5mm]
\ds+\frac{1}{2}\int_{\B_r(x_0)}{|dL|^2}\ dx^3+\frac{r}{2}\int_{\p \B_r(x_0)}{|\iota^\ast_{\p \B_r(x_0)}dL|^2}\ dvol_{\p B_\rho(x_0)}\\[5mm]
\ds=\frac{r}{2}\int_{\p \B_r(x_0)}\frac{|L(\p_r)|^2}{\ep}\ dvol_{\p \B_r(x_0)}+ \frac{r}{2}\,\int_{\p \B_r(x_0)}{|dL\res\p_r|^2}\ dvol_{\p \B_r(x_0)}
\end{array}
\ee
We have
\[
dL\res\p_r=\iota_{\p \B_r(x_0)}^\ast \ast dL=\iota_{\p \B_r(x_0)}^\ast (A\cdot\hat{u}-\xi)
\]
we have also
\[
\ast L=(\ast) L\wedge dr+L(\p_r)\,\om_{S^2}
\]
We have moreover
\[
\iota^\ast_{\p \B_r(x_0)} \ast L=-\ep\, \iota^\ast_{\p \B_r(x_0)}d\ast dL=-\ep\, \iota^\ast_{\p \B_r(x_0)}d(A\cdot \hat{u}-\xi)
\]
Hence
\[
|L(\p_r)|^2=|\iota_{\p \B_r(x_0)}^\ast \ast L|^2=\ep^2\,|\iota^\ast_{\p \B_r(x_0)}d(A\cdot\hat{u}-\xi)|^2
\]
moreover
\[
\int_{\B_r(x_0)}|d L|^2+\frac{|L|^2}{\ep}\ dx^3=\int_{\p \B_r(x_0)} L\,\dot{\wedge} \,(A\cdot\hat{u}-\xi)\ dvol_{\p \B_r(x_0)}
\]
from which we deduce
\[
\begin{array}{l}
\ds\int_{\B_r(x_0)}|d L|^2+\frac{|L|^2}{\ep}\ dx^3\\[5mm]
\ds\le t\,r\, \int_{\p \B_r(x_0)} \frac{|\iota^\ast_{\p \B_r} L|^2}{\ep}\ \ dvol_{\p \B_r(x_0)}+\frac{4}{t}\,\frac{\ep}{r^2}\, r\, \int_{\p \B_r(x_0)} |\iota^\ast_{\p \B_r(x_0)}(A\cdot\hat{u}-\xi)|^2\ dvol_{\p \B_r(x_0)}
\end{array}
\]
thus, taking $t:=1/9$ and adding
\[
\ds \frac{9}{2}\,\int_{B_\rho(x_0)}|d L|^2+\,\frac{|L|^2}{\ep}\ dx^3
\]
to the above identity we obtain
\[
\begin{array}{l}
\ds 4\,\int_{\B_r(x_0)}\frac{|L|^2}{\ep}\ dx^3+{5}\,\int_{\B_r(x_0)}{|dL|^2}\ dx^3+\frac{r}{2}\int_{\p \B_r(x_0)}{|\iota_{\p \B_r(x_0)}^\ast dL|^2}\ dvol_{\p \B_r(x_0)}\\[5mm]
\ds\le\frac{r}{2}\int_{\p \B_r(x_0)}\ep\,|\iota^\ast_{\p \B_r(x_0)}d(A\cdot\hat{u}-\xi)|^2\ dvol_{\p \B_r(x_0)}\\[5mm]
\ds+\lf(\frac{1}{2}+4\times 9^2\,\frac{\ep}{r^2}\rg)\, r\, \int_{\p \B_r(x_0)} |\iota^\ast_{\p \B_r(x_0)}(A\cdot\hat{u}-\xi)|^2\ dvol_{\p B_r(x_0)}
\end{array}
\]
Thus, using (\ref{tie}), we have
\[
\begin{array}{l}
\ds d(A\cdot\hat{u}-\xi)=F_A\cdot\hat{u}-\dashint_{\p \B_r(x_0)}F_A\cdot\hat{u}\ dvol_{\p \B_r(x_0)}\qquad\mbox{ on }\p \B_r(x_0)\\[5mm]

\end{array}
\]
and this implies
\be
\label{estL}
\begin{array}{l}
\ds 4\,\int_{\B_r(x_0)}\frac{|L|^2}{\ep}\ dx^3+{5}\,\int_{\B_r(x_0)}{|dL|^2}\ dx^3+\frac{r}{2}\int_{\p \B_r(x_0)}{|\iota_{\p \B_r(x_0)}^\ast dL|^2}\ dvol_{\p \B_r(x_0)}\\[5mm]
\ds\le {r}\,\int_{\p \B_r(x_0)}\ep\,|\iota^\ast_{\p \B_r(x_0)}F_A\cdot\hat{u}-(F_A\cdot\hat{u})_{x_0,r}|^2\ dvol_{\p \B_r(x_0)}\\[5mm]
\ds+\lf(\frac{1}{2}+4\times 9^2\,\frac{\ep}{r^2}\rg)\, r\, \int_{\p \B_r(x_0)} |\iota^\ast_{\p \B_r(x_0)}(A\cdot\hat{u}-\xi)|^2\ dvol_{\p B_r(x_0)}
\end{array}
\ee
Observe
\[
\int_{\p \B_r(x_0)}|F_A^T- (F_A^T)_{x_0,r}|^2\ \om_{\p \B_r(x_0)}=\int_{\p \B_r(x_0)}|F_A^T|^2\, \om_{\p \B_r(x_0)}-4\pi\,r^2\, |(F_A^T)_{x_0,r}|^2\ .
\]
Hence using also (\ref{estxi})
\be
\label{estL}
\begin{array}{l}
\ds 4\,\int_{\B_r(x_0)}\frac{|L|^2}{\ep}\ dx^3+{5}\,\int_{\B_r(x_0)}{|dL|^2}\ dx^3+\frac{r}{2}\int_{\p \B_r(x_0)}{|\iota_{\p \B_r(x_0)}^\ast dL|^2}\ dvol_{\p \B_r(x_0)}\\[5mm]
\ds\le {r}\,\int_{\p \B_r(x_0)}\ep\,|\iota^\ast_{\p \B_r(x_0)}F_A\cdot\hat{u}|^2\ dvol_{\p \B_r(x_0)}\\[5mm]
\ds+\lf(\frac{1}{2}+4\times 9^2\,\frac{\ep}{r^2}\rg)\, r\, \int_{\p \B_r(x_0)} |\iota^\ast_{\p \B_r(x_0)}(A\cdot\hat{u})|^2\ dvol_{\p B_r(x_0)}\\[5mm]
\ds +C \lf(1+\frac{\ep}{r^2}\rg)\, r\, (Y'(r))^{1/2} \lf(Y'(r)+r^{1/3}\,(Y'(r))^{5/6}\rg)\\[5mm]
\ds+C \lf(1+\frac{\ep}{r^2}\rg)\, r\,\lf(Y'(r)+r^{1/3}\,(Y'(r))^{5/6}\rg)^2
\end{array}
\ee

We introduced  $G:=\ast(K+L)$ and  it satisfies
\be
\label{eqG}
\lf\{
\begin{array}{l}
\ds  d d^\ast G+\frac{G}{\ep}=\frac{1}{8}\,\check{u}\cdot [d\check{u}\wedge d\check{u}]+d\ti{\frak{e}}= d\xi\qquad\mbox{ in }\B_r(x_0)\\[5mm]
\ds d G=0\qquad\mbox{ in }\B_r(x_0)\\[5mm]
\ds \iota_{\p \B_r(x_0)}^\ast d^\ast G=\iota_{\p \B_r(x_0)}^\ast A\cdot\hat{u}\qquad\mbox{ on }\p \B_r(x_0)
\end{array}
\rg.
\ee
Combining (\ref{estK}) and (\ref{estL}) is giving assuming $1>r\ge \sqrt{\ep}$ and $Y'(r)\le 1$
\be
\label{estG}
\begin{array}{l}
\ds\int_{\B_r(x_0)}\frac{|G|^2}{\ep}\ dx^3+\int_{\B_r(x_0)}{|d^\ast G|^2}\ dx^3\\[5mm]
\ds\le \frac{r}{3}\,\int_{\p \B_r(x_0)}\ep\,|\iota^\ast_{\p \B_r(x_0)}F_A\cdot\hat{u}|^2\ dvol_{\p \B_r(x_0)}\\[5mm]
\ds+\lf(\frac{1}{6}+C\,\frac{\ep}{r^2}\rg)\, r\, \,\int_{\p \B_r(x_0)} |\iota^\ast_{\p \B_r(x_0)}(A\cdot\hat{u})|^2\ dvol_{\p B_r(x_0)}\\[5mm]
\ds+C\, r\, (Y'(r))^{3/2}+C\, r^{4/3}\, (Y'(r))^{4/3}
\end{array}
\ee

\medskip

\noindent{\bf Existence and estimates for $\eta$}

\medskip

\be
\label{neta}
\lf\{
\begin{array}{l}
\ds d^\ast\,d\eta+\frac{\eta}{\ep^2}=0\qquad\mbox{ in }\B_r(x_0)\\[5mm]
\ds d^\ast\eta=0\qquad\mbox{ in }\B_r(x_0)\\[5mm]
\ds\iota_{\p B_\rho(x_0)}^\ast \eta=- \frac{1}{4}[\hat{u},\iota_{\p \B_r(x_0)}^\ast d_A\hat{u}]\ .
\end{array}
\rg.
\ee
following the same arguments as the one used to deduce (\ref{monL}) we obtain
\be
\label{monet}
\begin{array}{l}
\ds-\frac{1}{2}\int_{\B_r(x_0)}\frac{|\eta|^2}{\ep}\ dx^3+\frac{r}{2}\int_{\p \B_r(x_0)}\frac{|\iota^\ast_{\p \B_r(x_0)}\eta|^2}{\ep}\ dvol_{\p \B_r(x_0)}\\[5mm]
\ds+\frac{1}{2}\int_{\B_r(x_0)}\ep\,{|d\eta|^2}\ dx^3+\frac{r}{2}\int_{\p \B_r(x_0)}\ep\,{|\iota^\ast_{\p \B_r(x_0)}d\eta|^2}\ dvol_{\p B_\rho(x_0)}\\[5mm]
\ds=\frac{r}{2}\int_{\p \B_r(x_0)}\frac{|\eta(\p_r)|^2}{\ep}\ dvol_{\p \B_r(x_0)}+ \frac{r}{2}\,\int_{\p \B_r(x_0)}\ep{|d\eta\res\p_r|^2}\ dvol_{\p \B_r(x_0)}
\end{array}
\ee
Multiplying (\ref{neta}) by $\eta$ and integrating on $\B_r(x_0)$ is giving
\[
\int_{\B_r(x_0)}\ep\,\eta\wedge d(\ast d\eta)+\int_{\B_r(x_0)}\frac{|\eta|^2}{\ep}\ dx^3=0
\]
Hence
\[
-\int_{\B_r(x_0)}\ep\, d\lf(\eta\wedge \ast d\eta\rg)+\int_{\B_r(x_0)}\ep d\eta\wedge \ast d\eta+\int_{\B_r(x_0)}\frac{|\eta|^2}{\ep}\ dx^3=0
\]
We deduce for any $t>0$ and using (\ref{monet})
\[
\begin{array}{l}
\ds\int_{\B_r(x_0)}\ep\,|d\eta|^2+\frac{|\eta|^2}{\ep}\ dx^3= \ep\,\int_{\p\B_r(x_0)}\eta\wedge \ast d\eta\\[5mm]
\ds\le\frac{t}{2}\int_{\p \B_r(x_0)}|\iota_{\p \B_r(x_0)}^\ast\eta|^2\ dvol_{\p \B_r(x_0)}+\frac{\ep^2}{2t}\,\int_{\p \B_r(x_0)}\,|d\eta\res\p_r|^2\ dvol_{\p \B_r(x_0)}
\end{array}
\]
From (\ref{monet}) we have
\[
\begin{array}{l}
\ds\frac{\ep^2}{2t}\,\int_{\p \B_r(x_0)}\,|d\eta\res\p_r|^2\ dvol_{\p \B_r(x_0)}=\frac{\ep}{rt}\,\frac{r}{2}\int_{\p \B_r(x_0)}\,\ep\,|d\eta\res\p_r|^2\ dvol_{\p \B_r(x_0)}\\[5mm]
\ds\le \frac{\ep}{rt}\,\frac{r}{2}\int_{\p \B_r(x_0)}\frac{|\iota^\ast_{\p \B_r(x_0)}\eta|^2}{\ep}\ dvol_{\p \B_r(x_0)}+ \frac{\ep}{rt}\,\frac{1}{2}\int_{\B_r(x_0)}\ep\,{|d\eta|^2}\ dx^3\\[5mm]
\ds+ \frac{\ep}{rt}\,\frac{r}{2}\int_{\p \B_r(x_0)}\ep\,{|\iota^\ast_{\p \B_r(x_0)}d\eta|^2}\ dvol_{\p B_\rho(x_0)}\ .
\end{array}
\]
This gives
\[
\begin{array}{l}
\ds\int_{\B_r(x_0)}\ep\,|d\eta|^2+\frac{|\eta|^2}{\ep}\ dx^3\le \lf(\frac{\ep\,t}{2r}+\frac{\ep}{2\,rt}\rg)\ r\,\int_{\p \B_r(x_0)}\frac{|\iota_{\p \B_r(x_0)}^\ast\eta|^2}{\ep}\ dvol_{\p \B_r(x_0)}\\[5mm]
\ds+ \frac{\ep}{rt}\,\frac{1}{2}\int_{\B_r(x_0)}\ep\,{|d\eta|^2}\ dx^3+ \frac{\ep}{rt}\,\frac{r}{2}\int_{\p \B_r(x_0)}\ep\,{|\iota^\ast_{\p \B_r(x_0)}d\eta|^2}\ dvol_{\p B_\rho(x_0)}\ .
\end{array}
\]
We choose $t=r/\ep$ and we get
\be
\label{etbo}
\begin{array}{l}
\ds \lf(  1- \frac{\ep^2}{2\,r^2}\rg)\ \int_{\B_r(x_0)}\ep\,|d\eta|^2\ dx^3+\int_{\B_r(x_0)}\frac{|\eta|^2}{\ep}\ dx^3\\[5mm]
\ds\le \lf(\frac{1}{2}+\frac{\ep^2}{2\,r^2}  \rg)\,r\,\int_{\p\B_r(x_0)}\frac{|\iota_{\p \B_r(x_0)}^\ast\eta|^2}{\ep}\ dvol_{\p \B_r(x_0)}\\[5mm]
\ds\quad+\frac{\ep^2}{2\,r^2}\, r\,\int_{\p \B_r(x_0)}\ep\,{|\iota^\ast_{\p \B_r(x_0)}d\eta|^2}\ dvol_{\p B_\rho(x_0)}\ .
\end{array}
\ee
Recall
\[
d[\hat{u},d_A\hat{u}]=d_A[\hat{u},d_A\hat{u}]-[A\wedge [\hat{u},d_A\hat{u}]]
\]
Hence using (\ref{dAsqu}) we obtain
\[
\begin{array}{l}
\ds d[\hat{u},d_A\hat{u}]=[d_A\hat{u}\wedge d_A\hat{u}]+[\hat{u},[F_A,\hat{u}]]-[A\wedge [\hat{u},d_A\hat{u}]]\\[5mm]
\ds=[d_A\hat{u}\wedge d_A\hat{u}]+(F_A)^T-[A\wedge [\hat{u},d_A\hat{u}]]
\end{array}
\]
This gives
\[
\begin{array}{l}
\ds \lf(  1- \frac{\ep^2}{2\,r^2}\rg)\ \int_{\B_r(x_0)}\ep\,|d\eta|^2\ dx^3+\int_{\B_r(x_0)}\frac{|\eta|^2}{\ep}\ dx^3\\[5mm]
\ds\le \lf(\frac{1}{2}+\frac{\ep^2}{2\,r^2}  \rg)\,r\,\int_{\p\B_r(x_0)}\frac{\ds\lf|  \frac{1}{4}[\hat{u},\iota_{\p \B_r(x_0)}^\ast d_A\hat{u}]   \rg|^2}{\ep}\ dvol_{\p \B_r(x_0)}\\[5mm]
\ds+\frac{\ep^2}{32\,r^2}\, r\,\int_{\p \B_r(x_0)}\ep\,{\lf|\,\iota^\ast_{\p \B_r(x_0)}[d_A\hat{u}\wedge d_A\hat{u}]- \iota^\ast_{\p \B_r(x_0)}(F_A)^T+ \iota^\ast_{\p \B_r(x_0)}[A\wedge [\hat{u},d_A\hat{u}]]  \rg|^2}\ dvol_{\p B_\rho(x_0)}\ .
\end{array}
\]
Thus
\[
\begin{array}{l}
\ds \lf(  1- \frac{\ep^2}{2\,r^2}\rg)\ \int_{\B_r(x_0)}\ep\,|d\eta|^2\ dx^3+\int_{\B_r(x_0)}\frac{|\eta|^2}{\ep}\ dx^3\\[5mm]
\ds\le \lf(\frac{1}{8}+\frac{\ep^2}{8\,r^2}  \rg)\,r\,\int_{\p\B_r(x_0)}\frac{\ds\lf| \iota_{\p \B_r(x_0)}^\ast (d_A\hat{u})^T  \rg|^2}{\ep}\ dvol_{\p \B_r(x_0)}\\[5mm]
\ds+\frac{\ep^2}{16\,r^2}\, r\,\int_{\p \B_r(x_0)}\ep\,|\iota^\ast_{\p \B_r(x_0)}(F_A)^T|^2\ dvol_{\p \B_r(x_0)}\\[5mm]
\ds +\frac{\ep^2}{16\,r^2}\, r\,\int_{\p \B_r(x_0)}\ep\,\lf|\,\iota^\ast_{\p \B_r(x_0)}[d_A\hat{u}\wedge d_A\hat{u}]\rg|^2\ dvol_{\p \B_r(x_0)}\\[5mm]
\ds +\frac{\ep^2}{16\,r^2}\, r\,\int_{\p \B_r(x_0)}\ep\,\lf|\,\iota^\ast_{\p \B_r(x_0)}[A\wedge [\hat{u},d_A\hat{u}]]  \rg|^2\ dvol_{\p \B_r(x_0)}
\end{array}
\]
Using the $L^\infty$ bound on $|\nabla_Au|$ (\ref{linf}) this finally gives
\be
\label{etbop}
\begin{array}{l}
\ds \lf(  1- \frac{\ep^2}{2\,r^2}\rg)\ \int_{\B_r(x_0)}\ep\,|d\eta|^2\ dx^3+\int_{\B_r(x_0)}\frac{|\eta|^2}{\ep}\ dx^3\\[5mm]
\ds\le \lf(\frac{1}{8}+C\,\frac{\ep^2}{r^2}  \rg)\,r\,\int_{\p\B_r(x_0)}\frac{\ds\lf| \iota_{\p \B_r(x_0)}^\ast (d_A\hat{u})^T  \rg|^2}{\ep}\ dvol_{\p \B_r(x_0)}\\[5mm]
\ds+\frac{\ep^2}{16\,r^2}\, r\,\int_{\p \B_r(x_0)}\ep\,|\iota^\ast_{\p \B_r(x_0)}(F_A)^T|^2\ dvol_{\p \B_r(x_0)}\\[5mm]
\ds +C\,\frac{\ep}{r^2}\,r\,\int_{\p \B_r(x_0)}|\iota^\ast_{\p \B_r(x_0)}A|^2\ dvol_{\p \B_r(x_0)}
\end{array}
\ee
Classical elliptic estimates give
\[
\|\eta\|^2_{L^6(\B_r(x_0))}\le C\,\lf(\|d\eta\|^2_{L^2(\B_r(x_0))}+\|d^\ast\eta\|^2_{L^2(\B_r(x_0))}+\|\iota_{\p\B_r(x_0)}^\ast\eta\|^2_{L^4(\p\B_r(x_0))}\rg)
\]
Hence
\be
\label{eta6}
\begin{array}{l}
\ds\lf(  \int_{\B_r(x_0)}\ep^3\,|\eta|^6\ dx^3   \rg)^{1/3}\le \ C\ \int_{\B_r(x_0)}\ep\,|d\eta|^2\ dx^3\\[5mm]
\ds\quad+C\, \lf(   \ep^2\,\int_{\p\B_r(x_0)}{\ds\lf| \iota_{\p \B_r(x_0)}^\ast (d_A\hat{u})^T  \rg|^4}\ dvol_{\p \B_r(x_0)}\rg)^{1/2}
\end{array}
\ee
This gives for $\ep<r^2$
\be
\label{eta6}
\begin{array}{l}
\ds\lf(  \int_{\B_r(x_0)}\ep^3\,|\eta|^6\ dx^3   \rg)^{1/3}\le \ C\ r\, Y'(r)+C\,\sqrt{\frac{\ep}{r}}\,\sqrt{r\,Y'(r)}
\end{array}
\ee
Littlewood inequality gives
\[
\|\eta\|_{L^4(\B_r(x_0))}\le \|\eta\|^{1/4}_{L^2(\B_r(x_0))}\ \|\eta\|^{3/4}_{L^6(\B_r(x_0))}=\frac{1}{\ep^{1/4}}\,\lf\|\frac{\eta}{\sqrt{\ep}}\rg\|^{1/4}_{L^2(\B_r(x_0))}\ \|\sqrt{\ep}\,\eta\|^{3/4}_{L^6(\B_r(x_0))}
\]
This gives
\be
\label{eta4}
\begin{array}{l}
\ds\int_{\B_r(x_0)}\ep\,|\eta|^4\ dx^3\le \sqrt{\int_{\B_r(x_0)}\frac{|\eta|^2}{\ep}\ dx^3}\ \ \sqrt{  \int_{\B_r(x_0)}\ep^3\,|\eta|^6\ dx^3    }\\[5mm]
\ds\le C\,\sqrt{r\,Y'(r)}\ \lf(r\, Y'(r)+\sqrt{\frac{\ep}{r}}\,\sqrt{r\,Y'(r)}\rg)^{3/2}
\end{array}
\ee

The connection 1-form $A$ is extended inside $\p\B_r(x_0)$ by
\be
\label{connex}
\check{A}:=-\frac{1}{4}\lf[\check{u},d\check{u}\rg]+{\eta}^T+d^\ast G\ \check{u}
\ee
This gives
\[
d_{\check{A}}\check{u}=d\check{u}+[\check{A},\check{u}]=d\check{u}-\frac{1}{4}\lf[\lf[\check{u},d\check{u}\rg],\check{u}\rg]+[\eta^T,\check{u}]=[\eta^T,\check{u}]\ .
\]
Moreover
\[
\begin{array}{l}
\ds F_{\check{A}}=d\check{A}+\check{A}\wedge \check{A}= d\check{A}+\frac{1}{2}[\check{A},\check{A}]=-\frac{1}{4}\,d([\check{u},d \check{u}])+d{\eta}^T+d\check{u}\wedge d^\ast G+\check{u}\,dd^\ast G\\[5mm]
\ds+\frac{1}{32}[[\check{u},d\check{u}]\wedge[\check{u},d\check{u}]]+\frac{1}{2}\,[{\eta}^T,{\eta}^T]-\frac{1}{4}\,[{\eta}^T, [\check{u},d\check{u}]]-\frac{1}{4}d^\ast G\wedge[\check{u},[\check{u},d\check{u}]]+d^\ast G\wedge [\check{u},\eta^T]\\[5mm]
\end{array}
\]
We have
\[
\begin{array}{l}
\ds-\frac{1}{4}\,d([\check{u},d \check{u}])+\frac{1}{32}[[\check{u},d\check{u}]\wedge[\check{u},d\check{u}]]\\[5mm]
\ds=-\frac{1}{2}\,[\p_{x_1}\check{u},\p_{x_2}\check{u}]\ dx_1\wedge dx_2+\frac{1}{16}[[\check{u},\p_{x_1}\check{u}]\wedge[\check{u},\p_{x_2}\check{u}]]\ dx_1\wedge dx_2
\end{array}
\]
Recall the observation
\[
[[{\bf i},{\bf j}],[{\bf i},{\bf k}]]=-\,[2\,{\bf k},2\,{\bf j}]=4\,[{\bf j},{\bf k}]
\]
Thus we obtain the identity valid for any $\check{u}\in W^{1,2}(\B_r(x_0),\S^2)$
\be
\label{curb}
\begin{array}{l}
\ds-\frac{1}{4}\,d([\check{u},d \check{u}])+\frac{1}{32}[[\check{u},d\check{u}]\wedge[\check{u},d\check{u}]]
\ds=-\frac{1}{4}\,[\p_{x_1}\check{u},\p_{x_2}\check{u}]\ dx_1\wedge dx_2=-\frac{1}{8}\,[d\check{u}\wedge d\check{u}]
\end{array}
\ee
We have also
\[
-\frac{1}{4}d^\ast G\wedge[\check{u},[\check{u},d\check{u}]]=d^\ast G\wedge d\check{u}
\]
This gives
\be
\label{courb}
\begin{array}{l}
\ds  F_{\check{A}}=-\frac{1}{8}\,[d\check{u}\wedge d\check{u}]+ \check{u}\,dd^\ast G+F_{\eta^T}+\lf[\eta^T, \check{u}\, d^\ast G+\frac{1}{4}\, [d\check{u},\check{u}] \rg]
\end{array}
\ee
This gives in particular
\be
\label{courb-proj}
\lf\{
\begin{array}{l}
\ds F_{\check{A}}\cdot\check{u}=-\frac{G}{\ep}+d\ti{\frak{e}}+\check{u}\cdot F_{\eta^T}+\frac{1}{4}\,\check{u}\cdot\,[\eta^T,[d\check{u},\check{u}] ]\\[5mm]
\ds (F_{\check{A}})^T=(F_{\eta^T})^T+[\eta^T,\check{u}]\wedge d^\ast G
\end{array}
\rg.
\ee
Finally we have by minimality
\[
\begin{array}{l}
\ds\,Y(r):=CYMH({u},{A},\B_r(x_0))\le CYMH(\check{u},\check{A},\B_r(x_0))\\[5mm]
\ds =\frac{1}{2}\int_{\B_r(x_0)} |\check{A}|^2+\frac{|d_{\check{A}} \ti{u}|^2}{\ep}+\ep\,|F_{\check{A}}|^2\ dx^3+\frac{1}{2\,\ep^3}(1-|\ti{u}|^2)^2\ dx^3\\[5mm]
\ds=\frac{1}{2}\int_{\B_r(x_0)} \lf|   \frac{1}{4}[\check{u},d\check{u}]+\eta^T\rg|^2+|d^\ast G|^2  +\frac{|\nabla|\ti{u}||^2}{\ep} +|\ti{u}|^2\,\frac{|[\check{u},\eta]|^2}{\ep}\ dx^3\\[5mm]
\ds+\frac{1}{2}\int_{\B_r(x_0)}\ep\,\lf|\frac{G}{\ep}-\check{u}\cdot F_{\eta^T}+d\ti{\frak{e}}-\frac{1}{4}\,\check{u}\cdot\,[\eta^T,[d\check{u},\check{u}] ] \rg|^2\ dx^3\\[5mm]
\ds+\frac{1}{2}\int_{\B_r(x_0)}\ep\,\lf|  (F_{\eta^T})^T+[\eta^T,\check{u}]\wedge d^\ast G\rg|^2+\frac{1}{2\,\ep^3} (1-|\ti{u}|^2)^2\ dx^3
\end{array}
\]
Observe that
\[
\check{u}\cdot d\eta^T=-d\check{u}\,\dot{\wedge}\,\eta^T
\]
Hence
\[
\begin{array}{l}
\ds\int_{\B_r(x_0)}\ep\,|\check{u}\cdot F_{\eta^T}|^2\ dx^3=\int_{\B_r(x_0)}\ep\,|-d\check{u}\,\dot{\wedge}\,\eta^T+\frac{1}{2}\check{u}\cdot [\eta^T\wedge\eta^T]|^2\ dx^3\\[5mm]
\ds\le 2\,\ep\, \lf(\int_{\B_r(x_0)}|\nabla\check{u}|^3\ dx^3\rg)^{2/3}\ \lf(\int_{\B_r(x_0)}|\eta^T|^6\ dx^3\rg)^{1/3}+2\,\ep\,\int_{\B_r(x_0)}\, |\eta^T\wedge\eta^T|^2\ dx^3
\end{array}
\]
Using respectively (\ref{A-03}), (\ref{eta6}) and (\ref{eta4}) we obtain
\be
\label{FetaT}
\ds\int_{\B_r(x_0)}\ep\,|\check{u}\cdot F_{\eta^T}|^2\ dx^3\le C\, r\,(Y'(r))^2+C\, r \,\sqrt{\frac{\ep}{r^2}}\, (Y'(r))^{3/2}+C\, \lf(  \frac{\ep}{r}\rg)^{3/4}\ (r\, Y'(r))^{5/4}
\ee
We have also obviously
\be
\label{retaT}
\begin{array}{l}
\ds\int_{\B_r(x_0)}\ep\,\lf|\check{u}\cdot\,[\eta^T,[d\check{u},\check{u}] ] \rg|^2\ dx^3\le 2\,\ep\, \lf(\int_{\B_r(x_0)}|\nabla\check{u}|^3\ dx^3\rg)^{2/3}\ \lf(\int_{\B_r(x_0)}|\eta^T|^6\ dx^3\rg)^{1/3}\\[5mm]
\ds\quad\le C\, r\,(Y'(r))^2+C\, r \,\sqrt{\frac{\ep}{r^2}}\, (Y'(r))^{3/2}
\end{array}
\ee
Combining respectively (\ref{IV.9}), (\ref{epifrakh}), (\ref{A-03}), (\ref{estG}), (\ref{etbop}), (\ref{FetaT}), (\ref{retaT}) and (\ref{A-04}) we obtain for $r>\sqrt{\ep}$ that
\[
\begin{array}{l}
\ds\frac{1}{2}\int_{\B_r(x_0)} |\check{A}|^2+\frac{|d_{\check{A}} \ti{u}|^2}{\ep}+\ep\,|F_{\check{A}}|^2\ dx^3+\frac{1}{2\,\ep^3}(1-|\ti{u}|^2)^2\ dx^3\\[5mm]
\ds\le \frac{1}{2}\int_{\B_r(x_0)}\frac{1}{4}|d\check{u}|^2+\frac{1}{2}\,[\check{u},d\check{u}]\cdot\eta^T +|\eta^T|^2+\frac{|\nabla|\ti{u}||^2}{\ep}+\frac{1}{2\,\ep^3} (1-|\ti{u}|^2)^2\ dx^3\\[5mm]
\ds+\frac{1}{2}\int_{\B_r(x_0)}4\,\frac{|\eta^T|^2}{\ep^2}+\ep\,|d\eta^T|^2+|d^\ast G|^2+\frac{|G|^2}{\ep}+C\, r\,(Y'(r))^{3/2}+C\, r\, (Y'(r))^{5/4}+C\,r^2\, (Y'(r))^{9/8}\\[5mm]
\ds\le \frac{r}{2}\int_{\p \B_r(x_0)}\frac{1-\nu}{4}|\nabla_T\hat{u}|^2+(1-\nu)\,|\nabla_T|u||^2+\frac{1}{6\,\ep^3}(1-|{u}|^2)^2\ dvol_{\p \B_r(x_0)}+\lf(\frac{1}{2}+C\,\frac{\ep^2}{r^2}\rg)\,\frac{|\iota_{\p\B_r(x_0)}^\ast (d_Au)|^2}{\ep}\ dvol_{\p \B_r(x_0)}\\[5mm]
\ds\frac{r}{2}\,\int_{\p \B_r(x_0)}\frac{2\ep}{3}\,|\iota^\ast_{\p \B_r(x_0)}F_A\cdot\hat{u}|^2+\frac{1}{6} |\iota^\ast_{\p \B_r(x_0)}(A\cdot\hat{u})|^2\ dvol_{\p B_r(x_0)}\\[5mm]
\ds+C\,\frac{\ep}{r^2}\,r\,Y'(r)+C\, r\,(Y'(r))^{3/2}+C\, r\, (Y'(r))^{5/4}+C\,r^2\, (Y'(r))^{9/8}\\[5mm]
\end{array}
\]
Observing that
\[
|\iota_{\p \B_r(x_0)}^\ast A|^2=\lf|\frac{1}{4}[\hat{u},d\hat{u}]-\frac{1}{4}[\hat{u},d_A\hat{u}]\rg|^2+|\iota^\ast_{\p \B_r(x_0)}(A\cdot\hat{u})|^2
\]
which implies
\[
\int_{\B_r(x_0)}|\iota_{\p \B_r(x_0)}^\ast A|^2\ dvol_{\p\B_r(x_0)}\ge \int_{\B_r(x_0)}\frac{1}{4}|\nabla_T\hat{u}|^2+|\iota^\ast_{\p \B_r(x_0)}(A\cdot\hat{u})|^2\ dvol_{\p B_r(x_0)}-C\, \sqrt{\ep}\,Y'(r)
\]
Hence finally we have proved for $r>\sqrt{\ep}$, if $Y'(r)<\delta$ for $\delta$ small enough then there holds 
\be
\label{inY}
Y(r)\le (1-\nu)\,r\,Y'(r)+C\,\sqrt{\frac{\ep}{r^2}}\,r\,Y'(r)+C\, r\,(Y'(r))^{3/2}+C\, r\, (Y'(r))^{5/4}+C\,r^2\, (Y'(r))^{9/8}\ .
\ee
Let $r_0>R\,\sqrt{\ep}$ where $R$ is going to be chosen below. Let $\delta>0$ to be fixed later.  Assume $Y(r_0)\le r_0\,\delta/2$ and assume that there exists $\rho\in (R\,\sqrt{\ep},r_0)$ such that $Y(\rho)>\rho\,\delta$. We consider
\[
\rho_0:=\inf\lf\{\rho>R\,\sqrt{\ep}\ ;\ \forall\ s\in(\rho,r_0)\quad Y(s)\le  s\,\delta\rg\}
\]
Assume $\rho_0>R\,\sqrt{\ep}$.  This implies in particular that $Y(\rho_0)=\rho_0\,\delta$. Let $\rho\in(\rho_0,r_0)$. Assume $\rho\,Y'(\rho)\le Y(\rho)$ then, since $Y(\rho)\le \delta\,\rho$ we have $Y'(\rho)\le\delta$ and from (\ref{inY}) there holds
\be
\label{inY2}
\begin{array}{l}
\ds Y(r)\le (1-\nu)\,r\,Y'(r)+C\,\sqrt{\frac{\ep}{r^2}}\,r\,Y'(r)+C\, r\,(Y'(r))^{3/2}+C\, r\, (Y'(r))^{5/4}+C\,r^2\, (Y'(r))^{9/8}\\[5mm]
\ds \le (1-\nu+C\,R^{-1} +C\,\sqrt{\delta}+C\,\delta^{1/4}+C\,r_0\,\delta^{1/8})\,r\,Y'(r)
\end{array}
\ee
By having chosen $R^{-1}$ and $\delta$ small enough so that $(1-\nu+C\,R^{-1} +C\,\sqrt{\delta}+C\,\delta^{1/4}+C\,r_0\,\delta^{1/8}) $ we obtain $Y(r)<r Y'(r)$ which is a contradiction. Hence  for any $r\in(\rho_0,r_0)$ there holds
\[
\frac{d}{dr}\lf( \frac{Y(r)}{r}  \rg)=\frac{r\,Y'(r)-Y(r)}{r^2}\ge 0
\]
This implies
\[
\frac{Y(\rho_0)}{\rho_0}\le \frac{Y(r_0)}{r_0}\le \frac{\delta}{2}
\]
which is again a contradiction. Hence we deduce
\[
\forall\ r\in(A\,\sqrt{\ep},r_0)\qquad Y(r):=CYMH({u},{A},\B_r(x_0))\le r\,\delta\ .
\]
This ends the proof of  Lemma~\ref{lm-sca}.\hfill $\Box$
\subsection{Improved $L^\infty$ estimates of $\nabla_Au$.}
The goal of this subsection is to prove the following lemma.

\begin{Lm}
\label{lm-impLinf}
There exists $\delta>0$ such that, if $(u,A)$ is a smooth critical point of $CYMH_{\ep}^{1,1}$ in $\B_1(0)$ and if on $\B_{\sqrt{\ep}}(x_0)\subset \B_1(0)$  assume there holds respectively
\[
\frac{1}{2\sqrt{\ep}}\int_{\B_{\sqrt{\ep}}(x_0)} |{A}|^2+\frac{|d_{{A}} {u}|^2}{\ep}+\ep\,|F_{{A}}|^2\ dx^3+\frac{1}{2\,\ep^3}(1-|{u}|^2)^2\ dx^3\le \delta
\]
and
\[
1\ge |u|\ge 1-\delta
\]
then there exists $C>0$ universal such that
\be
\label{covinfim}
\|\nabla_Au\|_{L^\infty(\B_{\sqrt{\ep}/2}(x_0))}\le C\,\ep^{-1/4}\ .
\ee
\hfill  $\Box$
\end{Lm}
Before proving Lemma~\ref{lm-impLinf} we shall first prove the following lemma.
\begin{Lm}
\label{lm-linfty}
There exists $\sigma\in(0,1)$ independent of $\epsilon$ such that for any $\B_\rho(x_0)\subset \B_1(0)$ and any smooth critical point $(u,A)$ of $CYMH_\ep^{1,1}$
\[
1\ge|u|^2\ge 1-\sigma\qquad\mbox{ in }B_\rho(x_0)\quad\mbox{ and }\quad \|\sqrt{\ep}\, F_A\|_{L^2(\B_\rho(x_0))}+\|F_A\|_{L^{(3/2,1)}(B_\rho(x_0))}\le \sigma
\]
then
\be
\label{linfcovder}
\begin{array}{rl}
\ds\|\nabla_Au\|_{L^\infty(B_{\rho/2}(x_0))}&\ds\le C\ \rho^{-3}\ \int_{B_\rho(x_0)}|\nabla_Au|\ dx^3+ C\, \ep^{-1/2}\ \int_{\B_\rho(x_0)}\ep\,|F_A|^2+|A|^2\ dx^3\\[5mm]
\ds\qquad&\ds\quad+ C\, \ep^{-1/2}\ \sqrt{\int_{\B_\rho(x_0)}|A|^2\ dx^3}\ .
\end{array}
\ee
where $C$ is universal.
\hfill $\Box$
\end{Lm}
\noindent{\bf Proof of Lemma~\ref{lm-linfty}.} Recall for (\ref{bw-1}) the Bochner inequality
\be
\label{bw-01}
\begin{array}{l}
\ds\,\Delta \frac{|\nabla_A u|^2}{2\,\ep}=\frac{|(\nabla_A)(\nabla_A u)|^2}{\ep}+\frac{1}{\ep^3} |[u,(\nabla_A)u]|^2\\[5mm]
\ds+ \frac{2}{\ep}\sum_{l,m=1}^3[(\nabla_A)_lu,(\nabla_A)_mu]\cdot F_{lm}+\frac{1}{\ep^2}\,\sum_{m=1}^3[u,(\nabla_A)_mu]\cdot A_m\\[5mm]
\ds-\,\frac{1}{\ep^3}\, |\nabla_A u|^2\,(1-|u|^2)+\frac{1}{2\ep^3}\,|\nabla|u|^2|^2\ ,
\end{array}
\ee
We assume that on $B_\rho(x_0)$ there holds $|u|^2\ge 1-\sigma$ where $\sigma>0$ will be fixed later on. Let $t>0$ we have for any $v\ge 0$,  $v\in W^{1,2}$ and $v^2\in C^\infty$ 
\be
\label{sq-0}
\begin{array}{l}
\ds\Delta\sqrt{v^2+t^2}=\mbox{div}\lf[ \frac{v}{\sqrt{v^2+t^2}}\nabla v\rg]= \frac{v\,\Delta v}{\sqrt{v^2+t^2}}+\frac{|\nabla v|^2}{\sqrt{v^2+t^2}}-\frac{v^2\,|\nabla v|^2}{(v^2+t^2)^{3/2}}\\[5mm]
\ds=\frac{v\,\Delta v}{\sqrt{v^2+t^2}}+\frac{v^2\,|\nabla v|^2}{(v^2+t^2)^{3/2}}\ge\frac{1}{2}\frac{\Delta v^2-2\,|\nabla v|^2}{\sqrt{v^2+t^2}}
\end{array}
\ee
where $v\,\Delta v$ is a formal notation for $2^{-1}\,(\Delta v^2-2\,|\nabla v|^2)$
By Kato inequality we have 
\[
|(\nabla_A)(\nabla_A u)|^2\ge |\nabla|\nabla_Au||^2
\]
We deduce
\[
\begin{array}{l}
\ds\Delta\sqrt{|\nabla_Au|^2+t^2}\ge \frac{1}{2} \frac{\Delta|\nabla_A u|^2-2\,|\nabla|\nabla_Au||^2}{\sqrt{|\nabla_Au|^2+t^2}}\\[5mm]
\ds\ge \frac{1}{\ep^2} \frac{|[u,(\nabla_A)u]|^2}{\sqrt{|\nabla_Au|^2+t^2}}+\frac{1}{2\ep^2}\,\frac{|\nabla|u|^2|^2}{\sqrt{|\nabla_Au|^2+t^2}}-\,\frac{1}{\ep^2}\, \frac{|\nabla_A u|^2}{\sqrt{|\nabla_Au|^2+t^2}}\,(1-|u|^2)\\[5mm]
\ds +2\,\sum_{l,m=1}^3\frac{ [(\nabla_A)_lu,(\nabla_A)_mu]   }{\sqrt{|\nabla_Au|^2+t^2}}\cdot F_{lm}+\frac{1}{\ep}\,\sum_{m=1}^3\frac{[u,(\nabla_A)_mu]}{\sqrt{|\nabla_Au|^2+t^2}}\cdot A_m\
\end{array}
\]
Observe that
\[
|u|^2\,|\nabla_A u|^2=|[u,(\nabla_A)u]|^2+|u\cdot\nabla_Au|^2=|[u,(\nabla_A)u]|^2+4^{-1}\,|\nabla|u|^2|^2
\]
Hence
\[
|[u,(\nabla_A)u]|^2+\frac{1}{2}\,|\nabla|u|^2|^2\ge |u|^2\,|\nabla_A u|^2\ge (1-\sigma)\, |\nabla_Au|^2
\]
We have also
\[
 (1-|u|^2)\,|\nabla_Au|^2\le \sigma\, |\nabla_Au|^2
\]
We shall restrict to $1/6>\sigma>0$  in such a way that $(1-\sigma)\,\ge 2\,\sigma+2^{-1}$. Letting $t$ go to zero we obtain
\be
\label{Bochcovder}
\begin{array}{l}
\ds\Delta|\nabla_Au|\ge \frac{1}{2\,\ep^2}\, |\nabla_Au|-2\,|\nabla_A u|\, |F_A|-\frac{1}{\ep}\,|A|\ .
\end{array}
\ee
We claim that for  for $\|F_A\|_{L^{(3/2,\infty)}(B_\rho(x_0))}$ small enough  there exists a unique $w$ solving
\be
\label{eq-w}
\lf\{
\begin{array}{l}
\ds -\Delta w+\frac{1}{2\,\ep^2}\, w-2\,w\,|F_A|=\frac{1}{\ep}\,|A|\qquad\mbox{ in }\B_\rho(x_0)\\[5mm]
\ds w=0 \qquad\mbox{ on }\p \B_\rho(x_0)
\end{array}
\rg.
\ee
We have indeed the following a priori estimate
\[
\begin{array}{l}
\ds\int_{\B_\rho(x_0)}|\nabla w|^2+\frac{1}{2\ep^2}\, w^2\ dx^3\le 2\,\int_{\B_\rho(x_0)}|w|^2\,|F_A|\ dx^3+\int_{\B_\rho(x_0)}\frac{|w|}{\ep}\,|A|\ dx^3\\[5mm]
\ds\quad\le 2\, \||w|^2\|_{L^{(3,1)}(\B_\rho(x_0))}\ \|F_A\|_{L^{(3/2,\infty)}(\B_\rho(x_0))}+\lf(\frac{1}{\ep^2}\, \int_{\B_\rho(x_0)}\,w^2\ dx^3\rg)^{1/2}\, \|A\|_{L^2(\B_\rho(0))}
\end{array}
\]
We have
\[
\||w|^2\|_{L^{(3,1)}(B_\rho(x_0))}\le \|w\|^2_{L^{(6,2)}(\B_\rho(x_0))}\le C\ \int_{\B_\rho(x_0)}|\nabla w|^2\ dx^3
\]
Hence
\[
\begin{array}{l}
\ds\int_{\B_\rho(x_0)}|\nabla w|^2+\frac{1}{2\ep^2}\, w^2\ dx^3\le C\,   \int_{\B_\rho(x_0)}|\nabla w|^2\ dx^3\, \|F_A\|_{L^{(3/2,\infty)}(\B_\rho(0))}\\[5mm]
\ds +\int_{\B_\rho(x_0)}\frac{1}{4\,\ep^2}\, w^2\ dx^3+C\,\int_{\B_\rho(x_0)}|A|^2\ dx^3\\[5mm]
\end{array}
\]
We deduce for $\|F_A\|_{L^{(3/2,\infty)}(\B_\rho(0))}$ small enough the existence of $C>0$ such that 
\be
\label{bdnab}
\int_{\B_\rho(x_0)}|\nabla w|^2+\frac{1}{4\ep^2}\, w^2\ dx^3\le C\ \int_{\B_\rho(x_0)}|A|^2\ dx^3\ .
\ee
Let $x_1\in \B_\rho(x_0)$ such that $\B_{2\ep}(x_1)\subset \B_\rho(x_0)$ and denote $\ti{w}( y):=w(\ep\,y+x_1)$. We have
\[
-\Delta_y\ti{w}+ 2^{-1}\ \ti{w}=-\,\ep^2\,\Delta_xw(\ep\,y)+2^{-1}\ w(\ep\,y)=2\, w(\ep\,y)\,\ep^2\,|F_A|(\ep\,y)+\ep\,|A|(\ep\,y)
\]
We have by classical elliptic estimates 
\[
\begin{array}{l}
\ds\|\ti{w}\|_{L^\infty(\B_{1/2}(0))}\le C\,\|-\Delta_y\ti{w}+2^{-1}\ \ti{w}\|_{L^{(3/2,1)}(\B_{1}(0))}+C\,\|\ti{w}\|_{L^{2}(\B_{1}(0))}\\[5mm]
\ds\le C\,\ep^2\, \|w(\ep\,y)\|_{L^{(6,2)}(\B_1(0))}\ \|F_A(\ep\, y)\|_{L^2((\B_1(0))}+C\,\ep\, \|A(\ep\,y)\|_{L^2(\B_1(0))}+C\,\|\ti{w}\|_{L^{2}(\B_{1}(0))}\\[5mm]
\ds\le C\, \ep^2\, \ep^{-1/2}\, \|w\|_{L^{(6,2)}(\B_\ep(x_1))}\ \ep^{-2}\ \|\sqrt{\ep}\,F_A\|_{L^2(\B_\ep(x_1))}+C\,\ep^{-1/2}\,\|A\|_{L^2(\B_\ep(x_1))}+C\,\ep^{-1/2}\, \|\ep^{-1}\,w\|_{L^2(\B_\ep(x_1))}\\[5mm]
\ds\le C\,\ep^{-1/2}\, \lf(\|\nabla w\|_{L^2(B_\ep(x_1))}+\ep^{-1}\,\|w\|_{L^2(B_{2\ep}(x_1))}\rg)\ \|\sqrt{\ep}\,F_A\|_{L^2(\B_\ep(x_1))}+\ep^{-1/2}\,\|A\|_{L^2(\B_\rho(x_0))}\\[5mm]
\ds\le C\,\ep^{-1/2}\, \|A\|_{L^2(\B_\rho(x_0))}\ \|\sqrt{\ep}\,F_A\|_{L^2(\B_\ep(x_1))}+\ep^{-1/2}\,\|A\|_{L^2(\B_\rho(x_0))}
\end{array}
\]
This imples
\[
\ds |w(x_1)|\le C\, \ep^{-1/2}\ \int_{\B_\rho(x_0)}\ep\,|F_A|^2+|A|^2\ dx^3+ C\, \ep^{-1/2}\ \sqrt{\int_{\B_\rho(x_0)}|A|^2\ dx^3}\ .
\]
In particular
\be
\label{wlinf}
\|w\|_{L^\infty(\B_{\rho/2}(x_0))}\le  C\, \ep^{-1/2}\ \int_{\B_\rho(x_0)}\ep\,|F_A|^2+|A|^2\ dx^3+ C\, \ep^{-1/2}\ \sqrt{\int_{\B_\rho(x_0)}|A|^2\ dx^3}\ .
\ee
Let $w_-:=\min\{w,0\}$ multiplying (\ref{eq-w}) by $w_-$ and integrating by parts gives
\[
\begin{array}{l}
\ds\int_{\B_\rho(x_0)}|\nabla w_-|^2+\frac{1}{2\ep^2}\, w_-^2\ dx^3\le \sqrt{\ep}^{-1}\,\||w_-|^2\|_{L^2(\B_\rho(x_0))}\   \|\sqrt{\ep}\, F_A\|_{L^2(\B_\rho(x_0))}
\end{array}
\]
Littlewood inequality combined with Sobolev inequality is giving
\[
\begin{array}{l}
\ds\||w_-|^2\|_{L^2(\B_\rho(x_0))}=\|w_-\|^2_{L^4(\B_\rho(x_0))}\le \|w_-\|^{1/2}_{L^2(\B_\rho(x_0))}\ \|w_-\|^{3/2}_{L^6(\B_\rho(x_0))}\\[5mm]
\ds\le  C\, \|w_-\|^{1/2}_{L^2(\B_\rho(x_0))}\ \|\nabla w_-\|_{L^2(\B_\rho(x_0))}^{3/2}
\end{array}
\]
Thus
\[
\begin{array}{l}
\ds\int_{\B_\rho(x_0)}|\nabla w_-|^2+\frac{1}{2\ep^2}\, w_-^2\ dx^3\le \sqrt{\ep}^{-1}\,\||w_-|^2\|_{L^2(\B_\rho(x_0))}\   \|\sqrt{\ep}\, F_A\|_{L^2(\B_\rho(x_0))}\\[5mm]
\le C\, \|\ep^{-1}\,w_-\|^{1/2}_{L^2(\B_\rho(x_0))}\ \|\nabla w_-\|_{L^2(\B_\rho(x_0))}^{3/2}\ \|\sqrt{\ep}\, F_A\|_{L^2(\B_\rho(x_0))}\\[5mm]
\ds\le \int_{\B_\rho(x_0)}\frac{1}{4\ep^2}\, w_-^2\ dx^3+C\, \|\nabla w_-\|_{L^2(\B_\rho(x_0))}^{2}\ \|\sqrt{\ep}\, F_A\|^{4/3}_{L^2(\B_\rho(x_0))}
\end{array}
\]

For $\|\sqrt{\ep}\, F_A\|_{L^2(\B_\rho(x_0))}$ small enough we obtain $w_-\equiv 0$. Thus $w\ge 0$.
We have
\[
-\Delta (|\nabla_Au|-w)+\frac{1}{2\,\ep^2}\, (|\nabla_Au|-w)\le 2\,(|\nabla_Au|-w)\,|F_A|
\]
Let $v:=|\nabla_Au|-w$  there holds
\[
-\Delta_y{v}+\lf[\frac{1}{2\,\ep^2}-2\,|F_A|\rg]\ {v}\le 0
\]
Let $\chi_s(t)\in C^\infty_0({\R},{\R}_+)$ such that $\chi\equiv 1$ on $(s,+\infty)$ and $\chi\equiv 0$ on $(-\infty,-s)$ and such that $\chi'\ge 0$ and $\chi'\rightharpoonup 2^{-1}(sign(t)+1)$ as $s\rightarrow 0$. There holds
\[
-\chi_s\circ v\,\Delta v=-\mbox{div}(\chi_s\circ v\,\nabla v)+\chi'_s(v)\,|\nabla v|^2=-\Delta\lf(\int_0^{v(x)}\chi_s(t)\ dt\rg)+\chi'_s(v)\,|\nabla v|^2
\]
Hence we have
\[
-\Delta\lf(\int_0^{v(x)}\chi_s(t)\ dt\rg)+\chi'_s(v)\,|\nabla v|^2+\lf[\frac{1}{2\,\ep^2}-2\,|F_A|\rg]\ \chi_s\circ v\ {v}\le 0\ .
\]
Since $\chi'\ge 0$ we deduce
\[
-\Delta\lf(\int_0^{v(x)}\chi_s(t)\ dt\rg)+\lf[\frac{1}{2\,\ep^2}-2\,|F_A|\rg]\ \chi_s\circ v\ {v}\ .
\]
Passing to the limit as $s\rightarrow 0$ we finally get
\[
-\Delta v_++\lf[\frac{1}{2\,\ep^2}-2\,|F_A|\rg]\ v_+\le 0\ .
\]
Let
\[
V:=\lf[\frac{1}{2\,\ep^2}-2\,|F_A|\rg]\
\]
By assumption we have
\[
\|V_-\|_{L^{(3/2,1)}(B_\rho(x_0))}\le\sigma
\]
This gives that the negative part of the potential $V$ in the Schr\"odinger operator $-\Delta+V$ is in the {\it Kato Class} since
\[
\sup_{x\in \B_\rho(x_0)}\int_{\B_\rho(x_0)}\frac{|V_-(y)|}{|x-y|}\ dy^3\le \sup_{x\in \B_\rho(x_0)}\lf\|\frac{1}{|x-y|}\rg\|_{L^{(3,\infty)}\B_\rho(x_0)}\ \|F_A\|_{L^{(3/2,1)}(B_\rho(x_0))}\le\,C\,\sigma
\]
Using the main result in \cite{AiS} we obtain that the Harnack inequality holds and we have
\[
\|v_+\|_{L^\infty(B_{\rho/2}(x_0)}\le C\, \rho^{-3}\ \int_{B_\rho(x_0)}v_+\ dx^3
\]
Hence
\[
\|(|\nabla_Au|-w)_+\|_{L^\infty(B_{\rho/2}(x_0)}\le C\ \rho^{-3}\ \int_{B_\rho(x_0)}(|\nabla_Au|-w)_+\ dx^3\le C\ \rho^{-3}\ \int_{B_\rho(x_0)}|\nabla_Au|\ dx^3
\]
Combining this inequality with (\ref{wlinf}) we obtain
\be
\label{linka}
\begin{array}{rl}
\ds\|\nabla_Au\|_{L^\infty(B_{\rho/2}(x_0))}&\ds\le C\ \rho^{-3}\ \int_{B_\rho(x_0)}|\nabla_Au|\ dx^3+ C\, \ep^{-1/2}\ \int_{\B_\rho(x_0)}\ep\,|F_A|^2+|A|^2\ dx^3\\[5mm]
\ds\qquad&\ds\quad+ C\, \ep^{-1/2}\ \sqrt{\int_{\B_\rho(x_0)}|A|^2\ dx^3}\ .
\end{array}
\ee
This is (\ref{linfcovder}) and we have proved the  Lemma~\ref{lm-linfty}.\hfill $\Box$

\medskip

\noindent{\bf Proof of Lemma~\ref{lm-impLinf}. }
Assume 
\[
CYMH\lf({u},{A},\B_{\sqrt{\ep}}(x_0)\rg)\le \sqrt{\ep}\,\delta\ .
\]
Then in particular
\be
\label{cursma}
\frac{1}{\sqrt{\ep}}\int_{\B_{\sqrt{\ep}}(x_0)}\ep\,|F_A|^2\ dx^3<\delta\ .
\ee
We recall that the Lorentz norm $L^{(3/2,1)}$ on the ball $\B_r(x_0)$ of a function $f$ is equivalent to the following quasi-norm
\[
\int_0^{+\infty}\lf| \lf\{ x\in \B_r(x_0)\ ; \ |f(x)|>t  \rg\} \rg|^{2/3}\ dt
\]
It can be bonded by the $L^2$ norm in the following way. Let $M>0$ to be fixed later, there holds
\[
\begin{array}{l}
\ds \int_0^{+\infty}\lf| \lf\{ x\in \B_r(x_0)\ ; \ |f(x)|>t  \rg\} \rg|^{2/3}\ dt\\[5mm]
\ds\le \int_0^{M}\lf| \lf\{ x\in \B_r(x_0)\ ; \ |f(x)|>t  \rg\} \rg|^{2/3}\ dt+\int_M^{+\infty}\lf| \lf\{ x\in \B_r(x_0)\ ; \ |f(x)|>t  \rg\} \rg|^{2/3}\ dt\\[5mm]
\ds\le M\, \lf(\frac{4\pi}{3}\rg)^{2/3}r^2+ \lf(  \int_M^{+\infty}\frac{dt}{t^2} \rg)^{1/3}\, \lf(\int_M^{+\infty}t\ \lf| \lf\{ x\in \B_r(x_0)\ ; \ |f(x)|>t  \rg\} \rg|\ dt\rg)^{2/3}\\[5mm]
\ds\le C\, M\, r^2+ C\, M^{-1/3}\, \|f\|^{4/3}_{L^2(\B_r(x_0))}
\end{array}
\]
Choosing $M:= r^{-3/2}\,\|f\|_{L^2(\B_r(x_0))}$ one obtains
\be
\label{lore}
\|f\|_{L^{({3/2},1)}(\B_r(x_0))}\le C\, \sqrt{r}\, \|f\|_{L^2(\B_r(x_0))}\ .
\ee
Applying this inequality to $f:=|F_A|$ on the ball $B_r(x_0)$  we obtain
\be
\label{lorcur}
\|F_A\|_{L^{({3/2},1)}(\B_{\sqrt{\ep}}(x_0))}\le\ C\ \ep^{1/4}\, \|F_A\|_{L^{2}(\B_{\sqrt{\ep}}(x_0))}\le C\, \sqrt{\delta}\ .
\ee
Because of (\ref{cursma}) and (\ref{lorcur}), for $\delta$ small enough such that $C\, \sqrt{\delta}<\sigma$ where $\sigma>0$ is given by Lemma~\ref{lm-linfty}, and assuming $|u|\ge1-\sigma$ on $\B_{\sqrt{\ep}}(x_0)$ we have
\[
\begin{array}{l}
\ds\|\nabla_Au\|_{L^{\infty}(\B_{\sqrt{\ep}/2}(x_0))}\le C\ \ep^{-3/2}\,\int_{\B_{\sqrt{\ep}}(x_0)}|\nabla_Au|\ dx^3+C\,\sqrt{\delta}\ \ep^{-1/4}\\[5mm]
\ds\qquad\le C\, \ep^{-1/4}\,\lf(\int_{\B_{\sqrt{\ep}}(x_0)}\frac{|\nabla_Au|^2}{\ep}\ dx^3\rg)^{1/2}+C\,\sqrt{\delta}\ \ep^{-1/4}\le C\,\sqrt{\delta}\ \ep^{-1/4}
\end{array}
\]
This concludes the proof of Lemma~\ref{lm-impLinf}.\hfill $\Box$

\subsection{Monotonicity of the $L^2-$density of the normalised curvature $\sqrt{\ep}\,F_A$ under smallness assumptions.}

\begin{Lm}
\label{lm-longcur}{\bf [$L^2-$Density Control of $\sqrt{\ep}\,F_A$  for radii $r<\sqrt{\ep}$]}
There exists $\delta_0>0$ such that, if $(u,A)$ is a smooth critical point of $CYMH_{\ep}^{1,1}$ in $\B_1(0)$ and if on $\B_{\sqrt{\ep}}(x_0)\subset \B_1(0)$  then for any $0<\delta<\delta_0$ if
\[
\frac{1}{2\sqrt{\ep}}\int_{\B_{\sqrt{\ep}}(x_0)} |{A}|^2+\frac{|d_{{A}} {u}|^2}{\ep}+\ep\,|F_{{A}}|^2\ dx^3+\frac{1}{2\,\ep^3}(1-|{u}|^2)^2\ dx^3\le \delta
\]
and
\[
1\ge |u|\ge 1-\delta_0
\]
then there exists $C>0$ universal such that
\be
\label{transcur}
\|\sqrt{\ep}\,F_A\|_{L^{(3,\infty)}(\B_{\sqrt{\ep}/2}(x_0))}\le C\, \sqrt{\delta}\ .
\ee
In particular there holds 
\be
\label{montrans}
\forall \,r<\sqrt{\ep}\qquad\frac{1}{r}\int_{B_{r}(x_0)}\ep\,|F_A|^2\ dx^3\le    C\, \delta\ .
\ee
\hfill $\Box$
\end{Lm}
\noindent{\bf Proof of Lemma~\ref{lm-longcur}}.
Assuming
From the Bochner-Weitzenb\"ock formula (\ref{bw-2})we obtain in $\B_{\sqrt{\ep}}(x_0)$, 
\[
\begin{array}{l}
\ds-\Delta\,|F_A|^2+2\, |\nabla_A F_A|^2\le  \frac{C}{\ep}\,\lf( \frac{|\nabla_A u|^2}{\ep}+\frac{|A|^2}{2}+\ep |F_A|^2\rg)\ |F_A|
\end{array}
\]
This implies using the identity (\ref{sq-0}) for $v=|F_A|$ this time and Kato inequality $|\nabla_AF_A|\ge |\nabla|F_A||$
\be
\label{bo-cu}
-\Delta\,|F_A|\le  \frac{C}{\ep}\,\lf( \frac{|\nabla_A u|^2}{\ep}+\frac{|A|^2}{2}+\ep |F_A|^2\rg)\ .
\ee
Let
\[
f:= \frac{C}{\ep}\,\lf( \frac{|\nabla_A u|^2}{\ep}+\frac{|A|^2}{2}+\ep |F_A|^2\rg)
\]
and observe 
\[
\int_{\B_{\sqrt{\ep}}(x_0)}f(x)\ dx^3=\frac{C}{\ep}\,\lf( \frac{|\nabla_A u|^2}{\ep}+\frac{|A|^2}{2}+\ep |F_A|^2\rg)\le C\,\frac{\delta}{\sqrt{\ep}}
\]
Let $w$ be the solution of
\be
\label{wsys}
\lf\{
\begin{array}{l}
\ds -\Delta w= f\qquad\mbox{ in }\B_{\sqrt{\ep}}(x_0)\\[5mm]
\ds w=0\qquad\mbox{ on }\p\B_{\sqrt{\ep}}(x_0)
\end{array}
\rg.
\ee
Classical Calderon Zygmund estimates is giving
\[
\|w\|_{L^{(3,\infty)}(\B_{\sqrt{\ep}}(x_0))}\le C\, \|f\|_{L^1(\B_{\sqrt{\ep}}(x_0))}\le  C\,\frac{\delta}{\sqrt{\ep}}
\]
and the maximum principle is implying $w\ge 0$. Since $|F_A|-w$ is sub-harmonic, Harnack inequality is giving
\[
\forall x\in \B_{\sqrt{\ep}/2}(x_0)\qquad (|F_A|(x)-w(x))^+\le C\, \ep^{-3/2}\ \int_{\B_{\sqrt{\ep}}(x_0))}||F_A|(x)-w(x)|\ dx^3
\]
Thus $\forall x\in \B_{\sqrt{\ep}/2}(x_0)$
\be
\label{pointcur}
\begin{array}{l}
\ds |F_A|(x)\le w(x)+C\, \ep^{-3/2}\ \int_{\B_{\sqrt{\ep}}(x_0))}|F_A|(x)\ dx^3+C\,  \ep^{-3/2}\ \int_{\B_{\sqrt{\ep}}(x_0))}w(x)\ dx^3\\[5mm]
\ds\le w(x)+C\,  \ep^{-1}\,\lf(\frac{1}{\sqrt{\ep}}\,\int_{\B_{\sqrt{\ep}}(x_0))} \ep\,|F_A|^2\ dx^3  \rg)^{1/2}+C\, \ep^{-1/2}\,\|w\|_{L^{(3,\infty)}(\B_{\sqrt{\ep}}(x_0))}\\[5mm]
\ds\le  w(x)+C\,  \ep^{-1}\,\sqrt{\delta}
\end{array}
\ee
Thus 
\[
\|F_A\|_{L^{(3,\infty)}(\B_{\sqrt{\ep}/2}(x_0))}\le\|w\|_{L^{(3,\infty)}(\B_{\sqrt{\ep}/2}(x_0))}+C\,\ep^{-1}\,\sqrt{\ep}\,\sqrt{\delta}\le  C\,\frac{\delta}{\sqrt{\ep}}
\]
This concludes the proof of Lemma~\ref{lm-longcur}.\hfill $\Box$
\subsection{Monotonicity of the $L^2-$density of the connexion $A$ and the normalised covariant derivative $\nabla_Au/\sqrt{\ep}$ under smallness assumptions.}

\begin{Lm}
\label{lm-longcur}{\bf [Control  of the $L^{(3,1)}-$norm of $\sqrt{\ep}\,F_A$ on $\sqrt{\ep}$ balls]}
There exists $\delta_0>0$ such that, if $(u,A)$ is a smooth critical point of $CYMH_{\ep}^{1,1}$ in $\B_1(0)$ and if on $\B_{2\,\sqrt{\ep}}(x_0)\subset \B_1(0)$  then for any $0<\delta<\delta_0$ if
\[
\frac{1}{2\sqrt{\ep}}\int_{\B_{2\,\sqrt{\ep}}(x_0)} |{A}|^2+\frac{|d_{{A}} {u}|^2}{\ep}+\ep\,|F_{{A}}|^2\ dx^3+\frac{1}{2\,\ep^3}(1-|{u}|^2)^2\ dx^3\le \delta
\]
and
\[
1\ge |u|\ge 1-\delta_0
\]
then there exists $C>0$ universal such that
\be
\label{longcur}
\|\sqrt{\ep}\,F_A\|_{L^{(3,1)}(\B_{\sqrt{\ep}}(x_0))}\le C\, \sqrt{\delta}\ .
\ee
In particular there holds
\be
\label{moncurerr}
\int_{B_{\sqrt{\ep}}(x_0)}\frac{1}{|x-x_0|}\ep\,|F_A|^2\ dx^3\le C\, \|\ep\,| F_A|^2\|_{L^{(3/2,1)}(  \B_{\sqrt{\ep}}(x_0))}\le    C\, \delta\ .
\ee
\hfill $\Box$
\end{Lm}
\noindent{\bf Proof of Lemma~\ref{lm-longcur}}.

Let $x_1\in \B_{\sqrt{\ep}}(x_0)$ and $r<\sqrt{\ep}$. We recall the identity corresponding to (\ref{premon}) for $CYMH_\ep^{1,1}$
\[
\begin{array}{l}
\ds-\int_{\B_r}|A|^2+\frac{|\nabla_A u|^2}{\ep}\ dx^3+r\int_{\p \B_r}|A|^2+\frac{|\nabla_A u|^2}{\ep}\ dvol_{\p \B_r}\\[5mm]
\ds+\int_{\B_r}\ep\,|F_A|^2\ dx^3+r\int_{\p \B_r}\ep\,|F_A|^2\ dvol_{\p \B_r}\\[5mm]
\ds-\int_{\B_r}   \frac{3}{2}\frac{(1-|u|^2)^2}{\ep^3}\ dx^3         +r\,\frac{3}{2}\int_{\p\B_r}\frac{(1-|u|^2)^2}{\ep^3}\ dvol_{\p \B_r}\ge 0
\end{array}
\]
Now we compute
\[
\begin{array}{l}
\ds\frac{d}{dr}\lf(\frac{1}{\sqrt{r}}\,\int_{\B_r(x_1)} |A|^2+\frac{|\nabla_Au|^2}{\ep}+\ep\,|F_A|^2+\frac{3}{2}\,\frac{(1-|u|^2)^2}{\ep^3}\ \ dx_0\rg)\\[5mm]
\ds=-\frac{1}{2\,r^{3/2}}\,\int_{\B_r(x_1)} |A|^2+\frac{|\nabla_Au|^2}{\ep}+\ep\,|F_A|^2+\frac{3}{2}\,\frac{(1-|u|^2)^2}{\ep^3}\ \ dx_0\\[5mm]
\ds+\frac{1}{\sqrt{r}}\,\int_{\p\B_r(x_1)} |A|^2+\frac{|\nabla_Au|^2}{\ep}+\ep\,|F_A|^2+\frac{3}{2}\,\frac{(1-|u|^2)^2}{\ep^3}\ \ dvol_{\p \B_r(x_0)}\\[5mm]
\ds\ge \frac{1}{2\,r^{3/2}}\,\int_{\B_r(x_1)} |A|^2+\frac{|\nabla_Au|^2}{\ep}+\frac{3}{2}\,\frac{(1-|u|^2)^2}{\ep^3}\ dx^3-\frac{3}{2\,r^{3/2}}\,\int_{\B_r(x_1)} \ep\,|F_A|^2\ dx^3
\end{array}
\]
Using (\ref{montrans}) we have then
\[
\begin{array}{l}
\ds\frac{d}{dr}\lf(\frac{1}{\sqrt{r}}\,\int_{\B_r(x_1)} |A|^2+\frac{|\nabla_Au|^2}{\ep}+\ep\,|F_A|^2+\frac{3}{2}\,\frac{(1-|u|^2)^2}{\ep^3}\ \ dx_0\rg)\\[5mm]
\ds\ge -\,C\, r^{-1/2}\, \delta=-\frac{d}{dr}\lf(2\,C\,\delta\,r^{1/2}\rg)
\end{array}
\]
We deduce in particular the following Morrey norm control
\be
\label{morrey}
\sup_{x_1\in \B_{\sqrt{\ep}}(x_0)}\ \sup_{r<\sqrt{\ep}}\frac{1}{\sqrt{r}}\,\int_{\B_r(x_1)} |A|^2+\frac{|\nabla_Au|^2}{\ep}+\ep\,|F_A|^2\ dx^3\le C\,\delta\ \ep^{1/4}
\ee
Considering again (\ref{bo-cu}) and $w$ solving (\ref{wsys}) we have
\[
\sup_{x_1\in \B_{\sqrt{\ep}}(x_0)}\ \sup_{r<\sqrt{\ep}}\frac{1}{\sqrt{r}}\,\int_{\B_r(x_1)} f\ dx^3\le  C\,\delta\ \ep^{-3/4}
\]
Let $\ti{w}(y):=\ep\,w(\sqrt{\ep}\,y+x_0)$ and $\ti{f}(y):=\ep^2\, f(\sqrt{\ep}\,y+x_0)$. There holds
\[
\lf\{
\begin{array}{l}
\ds -\Delta_y \ti{w}= \ti{f}\qquad\mbox{ on }\B_1(0)\\[5mm]
\ds \ti{w}=0\qquad\mbox{ on }\p\B_1(0)
\end{array}
\rg.
\]
Moreover
\[
\sup_{y_1\in \B_1(0))}\ \sup_{\rho<1}\frac{1}{\sqrt{\rho}}\,\int_{\B_\rho(y_1)} \ti{f}\ dy^3\le  C\,\delta\ 
\]
Using a classical result by Adams (\cite{Ada} proposition 3.2 for $\la=\frac{5}{2}$ and $\al=2$) we have
\[
\|\ti{w}\|_{L^{(5,\infty)}(\B_1(0))}\le C\, \delta
\]
This gives
\[
\|w\|_{L^{(5,\infty)}(\B_{\sqrt{\ep}}(x_0))}\le C\, \delta\, \ep^{-7/10}\,
\]
We deduce 
\[
\|w\|_{L^{(3,1)}(\B_{\sqrt{\ep}}(x_0))}\le C\,\delta\,\sqrt{\ep}
\]
Injecting this information in (\ref{pointcur}) we finally obtain
\be
\label{L31cur}
\|F_A\|_{L^{(3,1)}(\B_{\sqrt{\ep}}(x_0))}\le C\,\delta\,\sqrt{\ep}
\ee
This concludes the proof of Lemma~\ref{lm-longcur}.\hfill $\Box$

\medskip

Finally we extend the energy density control given by Lemme~\ref{lm-sca} for radii larger than $\sqrt{\ep}$ to the whole range of radii.
\begin{Lm}
\label{lm-sca-wh} Let $0<r_1<1$. There exists $\delta_0>0$  independent of $\ep\in (0,1)$ but depending only on $r_1$ such that for any minimizer $(u,A)$ of $CYMH^{1,1}_\ep$ and any $\B_r(x_0)\subset\B_{r_1}(0)$  if for $r\ge\sqrt{\ep}$
\be
\label{hypdec}
\frac{1}{2\,r}\int_{\B_r(x_0)} |{A}|^2+\frac{|d_{{A}} {u}|^2}{\ep}+\ep\,|F_{{A}}|^2\ dx^3+\frac{1}{2\,\ep^3}(1-|{u}|^2)^2\ dx^3<\delta
\ee
and
\[
 1\ge |u|\ge 1-\delta_0\qquad\mbox{ in }\B_r(x_0)
\]
then, for any $0<\rho<r$ there holds
\be
\label{densc}
\frac{1}{2\,\rho}\int_{\B_\rho(x_0)} |{A}|^2+\frac{|d_{{A}} {u}|^2}{\ep}+\ep\,|F_{{A}}|^2\ dx^3+\frac{1}{2\,\ep^3}(1-|{u}|^2)^2\ dx^3<C\,\delta\ .
\ee
where $C>0$ is universal.
\hfill $\Box$
\end{Lm}
\noindent{\bf Proof of Lemma~\ref{lm-sca-wh}.}  Because of Lemma~\ref{lm-sca} it suffices to consider the case where $\rho<\sqrt{\ep}$. We have
\[
\frac{1}{2\,\sqrt{\ep}}\int_{\B_{\sqrt{\ep}}(x_0)} |{A}|^2+\frac{|d_{{A}} {u}|^2}{\ep}+\ep\,|F_{{A}}|^2\ dx^3+\frac{1}{2\,\ep^3}(1-|{u}|^2)^2\ dx^3<\delta
\]
Recall from remark~\ref{rm-mon} that one recast the monotonicity formula~\ref{III.1} for $CYMH^{1,1}_\ep$ in the following way
\[
\begin{array}{l}
\ds\frac{d}{dr}\lf(   \frac{1}{2r}\int_{\B_r(x_0)}|A|^2+\frac{|\nabla_Au|^2}{\ep}+\frac{3}{2\,\ep^3}(1-|u|^2)^2+\lf(\frac{2r}{|x-x_0|}-1\rg)\,\ep\,|F_A|^2\ dx^3 \rg)\\[5mm]
\ds=\frac{1}{r}\,\int_{\p \B_r(x_0)}|A_r|^2+\frac{|(\nabla_Au)_r|^2}{\ep}+\frac{1}{2\ep^3}(1-|u|^2)^2+\ep\,|F_A\res\p_r|^2\ dvol_{\p \B_r(x_0)}\ge 0
\end{array}
\]
Because of (\ref{moncurerr}) we have 
\[
\int_{\B_{\sqrt{\ep}}(x_0)}\frac{2}{|x-x_0|}\,\ep\,|F_A|^2\ dx^3 \le C\, \delta\qquad\mbox{ and }\qquad\frac{1}{\rho}\int_{\B_{\rho}(x_0)}\ep\,|F_A|^2\ dx^3 \le C\, \delta\ .
\]
Then there holds
\[
\begin{array}{l}
\ds\frac{1}{2\rho}\int_{\B_\rho(x_0)}|A|^2+\frac{|\nabla_Au|^2}{\ep}+\frac{3}{2\,\ep^3}(1-|u|^2)^2\ dx^3\le \frac{1}{2\,\sqrt{\ep}}\int_{\B_{\sqrt{\ep}}(x_0)} |{A}|^2+\frac{|d_{{A}} {u}|^2}{\ep}+\frac{1}{2\,\ep^3}(1-|{u}|^2)^2\ dx^3\\[5mm]
\ds\qquad+\frac{1}{\rho}\int_{\B_\rho(x_0)}\ep\,|F_A|^2\ dx^3+\int_{\B_{\sqrt{\ep}}(x_0)}\frac{2}{|x-x_0|}\,\ep\,|F_A|^2\ dx^3\le C\, \delta
\end{array}
\]
and  Lemma~\ref{lm-sca-wh} is proved.\hfill $\Box$
\subsection{The $\delta-$regularity Theorem for $CYMH_\ep^{1,1}$ independent of $\ep>0$ for smooth critical points}
 
 We are now proving the second main ``instrument'' for passing to the limit in the equations as $\ep\rightarrow 0$.
 \begin{Lm} {\bf[$\delta$-regularity theorem independent of $\ep>0$]} 
 \label{lm-deltareg}
For any $M>0$ and $0<r_1<1$ there exists $\delta_0>0$  and $C>0$ such that for any smooth critical point $(u,A)$  to $CYMH_\ep^{1,1}$ and any $r> C\,\ep\,\log \ep^{-1}$ satisfying
\[
CYMH_\ep^{1,1}(u,A,\B_1(0))<M
\]
 and any $x_0\in \B_{r_1}(0)$ such that $\B_r(x_0)\subset \B_{r_1}(0)$, if
 \[
  \sup_{x\in \B_{r/2}(x_0)}\ \sup_{\rho<r/2}\frac{1}{2\,\rho}\int_{\B_\rho(x)} |{A}|^2+\frac{|d_{{A}} {u}|^2}{\ep}+\ep\,|F_{{A}}|^2\ dx^3+\frac{1}{2\,\ep^3}(1-|{u}|^2)^2\ dx^3<\delta
 \]
 and
 \[
 1\ge|u|\ge 1-\delta\qquad\mbox{ in }\B_r(x_0)
 \]
 where $\delta<\delta_0$ then for any $p<+\infty$ there holds
 \[
 \|u\|_{W^{2,p}(\B_{r/2}(x_0))}+\| \nabla_Au\|_{W^{1,p}(\B_{r/2}(x_0))}\le C_{M,p,r}
 \]
 where $C_{M,p,r}$ only depends on $M$, $p$ and $r$.\hfill $\Box$
 \end{Lm}
\noindent{\bf Proof of Lemma~\ref{lm-deltareg}.} We denote $\hat{u}:=u/|u|$ on $\B_{r}(x_0)$ and we write
 \[
 -\nabla\hat{u}=\frac{1}{4}[\hat{u},[\hat{u},\nabla\hat{u}]]=\frac{1}{4}[\hat{u},[\hat{u},\nabla_A\hat{u}]]-\frac{1}{4}[\hat{u},[\hat{u},[A,\hat{u}]]]=[A,\hat{u}]+\frac{1}{4}[\hat{u},[\hat{u},\nabla_A\hat{u}]]
 \]
 Taking the divergence of this identity and using the fact that $A$ is divergence free is giving
 \[
 -\Delta\hat{u}=\sum_{l=1}^3[A_l,\p_{x_l}\hat{u}]+\frac{1}{4}\mbox{div}\lf([\hat{u},[\hat{u},\nabla_A\hat{u}]]\rg)
 \]
 Observe that (using quaternionic multiplication)
 \[
 \begin{array}{l}
 \ds\sum_{l=1}^3[A_l,\p_{x_l}\hat{u}]^j=2\,\sum_{l=1}^3\lf(A^{\bf j}_l\,\p_{x_l}u^{\bf k}-A^{\bf k}_l\,\p_{x_l}u^{\bf j}\rg)\ {\bf i}+2\,\sum_{l=1}^3\lf(A^{\bf k}_l\,\p_{x_l}u^{\bf i}-A^{\bf i}_l\,\p_{x_l}u^{\bf k}\rg)\ {\bf j}\\[5mm]
 \ds\qquad+2\,\sum_{l=1}^3\lf(A^{\bf i}_l\,\p_{x_l}u^{\bf j}-A^{\bf j}_l\,\p_{x_l}u^{\bf i}\rg)\ {\bf k}
 \end{array}
 \]
 Hence, in the orthonormal basis $({\bf i},{\bf j},{\bf k})$ of $\frak{su}(2)$ one has
 \[
 [A,\nabla u]=\Om\cdot \nabla u
 \]
 where
 \[
 \Om:=\lf(
 \begin{array}{ccc}
 0& -2\,A^{\bf k}& 2\,A^{\bf j}\\[3mm]
2\,A^{\bf k}& 0&  -2\,A^{\bf i}\\[3mm]
 -2\,A^{\bf j}& 2\,A^{\bf i}& 0
 \end{array}
 \rg)\in so(3)\otimes {\R}^3\ ,
 \]
 and the system satisfied by $\hat{u}$ can be recasted in 
 \be
 \label{anti}
 -\Delta \hat{u}=\Om\cdot \nabla \hat{u}+\mbox{div}\, f
 \ee
 where $f:=4^{-1}\,[\hat{u},[\hat{u},\nabla_A\hat{u}]]\in L^3(\B_r(x_0))$. Moreover, from (\ref{mordel}) there holds
 \[
  \sup_{x\in \B_{r/2}(x_0)}\ \sup_{\rho<r/2}\frac{1}{2\,\rho}\int_{\B_\rho(x)}|\Om|^2+|\nabla\hat{u}|^2\le C\,\delta
 \]
 In \cite{Rivom} the author discovered the importance of antisymmetry for the regularity and compactness theory of critical Schr\"odinger systems. The main ideas in \cite{Rivom} have been adapted to supercritical dimensions under small Morrey norm hypothesis in \cite{RivSt} and in \cite{Mos2} in the presence of a source term\footnote{In \cite{Mos2} the source term div$\,f$ is considered to be in $L^q(\B_r(x_0)$ while in our case it is in $W^{-1,3}(\B_r(x_0)$ but exactly the same arguments as in \cite{Mos2} are giving the corresponding result in our case.  }. From these works we can then conclude that, for $\delta>0$ small enough
 \be
 \label{L3hatu}
 \int_{\B_{r/4}(x_0)}|\nabla\hat{u}|^3\ dx^2\le C\,\delta+C\, \int_{\B_{r/4}(x_0)}\lf|[\hat{u},\nabla_A\hat{u}]\rg|^3\ dx^2\le C\ \delta\ .
 \ee
 Thus we deduce
 \be
 \label{L3AT}
  \int_{\B_{r/4}(x_0)}|A^T|^3\ dx^2\le C\,\delta\ .
 \ee
 We have for any $l,m\in\{1,2,3\}$
\[
\begin{array}{l}
\ds \hat{u}\cdot F_{ml}=\hat{u}\cdot (\nabla_A)_m A_l-\hat{u}\cdot (\nabla_A)_l A_m+\hat{u}\cdot [A_l,A_m]\\[5mm]
\ds=\p_{x_m}(\hat{u}\cdot A_l)-\p_{x_l}(\hat{u}\cdot A_m)-(\nabla_A)_m\hat{u}\cdot A_l+(\nabla_A)_l\hat{u}\cdot A_m+\hat{u}\cdot [A_l,A_m]\ .
\end{array}
\]
Hence for any $l\in\{1,2,3\}$
\[
\begin{array}{l}
\ds \sum_{m=1}^3\p_{x_m}\lf(\hat{u}\cdot F_{ml}\rg)=\Delta(\hat{u}\cdot A_l)-\sum_{m=1}^3\p^2_{x_lx_m}(\hat{u}\cdot A_m)\\[5mm]
\ds- \sum_{m=1}^3\p_{x_m}\lf((\nabla_A)_m\hat{u}\cdot A_l\rg)+ \sum_{m=1}^3\p_{x_m}\lf((\nabla_A)_l\hat{u}\cdot A_m+ \hat{u}\cdot [A_l,A_m]  \rg)
\end{array}
\]
We have using the fact that $A$ is divergence free
\[
-\sum_{m=1}^3\p^2_{x_lx_m}(\hat{u}\cdot A_m)=-\sum_{m=1}^3\p_{x_l}\lf( (\nabla_A)_m\hat{u}\cdot A_m  \rg)
\]
Moreover from (\ref{CYMHS-co})
\[
 \sum_{m=1}^3\p_{x_m}\lf(\hat{u}\cdot F_{ml}\rg)=\sum_{m=1}^3(\nabla_A)_m\hat{u}\cdot F_{ml}+\ep^{-1}\, \hat{u}\cdot A_l
\]
Hence there holds  for any $l\in\{1,2,3\}$
\[
\begin{array}{l}
\ds-\Delta(\hat{u}\cdot A_l)+\frac{\hat{u}\cdot A_l}{\ep}=-\p_{x_l}\lf( (\nabla_A)_m\hat{u}\cdot A_m  \rg)-\sum_{m=1}^3(\nabla_A)_m\hat{u}\cdot F_{ml}\\[5mm]
\ds- \sum_{m=1}^3\p_{x_m}\lf((\nabla_A)_m\hat{u}\cdot A_l\rg)+ \sum_{m=1}^3\p_{x_m}\lf((\nabla_A)_l\hat{u}\cdot A_m\rg)+ \sum_{m=1}^3\p_{x_m}\lf(\hat{u}\cdot [A_l,A_m] \rg)
\end{array}
\]
We have respectively 
\[
\hat{u}\cdot[A_l,A_m]=\hat{u}\cdot [A_l^T,A_m^T]\qquad\mbox{ and }\qquad (\nabla_A)_m\hat{u}\cdot A_n=(\nabla_A)_m\hat{u}\cdot A_n^T\ .
\]
and we finally deduce
\be
\label{longA}
\begin{array}{l}
\ds-\Delta(\hat{u}\cdot A_l)+\frac{\hat{u}\cdot A_l}{\ep}=-\p_{x_l}\lf( (\nabla_A)_m\hat{u}\cdot A^T_m  \rg)-\sum_{m=1}^3(\nabla_A)_m\hat{u}\cdot F_{ml}\\[5mm]
\ds- \sum_{m=1}^3\p_{x_m}\lf((\nabla_A)_m\hat{u}\cdot A^T_l\rg)+ \sum_{m=1}^3\p_{x_m}\lf((\nabla_A)_l\hat{u}\cdot A^T_m\rg)+ \sum_{m=1}^3\p_{x_m}\lf(\hat{u}\cdot [A^T_l,A^T_m] \rg)
\end{array}
\ee
Using classical Calderon Zygmund theory we obtain
\[
\begin{array}{l}
\|\hat{u}\cdot A\|_{L^{(3,\infty)}(\B_{r/8}(x_0))}+\|\nabla(\hat{u}\cdot A)\|_{L^{(3/2,\infty)}(\B_{r/8}(x_0))}\le C\ \|\nabla_A\hat{u}\cdot A^T\|_{L^{3/2}(\B_{r/4}(x_0))}+C\,\||A^T|^2\|_{L^{3/2}(\B_{r/4}(x_0))}\\[5mm]
\ds\qquad+C\,\|F_A\cdot\nabla_A\hat{u}\|_{L^{1}(\B_{r/4}(x_0))}+C\,\sqrt{r}^{-1}\,\|\hat{u}\cdot A\|_{L^{2}(\B_{r/4}(x_0))}\le C\, \sqrt{\delta}
\end{array}
\]
Hence in particular there holds
\be
\label{L32infcurlon}
\begin{array}{l}
\|\hat{u}\cdot F_A\|_{L^{(3/2,\infty)}(\B_{r/8}(x_0))}\le\,\|d(\hat{u}\cdot A)\|_{L^{(3/2,\infty)}(\B_{r/8}(x_0))}+ \|d\hat{u}\cdot A\|_{L^{(3/2,\infty)}(\B_{r/8}(x_0))}\\[5mm]
\ds\qquad+\|\hat{u}\cdot [A^T,A^T]\|_{L^{(3/2,\infty)}(\B_{r/8}(x_0))}\le C\, \sqrt{\delta}
\end{array}
\ee
Regarding the Transverse part of the curvature, we have the following inequality from Proposition~\ref{pr-bwt}
\be
\label{bwt-bis}
\begin{array}{l}
\ds -\mbox{div}\lf(|u|^{-4}\,\nabla |[u,F_A]|\rg)+\frac{1}{2\ep^2}|[u,F_A]|\le 8\, |u|^{-6}\, |\nabla_A\,u|^2\,|F_A|+4\, |u|^{-4}\,|\nabla_Au|\,|\nabla(F_A\cdot u)|\\[5mm]
\ds\qquad+4\,\ep^{-2}|u|^{-4}\,|\nabla_A u|^2+2\,\ep^{-1}\,|u|^{-4}\,|A|^2+4\,|u|^{-4}\,|F_A|^2\ .
\end{array}
\ee
Let
\[
\begin{array}{l}
\ds f:= 8\, |u|^{-6}\, |\nabla_A\,u|^2\,|F_A|+4\, |u|^{-4}\,|\nabla_Au|\,|\nabla(F_A\cdot u)|\\[5mm]
\ds\qquad+4\,\ep^{-2}|u|^{-4}\,|\nabla_A u|^2+2\,\ep^{-1}\,|u|^{-4}\,|A|^2+4\,|u|^{-4}\,|F_A|^2\ .
\end{array}
\]
Observe that $F_A$ satisfies the system
\[
\lf\{
\begin{array}{l}
\ds d_AF_A=0\\[5mm]
\ds d^\ast_A F_A=-\frac{[u,d_Au]}{\ep^2}-\frac{A}{\ep}
\end{array}
\rg.
\]
Hence
\[
\lf\{
\begin{array}{l}
\ds d\lf(\hat{u}\cdot F_A\rg)=d_A\hat{u}\,\dot{\wedge}\, F_A\\[5mm]
\ds d^\ast(\hat{u}\cdot F_A)=d_A\hat{u}\,\dot{\res} \,F_A-\frac{\hat{u}\cdot A}{\ep}
\end{array}
\rg.
\]
Classical elliptic estimates give
\be
\label{longcurv}
\begin{array}{l}
\ds\|\hat{u}\cdot F_A\|_{L^{(3,3/2)}(\B_{r/16}(x_0))}+\|\nabla (\hat{u}\cdot F_A)\|_{L^{3/2}(\B_{r/16}(x_0))}\\[5mm]
\ds\le C\, \ep^{-1}\,  \|\hat{u}\cdot A\|_{L^{3/2}(B_{r/8}(x_0))}+ C\,\|\nabla_A\hat{u}\|_{L^{6}(B_{r/8}(x_0))} \ \|F_A\|_{L^{2}(B_{r/8}(x_0))} + C\, r^{-1}\,\|\hat{u}\cdot F_A\|_{L^{(3/2,\infty)}(\B_{r/8}(x_0))}\\[5mm]
\ds\quad\le C\, \ep^{-1}\, {r}\,\sqrt{\delta}+C\, \ep^{-1}\, \delta^{2/3}\, r^{2/3}+C\,\sqrt{\frac{\delta}{\ep}}\ 
\end{array}
\ee
Hence
 \[
 \begin{array}{l}
\ds\int_{\B_{r/16}(x_0)}f(x)\ dx^3\le C\, \ep^{-1}\,\|\sqrt{\ep}\,|\nabla_Au|^2\|_{L^2(\B_r(x_0))}\, \|\sqrt{\ep}\,F_A\|_{L^2(\B_r(x_0))}\\[5mm]
\ds\quad+ C\,\|\nabla_Au\|_{L^3(\B_{r/16}(x_0))}\, \|\nabla(u\cdot F_A)\|_{L^{3/2}(\B_{r/16}(x_0))}+C\,\ep^{-1}\int_{\B_{r/8}(x_0)}\frac{|\nabla_Au|^2}{\ep}+|A|^2+\ep\,|F_A|^2\ dx^3\,\\[5mm]
\ds\quad\le C\,\ep^{-1}\,\sqrt{\delta}\ .
\end{array}
\]
 Denote $\rho:=r/16$. Let $w$ be the solution of
 \[
 \lf\{
\begin{array}{l}
\ds-\mbox{div}\lf(|u|^{-4}\,\nabla w\rg)+\frac{1}{4\,\ep^2|u|^4}w=f\qquad\mbox{ in }\B_{\rho}(x_0)\\[5mm]
 \ds w=0\qquad\mbox{ on }\p\B_{\rho}(x_0)
\end{array} 
 \rg.
 \]
 Since $f\ge 0$ and $w=0$ on $\p\B_\rho(x_0)$, the maximum principle is implying $w\ge 0$. Denoting $\ti{w}:=w/|u|^4$ we have
  \[
 \lf\{
\begin{array}{l}
\ds-\Delta \ti{w}+\frac{1}{4\,\ep^2}\ti{w}=f+4\,\mbox{div}\lf( \frac{\nabla|u|}{|u|}\, \ti{w} \rg)\qquad\mbox{ in }\B_{\rho}(x_0)\\[5mm]
 \ds w=0\qquad\mbox{ on }\p\B_{\rho}(x_0)
\end{array} 
 \rg.
 \]
 We have the a-priori estimate
 \[
 \|\ti{w}\|_{L^{(3,\infty)}(\B_{\rho}(x_0))}+\|\nabla \ti{w}\|_{L^{(3/2,\infty)}(\B_{\rho}(x_0))}\le C\,\|f\|_{L^{1}(\B_{r/16}(x_0))}+C\, \|\nabla |u|\|_{L^3(\B_{\rho}(x_0))}\  \|\ti{w}\|_{L^{(3,\infty)}(\B_{\rho}(x_0))}\ .
 \]
 Since
 \[
 \|\nabla |u|\|_{L^3(\B_{\rho}(x_0))}\le  \|\nabla_Au\|_{L^3(\B_{\rho}(x_0))}\le C\,\lf(\int_{\B_r(x_0)}\frac{|\nabla_Au|^2}{\ep}\ dx^3\rg)^{1/3}\le C\, \delta^{1/3}
 \]
For $\delta$ chosen small enough we have
\be
\label{tiw3inf}
 \|\ti{w}\|_{L^{(3,\infty)}(\B_{\rho}(x_0))}+\|\nabla \ti{w}\|_{L^{(3/2,\infty)}(\B_{\rho}(x_0))}\le C\,\|f\|_{L^{1}(\B_{\rho}(x_0))}\le C\, \ep^{-1}\,\sqrt{\delta}
\ee
Let $X:=4\,\frac{\nabla|u|}{|u|}\, \ti{w}$. We have then
\[
\|X\|_{L^{(3/2,\infty)}(\B_{\rho}(x_0))}\le  C\, \ep^{-1}\ \delta^{1/3}\,\sqrt{\delta}\ .
\]
 Let $h\in L^{(3,1)}(B_\rho(x_0))$ and $\varphi$ solution to
\[
\lf\{
\begin{array}{l}
\ds-\Delta \varphi+\frac{\varphi}{4\,\ep^2}= h\qquad\mbox{ in }B_\rho(x_0)\\[5mm]
\ds \varphi=0\qquad\mbox{ on }\p B_\rho(x_0)
\end{array}
\rg.
\]
We have
\[
\|\varphi\|_{L^{(3,\infty)} \B_\rho(x_0)}\le \|h\|_{L^{1}(B_\rho(x_0))}
\]
We extend $h$ outside $\B_\rho(x_0)$ by $0$ and we denote $\ti{h}$ this extension. Let
\[
\ti{\varphi}(x):=\frac{1}{4\pi}\int_{\R^3}\frac{1}{|x-y|}\,e^{-\frac{|x-y|}{2\,\ep}}\ \ti{h}(y)\ dy^3
\]
Recall
\[
\frac{d}{dr}\lf[  r^{-1}e^{-r/2\,\ep}\rg]=-[r^{-2}+r^{-1}\,2^{-1}\ep^{-1}]\,   e^{-r/2\ep}<0\quad\forall \ r>0\ .
\]
Hence $r\rightarrow r^{-1}e^{-r/\ep}$ is a diffeomorphism from $(0,+\infty)$ into itself. For $t\ge 1$ we denote $\psi_\ep(t)$ such that
\[
\psi_\ep(t)^{-1}\,\exp(-\psi_\ep(t)/2\,\ep)=t\quad\Leftrightarrow\quad \psi_\ep(t)=2\,\ep\,\log(t^{-1})+2\,\ep\, \log(\psi_\ep(t)^{-1})
\]
We have 
\[
\lf|\lf\{ x\ ;\   r^{-1}e^{-r/2\,\ep}>t \rg\}\rg|= |B^3_{ \psi_\ep(t)}(0)|=\frac{4\pi}{3}\,\lf(\psi_\ep(t)\rg)^3\ .
\]
hence
\[
\begin{array}{l}
\ds\lf[\sup_{t>e^{-1/\ep}}t^{3/2}\,\lf|\lf\{ x\ ;\   r^{-1}e^{-r/\ep}>t \rg\}\rg|\rg]^{2/3}= \lf[\sup_{t>e^{-1/\ep}}t^{3/2}\, \frac{4\pi}{3}\,\lf(\psi_\ep(t)\rg)^3 |\rg]^{2/3}=\lf(\frac{4\pi}{3}\rg)^{2/3} \sup_{t>e^{-1/\ep}}\ t\ \psi_\ep^2(t)
\end{array}
\]
We have since $\psi_\ep$ is decreasing
\[
\sup_{t>e^{-1/2\ep}}\ t\ \psi_\ep^2(t)= \sup_{t>e^{-1/2\ep}}\psi_\ep(t)\,\exp(-\psi_\ep(t)/2\ep)=\sup_{u< 1}u\, \exp(-u/2\ep)=2\,\ep\ .
\]
Thus
\[
\sup_{x\in{\R}^3} \lf\| {\mathbf 1}_{|x-y|<1} \frac{1}{|x-y|}\,e^{-\frac{|x-y|}{2\,\ep}}\rg\|_{L^{(3/2,\infty)}({\R}^3)}\le C\, \ep\ ,
\]
and 
\[
\begin{array}{l}
\ds\sup_{x\in \B_\rho(x_0)} \lf\| {\mathbf 1}_{|x-y|>1} \frac{1}{|x-y|}\,e^{-\frac{|x-y|}{2\,\ep}}\rg\|_{L^{(3/2,\infty)}(\B_\rho)}\\[5mm]
\ds\le C\,  \rho^2\,\sup_{x\in \B_\rho(x_0)} \lf\| {\mathbf 1}_{|x-y|>1} \frac{1}{|x-y|}\,e^{-\frac{|x-y|}{2\,\ep}}\rg\|_{L^\infty(\B_\rho)}\le C\, \rho^2\,e^{-1/2\,\ep}\ .
\end{array}
\]
Assuming $\rho>\ep$ we obtain
\[
\|\ti{\varphi}\|_{L^\infty\lf(\B_\rho(x_0)\rg)}\le \,C\, \ep\, \|h\|_{L^{(3,1)}(\B_\rho(x_0))}
\]
We have
\be
\label{dif-tivar}
\lf\{
\begin{array}{l}
\ds-\Delta (\varphi-\ti{\varphi})+\frac{\varphi-\ti{\varphi}}{4\,\ep^2}= 0\qquad\mbox{ in }\B_\rho(x_0)\\[5mm]
\ds \varphi-\ti{\varphi}=-\ti{\varphi}\qquad\mbox{ on }\p \B_\rho(x_0)
\end{array}
\rg.
\ee
Hence, thanks to the maximum principle we have
\[
\|{\varphi}\|_{L^\infty\lf(\B_\rho(x_0)\rg)}\le \,C\, \ep\, \|h\|_{L^{(3,1)}\lf(\B_\rho(x_0)\rg)}
\]
We have also, thanks to the generalised Young inequality in Lorentz spaces (see \cite{Riv-Func})
\[
\begin{array}{l}
\ds\|\nabla\ti{\varphi}\|_{L^{(3,1)}({\R}^3)}\le \lf\|\nabla_x\lf(  \frac{1}{|x-y|}\,e^{-\frac{|x-y|}{2\,\ep}}\rg)\rg\|_{L^{1}({\R}^3)}\  \|h\|_{L^{(3,1)}(\B_\rho(x_0))}\\[5mm]
\ds\quad\le C\ \ep\ \|h\|_{L^{(3,1)}(\B_\rho(x_0))}\ .
\end{array}
\]
We have then $\|\ti{\varphi}\|_{W^{1-1/3,(3,1)}(\p\B_\rho(x_0))}\le  C\, \ep\, \|h\|_{L^{(3,1)}(\B_\rho(x_0))}$. Inserting this inequality in (\ref{dif-tivar}) is giving
\[
\ds\|\nabla\varphi\|_{L^{(3,1)}({\R}^3)}\le  C\ \ep\ \|h\|_{L^{(3,1)}(\B_\rho(x_0))}\ .
\]
We have
\[
\begin{array}{l}
\ds\int_{\B_\rho(x_0)}\ti{w}\,h\ dx^3=\int_{B^3_\rho(x_0)}\ti{w}\,[-\Delta\varphi+4^{-1}\,\ep^{-2}\varphi]\ dx^3\\[5mm]
\ds=\int_{\B_\rho(x_0)}[-\Delta \ti{w}+4^{-1}\,\ep^{-2}\ti{w}]\ \varphi \ dx^3=\int_{B^3_\rho(x_0)}f+\mbox{div}X \ \varphi \ dx^3\\[5mm]
\ds\le \|f\|_{L^1(B_\rho(x_0))}\ \|{\varphi}\|_{L^\infty\lf(\B_\rho(x_0)\rg)}+\|X\|_{L^{(3/2,\infty)}(\B_{\rho}(x_0))}\ \|\nabla\varphi\|_{L^{(3,1)}(\B_{\rho}(x_0))}\\[5mm]
\ds\le  \,C\, \ep\, \lf( \|f\|_{L^1(\B_\rho(x_0))}+\|X\|_{L^{(3/2,\infty)}(\B_{\rho}(x_0))} \rg)\ \|h\|_{L^{(3,1)}\lf(\B_\rho(x_0)\rg)}\ .
\end{array}
\]
Thus
\be
\label{w32}
\|{w}\|_{L^{(3/2,\infty)}(\B_\rho(x_0))}\le C\,\sqrt{\delta}
\ee
We have using the fact that $2\,|u|^4>1$ in $\B_\rho(x_0)$ we have
\[
\begin{array}{l}
\ds -\mbox{div}\lf(\frac{1}{|u|^{4}}\,\nabla \lf(|[u,F_A]|-w\rg)\rg)+\frac{1}{4\,\ep^2\,|u|^4}(|[u,F_A]|-w)\\[5mm]
\ds\quad\le -\frac{1}{2\ep^2}\,|[u,F_A]|+\frac{1}{4\,\ep^2|u|^4}|[u,F_A]|\le 0\ .
\end{array}
\]
This gives
\[
-\Delta\lf(\frac{|[u,F_A]|-w}{|u|^4}\rg)+\frac{1}{4\,\ep^2}\frac{(|[u,F_A]|-w)}{|u|^4}\le 4\,\mbox{div}\lf( \frac{\nabla |u|}{|u|}\, \frac{|[u,F_A]|-w}{|u|^4}  \rg)\ .
\]
Let $\sigma$ be the solution of
\[
\lf\{
\begin{array}{l}
\ds-\Delta \sigma+\frac{\sigma}{4\,\ep^2}=4\,\mbox{div}\lf( \frac{\nabla |u|}{|u|}\, \frac{|[u,F_A]|-w}{|u|^4}  \rg) \qquad\mbox{ in }\B_\rho(x_0)\ .\\[5mm]
\ds \sigma=0\qquad\mbox{ on }\p \B_\rho(x_0)\ .
\end{array}
\rg.
\]
Classical elliptic estimates are giving
\[
\begin{array}{l}
\ds\|\sigma\|_{L^{(3/2,\infty)}(\B_\rho(x_0))}\le C\ \lf\| \frac{\nabla |u|}{|u|}\, \frac{|[u,F_A]|-w}{|u|^4} \rg\|_{L^1(\B_\rho(x_0))}\le C \int_{\B_r(x_0)} |\nabla |u||\,\lf(|F_A|+ |w|\rg)\ dx^3\\[5mm]
\ds \le C\ \|\sqrt{\ep}^{-1}\,\nabla_Au\|_{L^2(\B_r(x_0))}\ \|\sqrt{\ep}\,\nabla_Au\|_{L^2(\B_r(x_0))}+C\ \|\nabla_Au\|_{L^{(3,1)}(\B_r(x_0))}\  \|w\|_{L^{(3/2,\infty)}(\B_r(x_0))}
\end{array}
\]
Observe that using $|\nabla_A u|(x)\le C\,\ep^{-1}$ for all $x$ in $\B_r(x_0)$ we have
\be
\label{L31cov}
\begin{array}{l}
\ds\|\nabla_Au\|_{L^{(3,1)}(\B_r(x_0))}\le \int_0^{C\,\ep^{-1}}\lf| \lf\{ x\in \B_r(x_0)\ :\ |\nabla_A u|(x)>t  \rg\}  \rg|^{1/3}\ dt\\[5mm]
\ds\quad= \int_0^{C\,\ep^{-1}}\lf| \lf\{ x\in \B_r(x_0)\ :\ \frac{|\nabla_A u|}{\sqrt{\ep}}(x)>\frac{t}{\sqrt{\ep}} \rg\}  \rg|^{1/3}\ dt\\[5mm]
\ds\quad=\sqrt{\ep}\int_0^{C\,\ep^{-3/2}}\lf| \lf\{ x\in \B_r(x_0)\ :\ \frac{|\nabla_A u|}{\sqrt{\ep}}(x)>s \rg\}  \rg|^{1/3}\ ds\\[5mm]
\ds\quad\le \sqrt{\ep}\, \lf(\int_0^{C\,\ep^{-3/2}} \frac{ds}{\sqrt{s}}\rg)^{2/3}\lf(\int_0^{+\infty}s\,\lf| \lf\{ x\in \B_r(x_0)\ :\ \frac{|\nabla_A u|}{\sqrt{\ep}}(x)>s \rg\}  \rg|\ ds\rg)^{1/3}\\[5mm]
\ds\quad\le  \|\sqrt{\ep}^{-1}\,|\nabla_Au|\|_{L^{2}(\B_r(x_0))}^{2/3}\le C\ \delta^{1/3}\ .
\end{array}
\ee
Hence
\[
\|\sigma\|_{L^{(3/2,\infty)}(\B_\rho(x_0))}\le C\, \delta^{5/6}\ .
\]
Let $\chi_\rho\in C^\infty(B_\rho(x_0), {\R}_+)$ such that $\chi_\rho\equiv 1$ on $B_{3\rho/4}(x_0)$ and $|\nabla^l \chi_\rho|\le C_l\,\rho^{-l}$. Denote
\[
v= \frac{|[u,F_A]|-w}{|u|^4}-\sigma \ .
\]
Let $x\in B_{\rho/2}(x_0)$ at which then $\chi_\rho(x)=1$, we have 
\[
\begin{array}{l}
\ds0\ge \frac{1}{4\pi}\int_{\B_\rho(x_0)}\chi_\rho(y)\, |x-y|^{-1}\,\exp(-|x-y|/2\,\ep)\, \lf(-\Delta_y v+\frac{v}{4\ep^2}\rg)(y)\ dy^3\\[5mm]
\ds=\frac{1}{4\pi}\int_{\B_\rho(x_0)}\nabla_y \chi_\rho(y)\,\, |x-y|^{-1}\,\exp(-|x-y|/2\ep)\, \nabla_y v\ dy^3\\[5mm]
\ds+\frac{1}{4\pi}\int_{\B_\rho(x_0)} \chi_\rho(y)\,\, \nabla_y(|x-y|^{-1}\,\exp(-|x-y|/2\ep))\, \nabla_y v\ dy^3\\[5mm]
\ds+\frac{1}{4\pi}\int_{\B_\rho(x_0)}\chi_\rho(y)\, |x-y|^{-1}\,\exp(-|x-y|/2\ep)\,\frac{v}{4\ep^2}(y)\ dy^3\\[5mm]
\ds=\frac{1}{2\pi}\int_{\B_\rho(x_0)} \nabla_y\chi_\rho(y)\,\, \nabla_y(|x-y|^{-1}\,\exp(-|x-y|/2\ep))\, v(y)\ dy^3\\[5mm]
\ds-\frac{1}{4\pi}\int_{\B_\rho(x_0)}\Delta_y \chi_\rho(y)\,\, |x-y|^{-1}\,\exp(-|x-y|/2\ep)\, v(y)\ dy^3
+v(x)
\end{array}
\]
We have respectively for any $x\in B_{\rho/2}(x_0)$ 
\[
\begin{array}{l}
\ds \lf|\frac{1}{2\pi}\int_{\B_\rho(x_0)} \nabla_y\chi_\rho(y)\,\, \nabla_y(|x-y|^{-1}\,\exp(-|x-y|/2\ep))\, v(y)\ dy^3\rg|\\[5mm]
\ds\quad\le C\,(\rho^{-3}+\rho^{-2}\ep^{-1})\ \exp(-\rho/8\,\ep)\ \|v\|_{L^1(\B_\rho(x_0))}\ ,
\end{array}
\]
and
\[
\begin{array}{l}
\ds \lf|\frac{1}{4\pi}\int_{\B_\rho(x_0)}\Delta_y \chi_\rho(y)\,\, |x-y|^{-1}\,\exp(-|x-y|/2\ep)\, v(y)\ dy^3\rg|\\[5mm]
\ds\quad\le C\,\rho^{-3}\ \exp(-\rho/8\,\ep)\ \|v\|_{L^1(\B_\rho(x_0))}\ .
\end{array}
\]
We have
\[
\|v\|_{L^1(\B_\rho(x_0))}\le C\,\sqrt{\delta}+C\,\sqrt{\ep}^{-1}\,\|\sqrt{\ep}\,|F_A|\|_{L^2(\B_\rho(x_0))}\le C\,\sqrt{\ep}^{-1}\,\rho^{3/2}\,\sqrt{\delta}\ .
\]Thus
\[
\forall\ x\in\B_{\rho/2}(x_0)\qquad |v(x)|\le C\, \rho^{-1/2}\ep^{-3/2}\,\sqrt{\delta}\ \exp(-\rho/8\,\ep)\ 
\]
 Hence assuming $r>C_\delta\, \ep\,\log1/\ep$ for some $C_\delta>0$ depending on $\delta$ we deduce
 \be
 \label{L32infcurtra}
 \|[F_A,u]\|_{L^{(3/2,\infty)}(\B_{r/32}(x_0))}\le C\, \sqrt{\delta}\ .
 \ee 
 Combining this fact with (\ref{L32infcurlon}) we obtain
 \be
 \label{L32infcur}
 \|F_A\|_{L^{(3/2,\infty)}(\B_{r/32}(x_0))}\le C\, \sqrt{\delta}\ .
 \ee
 Recall from (\ref{Bochcovder})
 \[
\begin{array}{l}
\ds\Delta|\nabla_Au|\ge \frac{1}{2\,\ep^2}\, |\nabla_Au|-2\,|\nabla_A u|\, |F_A|-\frac{1}{\ep}\,|A|\ .
\end{array}
\]
Denote $\rho:= r/32$. We have seen that $\|F_A\|_{L^{(3/2,\infty)}(B_\rho(x_0))}$ small enough  there exists a unique $w$ solving
\be
\label{eq-w}
\lf\{
\begin{array}{l}
\ds -\Delta w+\frac{1}{2\,\ep^2}\, w-2\,w\,|F_A|=\frac{1}{\ep}\,|A|\qquad\mbox{ in }\B_\rho(x_0)\\[5mm]
\ds w=0 \qquad\mbox{ on }\p \B_\rho(x_0)
\end{array}
\rg.
\ee
it is non negative and it satisfies
\be
\label{estw}
\|w\|_{L^{(6,2)}\B_\rho(x_0)}^2+\int_{\B_\rho(x_0)}|\nabla w|^2+\frac{1}{4\ep^2}\, w^2\ dx^3\le C\ \int_{\B_\rho(x_0)}|A|^2\ dx^3\ .
\ee
We have also
\[
-\Delta(|\nabla_Au|-w)+\lf(\frac{1}{2\,\ep^2}-2\,|F_A|\rg)\,(|\nabla_Au|-w)\le 0
\]
Denote $\zeta:=(|\nabla_Au|-w)$. Let $\chi_\rho\in C^\infty(B_\rho(x_0), {\R}_+)$ such that $\chi_\rho\equiv 1$ on $B_{3\rho/4}(x_0)$ and $|\nabla^l \chi_\rho|\le C_l\,\rho^{-l}$. We have
\be
\label{chirhov}
-\Delta(\chi_\rho\,\zeta)+\lf(\frac{1}{2\,\ep^2}-2\,|F_A|\rg)\,(\chi_\rho\,\zeta)\le-\Delta\chi_\rho\ \zeta-2\,\nabla\chi_\rho\,\nabla \zeta=\Delta\chi_\rho\ \zeta-2\,\mbox{div}\lf( \nabla\chi_\rho\, \zeta\rg)\ .
\ee
Let $x_1\in \B_{\rho/2}(x_0)$. We consider Let $G_{\ep,x_1}$ be the solution to
\be
\label{Gep}
-\Delta G_{\ep,x_1}+\lf(\frac{1}{2\,\ep^2}-2\,|F_A|\,{\bf 1}_{\B_r(x_0)}\rg)\,G_{\ep,x_1}=\delta_{x_1}\ .
\ee
The existence and uniqueness of $G_\ep$ in $L^2({\R}^3)$ is the conclusion of the following a-priori bound and and a fixed point argument : we write
\[
G_{\ep,x_1}:=\frac{1}{4\pi}\frac{e^{-|x|/\ep}}{|x|}\star \lf(\delta_{x_1}+2\,|F_A|\,{\bf 1}_{\B_r(x_0)}\, G_{\ep,x_1}\rg)
\]
and we deduce
\[
\begin{array}{l}
\ds\|G_{\ep,x_1}\|_{L^2({\R}^3)}\le \frac{1}{4\pi}\ \lf\|  \frac{e^{-\frac{|x|}{\ep}}}{|x|} \rg\|_{L^2({\R}^3)}\, \lf(1+2\,\|F_A\|_{L^2(\B_r(x_0))}\,\|G_{\ep,x_1}\|_{L^2({\R}^3)}\rg)\\[5mm]
\ds\qquad\le C\, \sqrt{\ep}\, \lf(1+2\,\|\sqrt{\ep}\, F_A\|_{L^2(\B_r(x_0))}\,\sqrt{\ep}^{-1}\,\|G_{\ep,x_1}\|_{L^2({\R}^3)}\rg)\\[5mm]
\ds\qquad\le C\, \sqrt{\ep}\, \lf(1+C\,\sqrt{\delta}\,\sqrt{\ep}^{-1}\,\|G_{\ep,x_1}\|_{L^2({\R}^3)}\rg)
\end{array}
\]
Hence for $\delta$ small enough there is a unique solution to this equation in $L^2({\R}^3)$ and there holds
\[
\lf\|\frac{G_{\ep,x_1}}{\sqrt{\ep}}  \rg\|_{L^2({\R}^3)}\le C
\]
We have also by Young inequality
\[
\|\nabla G_{\ep,x_1}\|_{L^{(3/2,\infty)}(\R^3)}\le \frac{1}{4\pi}\lf\|\nabla\lf(\frac{e^{-|x|/\ep}}{|x|}\rg)\rg\|_{L^{(3/2,\infty)}(\R^3)}\  \lf(1+2\,\lf\||F_A|\,{\bf 1}_{\B_r(x_0)}\, G_{\ep,x_1}\rg\|_{L^1(\B_r(x_0))}\rg)\le C\,(1+\sqrt{\delta})
\]
Observe also that from (\ref{Gep}) we have $G_{\ep,x_1}\ge 0$. We multiply (\ref{chirhov}) by $G_{\ep,x_1}$ and after integration by parts we obtain
\[
(|\nabla_Au|-w)(x_1)=\zeta(x_1)\le C\,\rho^{-2}\,\|G_{\ep,x_1}\, \zeta\|_{L^1(\B_\rho(x_0))}+C\,\rho^{-1}\,\|\nabla G_{\ep,x_1}\, \zeta\|_{L^1(\B_\rho(x_0))}\ .
\]
Using (\ref{L31cov}) and (\ref{estw}) we deduce
\[
\forall\, x\in B_{\rho/2}(x_0)\ \qquad |\nabla_A u|(x)\le w(x)+C\,\sqrt{\ep}\ \delta^{1/3}\,\rho^{-3/2}+C\,\rho^{-1}\,\delta^{1/3}\ .
\]
Combining this inequality with (\ref{estw}) one finally gets
\[
\|\nabla_Au\|_{L^6(\B_{r/64}(x_0))}\le C\, \delta^{1/3}
\]
Bootstrapping this information in the whole ``circuit'' starting from (\ref{anti}) one gets ultimately $L^\infty$ bounds on $A$, $\nabla_Au$ and $F_A$ on a smaller ball $\B_{tr}(x_0)$ where $0<t<1$ is universal.
Regarding the derivatives  we observe firs that the equation on $|u|$ 
\[
-2^{-1}\,\Delta(1-|u|^2)+\frac{|u|^2}{\ep^2}(1-|u|^2)=|\nabla_Au|^2
\]
is implying a uniform bound independent of $\ep$ of $\|1-|u|^2\|_{W^{2,p}(\B_{tr/2}(x_0))}$ for any $p<+\infty$. Then the first equation in the E.L. system (\ref{CYMHS-co}) is giving
\[
\sum_{l=1}^3(\nabla_A)_l[u,(\nabla_A)_lu]=0
\]
Recall also that
\[
d_A(d_Au)=[F_A,u]
\]
Thus
\[
d_A[u,d_Au]=[u,[F_A,u]]+[d_Au\wedge d_Au]
\]
This implies the system
\be
\label{sysdAu}
\lf\{
\begin{array}{l}
\ds d^\ast [u,d_Au]=-[A\res d_Au]\\[5mm]
\ds d[u,d_Au]=-[A\,\wedge\, d_Au]+[u,[F_A,u]]+[d_Au\wedge d_Au]
\end{array}
\rg.
\ee
We deduce from this system and the $L^\infty$ bounds on $A$, $F_A$, and $d_Au$ a uniform bound independent of $\ep$ of $\|[u,d_Au]\|_{W^{1,p}(\B_{tr/2}(x_0))}$ for any $p<+\infty$.  Injecting this bounds in (\ref{anti}) is giving that $\nabla^2 \hat{u}$ is uniformly bounded in $W^{2,p}(\B_{tr/2}(x_0))$ for any $p<+\infty$. This concludes the proof of Lemma~\ref{lm-deltareg} is proved.\hfill $\Box$
\subsection{Concentration Compactness for $CYMH_\ep^{\la,\mu}$ minimizers in the limit $\ep\rightarrow 0$.}
This subsection is devoted to the proof of Theorem~\ref{th-cons-comp}.

\medskip

 \noindent{\bf Proof of Theorem~\ref{th-cons-comp}.} To simplify notations we take $\la=1$ and $\mu=1$ in the proof. Let  $\phi$ in $C^2(\p\B_1(0),{\S}^2)$ with zero degree. Denote by ${\mathcal E}_\phi$ the subspace of ${\mathcal E}$ given by elements (\ref{matE}) such that $u$ equals to $\phi$ at the boundary.  Let $(u^k,A^k)$ be a sequence of minimizers of $CYMH_{\ep_k}^{1,1}$ in ${\mathcal E}_\phi$ for some sequence $\ep_k\rightarrow 0$ .The existence of $(u_k,A_k)$ is given by  Proposition~\ref{pr-II.1}.   From Lemma~\ref{lm-bd} we have the existence of $M>0$ such that
 \[
 \sup_{k\in{\N}} \ CYMH_{\ep_k}^{1,1}(u^k,A^k)\le M\ .
 \]
 From theorem~\ref{th-reg} we have that $(u^k,A^k)$ is smooth inside $\B_1(0)$.

 \medskip
 
Let $r_1<1$ and $\delta_0>0$ given by Lemma~\ref{lm-sca-wh} and Lemma~\ref{lm-deltareg}. Because of Lemma~\ref{lm-zero-set}, modulo extraction of a subsequence that we keep denoting
 $(u^k,A^k)$, there exists a fixed number $Q$ a fixed family of points $(a_1\cdots a_Q)\in \ov{\B}_{r_1}(0)$ and a finite sequence of $Q$ points $x_1^k\cdots x_Q^k$ such that for any $i=1\cdots Q$ we have 
 \[
 \forall \,i=1\cdots Q\qquad x_i^k\rightarrow a_i\in \B_{r_1}(0)\qquad.
 \]
 Moreover
 \[
 \forall x\,\in  \B_{r_1}(0)\setminus \bigcup_{i=1}^Q \B_{\ep_k}(x_i^k)\qquad   \qquad   |u^k|(x)\ge 1-\delta_0\ .
 \]
 We introduce the following Radon measure
 \[
 \mu_k:=\lf(|{A}^k|^2+\frac{|d_{{A}^k} {u}^k|^2}{\ep}+\ep\,|F_{{A^k}}|^2\ +\frac{1}{2\,\ep^3}(1-|{u}^k|^2)^2\rg)\ dx^3
 \]
 We extract a subsequence such that
 \[
 \mu_k\rightharpoonup \mu_\infty\qquad\mbox{ in Radon measures.} 
 \]
 Let $\B_r(x_0)\subset \B_{r_1}(0)$ such that
 \[
\frac{1}{2\,r}\mu_\infty(\B_r(x_0))<\frac{\delta}{2}\ ,
\]
where $\delta_0>\delta>0$ is given by Lemma~\ref{lm-sca-wh}. The we have  for any $ x\in \B_{r/2}(x_0)$ and $k$ large enough (independent of $x\in  \B_{r/2}(x_0)$)
\[
\begin{array}{l}
\ds\frac{1}{r}\int_{\B_{r/2}(x)} |{A^k}|^2+\frac{|d_{{A}^k} {u}^k|^2}{\ep}+\ep\,|F_{{A}^k}|^2\ dx^3+\frac{1}{2\,\ep_k^3}(1-|{u}^k|^2)^2\ dx^3<{\delta}\\[5mm]
\ds\le 2\ \frac{1}{2\,r}\int_{\B_r(x_0)} |{A}^k|^2+\frac{|d_{{A}^k} {u}^k|^2}{\ep}+\ep\,|F_{{A}^k}|^2\ dx^3+\frac{1}{2\,\ep_k^3}(1-|{u}^k|^2)^2\ dx^3<\delta\ .
\end{array}
\]
 Applying Lemma~\ref{lm-sca-wh} we then have
 \be
 \label{mordel}
 \sup_{x\in \B_{r/2}(x_0)}\ \sup_{\rho<r/2}\frac{1}{2\,\rho}\int_{\B_\rho(x)} |{A}^k|^2+\frac{|d_{{A}^k} {u}^k|^2}{\ep}+\ep\,|F_{{A}^k}|^2\ dx^3+\frac{1}{2\,\ep_k^3}(1-|{u}^k|^2)^2\ dx^3<C\,\delta
 \ee
for some universal $C>0$. Applying Lemma~\ref{lm-deltareg} we have then modulo extraction of a subsequence the fact that
\[
A^k\rightarrow A^\infty\quad,\quad \nabla u^k\rightarrow \nabla u^\infty \quad\mbox{ and }\nabla_{A^k}u^k\rightarrow 0\qquad\mbox{ in }L^\infty(\B_{r/2}(x_0))
\]
for some limiting connection moreover $\nabla_{A^k}$ being uniformly bounded in $L^2(\B_{r/2}(x_0))$ we have
\[
|(A^k)^T|^2+\ep^k\,|F_{A^k}|^2\ dx^3 \rightharpoonup 4^{-1}\,|\nabla u^\infty|^2\ dx^3\qquad\mbox{ in }{\mathcal M}(\B_{r/2}(x_0))\ .
\]
From (\ref{Bochcovder}) we then obtain that
\[
\begin{array}{l}
\ds -\Delta |\nabla_A^k u^k|^2+\frac{1}{2\ep_k^2}\,|\nabla_{A^k} u^k|^2\le f^k+\ep_k^{-1}g^k
\end{array}
\]
where both $f^k$ and $g^k$ are uniformly bounded in $L^p(\B_{r/2}(x_0))$ for any $p<+\infty$. This implies using the arguments from the previous subsection that 
\[
\forall x\in B_{r/4}(x_0)\quad |\nabla_{A^k}u^k|^2(x)\le w^k(x)+ C_r\ \ep_k^{-3/2}\ e^{-r/8\ep_k}
\]
where
\[
\ds -\Delta w^k+\frac{1}{2\ep_k^2}\,w^k= (f^k+\ep_k^{-1}g^k)\ {\mathbf 1}_{\B_{r/2}(x_0)}\quad\mbox{ in }{\R}^3
\]
Hence
\[
w^k:=\frac{1}{4\pi}\frac{e^{-|x|/\ep_k}}{|x|}\star (f+\ep_k^{-1}g)\ {\mathbf 1}_{\B_{r/2}(x_0)}
\]
and we have for any $p<+\infty$
\[
\|w^k\|_{L^p({\R}^3)}\le \frac{1}{4\pi} \lf\|\frac{e^{-|x|/\ep_k}}{|x|}\rg\|_{L^1({\R}^3)}\ \lf(\|f^k\|_{L^p(\B_{r/2}(x_0))}+\ep_k^{-1}\,\|g^k\|_{L^p(\B_{r/2}(x_0))}\rg)=O(\ep_k)
\]
We deduce for any $p<+\infty$
\[
\int_{\B_{r/4}(x_0)}|\nabla_{A^k}u^k|^{2p}(x)\ dx^3\le O(\ep_k^p)
\]
Recall that 
\[
-\frac{1}{2}\,\Delta(1-|u^k|^2)+\frac{|u^k|^2}{\ep_k^2}(1-|u^k|^2)=|\nabla_{A^k}u^k|^2
\]
Multiplying the equation by $\ep_k^{-1}\,\chi_r(x) (1-|u_k|)$ where $\chi_r\in C^\infty_0(\B_{r/4}(x_0),{\R}_+)$ and $\chi_r\equiv 1$ on ${\B}_{r/8}(x_0)$. integrating by parts is giving
\[
\begin{array}{l}
\ds\int_{B_r(x_0)}\ \chi_r\ \frac{|\nabla|u^k|^2|^2}{\ep_k}+\chi_r\ \frac{|u^k|^2}{\ep_k^3}(1-|u^k|^2)^2\ dx^3\le \int_{B_r(x_0)}\ \chi_r\ \frac{|\nabla_{A^k}u^k|^2}{\ep_k}\ (1-|u^k|^2)\\[5mm]
\ds+\frac{1}{4}\int_{B_r(x_0)}\ \nabla\chi_r\cdot \frac{1}{\ep_k}\,\nabla(1-|u^k|^2)^2\ dx^3\\[5mm]
\ds\le \lf( \int_{B_r(x_0)}\ \chi_r\ \frac{|\nabla_{A^k}u^k|^4}{\ep^2_k} \ dx^3\rg)^{1/2}\ \lf( \int_{B_r(x_0)}\ \chi_r\ (1-|u|^2)^2\ dx^3 \rg)^{1/2}\\[5mm]
\ds+\frac{C}{r}\ \lf( \int_{B_r(x_0)}\ \ \frac{|\nabla|u^k|^2|^2}{\ep_k}  \rg)^{1/2}\ \lf( \int_{B_r(x_0)}\ \ \frac{(1-|u|^2)^2}{\ep_k}\ dx^3 \rg)^{1/2}=o_k(1)
\end{array}
\]
Hence
\[
|\nabla|u^k||^2+\frac{1}{2\ep_k^3}(1-|u^k|^2)^2\ dx^3\rightharpoonup 0\qquad\mbox{ in }{\mathcal M}(\B_{r/8}(x_0))
\]
Multiplying (\ref{longA}) by $\chi_\rho\ A^k\cdot \hat{u}^k$ and integrating by parts is giving 
\[
\begin{array}{l}
\ds\int_{B_r(x_0)}\ \chi_r\ |\nabla(\hat{u}^k\cdot A^k)|^2+\frac{|\hat{u}^k\cdot A^k|^2}{\ep}\le \int_{B_r(x_0)}\ \nabla \chi_r\ \hat{u}^k\cdot A^k\cdot\nabla(\hat{u}^k\cdot A^k) \ dx^3\\[5mm]
\ds+\int_{B_r(x_0)}\ \nabla\chi_r\,\nabla_{A^k}u^k\cdot A^k-\chi_r\, \nabla_{A^k}u^k\res F_{A^k}+\sum_{m=1}^3\p_{x_m}\chi_r\,(\nabla_{A^k})_mu^k\cdot A^k\ dx^3\\[5mm]
\ds-\int_{B_r(x_0)}\ \sum_{m=1}^3\p_{x_m}\chi_r\,(\nabla_{A^k})u^k\cdot A^k_m-\int_{B_r(x_0)}\ \sum_{m=1}^3\p_{x_m}\chi_r\,\hat{u}\cdot [A^k,A^k_m] =O(1)
\end{array}
\]
Hence
\[
\int_{\B_{r/8}(x_0)}\frac{|A^k\cdot\hat{u}^k|^2}{\ep}=O(1)
\]
This implies
\[
|A^k\cdot\hat{u}^k|^2\ dx^3\ \rightharpoonup\ 0\qquad\mbox{ in }{\mathcal M}(\B_{r/8}(x_0))
\]
and we have also  the ``transversal condition''
\[
A^\infty\cdot u^\infty=0\ .
\]
We write
\[
F_{A^k}=d((A^k)^T)+(A^k)^T\wedge (A^k)^T+d(A^k\cdot\hat{u}^k\,\hat{u}^k)+2\,A^k\cdot\hat{u}^k\wedge [u^k,(A^k)^T]
\]
Hence
\[
\begin{array}{l}
[\hat{u}^k,F_{A^k}]=[\hat{u}^k,d((A^k)^T)]+A^k\cdot\hat{u}^k\wedge [\hat{u}^k, d\hat{u}^k]-8\,A^k\cdot\hat{u}^k\wedge (A^k)^T\\[5mm]
=d\lf([\hat{u}^k,(A^k)^T]\rg)-[d\hat{u}^k\wedge (A^k)^T]+A^k\cdot\hat{u}^k\wedge [\hat{u}^k, d\hat{u}^k]-8\,A^k\cdot\hat{u}^k\wedge (A^k)^T\ .
\end{array}
\]
We have respectively
\[
[\hat{u}^k,(A^k)^T]=-d_{A^k}\hat{u}^k+d\hat{u}^k\qquad\mbox{ and }\qquad-[d\hat{u}^k\wedge (A^k)^T]=-[d_{A^k}\hat{u}^k\wedge (A^k)^T]+[[u^k,(A^k)^T]\wedge (A^k)^T]
\]
and, using Jacobi
\[
\begin{array}{l}
\ds[[u^k,(A^k)^T]\wedge (A^k)^T]=\sum_{lm=1}^3[[u^k,(A^k)_l^T], (A^k)_m^T]-[[u^k,(A^k)_m^T], (A^k)_l^T]\ dx_l\wedge dx_m\\[5mm]
\ds\quad=\sum_{lm=1}^3[[(A^k)_m^T, (A^k)_l^T],u^k]\ dx_l\wedge dx_m=0
\end{array}
\]
where we used that $[(A^k)_m^T, (A^k)_l^T]$ is parallel to $u^k$. Thus
\be
\label{transcur}
[\hat{u}^k,F_{A^k}]=-d\lf(  d_{A^k}\hat{u}^k  \rg)-[d_{A^k}\hat{u}^k\wedge (A^k)^T]+A^k\cdot\hat{u}^k\wedge [\hat{u}^k, d\hat{u}^k]-8\,A^k\cdot\hat{u}^k\wedge (A^k)^T\ .
\ee
This implies in particular
\be
\label{bdtranscur}
\limsup_{k\rightarrow +\infty}\lf\|\frac{[\hat{u}^k,F_{A^k}]}{\sqrt{\ep}}\rg\|_{H^{-1}+L^2(\B_{r/8}(x_0))}<+\infty
\ee
Modulo extraction of a subsequence there exists $F$ and $\al$ such that
\be
\label{weakcurv}
\frac{[\hat{u}^k,F_{A^k}]}{\sqrt{\ep}}\rightharpoonup \om\qquad\mbox{ in }\qquad H^{-1}+L^2(\B_{r/8}(x_0))
\ee
and
\[
\frac{[\hat{u}^k,d_{A^k}\hat{u}^k]}{\sqrt{\ep}}\rightharpoonup \al\quad\mbox{ weakly in }L^2(\B_{r/8}(x_0))
\]
Taking into account that
\[
d|u^k|\ \longrightarrow\ 0 \qquad\mbox{ strongly in }L^2(\B_{r/8}(x_0))\ ,
\]
from (\ref{sysdAu}) there holds
\be
\label{sysdAulim}
\lf\{
\begin{array}{l}
\ds d^\ast_{A^\infty} \al=0\\[5mm]
\ds d_{A^\infty}\al=\om\\[5mm]
\ds \al\cdot u^\infty=0
\end{array}
\rg.
\ee
From (\ref{CYMHS-co}) we have
\[
-[\hat{u}^k,A^k]=\ep_k\,[\hat{u}^k,d^\ast_{A^k}F_{A_k}]+\frac{1}{\ep_k}\,[\hat{u}^k,[u^k,d_{A^k}u^k]]=\ep_k\,d^\ast_{A^k}[\hat{u}^k,F_{A_k}]-\ep_k\,[d_{A^k}\hat{u}^k\res F_{A_k}]+\frac{1}{\ep_k}\,[\hat{u}^k,[u^k,d_{A^k}u^k]]
\]
Hence
\[
\begin{array}{l}
\ds\frac{1}{|u^k|^2}\lf[\hat{u}^k,\frac{[\hat{u}^k,d_{A^k}\hat{u}^k]}{\sqrt{\ep_k}}\rg]=\lf[\hat{u}^k,\frac{[u^k,d_{A^k}u^k]}{\sqrt{\ep_k}}\rg]\\[5mm]
\ds=-\sqrt{\ep_k}\,[\hat{u}^k,A^k]-\ep^{3/2}_k\,d^\ast[\hat{u}^k,F_{A_k}]- \ep^{3/2}_k\,[A^k\res[\hat{u}^k,F_{A_k}]]+\ep^{3/2}_k\,[d_{A^k}\hat{u}^k\res F_{A_k}]\rightarrow 0\quad\mbox{in }L^1+H^{-1}(\B_r(x_0))
\end{array}
\]
We deduce
\[
[u^\infty,\al]=0
\]
Since $u^\infty\cdot\al =0$ we deduce $\al=0$ and $\om=0$. 
From (\ref{sysdAu}) we have then
\be
\label{comp-dAu}
\lf\{
\begin{array}{l}
\ds d^\ast\lf(\frac{[\hat{u}^k,d_{A^k}\hat{u}^k]}{\sqrt{\ep}}\rg)\rightharpoonup 0\qquad\mbox{ weakly in }L^2(\B_{r/8}(x_0))\\[5mm]
\ds \frac{[u^k,[\hat{u}^k,F_{A^k}]]}{\sqrt{\ep}}-d\lf(\frac{[u^k,d_{A^k}\hat{u}^k]}{\sqrt{\ep}}\rg)\rightharpoonup 0\qquad\mbox{ weakly in }L^2(\B_{r/8}(x_0))
\end{array}
\rg.
\ee
Multiplying (\ref{bwt}) by $\ep\,\chi_r\,|[u^k,F_{A^k}]|$ (where we recall that  $\chi_r\in C^\infty_0(\B_{r/4}(x_0),{\R}_+)$ and $\chi_r\equiv 1$ on ${\B}_{r/8}(x_0)$) and integrating by parts is giving
\[
\begin{array}{l}
\ds\int_{B_r(x_0)}\ \chi_r\ \frac{\ep_k}{|u^k|^4}\lf|  \nabla{|[u^k,F_{A^k}]|}\rg|^2+\frac{1}{2\ep_k}|[u^k,F_{A^k}]|^2\le \int_{B_r(x_0)}\ \frac{\ep_k}{2|u^k|^4}\,\nabla\chi_r\cdot\nabla |[u^k,F_{A^k}]|^2\ dx^3\\[5mm]
\ds\qquad+\int_{B_r(x_0)}\ \chi_r \ |[u^k,F_{A^k}]|\ \lf(8\, |u^k|^{-6}\, \ep_k\,|\nabla_{A^k}\,u^k|^2\,|F_{A^k}|+4\, |u^k|^{-4}\,|\nabla_{A^k}u^k|\,\ep_k\,|\nabla(F_{A^k}\cdot u^k)|\rg)\ dx^3\\[5mm]
\ds\qquad+\int_{B_r(x_0)}\ \chi_r \ |[u^k,F_{A^k}]|\, \lf(4\,\ep_k^{-1}|u^k|^{-4}\,|\nabla_{A^k} u^k|^2+2\,|u^k|^{-4}\,|A^k|^2+4\,\ep_k\,|u^k|^{-4}\,|F_{A^k}|^2\rg)\ dx^3=O(1)

\end{array}
\]
where we have used (\ref{longcurv}) and the fact that
\[
\int_{B_r(x_0)}\ \frac{\ep_k}{2|u^k|^4}\,\nabla\chi_r\cdot\nabla |[u^k,F_{A^k}]|^2\ dx^3=-\int_{B_r(x_0)}\ \mbox{div}\lf(\frac{\ep_k}{2|u^k|^4}\,\nabla\chi_r\rg)\ |[u^k,F_{A^k}]|^2\ dx^3\ .
\]
Thus we deduce that
\[
\limsup_{k\rightarrow +\infty}\lf\| \frac{[u^k,[\hat{u}^k,F_{A^k}]]}{\sqrt{\ep}}\rg\|_{L^2(\B_{r/8}(x_0))}<+\infty
\]
This implies that the weak convergence in (\ref{weakcurv}) holds in $L^2(\B_{r/8}(x_0))$ and we have from (\ref{comp-dAu}) and the fact that $\om=0$
\[
\lf\{
\begin{array}{l}
\ds d^\ast\lf(\frac{[\hat{u}^k,d_{A^k}\hat{u}^k]}{\sqrt{\ep}}\rg)\rightharpoonup 0\qquad\mbox{ weakly in }L^2(\B_{r/8}(x_0))\\[5mm]
\ds d\lf(\frac{[u^k,d_{A^k}\hat{u}^k]}{\sqrt{\ep}}\rg)\rightharpoonup 0\qquad\mbox{ weakly in }L^2(\B_{r/8}(x_0))
\end{array}
\rg.
\]
Using Rellich Kondrachov we deduce that $\frac{[\hat{u}^k,d_{A^k}\hat{u}^k]}{\sqrt{\ep}}$ is pre-compact in $L^2(\B_{r/16}(x_0))$ and there holds
\[
\frac{|[u^k,d_{A^k}u^k]|^2}{\ep_k}\ dx^3\rightharpoonup 0\qquad\mbox{ in }{\mathcal M}(\B_{r/8}(x_0))
\]
Hence we have proved that 
\be
\label{muinf}
\frac{\mu_\infty(\B_{r}(x_0))}{r}<\delta\qquad\Longrightarrow\qquad \mu_\infty\res \B_{r/16}(x_0)=\frac{1}{4}\,|\nabla u^\infty|^2\ dx_3\res\B_{r/16}(x_0)
\ee
Let $\sigma>0$ and introduce
\[
K_\sigma:=\lf\{x\in \B_{r_1}(0)\setminus\bigcup_{i=1}^Q\B_\sigma(a_i)\ ;\ \liminf_{r\rightarrow 0} \frac{\mu_\infty(\B_{r}(x))}{r}\ge\delta\rg\}
\]
Let $x\in \B_{r_1}(0)\setminus\lf(K_\sigma\cup\bigcup_{i=1}^Q\B_\sigma(a_i)\rg)$ then there exists $\rho>0$ such that
\[
\frac{\mu_\infty(\B_{\rho}(x))}{\rho}<\delta\ .
\]
Using (\ref{muinf}) we have that
\[
\B_{\rho/16}(x_0)\cap K_\sigma=\emptyset\ .
\]
Hence $K_\sigma$ is relatively closed in $\B_{r_1}(0)\setminus\bigcup_{i=1}^Q\B_\sigma(a_i)$ moreover
\[
{\mathcal H}^1(K_\sigma)\le C\,\delta^{-1}\ \mu_\infty(\B_{r_1}(0))\ .
\]
where $C$ is universal. From the work of S.Eilenberg and O.G. Harrold we know that compact connected sets of finite one dimensional Hausdorff measure are rectifiable (\cite{EH} theorem 2 page 139). Let
\[
(K^i_\sigma)_{i\in I}
\] 
be the connected components of $K_\sigma$. Since $K_\sigma$ is compact each of its connected components are compact.  Let $d_i$
\[
d_i:=\dist\lf\{ K_\sigma^i,\bigcup_{j\ne i}K_\sigma^j\rg\}
\]
It is clear since $K_\sigma^i$ and $\bigcup_{j\ne i}K_\sigma^j$ are disjoint compact sets that $d_i>0$. Then, by definition the open sets
\[
U_i:=\{x\in \B_1(0)\quad;\quad \mbox{dist}(x, K_i)<d_i\}
\]
are disjoint to each other then there can only be a countable union of them and $I$ is at most countable. We deduce that each of the $K_\sigma$ are 1-rectifiable.
This concludes the proof of theorem~\ref{th-cons-comp}.\hfill $\Box$
\subsection{Overloading and Clearing out Monopoles.}
In this subsection we are going to increase the coupling constant $\mu$ in front of the $YMH$ part of the energy in order to penalise the possible formation of Monopoles.
We obtain a result according to which there is a threshold for $\la\mu$ independent of $\ep$ above which no Monopole can be formed in the asymptotic $\ep\rightarrow 0$ with $\la$ and $\mu$ fixed. In this regime the minimal $CYMH$ energy is converging towards the Brezis Coron Lieb relaxed energy. Below this threshold Monopole with non zero topological charge can be formed\footnote{The topological singularities for the Higgs field is characterised by the local integral of $F_A\cdot d_Au$ (see \cite{JT})} and at the limit the harmonic map extension of $\phi$ given by $u^\infty$ may have singularities which are not penalised and ''connected'' by $\mu_\infty$.  This is the object of Theorem~\ref{th-clearout} that we are now proving.

\medskip

\noindent{\bf Proof of theorem~\ref{th-clearout} }From lemma~\ref{lm-bd} we have that  that for any $\phi$ of degree 0 and  $\phi\in C^2(\p \B^3,\S^2)$,  there exists a constant $C_\phi>0$ independent of $\la\ge 1$ and $\mu\ge 1$ such that
\[
\limsup_{\ep\rightarrow 0}\min\lf\{CYMH_\ep^{\la,\mu}(u,A)\ ;\ (u,A)\in {\mathcal E}_\phi\rg\}<C_\phi
\]
From Lemma~\ref{lm-zero-set} we choosing $\la/\mu$ bounded, there exists $N_\phi$ such that for any $\la\mu>\La_\phi$ and $\ep$ small enough
\[
1\ge |u|\ge 1/2\qquad\mbox{in }\ov{\B_1(0)}
\]
moreover for any $1>t>1/2$
\[
\lf| \lf\{ x\in  \B_1(0) \ ;\ |u(x)|<t \rg\}\rg|\le  \frac{C_{M,\phi}}{\la\mu}\  \ep^3\ t^{-5}\ \lf(1+\sqrt{\frac{\la}{\mu}}\rg)^3\ .
\]
Denote on $\B_1(0)$ $\hat{u}:=u/|u|$. There holds
\[
\begin{array}{l}
\ds\frac{1}{2}\,\int_{\B_1(0)}|A|^2\ dx^3\ge \frac{1}{2}\,\int_{\B_1(0)}|A^T|^2\ dx^3= \frac{1}{2}\,\int_{\B_1(0)}\frac{1}{4}\lf|[\hat{u},A]\rg|^2\ dx^3\\[5mm]
\ds\ge \frac{1}{8}\,\int_{\B_1(0)}\frac{1}{4}\lf|d\hat{u}-(d_A\hat{u})^T\rg|^2\ dx^3
\end{array}
\]
Let $t<1$, we have 
\[
\begin{array}{l}
\ds\frac{1}{8}\,\int_{\B_1(0)}\lf|d\hat{u}-(d_A\hat{u})^T\rg|^2\ dx^3\ge\frac{1}{8}\,\int_{\B_1(0)}\lf|d\hat{u}\rg|^2-2\,d\hat{u}\cdot (d_A\hat{u})^T+|d_A\hat{u}|^2\\[5mm]
\ds\ge \frac{1}{8}\,\int_{\B_1(0)}(1-t^2)\,\lf|d\hat{u}\rg|^2- \frac{1}{t^2}\, |(d_A\hat{u})^T|^2\ dx^3\ge \frac{1}{8}\,\int_{\B_1(0)}(1-t^2)\,\lf|d\hat{u}\rg|^2- \frac{1}{|u|^2\,t^2}\, |d_A{u}|^2\ dx^3
\end{array}
\]
Thus combining the two previous inequalities we have
\[
\frac{1}{2}\,\int_{\B_1(0)}|A|^2+\mu\,\frac{|d_Au|^2}{\ep}\ dx^3\ge \frac{1}{8}\,\int_{\B_1(0)}(1-t^2)\,\lf|d\hat{u}\rg|^2+\lf(\frac{\mu}{\ep}- \frac{4}{\,t^2}\rg)\, |d_A{u}|^2\ dx^3\ .
\]
Choosing $t^2:=4\,\ep/\mu$ we have finally
\[
CYMH_\ep^{\la,\mu}(u,A)\ge \frac{1}{8}\,\int_{\B_1(0)}\lf(1-4\,\frac{\ep}{\mu}\rg)\,\lf|d\hat{u}\rg|^2\ dx^3\ .
\]
Because of theorem~\ref{th-reg} we have $\hat{u}\in W^{1,2}\cap C^\infty_{loc}(\B^3,{\S}^2)$, and $\hat{u}=\phi$ on $\p \B^3_1(0)$. This implies in particular
\[
d(\hat{u}^\ast\om_{S^2})=0\qquad\mbox{ in }{\mathcal D}'(\B^3)\ .
\]
where $\om_{S^2}$ is the volume form on $\S^2$. From \cite{BCL}, \cite{GMS-1}, \cite{GMS-2} and \cite{BBC} we have
\be
\label{BCLGMS}
\inf\lf\{ \frac{1}{8}\,\int_{\B_1(0)}|du|^2\ dx^3+\pi\, L_\phi(u) \  ;\  u\in W_\phi^{1,2}(\B^3,\S^2)\rg\}=\inf\lf\{ \frac{1}{8}\,\int_{\B_1(0)}|du|^2\ dx^3 ;\  u\in C^1_\phi(\B^3,\S^2)\rg\}
\ee
where 
\[
L_\phi(u):=\frac{1}{4\pi}\sup\lf\{ \int_{\B^3}d\xi\wedge u^\ast\om_{S^2} -\int_{\p\B^3}\xi\ \phi^\ast \om_{S^2}\ ;\ \|d\xi\|_\infty\le 1\rg\}\ .
\]
Hence we have proved
\[
\liminf_{\ep\rightarrow 0}\min\lf\{CYMH_\ep^{\la,\mu}(u,A)\ ;\ (u,A)\in {\mathcal E}_\phi\rg\}\ge\inf\lf\{ \frac{1}{8}\,\int_{\B_1(0)}|du|^2\ dx^3+\pi\, L_\phi(u) \  ;\  u\in W_\phi^{1,2}(\B^3,\S^2)\rg\}
\]
Using now (\ref{fadeev}) and the proof of Lemma~\ref{lm-bd} we have for any smooth extension $u\in C^2(\B^3,\S^2)$ of $\phi$ in $\B^3$
\[
\liminf_{\ep\rightarrow 0}\min\lf\{CYMH_\ep^{\la,\mu}(u,A)\ ;\ (u,A)\in {\mathcal E}_\phi\rg\}\le\frac{1}{8}\,\int_{\B_1(0)}|du|^2\ dx^3
\]
Then because of (\ref{BCLGMS}) we have the reverse inequality
\[
\inf\lf\{ \frac{1}{8}\,\int_{\B_1(0)}|du|^2\ dx^3+\pi\, L_\phi(u) \  ;\  u\in W_\phi^{1,2}(\B^3,\S^2)\rg\}\ge \liminf_{\ep\rightarrow 0}\min\lf\{CYMH_\ep^{\la,\mu}(u,A)\ ;\ (u,A)\in {\mathcal E}_\phi\rg\}\ .
\]
This concludes the proof of Theorem~\ref{th-clearout}.\hfill $\Box$
\section{ The Asymptotic when  $\mu=\tau/\ep$  for $\tau=\tau_0$ fixed and $\ep\rightarrow 0$ }
\reset
In this section we study sequences of minimizers of
\[
CYMH^{1,\tau/\ep}_\ep(u,A):=\frac{1}{2}\int_{\B_1(0)}|A|^2\ dx^3+\frac{\tau}{2}\int\,\frac{|\nabla_Au|^2}{\ep^2}+|F_A|^2+\frac{1}{2\ep^4}(1-|u|)^2\ dx^3
\]
when $\tau$ is fixed and $\ep\rightarrow 0$.
\subsection{Uniform convergence of $|u|$ to 1 and an a-priori $L^\infty$ estimate on $\nabla_Au$.}

\begin{Lm}
\label{lm-apestinf}
Let $\phi\in C^2(\B^3,\S^2)$, $\tau\in(0,1)$,$\la\ge 1$ and let $\ep_k\rightarrow 0$. Let $(u^k,A^k)$ be a sequence of minimizers of $CYMH_{\ep^k}^{\la,\tau/\ep_k}$ in ${\mathcal E}_\phi$. Then we have
\be
\label{prel-bd-B}
\|\nabla_{A^k}  u^k\|_{L^{\infty}(\B_1(0)} \le C\,\la\,\sqrt{\frac{\,M}{\ep\,\tau}}\qquad\mbox{ and }\qquad \|1-|u^k|\|_{L^\infty(\B_1(0)} \le C\,\sqrt{\frac{\ep\, \la\,M}{\tau}}\ .
\ee
\hfill $\Box$
\end{Lm}
\noindent{\bf Proof of Lemma~\ref{lm-apestinf}.} We present the proof in the case when $\B_r(x_0)\subset \B^3$. A similar proof can be adapted to the case where $x_0\in \p\B^3$ using elliptic estimates up to the boundary and using the assumptions on $\phi$ (i.e. $\|\phi\|_{C^2(\p\B^3)}<+\infty$). 
We first observe that thanks to lemma~\ref{lm-bd}
\[
\limsup_{k\rightarrow +\infty}CYMH^{\la,\tau/\ep_k}_{\ep_k}(u^k,A^k)=\limsup_{k\rightarrow +\infty}\inf\lf\{CYMH^{1,\tau/\ep_k}_{\ep_k}(u,A)\ ;\ (u,A)\in{\mathcal E}_\phi\rg\}<+\infty\ .
\]
Let
\[
M:=\sup_{k\in {\N}}CYMH^{\la,\tau/\ep_k}_{\ep_k}(u^k,A^k)\ .
\]
From now on we shall omit the subscript $k$ when there is no ambiguity. We have
\[
\begin{array}{l}
\ds\int_{\B_r(x_0)}|F_A|^{3/2}\ dx^3\le C\, (r^3)^{1/4}\ \lf(\int_{B_r(x_0)}|F_A|^{2}\ dx^3\rg)^{3/4}\\[5mm]
\ds\le C\, \lf(\frac{r}{\tau}\rg)^{3/4}\ \lf(\tau\,\int_{B_r(x_0)}\,|F_A|^{2}\ dx^3\rg)^{3/4}\le C\ \lf(M\,\frac{r}{\tau}\rg)^{3/4}\
\end{array}
\]
Hence there exists $\al\in (0,1)$ independent of $\ep$ but only on $M/\tau$ such that
\[
\lf(\int_{\B_{\al}(x_0)}|F_A|^{3/2}\ dx^3\rg)^{2/3}<\min\{\delta_{3/2},\delta_2\}:=\delta\ .
\]
where $\delta_{3/2}$ and $\delta_2>0$ are given by (\ref{II.0}). 
\be
\label{al}
\al=C\, \delta^{4/3}\,\frac{\tau}{M}\ .
\ee
We extract a Coulomb Gauge $A^g:=g^{-1}dg+g^{-1}\,A\, g$ on $\B_{\al}(x_0)$. We apply (\ref{II.1}) for $p=3/2$ and $p=2$ and we get
\be
\label{Ag}
\lf\{
\begin{array}{l}
\|A^g\|_{L^{(3,3/2)}(\B_{\al}(x_0))}+\|\nabla(A^g)\|_{L^{3/2}(\B_{\al}(x_0))}\le C\ \|F_{A}\|_{L^{3/2}(\B_{\al}(x_0))}\le C\, \delta\\[5mm]
\ds\|A^g\|_{L^{(6,2)}(\B_{\al}(x_0))}+\|\nabla(A^g)\|_{L^2(\B_{\al}(x_0))}\le C\ \|F_{A}\|_{L^2(\B_{\al}(x_0))}\le C \sqrt{\frac{M}{\tau}}\ .
\end{array}
\rg.
\ee
We deduce in particular thanks to the generalised Littlewood inequality
\[
\begin{array}{l}
\ds\|A^g\|^2_{L^{(3,1)}(B_{\al}(x_0))}\le \|A^g\|_{L^2(B_{\al}(x_0))}\, \|A^g\|_{L^{(6,2)}(B_{\al}(x_0))}\\[5mm]
\ds\quad\le\, C\, \sqrt{\al}  \|A^g\|_{L^3(B_{\al}(x_0))}\, \|A^g\|_{L^{(6,2)}(B_{\al}(x_0))}\le C\, \sqrt{\al}\, \delta\,  \sqrt{\frac{M}{\tau}}=C\, \delta^{5/3}
\end{array}
\]
Observe that $F_{A^g}$ satisfies in $B_{\ep^p}(x_1)$ for any $x_1\in B_{\al/2}(x_0)$ and $p>0$
\[
\lf\{
\begin{array}{l}
\ds d F_{A^g}=-A^g\wedge F_{A^g}\\[5mm]
\ds d^\ast F_{A^g}=A^g\res F_{A^g}-\, g^{-1}\frac{[u,\nabla_A u]}{\ep^2}\,g-\frac{1}{\tau}\,g^{-1}\,A\,g\ .
\end{array}
\rg.
\]
This gives
\[
-\,\Delta A^g+d^\ast\lf(  A^g\wedge A^g  \rg)-A^g\res dA^g-A^g\res(A^g\wedge A^g)=-\, \frac{[u^g,\nabla_{A^g} u^g]}{\ep^2}-\frac{1}{\tau}\,g^{-1}\,A\,g
\]
Let $B\in W^{1,2}(\wedge^1B_{2\ep^p}(x_0),\frak{su}(2))$ and $ D\in L^2(\wedge^2B_{2\ep^p}(x_0),\frak{su}(2))$ solving
\[
-\,\Delta B+d^\ast\lf(  A^g\wedge B  \rg)-A^g\res dB-A^g\res(A^g\wedge B)= D
\]
Using the fact that $\|A^g\|_{L^{(3,1)}(B_{\al}(x_1))}$ is small enough, classical elliptic estimates give
\[
\|B\|_{L^\infty(B_{\ep^p}(x_1))}+\|\nabla B\|_{L^{(3,1)}(B_{\ep^p}(x_1))}\le C\, \|D\|_{L^{(3/2,1)}(B_{2\,\ep^p}(x_1))}+C\, \ep^{-p}\,\|B\|_{L^3(B_{2\,\ep^p}(x_1))}\ .
\]
Hence
\[
\begin{array}{l}
\ds\|A^g\|_{L^\infty(B_{\ep^p}(x_1))}+\|\nabla A^g\|_{ L^{(3,1)}(B_{\ep^p}(x_1)) }\\[5mm]
\ds\qquad\le C\, \ep^{-2}\,\|\nabla_Au\|_{L^{(3/2,1)}(B_{2\,\ep^p}(x_1))}+ C\, \tau^{-1}\,\|A\|_{L^{(3/2,1)}(B_{2\ep^p}(x_1))}+C\, \ep^{-p}\,\|A^g\|_{L^3(B_{2\ep^p}(x_1))}\ .
\end{array}
\]
We have using H\"older inequalities in Lorentz spaces
\[
\|A^g\|_{L^3(B_{2\,\ep^p}(x_1))}\le C\, \ep^{p/2}\,\|A^g\|_{L^6(B_{2\,\ep^p}(x_1))}\ ,
\]
and
\[
\lf\{
\begin{array}{l}
\ds\|\nabla_Au\|_{L^{(3/2,1)}(B_{2\,\ep^p}(x_0))}\le C\, \sqrt{\ep^p}\ \|\nabla_Au\|_{L^{2}(B_{2\ep^p}(x_0))}\\[5mm]
\ds\|A\|_{L^{(3/2,1)}(B_{2\,\ep^p}(x_0))}\le C\ \sqrt{\ep^p}\ \|A\|_{L^{2}(B_{2\ep^p}(x_0))}\ .
\end{array}
\rg.
\]
Thus
\[
\begin{array}{l}
\ds\|A^g\|_{L^\infty(B_{\ep^p}(x_1))}+\|\nabla A^g\|_{ L^{(3,1)}(B_{\ep^p}(x_1)) }\\[5mm]
\ds\quad\le C\, \ep^{-1+p/2}\ \|\ep^{-1}\,\nabla_Au\|_{L^{2}(B_{2\,\ep^p}(x_1))}+C \tau^{-1}\ \ep^{p/2}\ \|A\|_{L^{2}(B_{2\,\ep^p}(x_1))}+C\, \ep^{-p/2}\, \|A^g\|_{L^6(B_{2\,\ep^p}(x_1))}\ .
\end{array}
\]
Choosing $-1+p/2=-p/2$ that is $p=1$ we obtain

\be
\label{Aglinf}
\|A^g\|_{L^\infty(B_{\al/2}(x_0))}\le C\, \sqrt{\frac{M}{\ep\,\tau}}\ .
\ee
\medskip

The Euler Lagrange equation (\ref{CYMHS-co})
 is implying
\be
\label{moduinf}
-2^{-1}\Delta(1-|u|^2)+\,\la\,\frac{|u|^2}{\ep^2}\, (1-|u|^2)=|\nabla_Au|^2\ge 0 .
\ee
Since we are assuming $|u|\equiv 1$ on $\p \B_1(0)$, thanks to the maximum principle we have
\be
\label{linfty-module2}
\forall x\in \B_1(0)\qquad|u|(x)\le 1\ .
\ee
We have also
\[
\begin{array}{l}
\ds\sum_{l=1}^3(\nabla_A^g)_l((\nabla_A^g)_l u^g)=\sum_{l=1}^3\p_{x_l}(\p_{x_l}u^g+[A^g_l,u])+[A^g_l, \p_{x_l}u^g+[A^g_l,u^g]]\\[5mm]
\ds=\Delta u^g+2\,\sum_{l=1}^3[A_l^g,\p_{x_l}u^g]+[A_l^g[A_l^g,u]] 
\end{array}
\]
thus
\be
\label{eq-ug2}
\Delta u^g=-2\,[A^g;\nabla u^g]-[A^g[A^g,u^g]]-\frac{\la}{\ep^2}u^g\,(1-|u|^2)\ .
\ee
We deduce using classical Calderon Zygmund theory for any $p\ge 0$ and $x_1\in B_{\al/2}(x_0)$
\[
\begin{array}{l}
\ds\|\Delta u^g\|_{L^2\lf(B_{\ep^p}(x_1)\rg)}\le C\ \|A^g\|_{L^\infty(B_{\ep^p}(x_1))}\,\|\nabla u^g\|_{L^2\lf(B_{\ep^p}(x_1)\rg)}\\[5mm]
\ds\quad+C\,\|A^g\|^2_{L^\infty\lf(B_{\ep^p}(x_1)\rg)} \, \ep^{3p/2}+C\,\frac{\la}{\ep^{2}}\, \lf( \int_{B_{\ep^p}(x_1)}(1-|u|^2)^2\ dx^3  \rg)^{1/2}+C\, \ep^{-p}\,\|\nabla u^g\|_{L^2\lf(B_{\ep^p}(x_1)\rg)}
\end{array}
\]
Then we bound
\[
\|\nabla u^g\|_{L^2\lf(B_{\ep^p}(x_1)\rg)}\le\lf[\|\nabla_Au\|_{L^2\lf(B_{\ep^p}(x_1)\rg)}+\|A^g\|_{L^2\lf(B_{\ep^p\al}(x_0)\rg)}\rg]\
\]
We have
\[
\|A^g\|_{L^2\lf(B_{\ep^p}(x_1)\rg)}\le C\,\ep^{p}\ \|A^g\|_{L^6\lf(B_{\ep^p}(x_1)\rg)}\le C\,\ep^p\,\sqrt{\frac{M}{\tau}}\ .
\]
We deduce
\[
\|\nabla u^g\|_{L^2\lf(B_{\ep^p}(x_1)\rg)}\le\, C\, \ep\,\sqrt{\frac{M}{\tau}}+C\,\ep^p\,\sqrt{\frac{M}{\tau}}\ .
\]
Hence
\[
\|\Delta u^g\|_{L^{2}\lf(B_{\ep^p}(x_1)\rg)}\le C\, {\frac{M}{\tau}}\ \lf( \sqrt{\ep}+\ep^{p-1/2}  \rg)+C\, {\frac{M}{\tau}}\ \ep^{3p/2-1}+C\, \sqrt{\frac{\la\,M}{\tau}}+C\, \sqrt{\frac{M}{\tau}}\, (1+\ep^{1-p})
\]
This implies for $p=1$
\[
\|\Delta u^g\|_{L^{(3/2,1)}\lf(\B_{\ep}(x_1)\rg)}\le C\,\sqrt{\ep}\, \|\Delta u^g\|_{L^{2}\lf(\B_{\ep}(x_1)\rg)}\le C\,\sqrt{\frac{\ep\, \la\,M}{\tau}}
\]
This gives
\[
\begin{array}{l}
\ds\|\nabla u^g\|_{L^{(3,1)}\lf(B_{\ep/2}(x_1)\rg)}\le\, C\,\|\Delta u^g\|_{L^{(3/2,1)}\lf(\B_{\ep}(x_1)\rg)}+C\, \ep^{-1/2}\,\|\nabla u^g\|_{L^2\lf(B_{\ep}(x_1)\rg)}\le C\,\sqrt{\frac{\ep\, \la\,M}{\tau}}
\end{array}
\]
We have also from (\ref{Aglinf})
\[
\ds\|A^g\|_{L^{(3,1)}\lf(B_{\ep/2}(x_1)\rg)}\le  C\,\sqrt{\frac{\ep\, \,M}{\tau}}
\]
Thus
\[
\|\nabla_{A} u\|_{L^{(3,1)}\lf(B_{\ep/2}(x_1)\rg)}=\|\nabla_{A^g} u^g\|_{L^{(3,1)}\lf(B_{\ep/2}(x_1)\rg)}\le C\,\sqrt{\frac{\ep\, \la\,M}{\tau}}
\]
Inserting this bound in the r.h.s of (\ref{moduinf}) we obtain
\[
\begin{array}{l}
\ds\|1-|u|^2\|_{L^{\infty}\lf(B_{\ep/4}(x_1)\rg)}\le C\, \||\nabla_{A} u|^2\|_{L^{(3/2,1)}\lf(B_{\ep/2}(x_1)\rg)}+\sqrt{\dashint_{B_{\ep/2}(x_1)}(1-|u|)^2\ dx^3}\\[5mm]
\ds\qquad\le C\, \frac{\ep\, \la\,M}{\tau}+C\,\sqrt{\frac{\ep\, \la\,M}{\tau}}
\end{array}
\]
Injecting the previous estimates in (\ref{eq-ug2}) is giving
\[
\begin{array}{l}
\ds\|\Delta u^g\|_{L^{(3,1)}\lf(\B_{4^{-1}\,\ep}(x_1)\rg)}\le C\ \|A^g\|_{L^\infty(B_{2^{-1}\,\ep}(x_1))}\,\|\nabla u^g\|_{L^{(3,1)}\lf(B_{2^{-1}}(x_1)\rg)}\\[5mm]
\ds\quad+C\,\ep\,\|A^g\|^2_{L^\infty\lf(B_{2^{-1}\,\ep}(x_0)\rg)} +C\,\frac{\la}{\ep^2}\,\lf( \int_{B_{4^{-1}\ep}(x_1)}(1-|u|^2)^2\ dx^3  \rg)^{1/3}\,\|1-|u|^2\|_{L^{\infty}\lf(B_{\ep/4}(x_1)\rg)}^{1/3}
\end{array}
\]
Thus
\be
\label{Del31}
\|\Delta u^g\|_{L^{(3,1)}\lf(\B_{4^{-1}\,\ep}(x_1)\rg)}\le \frac{M}{\tau}\,(1+\sqrt{\la})+C\,\frac{\la}{\ep^{1/2}}\,\sqrt{\frac{ \,M}{\tau}}
\ee
Calderon Zygmund estimates implies
\[
\begin{array}{l}
\|\nabla^2  u^g\|_{L^{(3,1)}\lf(\B_{8^{-1}\,\ep}(x_1)\rg)}\le C\ \|\Delta u^g\|_{L^{(3,1)}\lf(B_{4^{-1}\ep}(x_1)\rg)}+C\,\ep^{-1}\,\|\nabla u^g\|_{L^{(3,1)}\lf(B_{4^{-1}\ep}(x_1)\rg)}\\[5mm]
\ds\quad\le\,  \frac{M}{\tau}\,(1+\sqrt{\la})+ C\,\la\,\sqrt{\frac{\,M}{\ep\,\tau}}
\end{array}
\]
Hence
\[
\begin{array}{l}
\ds\|\nabla  u^g\|_{L^{\infty}\lf(\B_{8^{-1}\,\ep}(x_1)\rg)}\le C\, \|\nabla^2  u^g\|_{L^{(3,1)}\lf(B_{8^{-1}\ep}(x_1)\rg)}+C\, \ep^{-1}\, \|\nabla u^g\|_{L^{(3,1)}\lf(B_{4^{-1}\,\ep}(x_1)\rg)}\\[5mm]
\ds\quad\le   C\,\la\,\sqrt{\frac{\,M}{\ep\,\tau}}
\end{array}
\]
We have then proved 
\be
\label{prel-bd}
\|\nabla_A  u\|_{L^{\infty}\lf(\B_{2^{-1}\al}(x_0)\rg)} \le C\,\la\,\sqrt{\frac{\,M}{\ep\,\tau}}\qquad\mbox{ and }\qquad \|1-|u|\|_{L^\infty(\B_{2^{-1}\al}(x_0))} \le C\,\sqrt{\frac{\ep\, \la\,M}{\tau}}
\ee
This holds for any $x_0\in\B_1(0)$ such that $\B_{\ep\al}(x_0)\subset \B_1(0)$ . Thus we deduce (\ref{linf}) in the $\ep\al-$interior of $\B_1(0)$ and taking then $x_0\in\p \B^3(0)$ we can implement the same argument
on the quasi half ball $\B_{\ep}(x_0)\cap\B^3_1(0)$ taking into account the $C^2$ norm control of $u=\phi$ on $\p\B^3_1(0)\cap \B_{\al\ep}(x_0)$ and lemma~\ref{lm-linfty-b} is proved. \hfill $\Box$

\medskip
\subsection{ The Limiting Fields for Minimizers of $CYMH^{\la,\tau/\ep_k}_{\ep_k}(u,A)$ as $\ep_k\rightarrow 0$.}
This subsection is devoted to the proof of Theorem~\ref{th-harmflu}.

\medskip

\noindent{\bf Proof of theorem~\ref{th-harmflu}.} The previous Lemma is implying in particular 
\[
\begin{array}{l}
\ds\|d_Au\|_{L^6(\B_1(0))}\le \lf(\int_{\B_1(0)}|d_Au|^6\ dx^3\rg)^{1/6}\le \lf(\|d_Au\|^4_{L^\infty(\B_1(0))}\,\int_{\B_1(0)}|d_Au|^2\ dx^3\rg)^{1/6}\\[5mm]
\ds\quad\le C\ \lf(\la^4\,\lf(\frac{\,M}{\tau}\rg)^2\,\int_{\B_1(0)}\frac{|d_Au|^2}{\ep^2}\ dx^3\rg)^{1/6}\le C\ \la^{2/3}\ \sqrt{\frac{M}{\tau}}
\end{array}
\]

We keep denoting $\hat{u}:=u/|u|$. We have
\[
\tau\,\hat{u}\cdot d_A^\ast F_A=-\hat{u}\cdot A\qquad\mbox{in }\B_1(0)\ .
\]
Hence there holds
\[
\lf\{
\begin{array}{l}
\ds d\lf(  \hat{u}\cdot  F_A \rg)=d_A\hat{u}\,\dot{\wedge} \,F_A\qquad\mbox{in }\B_1(0)\\[5mm]
\ds d^\ast \lf(\hat{u}\cdot F_A\rg)=d_A\hat{u}\ \dot{\res} \ F_A-\tau^{-1}\,\hat{u}\cdot A\qquad\mbox{in }\B_1(0)\\[5mm]
\ds\iota_{\p\B_1(0)}\ast \hat{u}\cdot F_A=0
\end{array}
\rg.
\]
Hence classical elliptic estimates are giving
\[
\begin{array}{l}
\ds\|\hat{u}\cdot F_A\|_{L^{3,3/2}(\B_1(0))}+\|\nabla\lf(\hat{u}\cdot F_A\rg)\|_{L^{3/2}(\B_1(0))}\le C\ \|d_A\hat{u}\|_{L^6(\B_1(0))}\ \|F_A\|_{L^2(\B_1(0))}+C\ \tau^{-1}\ \|A\|_{L^2(\B_1(0))}\\[5mm]
\ds\quad\le   C\ \la^{2/3}\ \frac{M}{\tau}+C\, \frac{\sqrt{M}}{\tau}
\end{array}
\]
From (\ref{bwt}) we have the existence of $f$ such that $\|f\|_{L^1(\B_1(0))}$ is uniformly bounded and
\[
-\mbox{ div}\lf(\frac{1}{|u|^4}\nabla||[u,F_A]|  \rg)+\frac{1}{2\ep^2}|[u,F_A]|\le f\ .
\]
We argue similarly as for the proof of (\ref{L32infcur}) and we first introduce $w$ solution to
\[
\lf\{
\begin{array}{l}
\ds-\mbox{ div}\lf(\frac{1}{|u|^4}\nabla w  \rg)+\frac{1}{2\ep^2|u|^4}w= f\qquad\mbox{ in }\B_1(0)\\[5mm]
\ds w=0\qquad\mbox{ in }\p\B_1(0)
\end{array}
\rg.
\]
We first obtain the counterpart of (\ref{w32}) which is in the present case
\[
\|w\|_{L^{3,\infty}(\B_1(0))}=O(1)\qquad\mbox{ and }\qquad\|w\|_{L^{(3/2,1)}(\B_1(0))}=O(\ep^{-1})
\]
from which we also deduce by interpolation
\[
\|w\|_{L^{2}(\B_1(0))}=O(\ep^{-1/2})
\]
Then we compare $w$ and $|[u,F_A]|$ and we obtain for any ball $\B_{r_1}(0)$ where $r_1<1$
\[
\|[F_A,u]\|_{L^{3,\infty}(\B_{r_1}(0))}=O(1)\quad,\quad\|[F_A,u]\|_{L^{2}(\B_{r_1}(0))}=O(\ep^{-1/2})\quad\mbox{ and }\quad\|[F_A,u]\|_{L^{(3/2,1)}(\B_{r_1}(0))}=O(\ep^{-1})\ .
\]
Now we recall from (\ref{bw-1}) and using (\ref{sq-0}) similarly as in the proof of (\ref{Bochcovder})
\[
-\Delta|\nabla_Au|+\lf(\frac{1}{2\ep^2}-2\,|[F_A,u]|\rg) \ |\nabla_Au|\le 2\,|\nabla_Au|\,|F_A\cdot u|+\frac{|A|}{\tau}
\]
We have
\[
\||\nabla_Au|\,|F_A\cdot u|\|_{L^2(\B_1(0))}\le \|\nabla_Au\|_{L^6(\B_1(0))}\ \|F_A\cdot u|\|_{L^3(\B_1(0))}=O(1)
\]
We introduce this time $w$ such that
\[
\lf\{
\begin{array}{l}
\ds -\Delta\ti{w}+\lf(\frac{1}{2\ep^2}-2\,|[F_A,u]|\rg) \ \ti{w}= 2\,|\nabla_Au|\,|F_A\cdot u|+\frac{|A|}{\tau}\qquad\mbox{ in }\B_{r_1}(0)\\[5mm]
\ds \ti{w}=0\qquad\mbox{ on }\p \B_{r_1}(0)\ .
\end{array}
\rg.
\]
Since $\|[F_A,u]\|_{L^{(3/2,1)}(\B_{r_1}(0))}=o(1)$ we have that $\ti{w}$ is unique and non negative in $\B_{r_1}(0)$ (see (\ref{bdnab}). Moreover multiplying by $\ti{w}/\ep^2$ and integrating by parts is giving, using the smallness of $|[F_A,u]\|_{L^{(3/2,\infty)}(\B_{r_1}(0))}$ as in the proof of (\ref{bdnab}),
\[
\begin{array}{l}
\ds\int_{\B_{r_1}(0)}|\nabla \ti{w}|^2+\frac{1}{2\ep^2}|\ti{w}|^2\le \int_{\B_{r_1}(0)}\ 2\,|[F_A,u]| \ |\ti{w}|^2\ dx^3+\int_{\B_{r_1}(0)}\ \lf(2\,|\nabla_Au|\,|F_A\cdot u|+\frac{|A|}{\tau}\rg)\ \ti{w}\ dx^3\\[5mm]
\ds \le\, 2\,\||[F_A,u]|\|_{L^{(3/2,\infty)}(\B_{r_1}(0))} \ \||\ti{w}|^2\|_{L^{(3,1)}(\B_{r_1}(0))}+\|\ti{w}\|_{L^{2}(\B_{r_1}(0))}\ \lf\|2\,|\nabla_Au|\,|F_A\cdot u|+\frac{|A|}{\tau}\rg\|_{L^{2}(\B_{r_1}(0))}\\[5mm]
\ds\le o(1)\,\int_{\B_{r_1}(0)}|\nabla \ti{w}|^2\ dx^3+O(1)\,\|\ti{w}\|_{L^{2}(\B_{r_1}(0))}\
\end{array}
\]
We deduce
\[
\int_{\B_{r_1}(0)}|\nabla \ti{w}|^2+\frac{1}{2\ep^2}|\ti{w}|^2=O(\ep^2)
\]
Moreover for $V:={\mathbf 1}_{B_{r_1}(0)}\,\lf(\frac{1}{2\ep^2}-2\,|[F_A,u]|\rg)$ there holds
\[
\sup_{x\in \B_{r_1}(0)}\int_{\B_{r_1}(0)}\frac{|V_-(y)|}{|x-y|}\ dy^3\le \sup_{x\in \B_{r_1}(0)}\lf\|\frac{1}{|x-y|}\rg\|_{L^{(3,\infty)}\B_{r_1}(0)}\ \|[F_A,u]\|_{L^{(3/2,1)}(B_{r_1}(0))}\le\,O(\ep)
\]
This gives that the negative part of the potential $V$ in the Schr\"odinger operator $-\Delta+V$ is in the {\it Kato Class}. Using the main result in \cite{AiS} we obtain that the Harnack inequality holds and we have for any $r_2<r_1$
\[
\|\ti{w}_+\|_{L^\infty(B_{r_2}(x_0)}\le C\, \ \int_{B_{r_1}(0)}\ti{w}_+\ dx^3
\]
Hence arguing as for getting (\ref{linka}) we finally obtain for any radius $r_1<1$
\be
\label{decnab}
\int_{B_{r_1}(0)}\frac{|\nabla|\nabla_Au||^2}{\ep^2}+\frac{|\nabla_Au|^2}{\ep^4}\ dx^3=O(1)\ .
\ee
We recall that for $\al$ given by (\ref{al}) and any ball $\B_r(x_0)\subset \B_1(0)$ and $r<\al$ there exists a Coulomb gauge $A^g$ satisfying (\ref{Ag}) and 
\[
\lf\{
\begin{array}{l}
\ds d F_{A^g}=-A^g\wedge F_{A^g}\\[5mm]
\ds d^\ast F_{A^g}=A^g\res F_{A^g}-\,\frac{[u^g,\nabla_{A^g} u^g]}{\ep^2}\,-\frac{1}{\tau}\,g^{-1}\,A\,g\ .
\end{array}
\rg.
\]
This gives
\[
\|\nabla (F_{A^g})\|_{L^2(\B_{r/2}(x_0))}+\|F_A\|_{L^6(\B_{r/2}(x_0))}=O(1)
\]
In particular $F_{A^g}$ is pre-compact in $L^p(\B_{r/2}(x_0))$ for any $p<6$. Since $dg$ is uniformly bounded in $L^2(\B_{r/2}(x_0))$, using Rellich Kondrachov theorem we deduce that $F_A:=g\,F_{A^g}\, g^{-1}$ is pre-compact in $L^p(\B_{r/2}(x_0))$ for any $p<6$..

\medskip

We extract a subsequence from $(u^k,A^k)$ such that
\[
u^k\rightharpoonup u^\infty\qquad\mbox{ weakly in }W^{1,2}_\phi(\B_1(0),{\S}^2)\qquad\mbox{ and }\qquad A^k\rightharpoonup A^\infty\qquad\mbox{ weakly in }L^2(\Om^1(\B_1(0))\ .
\]
We also extract a subsequence such that
\[
F_{A^k}=dA^k+\frac{1}{2}[A^k,A^k]\rightharpoonup F^\infty\qquad\mbox{ weakly in }L^2(\Om^2(\B_1(0))\ .
\]
Because of the above we have in fact
\[
F_{A^k}\ \longrightarrow\ F^\infty\qquad\mbox{ strongly in }L^p_{loc}(\Om^2(\B_1(0))\qquad\forall p<6\ .
\]
Observe that on any ball $\B_r(x_0)\subset \B_1(0)$ and $r<\al$, the gauge $(A^k)^{g^k}$ is bounded in $W^{2,2}(\B_{r/2}(x_0))$ hence modulo extraction of a subsequence we have for any $p<6$
\[
(A^k)^{g^k}\ \longrightarrow\ B^\infty\qquad\mbox{ strongly in }W^{1,p}(\Om^1(\B_{r/2}(x_0)))
\]
and 
\[
 F_{(A^k)^{g^k}}=(g^{k})^{-1}\,F_{A^k}\,g^k\longrightarrow\ F_{B^\infty}\qquad\mbox{ strongly in }L^p(\Om^1(\B_{r/2}(x_0)))
\]
Since $dg^k$ is uniformly bounded in $L^2$ one can extract a subsequence such that
\[
g^k\ \rightharpoonup\  g^\infty \qquad\mbox{ weakly in }W^{1,2}(\B_r(x_0))
\]
By Rellich Kondrachov we have that $g^k$ is converging strongly to $g^\infty$ in $L^p(B_r(x_0))$. Hence there holds
\[
F_{B^\infty}=(g^\infty)^{-1}\,F^\infty\, g^\infty\ .
\]
We have also
\[
(A^k)^{g^k}=(g^{k})^{-1}\,A^k\,g^k+(g^k)^{-1}dg^k\ \rightharpoonup\ (g^\infty)^{-1}\ A^\infty\, g^\infty+(g^\infty)^{-1}\,dg^\infty\qquad\mbox{ in }{\mathcal D}'(\B_{r/2}(x_0))
\]
Thus
\[
B^\infty=(g^\infty)^{-1}\ A^\infty\, g^\infty+(g^\infty)^{-1}\,dg^\infty
\]
which implies finally
\[
F_{A^\infty}=F^\infty
\]
and we then have\footnote{In a funny way we have been able to prove the strong convergence of the curvature of $A^k$ to the curvature of $A^\infty$ in $L^2_{loc}(\Om^1(\B_1(0))\otimes \frak{su}(2))$ but we only know at this stage that
$A^k$ is weakly converging to $A^\infty$ in $L^2(\B_1(0)\otimes \frak{su}(2))$.}
\[
F_{A^k}\ \longrightarrow\ F_{A^\infty}\qquad\mbox{ strongly in }L^2_{loc}(\Om^2(\B_1(0))\ .
\]

\medskip

\noindent{\bf The strong convergence of $A^k$ to $A^\infty$ in $W^{1,2}_{loc}(\B_1(0))$.} We claim now that the convergence of $A^k$ towards $A^\infty$ is strong in $L^2_{loc}(\Om^1(\B_1(0))$.

For any radius $r_1<1$, covering $B_{r_1}(0)$ by balls a uniformly bounded number of balls $\B_r(x_0)\subset \B_1(0)$ and $r<\al$ we finally obtain
\[
\|F_A\|_{L^6(\B_{r_1}(0))}=O(1)\ .
\]
This gives in particular for any $\B_r(x_0)\subset \B_{r_1}(0)$
\be
\label{curvmon}
\int_{B_r(x_0)}\frac{1}{|x-x_0|}|F_A|^2\ dx^3\le r\ \|F_A\|^2_{L^6(\B_{r_1}(0))}\le C\, r\ .
\ee
Using (\ref{trumon}) we have
\be
\label{sss}
\begin{array}{l}
\ds\frac{d}{dr}\lf(   \frac{1}{2r}\int_{\B_r(x_0)}|A|^2+\tau\,\frac{|\nabla_Au|^2}{\ep^2}+\frac{3\la\tau}{2\ep^4}(1-|u|^2)^2+\lf(\frac{2r}{|x-x_0|}-1\rg)\tau\,|F_A|^2\ dx^3 \rg)\\[5mm]
\ds=\frac{1}{r}\,\int_{\p \B_r(x_0)}|A_r|^2+\tau\frac{|(\nabla_Au)_r|^2}{\ep^2}+\frac{\la\,\tau}{2\ep^4}(1-|u|^2)^2+\,\tau\,|F_A\res\p_r|^2\ dvol_{\p \B_r(x_0)}
\end{array}
\ee
Writing $$F(r):=\frac{1}{2r}\int_{\B_r(x_0)}\lf(\frac{2r}{|x-x_0|}-1\rg)\tau\,|F_A|^2\ dx^3$$
we have
\[
0\le F(r)\le C\,r
\]
and
\[
r\longrightarrow F(r)+\frac{1}{2r}\int_{\B_r(x_0)}|A|^2+\tau\,\frac{|\nabla_Au|^2}{\ep^2}+\frac{3\la\tau}{2\ep^4}(1-|u|^2)^2\ dx^3
\]
is an increasing function of $r\in [0,r_1-|x_0|]$ for any $x_0\in \B_{r_1}(0)$ and $r_1<1$ is arbitrary.
 Then we prove the following Lemma
 \begin{Lm}
 \label{lm-delreg}
 Let $r_1<1$. There exists $\delta>0$ and $r_0>0$ such that for any $r<r_0$ and any $x_0\in \B_{r_1}(0)$, if $\B_r(x_0)\subset \B_{r_1}(0)$ and if
 \be
 \label{smf}
 \frac{1}{2r}\int_{\B_r(x_0)}|A|^2+\tau\,\frac{|\nabla_Au|^2}{\ep_k^2}+\frac{3\la\tau}{2\ep_k^4}(1-|u|^2)^2\ dx^3<\delta
 \ee
 then
 \[
 \|\nabla A^k\|_{L^2(\B_{r/2}(x_0))}\le C_r
 \]
 where $C_r$ is independent of $k$.\hfill $\Box$.
 \end{Lm}
 \noindent{\bf Proof of Lemma~\ref{lm-delreg}.} We choose $r_0<2^{-1}-2^{-1}\,r_0$ such that 
 \[
 \sup_{x_1\in \B_{r_1}(0)}\int_{B_{r_0}(x_1)}\frac{1}{|x-x_1|}|F_A|^2\ dx^3\le r_0\ \|F_A\|^2_{L^6(\B_{2^{-1}(r_1+1)}(0))}\le \delta
 \]
 where $\delta>0$ shall be fixed later on. Because of  (\ref{sss}) we have
 \[
 \max_{x\in {\B}_{r_0/2}(x_0)}\max_{\rho<r_0/4}\frac{1}{\rho}\int_{\B_\rho(x)}|A|^2\ dx^3<4\delta
 \]
 Recall moreover that choosing $r_0<\al$ where $\al$ is given by (\ref{al}) there exists a Gauge change and a one form form $d^\ast \xi$ whose norm is uniformly bounded in $W^{1,2}(\B_{r_0}(x_0))$ independently of $k$ such that
 \[
 g\,dg^{-1}=-A+g\,d^\ast \xi \,g^{-1}
 \] 
 and the Coulomb condition is implying
 \be
 \label{quahar}
 d^\ast(g\,d g^{-1})=dg\wedge d\ast \xi\,g^{-1}+g\, d\ast\xi\wedge dg^{-1}
 \ee
 Arguing exactly like in the proof of Theorem~\ref{th-delreg} we obtain the existence of $\beta>0$ such that
 \[
  \max_{x\in {\B}_{r_0/2}(x_0)}\max_{\rho<r_0/4}\frac{1}{\rho^{1+\beta}}\int_{\B_\rho(x)}|dg|^2\ dx^3<C\, \delta
 \]
this gives in particular a uniform $C^{0,\beta/2}$ bound of $g$ in ${\B}_{r_0/2}(x_0)$ controlled by $\sqrt{\delta}$. This permits to treat the operator $d^\ast g\,d\cdot$ as an elliptic operator. One observes then that the equation (\ref{quahar}) is becoming sub-critical for this Morrey norm and, since $d^\ast\xi$ is uniformly bounded in $W^{1,2}(\B_{r_0}(x_0))$
we obtain, after iterating a finite amount of steps in the elliptic system (\ref{quahar}),  for some $t>0$ independent of $k$ and $r$
\[
\|dg\|_{W^{1,2}(\B_{tr}(x_0))}\le C\ \|d^\ast\xi\|_{W^{1,2}(\B_{r_0}(x_0))}+r^{-1}\|dg\|_{L^2(\B_{r}(x_0))}
\]
Hence we have proved
\[
\|A^k\|_{W^{1,2}(\B_{tr}(x_0))}<C_r
\]
where $C_r$ is independent of $k$. Having chosen $\delta$ small enough one can always assume that $t=1/2$. This concludes the proof of Lemma~\ref{lm-delreg}.\hfill $\Box$

\medskip

We shall now prove the following lemma

\begin{Lm}
\label{lm-0dens}
Under the above notations we have for any $x_0\in \B_1(0)$
\be
\label{dens-0}
\lim_{\rho\rightarrow 0}\limsup_{k\rightarrow +\infty}\frac{1}{\rho}\int_{\B_\rho(x)}|A^k|^2\ dx^3=0\ .
\ee
and $A^k$ is uniformly bounded in $W^{1,2}(\B_{r_1}(0))$ for any $1>r_1>0$.
\hfill$\Box$
\end{Lm}
\noindent{\bf Proof of Lemma~\ref{lm-0dens}.} Assume the result would not be true. Hence there would be a point $x_0\in\B_1(0)$ such that
\[
\lim_{\rho\rightarrow 0}\limsup_{k\rightarrow +\infty}\frac{1}{\rho}\int_{\B_\rho(x)}|A^k|^2\ dx^3>0
\]
Modulo extraction of a subsequence we can assume
\[
\mu_k:=|A^k|^2+\tau\,\frac{|\nabla_{A^k}u^k|^2}{\ep_k^2}+\frac{3\la\tau}{2\ep_k^4}(1-|u_k|^2)^2\ \rightharpoonup\ \mu_\infty\ \qquad\mbox{ in }{\mathcal M}(\B_1(0))
\]
Because of (\ref{decnab}) there holds 
\[
\frac{|\nabla_{A^k}u^k|^2}{\ep_k^2}\ \rightharpoonup\ 0 \qquad\mbox{ in }{\mathcal M}(\B_{r_1}(0))
\]
for any $r_1<1$. Let $\chi_{r_1}$ be a smooth function such that $\chi_{r_1}\equiv 1$ on $\B_{r_1}(0)$ and $\chi_{r_1}\in C^\infty_0(\B_{(r_1+1)/2}(0))$. From (\ref{moduinf}) we have 
\[
-\int_{\B_{r_1}(0)} \chi_{r_1}(x) \frac{(1-|u|^2)}{\ep^2}\,\Delta\, (1-|u|^2)+\la\,\chi_{r_1}\, \frac{|u|^2}{\ep^4}\ (1-|u|^2)^2\ dx^3=\int_{\B_{r_1}(0)} \chi_{r_1}(x) \ \frac{|\nabla_Au|^2}{\ep^2}\,(1-|u|^2)\ dx^3
\]
Hence
\[
\begin{array}{l}
\ds\int_{\B_{1}(0)} \chi_{r_1}(x)  \frac{|\nabla(1-|u|^2)|^2}{\ep^2}+\la\,\chi_{r_1}\, \frac{|u|^2}{\ep^4}\ \ (1-|u|^2)^2\ dx^3\\[5mm]
\ds\le \int_{\B_{1}(0)} \chi_{r_1}(x) \ \frac{|\nabla_Au|^2}{\ep^2}\,(1-|u|^2)\ dx^3+\frac{1}{2}\, \int_{\B_{r_1}(0)} \Delta\chi_{r_1}(x) \ \frac{(1-|u|^2)^2}{2\,\ep^2}\ dx^3\\[5mm]
\ds\le \|1-|u|^2\|_{L^\infty(\B_{(r_1+1)/2}(0))}\ \int_{\B_{(r_1+1)/2}(0)}  \frac{|\nabla_Au|^2}{\ep^2}\ dx^3+\ep^2\  \int_{\B_{1}(0)} \ \frac{(1-|u|^2)^2}{2\,\ep^4}\ dx^3=o(1)
\end{array}
\]
Hence in particular
\[
\frac{3\la\tau}{2\ep_k^4}(1-|u_k|^2)^2\ \rightharpoonup\ 0 \qquad\mbox{ in }{\mathcal M}(\B_{r_1}(0))
\]
Observe that thanks to (\ref{sss}) and (\ref{curvmon}) 
\[
\frac{d}{dr}\lf(  \frac{1}{2r}\mu_\infty(B_r(x_0))+\frac{1}{2r} \ \int_{\B_r(x_0)}\,\lf(\frac{2r}{|x-x_0|}-1\rg)\tau\,|F_{A^\infty}|^2\ dx^3\rg)\ge 0
\]
Since $F_{A^\infty}\in L^6(\B_{r_1}(0))$ we have  for any $r>0$
\[
\int_{B_{r}(x_0)}\frac{1}{|x-x_1|}|F_{A^\infty}|^2\ dx^3\le r\ \|F_A\|^2_{L^6(\B_{r_1}(0))}
\]
Hence the following limit
\[
\lim_{\rho\rightarrow 0} \frac{\mu_\infty(B_\rho(x_0))}{2\,\rho}\qquad\mbox{ exists.}
\]
Assume
\[
\lim_{\rho\rightarrow 0} \frac{\mu_\infty(B_\rho(x_0))}{2\,\rho}<\delta
\]
Because of Lemma~\ref{lm-delreg} we have that there exists a fixed radius $\rho$ such that $\|\nabla A^k\|_{W^{1,2}(\B_\rho(x_0))}$ is uniformly bounded. Hence, if (\ref{dens-0}) is not true then we have
\be
\label{lowden}
\lim_{\rho\rightarrow 0}\lim_{k\rightarrow +\infty}\frac{1}{2\,\rho}\int_{\B_\rho(x)}|A^k|^2\ dx^3=\lim_{\rho\rightarrow 0}\frac{\mu_\infty(B_\rho(x_0))}{2\,\rho}\ge \delta\ .
\ee
Recall the existence of a Coulomb gauge $d^\ast\xi^k$ on a fixed ball $\B_{r_0}(x_0)$ such that
\[
A^k=(g^k)^{-1}\,dg^k+(g^k)^{-1}\,d^\ast\xi^k\,g^k\quad\mbox{ and }\quad\limsup_{k\rightarrow +\infty}\|d^\ast\xi^k\|_{W^{1,6}(\B_{r_0}(x_0))}<+\infty\ .
\]
Hence there exists $\rho_k\rightarrow 0$ such that $\ti{g}^k(y):=g^k(\rho_k\,y+x_0)$ satisfies
\[
\ti{g}^k\rightharpoonup \ti{g}^\infty\quad\mbox{ weakly in }W^{1,2}_{loc}({\R}^3,SU(2))\qquad\mbox{ and }\qquad \lim_{k\rightarrow +\infty}\int_{\B_1(0)}|d\ti{g}^k|^2\ dy^3\ge \delta\ .
\]
Moreover for any $R>0$
\[
\lim_{k\rightarrow +\infty}\int_{\B_R(0)}\frac{1}{|y|}\lf|\frac{\p \ti{g}^k}{\p r}\rg|^2\ dy^3=0\ .
\]
Now one can argue word by word as in the proof of theorem~\ref{th-reg} to obtain a contradiction and this concludes the proof of Lemma~\ref{lm-0dens}.\hfill $\Box$

\medskip

\noindent{\bf End of the proof of Theorem~\ref{th-harmflu}.} Thanks to the previous lemma we have that $\|A^k\|_{W^{1,2}(\B_{r_1}(0))}$ is uniformly bounded for any $r_1<1$ and a classical bootstrap arguments leads to the strong convergence to $A^k$ and $u^k$ in any $C^l_{loc}(\B_1(0))$. 
We have moreover $|u^\infty|=1$ on $\B^3$ and
\[
d u^k+[A^k,u^k]\rightarrow 0\qquad\mbox{ in } L^2({\B}_1(0))
\]
Thus the transverse part to $A^\infty$ with respect to $u^\infty$ is given by
\[
(A^\infty)^T=-\frac{1}{4}[u^\infty,[u^\infty,A^\infty]]=-\frac{1}{4}[u^\infty,du^\infty]\ .
\]
Moreover $d^\ast A^\infty=0$
which implies
\[
d^\ast\lf(\frac{1}{4}[u^\infty,du^\infty]\rg)=d^\ast((A^\infty)\cdot u^\infty\ u^\infty)
\]
Hence
\[
\begin{array}{l}
\ds F_{A^\infty}=-\frac{1}{4}d\lf([u^\infty,du^\infty]\rg)+d\lf(A^\infty\cdot u^\infty\,u^\infty\rg)+\frac{1}{32}\lf[[u^\infty,du^\infty]\,\wedge\,[u^\infty,du^\infty]\rg]\\[5mm]
\ds-\frac{1}{8}\lf[ [u^\infty, du^\infty],u^\infty\rg]\wedge A^\infty\cdot u^\infty- \frac{1}{8}\,A^\infty\cdot u^\infty\wedge [u^\infty,[u^\infty, du^\infty]]

\end{array}
\]
Recall from (\ref{curb})
\[
\begin{array}{l}
\ds-\frac{1}{4}\,d([{u}^\infty,d{u}^\infty])+\frac{1}{32}[[{u}^\infty,d{u}^\infty]\wedge[{u}^\infty,d{u}^\infty]]=-\frac{1}{8}\,[d{u}^\infty\wedge d{u}^\infty]
\end{array}
\]
This gives
\[
\begin{array}{rl}
\ds F_{A^\infty}&\ds=-\frac{1}{8}\,[d{u}^\infty\wedge d{u}^\infty] +d\lf(A^\infty\cdot u^\infty\,u^\infty\rg) +A^\infty\cdot u^\infty\wedge du^\infty\\[5mm]
 &\ds=-\frac{1}{8}\,[d{u}^\infty\wedge d{u}^\infty] +d\lf(A^\infty\cdot u^\infty\rg)\,u^\infty\
\end{array}
\]
We deduce in particular
\[
[u^\infty,F_{A^\infty}]=0\qquad\mbox{ and }\qquad d\lf(u^\infty\cdot F_{A^\infty}\rg)=0\ .
\]
We had
\[
u^k\cdot d^\ast_{A^k}F_{A^k}=-\frac{1}{\tau} u^k\cdot A^k
\]
which also gives
\[
d^\ast\lf(u^k\cdot F_{A^k}\rg)-d_{A^k}u^k\,\dot{\res}\,F_{A^k} =-\frac{1}{\tau} u^k\cdot A^k
\]
Passing to the limit is giving
\[
d^\ast\lf(u^\infty\cdot F_{A^\infty}\rg)=-\frac{1}{\tau} u^\infty\cdot A^\infty
\]
Let $G^\infty$ be the real valued 2-form given by $2^{-1}\,G^\infty:=-u^\infty\cdot F_{A^\infty}$. We have
\[
d^\ast\lf(\frac{1}{4}[u^\infty,du^\infty]\rg)+\tau\,d^\ast\lf(\frac{1}{2}\,d^\ast G^\infty\ u^\infty\rg)=0
\]
Moreover
\[
G^\infty=\frac{1}{4}\,u^\infty\cdot[d{u}^\infty\wedge d{u}^\infty]-\tau\,d\,d^\ast G^\infty
\]
Observe that we have in general for any map in $W^{1,2}(\B^3_1(0),\S^2)$
\[
\begin{array}{l}
\ds\frac{1}{4}\,u\cdot[d{u}\,\wedge\, d{u}]=\frac{1}{2} u\cdot[\p_{x_1}{u}\,\wedge\, \p_{x_2}{u}]\ dx_1\wedge dx_2\\[5mm]
\ds= u^{\bf i} \lf(\p_{x_1}u^{\bf j}\,\p_{x_2}u^{\bf k}-\p_{x_1}u^{\bf k}\,\p_{x_2}u^{\bf j}\rg)\ dx_1\wedge dx_2 -u^{\bf j} \lf(\p_{x_1}u^{\bf k}\,\p_{x_2}u^{\bf i}-\p_{x_1}u^{\bf i}\,\p_{x_2}u^{\bf k}\rg)\ dx_1\wedge dx_2\\[5mm]
\ds\ +u^{\bf k} \lf(\p_{x_1}u^{\bf i}\,\p_{x_2}u^{\bf j}-\p_{x_1}u^{\bf j}\,\p_{x_2}u^{\bf i}\rg)\ dx_1\wedge dx_2= u^\ast \om_{\S^2}
\end{array}
\]
where $\om_{\S^2}$ is the volume form on $\S^2$. We have also
\[
\frac{1}{2}[u,du]=\lf(u^{\bf j}\, du^{\bf k}-u^{\bf k}\, du^{\bf j}\rg)\ {\bf i}+\lf(u^{\bf k}\, du^{\bf i}-u^{\bf i}\, du^{\bf k}\rg)\ {\bf j}+\lf(u^{\bf i}\, du^{\bf j}-u^{\bf j}\, du^{\bf i}\rg)\ {\bf k}\ .
\]
which gives
\[
\frac{1}{4}[du\,\wedge\, du]=du\wedge du
\]
Hence
\[
F_{A^\infty}=-\frac{1}{2}du\wedge du+\frac{\tau}{2}\,dd^\ast G^\infty\ u^\infty=-\frac{1}{2} \,\lf(u^\ast\om_{\S^2}-\tau\,dd^\ast G^\infty\rg)\ u^\infty
\]
Finally\footnote{Recall that on one forms $\beta$  there holds 
\[
d^\ast \beta=-\ast d\ast\beta
\]
while for two forms $F$ there holds
\[
d^\ast F=\ast d\ast F
\]}
 we have prove that $(u^\infty, G^\infty)$ satisfy  the system 
\[
\lf\{
\begin{array}{l}
\ds -d\ast\lf(u^\infty\wedge d u^\infty\rg)=\tau\, d\ast G^\infty\wedge du^\infty\qquad\mbox{ in }\B_1(0)\\[5mm]
\ds \tau\,d\,d^\ast G^\infty+G^\infty=(u^\infty)^\ast\om_{S^2}\qquad\mbox{ in }\B_1(0)\\[5mm]
\ds\iota^\ast_{\p\B_1(0)}\ast G^\infty=0
\end{array}
\rg.
\]
and $F_{A^\infty}=-\,2^{-1}\,G^\infty\ u^\infty$. Implementing the above arguments at balls centred at the boundary we obtain that $(u^\infty,A^\infty)$ are smooth up to the boundary.

\medskip

In order to complete the proof of theorem~\ref{th-harmflu} it remains to show that $(u^\infty, \eta^\infty=\tau\,d^\ast G^\infty)$ is an absolute minimizer of ${\mathcal R}_\tau$ under the only constraint $u^\infty=\phi$ on $\p B^3_1(0)$.
We clearly have
\[
\inf\lf\{ {\mathcal R}_\tau(u,\eta)\ ;\ u\in C^1(\B^3_1(0),\S^2)\cap W^{1,2}_\phi(\B^3,\S^2)\quad,\eta\in C^1(\Om_1(\B_1^3(0))\rg\}\le {\mathcal R}_\tau(u^\infty,\eta^\infty)\
\]
Let $u\in W^{1,4}_\phi(\B^3,\S^2)$ and let $\eta\in L^2(\Om_1(\B^3_1(0)))$ such that $d\eta\in L^2(\Om_2(\B^3_1(0)))$. and We consider
\[
A:= -\frac{1}{4}\,[u,du]+\frac{\eta}{2}\, u\ .
\]
We compute
\[
\begin{array}{rl}
\ds F_A&\ds =-\frac{1}{4}\,d[u,du]+\frac{1}{2}d\eta\, u-\eta\wedge du+\frac{1}{32}\,\lf[ [u,du]\,\wedge\,[u,du]  \rg]\\[5mm]
\ds&\ds\quad -\frac{1}{16}\,[[u,du],u]\,\wedge \eta-\frac{1}{16}\,\eta\wedge [u,[u,du]]\\[5mm]
\ds&\ds\quad=-\frac{1}{8}\,[du\,\wedge du]+\frac{1}{2}\,d\eta\, u=\frac{1}{2}\,\lf(d\eta-u^\ast\om\rg)\ u
\end{array}
\]
Hence for any $\ep>0$ there holds
\[
\frac{1}{2}\int_{\B^3_1(0)}|A|^2+\tau\,|F_A|^2+\tau\,\frac{|\nabla_A u|^2}{\ep^2}+\tau\la\frac{(1-|u|^2)}{\ep^4}\ dx^3=\frac{1}{8}\int_{\B^3_1(0)}\lf(|du|^2+|\eta|^2\rg)+\tau\,\lf|d\eta-u^\ast\om\rg|^2\ dx^3
\]
Hence
\[
\begin{array}{l}
\ds{\mathcal R}_\tau(u^\infty,\eta^\infty)=\frac{1}{2}\int_{\B^3_1(0)}|A^\infty|^2+\tau\,|F_{A^\infty}|^2\ dx^3\le \limsup_{k\rightarrow +\infty}CYMH_{\ep_k}^{\la,\tau/\ep^k}(u^k,A^k)\\[5mm]
\ds \le CYMH_{\ep_k}^{\la,\tau/\ep^k}(u,A)\le {\mathcal R}_\tau(u,\eta)
\end{array}
\]
This concludes   the proof of theorem~\ref{th-harmflu}.\hfill $\Box$
\section{The Study of ${\mathcal R}_\tau$ and the asymptotics $\tau\rightarrow +\infty$ and $\tau\rightarrow 0$.}
\reset
\subsection{Some fundamental facts on ${\mathcal R}_\tau$.}
\begin{Lm}
\label{lm-hopf}
Let $u\in W^{1,3}(\B_1^3(0),\S^2)$ 
and let $\eta\in L^3(\Om_1(\B_1(0)))$ such that
\[
d\eta:=u^\ast \om_{\S^2}\ .
\]
Then there exists $\check{u}\in W^{1,3}(\B^3,\S^3)$ such that $u=\frak{h}\circ u$ where $\frak{h}$ is the Hopf fibration and 
\[
|du|^2+|\eta|^2=4\,|d\hat{u}|^2 \qquad\mbox{ almost everywhere}\ .
\]
\hfill $\Box$
\end{Lm}
\noindent{\bf Proof of Lemma~\ref{lm-hopf}.} The Hopf fibration $\frak{h}$ is the map given by
\[
\frak{h}\ :\ (w_1,w_2)\in \S^3\subset {\C}^2\quad\longrightarrow\quad (2 w_1\,\ov{w}_2\,,\,|w_1|^2-|w_2|^2)
\]
It is transversally conformal that is the restriction to every horizontal plane given by the Kernel of $\al$ where
\[
\al:=z_1\,dz_2-z_2\,dz_1+z_3\,dz_4-z_4\,dz_1
\]
is a conformal linear map moreover for any horizontal vector $X^H\in \mbox{Ker}(\al)$ there holds
\[
|\frak{h}_\ast X^H|_{S^2}= 2\, |X^H|_{S^3}
\]
in such a way that $|\nabla \frak{h}|\equiv 2\,\sqrt{2}$ on $S^3$ (see \cite{Riv}). Observe that the $1-$form $\al$ satisfies $|\al|_{S^3}\equiv 1$. We have also from \cite{Riv}
\[
d\al =2\, dx_1\wedge dx_2+2\, dx_3\wedge dx_4=\frac{1}{2}\, \frak{h}^\ast\om_{\S^2}\ .
\]
Since smooth maps are dense in $W^{1,3}(\B^2,\S^2)$ we can assume that $u$ is $C^\infty$. Let $u$ be a smooth map from $\B^3_1(0)$ into $\S^2$ and let $\eta$ such that
\[
d\eta=u^\ast\om_{\S^2}
\]
Since $u$ is smooth from $\B^3_1(0)$ into $\S^2$ there exists a smooth lift $\ti{u}$ from $\B^3_1(0)$ into $\S^3$ such that $\frak{h}\circ \ti{u}=u$ and there holds
\[
d(\ti{u}^\ast\al)=\ti{u}^\ast d\al=\frac{1}{2}\,\ti{u}^\ast \frak{h}^\ast\om_{\S^2}=\frac{1}{2}\,u^\ast\om_{\S^2}\ .
\]
Hence
\[
d(\ti{u}^\ast\al-\frac{1}{2}\,\eta)=0\ .
\]
Let $\varphi$ such that $d\varphi=\ti{u}^\ast\al-\frac{1}{2}\,\eta$. Consider the map $\check{u}$ into $S^3$ such that $\check{u}:=e^{-i\varphi}\,\ti{u}$. Observe that we still have $\frak{h}\circ\check{u}=u$. For any $g\,:\,\B^3\rightarrow \S^3$ there holds
\[
g^\ast\al= g_1\,dg_2-g_2\,dg_1+g_3\,dg_4-g_4\,dg_1=\Im\lf( (g_1-ig_2)\, d(g_1+ig_2)+(g_3-ig_4)\, d(g_3+ig_4)\rg)
\]
Hence writing $\ti{u}:=(a,b)$ we have
\[
\check{u}^\ast\al=(e^{-i\varphi}\,\ti{u})^\ast\al= \Im\lf( e^{i\varphi}\, \ov{a}\,d( e^{-i\varphi}\, {a})+e^{i\varphi}\, \ov{b}\,d( e^{-i\varphi}\, {b})  \rg)=-d\varphi+\ti{u}^\ast\al=2^{-1}\,\eta\ .
\]
Let $(e_1^\ast, e_2^\ast)$ be an orthonormal basis of $H_z$ the horizontal plane at $z\in\S^3$. Thus  $(e_1^\ast(z), e_2^\ast(z),\al(z))$ is an orthonormal basis of $T^\ast\S^3$ and we have 
\[
|\nabla \check{u}|^2=\sum_{l=1}^3 |\p_{x_l}\check{u}|^2=\sum_{l=1}^3 <e_1^\ast(\check{u}),\p_{x_l}\check{u}>^2+<e_1^\ast(\check{u}),\p_{x_l}\check{u}>^2+<\al(\check{u}),\p_{x_l}\check{u}>^2
\]
We have for instance
\[
|\check{u}^\ast\al|^2=\sum_{l=1}^3<\check{u}^\ast\al,\p_{x_l}>^2=\sum_{l=1}^3<\al,\check{u}_\ast\p_{x_l}>^2=\sum_{l=1}^3<\al(\check{u}),\p_{x_l}\check{u}>^2
\]
Hence
\[
|\nabla \check{u}|^2=|\check{u}^\ast e_1^\ast|^2+|\check{u}^\ast e_2^\ast|^2+|\check{u}^\ast\al|^2=|\check{u}^\ast e_1^\ast|^2+|\check{u}^\ast e_2^\ast|^2+\frac{1}{4}|\eta|^2
\]
Since $\frak{h}$ is conformal from $H_z$ into $T_{\frak{h}(z)}\S^2$ with conformal factor equal to $2$ we have that $(e^\ast_1,e^\ast_2)=\frac{1}{2}\,\frak{h}^\ast (f_1^\ast, f_2^\ast)$ where $(f_1^\ast, f_2^\ast)$ is an orthonormal basis of $T_{\frak{h}(z)}\S^2$. This implies
\[
|\check{u}^\ast e_1^\ast|^2+|\check{u}^\ast e_2^\ast|^2=\frac{1}{4}\, \lf(|u^\ast f_1^\ast|^2+|u^\ast f_2^\ast|^2\rg)=\frac{1}{4}|\nabla u|^2\ .
\]
Hence we have proved the lemma~\ref{lm-hopf}.\hfill$\Box$

\medskip

It would be interesting to extend this result to maps in $W^{1,2}(\B^3,\S^2)$  which are strongly approximable by smooth maps  - which is equivalent to $d(u^\ast\om)=0$ (\cite{Be1}) - and for which there exists a 1-form in $L^2$ such that $d\eta=u^\ast\om$ (see open problems ~\ref{ope-lift} and \ref{ope-lift0}). This is also closely related to the following question considered in \cite{AK} : Let $A\in L^2(\Om_1(\B^3_1(0))\otimes \frak{su}(2))$ such that
\[
F_A=dA+\frac{1}{2}[A\,\wedge A]=0
\]
does there exists $g\in W^{1,2}(\B^3_1(0),SU(2))$ such that
\[
A=g^{-1}\, dg\  ?
\]
We now prove the following proposition.
\begin{Prop}
\label{pr-crit-refa}
$(u,\eta)$ is a smooth critical point of
\[
{\mathcal R}_\tau(u,\eta):=\frac{1}{8}\int_{\B^3_1(0)}|du|^2+|\eta|^2+\tau\,|d\eta-u^\ast\om|^2\ dx^3
\]
among any choice of one form $\eta$ and  smooth $u\in \S^2$ satisfying $u=\phi$ on $\p\B^3_1(0)$. Then $\eta$ satisfies the Coulomb condition
\[
d^\ast\eta =0
\]
and $(u,G)$ where $G:=u^\ast\om_{\S^2}-d\eta$ satisfy (\ref{lond}) that is
\[
\lf\{
\begin{array}{l}
\ds- d\ast\lf(u\times d u\rg)=\tau\, d\ast G\wedge du=-\,\tau\,d\lf(  d\ast G\ u  \rg)\qquad\mbox{ in }\B_1(0)\\[5mm]
\ds \tau\,d\,d^\ast G+G=u^\ast\om_{S^2}\qquad\mbox{ in }\B_1(0)\\[5mm]
\ds dG=0\qquad\mbox{ in }\B_1(0)\\[5mm]
\ds\iota^\ast_{\p\B_1(0)}\ast G=0
\end{array}
\rg.
\]
We have in particular
\[
\eta=\tau \,d^\ast G\ \qquad\mbox{ and }\quad{\mathcal R}_\tau(u,\eta):=\frac{1}{8}\int_{\B^3_1(0)}|du|^2+\tau^2\,|d^\ast G|^2+\tau\,|G|^2\ dx^3
\]
Let $B$ such that
\[
dB=-\ast\lf(u\times d u\rg)+\tau \,d\ast G\ u 
\]
then $u$ satisfies
\be
\label{antisy}
\ast du=u\times dB\ .
\ee
\hfill $\Box$
\end{Prop}
\noindent{\bf Proof of Proposition~\ref{pr-crit-refa}.} Assume $(u,\eta)$ is a smooth critical point of
\[
{\mathcal R}_\tau(u,\eta):=\frac{1}{8}\int_{\B^3_1(0)}|du|^2+|\eta|^2+\tau\,|d\eta-u^\ast\om|^2\ dx^3=\frac{1}{8}\int_{\B^3_1(0)}|du|^2+|\eta|^2+\tau\,\lf|d\eta\,u- \frac{1}{2}\,du\wedge du\rg|^2\ dx^3
\]
where $u\in \S^2$ prescribing only $u=\phi$ on $\p\B^3_1(0)$. Then for any $\zeta\in C^\infty(\Om_1(\B^3))$ and $v\in C^\infty_0(\B^3_1(0),{\R}^3)$ such that $v\cdot u=0$ at every point there holds
\[
\begin{array}{l}
\ds 0=\int_{\B^3_1(0)} du\cdot dv+\eta\cdot\zeta+\tau \lf(d\eta\, u- \frac{1}{2}\,du\wedge du\rg)\cdot \lf( d\zeta\ u-dv\wedge du+d\eta\, v\rg)\ dx^3\\[5mm]
\ds=\int_{\B^3_1(0)} d^\ast du\cdot v+\eta\cdot\zeta+\tau\,d^\ast (d\eta-u^\ast\om)\cdot \zeta+\tau\,(u^\ast\om-d\eta)\, u\cdot dv\wedge du\\[5mm]
\ds+\int_{\p \B_1^3(0)} \ast(d\eta-u^\ast\om)\wedge \zeta
\end{array}
\]
We have using the fact $v\cdot du\wedge du=0$
\[
\begin{array}{rl}
\ds u\cdot dv\wedge du&\ds=\sum_{i=1}^3 u^i\, (dv^{i+1}\wedge du^{i-1}-dv^{i-1}\wedge du^{i+1})=\sum_{i=1}^3 d\lf(u^i\, (v^{i+1}\, du^{i-1}-v^{i-1}\, du^{i+1})\rg)\\[5mm]
&\ds\quad-\sum_{i=1}^3 v^{i}\,du^{i-1}\,\wedge du^{i+1}+ \sum_{i=1}^3v^{i}\,du^{i+1}\wedge du^{i-1}\\[5mm]
 &\ds=\sum_{i=1}^3 d\lf(u^i\, (v^{i+1}\, du^{i-1}-v^{i-1}\, du^{i+1})\rg)\ .
\end{array}
\]
Thus
\[
\int_{\B^3_1(0)} \tau\,(u^\ast\om-d\eta)\, u\cdot dv\wedge du+\tau =\int_{\B^3_1(0)} \tau\,d^\ast(u^\ast\om_{\S^2}-d\eta)\,\cdot\sum_{i=1}^3 v_i\,\lf( u^{i-1}\, du^{i+1}-u^{i+1}\,du^{i-1}  \rg)
\]
This gives respectively
\[
\lf\{
\begin{array}{l}
\ds\tau\,d^\ast d\eta+\eta=\tau\,d^\ast(u^\ast\om_{\S^2})\qquad\mbox{ in }\B^3_1(0)\\[5mm]
\ds\iota_{\p \B^3}^\ast\ast(d\eta-u^\ast\om)=0
\end{array}
\rg.
\]
and
\[
u\times d^\ast du-\tau\,d^\ast(u^\ast\om_{\S^2}-d\eta)\cdot u\times\lf(u\times du\rg)=0
\]
This gives
\[
d^\ast (u\times du)+\tau\, du\wedge\,d\ast((u^\ast\om_{\S^2})-d\eta)=0
\]
Let $G:=u^\ast\om_{\S^2}-d\eta$, we have
\[
\tau\,d^\ast G=\eta
\]
and
\[
\tau\,dd^\ast G=d\eta=-G+u^\ast\om_{\S^2}
\]
This concludes the proof  of Proposition~\ref{pr-crit-refa}.\hfill $\Box$

\subsubsection{Gauge change in ${\mathcal R}_\tau$.}

\begin{Lm}
\label{lm-gaucha}
Denote
\[
A(u,\eta):= -\frac{1}{4}\,[u,du]+\frac{\eta}{2}\, u\ .
\]
With this notation one has
\[
F_{A(u,\eta)}:=\frac{1}{2}\,\lf(d\eta-u^\ast\om\rg)\ u
\]
and there holds
\[
|A(u,\eta)|^2+\tau\,|F_{A(u,\eta)}|^2=\frac{1}{4}\lf(|du|^2+|\eta|^2+\tau\,\lf|d\eta-u^\ast\om\rg|^2\rg)
\]
Moreover for any map $g$ into $SU(2)$ there holds
\be
\label{gaucha}
(A(u,\eta))^g=A(u^g, \eta^g)\
\ee
where
\[
u^g:=g^{-1}\,u\,g\qquad\mbox{ and }\qquad\eta^g:=\eta+2\,<dg\,g^{-1},u>\ .
\]
\hfill $\Box$
\end{Lm}
\noindent{\bf Proof of Lemma~\ref{lm-gaucha}}
The first identities have been proved above. It remains to prove the gauge change formula. Let $g$ be a map into $SU(2)$, there holds
\[
(A(u,\eta))^g=g^{-1}\,A(u,\eta)\, g+g^{-1}\,dg= -\frac{1}{4}\,g^{-1}[u,du]\,g+\frac{\eta}{2}\, g^{-1}\,u g+g^{-1}\,dg\ .
\]
Denote $u^g:=g^{-1}\,u\,g$ and observe that
\[
\begin{array}{l}
\ds [u^g,du^g]=[g^{-1}\,u\,g, g^{-1}\,du\,g+dg^{-1}\, u\, g+g^{-1}\,u\, dg]=g^{-1}[u,du]\,g+[u^g, [u^g,g^{-1}\,dg]]\\[5mm]
\ds=g^{-1}[u,du]\,g-4\, g^{-1}\,dg+4 <g^{-1}\,dg,u^g>\ u^g
\end{array}
\]
Thus
\[
(A(u,\eta))^g=-\frac{1}{4}\,[u^g,du^g]+\frac{1}{2}\,(\eta+2\,<g^{-1}\,dg,u^g> )\, u^g=A(u^g, \eta+2\,<dg\,g^{-1},u>)\ .
\]
This concludes the proof of Lemma~\ref{lm-gaucha}.\hfill $\Box$
\subsubsection{Monotonicity for Stationary critical Points of ${\mathcal R}_\tau$.}

\begin{Lm}
\label{lm-moncalR}
Let $u\in C^1(\B^3,\S^2)$ and $\eta\in C^1(\Om_1(\B^3))$ be a  critical point to ${\mathcal R}_\tau$. Then  for any $x_0\in \B_1^3(0)$ and $r>0$ such that $\B_r(x_0)\subset \B_1(0)$
\be
\label{moncalR}
\begin{array}{l}
\ds\frac{d}{dr}\lf(\frac{1}{r}\int_{\B_r(x_0)}|d u|^2+|\eta|^2\ dx^3+\lf(\frac{2}{|x-x_0|}-1\rg)\tau\,|d\eta-u^\ast\om_{\S^2}|^2\ dx^3\rg)\\[5mm]
\ds=\frac{2}{r}\int_{\p \B_r}\lf|\p_ru\rg|^2+\lf|\eta\res\p_r\rg|^2\ dvol_{\p \B_r}+\frac{2}{r}\, \int_{\B_r}\lf|(d\eta-u^\ast\om_{\S^2})\res\p_r\rg|^2\ dvol_{\p \B_r}
\end{array}
\ee
\hfill $\Box$
\end{Lm}
\noindent{\bf Proof of Lemma~\ref{lm-moncalR}.} Since $(u,\eta)$ is assumed to be smooth and critical for ${\mathcal R}_\tau$ it fulfils the stationarity condition. Let $X\in C^\infty_0({\B^3},{\R}^3)$ and denote $\varphi_t(x):=x+t\,X(x)$. We have 
\[
\lf.\frac{d}{dt}\lf(\varphi_t^\ast \eta\rg)\rg|_{t=0}=\lf.\frac{d}{dt}\lf(\varphi_t^\ast \lf(\sum_{l=1}^3 \eta_l\ dx_l\rg)\rg)\rg|_{t=0}=\sum_{k=1}^3\sum_{l=1}^3\frac{\p \eta_l}{\p x_k}\ X_k\ dx_l+\sum_{k=1}^3\sum_{l=1}^3\eta_l\frac{\p X_l}{\p x_k}\ dx_k
\]
moreover
\[
\begin{array}{l}
\ds\lf.\frac{d}{dt}\lf(d(u\circ\varphi_t\rg)\rg|_{t=0}=\lf.\frac{d}{dt}\lf(\varphi_t^\ast du\rg)\rg|_{t=0}=\lf.\frac{d}{dt}\lf(\varphi_t^\ast \lf(\sum_{l=1}^3\p_{x_l}u\, dx_l\rg)\rg)\rg|_{t=0}\\[5mm]
\ds=\sum_{k=1}^3\sum_{l=1}^3\frac{\p}{\p x_k}\lf(\p_{x_l}u\rg)\ X_k\ dx_l+\sum_{k=1}^3\sum_{l=1}^3(\p_{x_l}u) \frac{\p X_l}{\p x_k}\ dx_k
\end{array}
\]
and, denoting $F_A:=2^{-1}\,(d\eta-u^\ast\om_{\S^2})$ we have
\[
\begin{array}{l}
\ds\frac{d}{dt}\lf(\phi_t^\ast F_A\rg)=\frac{d}{dt}\lf(\phi_t^\ast\sum_{l<m} (\p_{x_l}A_m-\p_{x_m}A_l+[A_l,A_m])\ dx_l\wedge dx_m\rg)\\[5mm]
\ds=\sum_{k=1}^3\sum_{l<m} \frac{\p}{\p x_k}(\p_{x_l}A_m-\p_{x_m}A_l+[A_l,A_m])\ X_k\ dx_l\wedge dx_m\\[5mm]
\ds\quad+\sum_{k=1}^3\sum_{l<m} F_{lm}\, \lf(\frac{\p X_l}{\p x_k}\ dx_k\wedge dx_m+\frac{\p X_m}{\p x_k}\ dx_l\wedge dx_k\rg)
\end{array}
\]
Following step by step the same arguments as in the proof of Lemma~\ref{lm-mono} we obtain for any $x_0\in \B_1^3(0)$ and $r>0$ such that $\B_r(x_0)\subset \B_1(0)$
\[
\begin{array}{l}
\ds-\frac{1}{2}\int_{\B_r}|d u|^2+|\eta|^2\ dx^3+\frac{r}{2}\int_{\p \B_r}|d u|^2+|\eta|^2\ \ dvol_{\p \B_r}\\[5mm]
\ds+\frac{1}{2}\int_{\B_r}\tau\,|d\eta-u^\ast\om|^2\ dx^3+\frac{r}{2}\int_{\p \B_r}\tau\,|d\eta-u^\ast\om_{\S^2}|^2\ dvol_{\p \B_r}\\[5mm]
\ds=r\,\int_{\p \B_r}\lf|\p_ru\rg|^2\ dvol_{\p B_r}+r\,\int_{\p \B_r}\lf|\eta\res\p_r\rg|^2\ dvol_{\p \B_r}+r\, \int_{\p\B_r}\tau\,\lf|(d\eta-u^\ast\om_{\S^2})\res\p_r\rg|^2\ dvol_{\p \B_r}
\end{array}
\]
This gives lemma~\ref{lm-moncalR}.\hfill $\Box$

\subsection{$\delta-$regularity for minimizers of ${\mathcal R}_\tau$ independent of $\tau\in(0,1)$}

We first establish the following Lemma.
\begin{Lm}
\label{lm-montau}
Let $0<r_1<1$. There exists $\delta_0>0$ and $R>1$ independent of $\tau\in (0,1)$ but depending only on $r_1$ such that for any minimizer $(u,\eta)$ of ${\mathcal R}_\tau$ and any $\B_r(x_0)\subset\B_{r_1}(0)$ and $r>R\,\sqrt{\tau}$ if
\be
\label{smuet-1}
\frac{1}{r}\int_{\B_r(x_0)} |du|^2+|\eta|^2+\tau\,|d\eta -u^\ast\om|^2\ dx^3<\delta
\ee
 for any $\delta<\delta_0$. Then, for any $R\sqrt{\tau}<\rho<r$ there holds
\be
\label{smuet-2}
\frac{1}{\rho}\int_{\B_\rho(x_0)} |du|^2+|\eta|^2+\tau\,|d\eta -u^\ast\om|^2\ dx^3<2\,\delta
\ee
\hfill $\Box$
\end{Lm}
Let $\delta>0$ to be fixed later on. Let $x_0\in \B_1(0)$ and $\B_r(x_0)\subset \B_1(0)$. Assume
\[
\int_{\p \B_r(x_0)}|du|^2+|\eta|^2+\tau\,|d\eta -u^\ast\om|^2\ dvol_{\p \B_r(x_0)}<\delta
\]
We assume $\delta$ small enough so that one can apply Lemma~\ref{lm-ext} to $u$ and we call $\ti{u}$ the harmonic extension of Thanks to \cite{Hel} Lemma 5.1.4, if $\delta<8\pi/3$ then there exists $e_1,e_2\in W^{1,2}(\p\B_r(x_0), \S^2)$ such that $e_1\cdot e_2=0$ and $u=e_1\times e_2$ moreover
\[
\int_{\p \B_r(x_0)} |\nabla_T e_i|^2\ dvol_{\p  \B_r(x_0)}\le C\, \int_{\p \B_r(x_0)} |\nabla_T u|^2\ dvol_{\p  \B_r(x_0)}
\]
and it satisfies $d^{(\ast)}(e_1\cdot de_2)=0$. We have on $\p\B_r(x_0)$, $u^\ast\om_{\S^2}=de_1\wedge de_2$. Let $v\in W^{1,2}(\p\B_r(x_0), \S^2)$ be the solution of average $0$ on $\p\B_r(x_0)$ of
\[
d(\ast)\, dv=u^\ast\om_{\S^2}=de_1\wedge de_2\qquad\mbox{ on }\p \B_r(x_0)\ .
\]
We have
\[
\int_{\p \B_r(x_0)} |\nabla_T v|^2\ dvol_{\p  \B_r(x_0)}\le C\, \lf(\int_{\p \B_r(x_0)} |\nabla_T u|^2\ dvol_{\p  \B_r(x_0)}\rg)^2
\]
where $C>0$ is universal. Let $\frak{f}$ be the 1-form solving
\[
\lf\{
\begin{array}{l}
d\frak{f}=\ti{u}^\ast\om\qquad\mbox{ in }\B_r(x_0)\\[5mm]
d^\ast\frak{f}=0\qquad\mbox{ in }\B_r(x_0)\\[5mm]
\iota_{\p\B_r(x_0)}^\ast \frak{f}=(\ast) dv
\end{array}
\rg.
\]
Calderon Zygmund estimates are giving
\[
\begin{array}{rl}
\ds\| d\frak{f}\|_{L^{3/2}(\B_r(x_0))}+\|\frak{f}\|_{L^3(\B_r(x_0))}&\ds\le C\, \|\ti{u}^\ast\om\|_{L^{3/2}(\B_r(x_0))}+C\, \|dv\|_{L^2(\p\B_r(x_0))}\\[5mm]
\ds\qquad&\ds\le C\int_{\p \B_r(x_0)}|\nabla_Tu|^2\ dvol_{\p \B_r(x_0)}
\end{array}
\]
Let $\xi$ be the one form minimizing
\[
\inf\lf\{  \int_{\B_r(x_0)}\tau\,|d\xi|^2+|\xi|^2\ dx^3\quad;\quad d^\ast\xi=0\quad\mbox{ and }\quad \iota_{\p \B_r(x_0)}^\ast\xi=\eta-(\ast) dv \rg\}
\]
It satisfies
\be
\label{xi}
\lf\{
\begin{array}{l}
\ds \tau d^\ast d\xi+\xi=0\qquad\mbox{ in }\quad \B_r(x_0)\\[5mm]
\ds d^\ast\xi=0\qquad\mbox{ in }\quad \B_r(x_0)\\[5mm]
\ds  \iota_{\p \B_r(x_0)}^\ast\xi=\eta-(\ast) dv
\end{array}
\rg.
\ee
and we choose to extend $\eta$ by $\ti{\eta}:=\xi+\frak{f}$. Multiplying the first equation of (\ref{xi}) by $\xi$ and integrating by parts is giving for any $t>0$ that we are going to fix later on
\be
\label{idxi}
\begin{array}{l}
\ds\int_{\B_r(x_0)}\tau\,|d\xi|^2+|\xi|^2\ dx^3=\tau\,\int_{\p \B_r(x_0)} \xi\wedge \ast d\xi=\tau\,\int_{\p \B_r(x_0)} (\eta-(\ast) dv)\,\wedge\, \ast d\xi\\[5mm]
\ds\le t^2\, \frac{\tau}{2}\, \int_{\p \B_r(x_0)} \lf|\iota^\ast_{\p \B_r(x_0)}(\eta-(\ast) dv)\rg|^2 \ dvol_{\p \B_r(x_0)}+t^{-2}\, \frac{\tau}{2}\, \int_{\p \B_r(x_0)} \lf|d\xi\res\p_r\rg|^2 \ dvol_{\p \B_r(x_0)}
\end{array}
\ee
Using the corresponding identity to (\ref{monet}) we have
\[
\begin{array}{l}
\ds-\frac{1}{2}\int_{\B_r(x_0)}{|\xi|^2}\ dx^3+\frac{r}{2}\int_{\p \B_r(x_0)}{|\iota^\ast_{\p \B_r(x_0)}\xi|^2}\ dvol_{\p \B_r(x_0)}\\[5mm]
\ds+\frac{\tau}{2}\int_{\B_r(x_0)}\,{|d\xi|^2}\ dx^3+\frac{r}{2}\,\tau\,\int_{\p \B_r(x_0)}\,{|\iota^\ast_{\p \B_r(x_0)}d\xi|^2}\ dvol_{\p B_\rho(x_0)}\\[5mm]
\ds=\frac{r}{2}\int_{\p \B_r(x_0)}{|\xi(\p_r)|^2}\ dvol_{\p \B_r(x_0)}+ \frac{r}{2}\,\tau\,\int_{\p \B_r(x_0)}\,{|d\xi\res\p_r|^2}\ dvol_{\p \B_r(x_0)}
\end{array}
\]
This gives
\be
\label{pohoxi}
\begin{array}{l}
\ds t^{-2}\, \frac{\tau}{2}\, \int_{\p \B_r(x_0)} \lf|d\xi\res\p_r\rg|^2 \ dvol_{\p \B_r(x_0)}\le\, t^{-2}\, \frac{1}{2}\int_{\p \B_r(x_0)}{|\iota^\ast_{\p \B_r(x_0)}(\eta-(\ast) dv)|^2}\ dvol_{\p \B_r(x_0)}\\[5mm]
\ds \quad+ t^{-2}\,\frac{\tau}{2 r}\, \int_{\B_r(x_0)}\,{|d\xi|^2}\ dx^3+ t^{-2}\,\frac{\tau}{2}\int_{\p \B_r(x_0)}\,{|\iota^\ast_{\p \B_r(x_0)}(d\eta-u^\ast\om)|^2}\ dvol_{\p B_\rho(x_0)}
\end{array}
\ee
Combining (\ref{idxi}) and (\ref{pohoxi}) is giving
\[
\begin{array}{l}
\ds\lf(1-t^{-2}\,\frac{1}{2 r}\rg)\, \int_{\B_r(x_0)}\tau\,|d\xi|^2+|\xi|^2\ dx^3\le  \frac{1}{2}\,\lf(t^2\, {\tau}+t^{-2}\,\rg)\, \int_{\p \B_r(x_0)} \lf|\iota^\ast_{\p \B_r(x_0)}(\eta-(\ast) dv)\rg|^2 \ dvol_{\p \B_r(x_0)}\\[5mm]
\ds\qquad+ t^{-2}\,\frac{\tau}{2}\int_{\p \B_r(x_0)}\,{|\iota^\ast_{\p \B_r(x_0)}(d\eta-u^\ast\om)|^2}\ dvol_{\p B_\rho(x_0)}
\end{array}
\]
We choose $t^2=r/\tau$ and we restrict to $r^2>2\,\tau$ and we obtain
\be
\label{bdxi}
\begin{array}{l}
\ds \frac{3}{4}\int_{\B_r(x_0)}\tau\,|d\xi|^2+|\xi|^2\ dx^3\le\frac{r}{2}\,\lf(1+\frac{\tau}{r^2}\rg)\, \int_{\p \B_r(x_0)} \lf|\iota^\ast_{\p \B_r(x_0)}(\eta-(\ast) dv)\rg|^2 \ dvol_{\p \B_r(x_0)}\\[5mm]
\ds+\frac{\tau}{2\,r^2}\,r\,\int_{\p \B_r(x_0)}\,\tau\,{|\iota^\ast_{\p \B_r(x_0)}(d\eta-u^\ast\om)|^2}\ dvol_{\p B_\rho(x_0)}
\end{array}
\ee
We have 
\[
\begin{array}{l}
\ds\int_{\B_r(x_0)}|d\ti{u}|^2+|\ti{\eta}|^2+\tau\,|d\ti{\eta}-\ti{u}^\ast\om_{\S^2}|^2\ dx^3\le r\,(1-\nu)\int_{\p \B_r(x_0)}\,|du|^2\ dvol_{\p B_\rho(x_0)}\\[5mm]
\ds + \int_{\B_r(x_0)}|\xi+\frak{f}|^2+\tau\,|d\xi|^2\ dx^3
\end{array}
\]
We bound
\[
\begin{array}{l}
\ds\int_{\B_r(x_0)}|\xi+\frak{f}|^2+\tau\,|d\xi|^2\ dx^3\le \frac{6}{5}\int_{\B_r(x_0)}|\xi|^2+\tau\,|d\xi|^2\ dx^3+6\, \int_{\B_r(x_0)}|\frak{f}|^2\ dx^3\\[5mm]
\ds\le \frac{4}{5}\, \lf(1+\frac{\tau}{r^2}\rg)\, r\,\int_{\p \B_r(x_0)} \lf|\iota^\ast_{\p \B_r(x_0)}(\eta-(\ast) dv)\rg|^2 \ dvol_{\p \B_r(x_0)}\\[5mm]
\ds\quad+\frac{4}{5}\,\frac{\tau}{r^2}\,r\,\int_{\p \B_r(x_0)}\,\tau\,{|\iota^\ast_{\p \B_r(x_0)}(d\eta-u^\ast\om)|^2}\ dvol_{\p B_\rho(x_0)}
+C\, r\, \lf(\int_{\p \B_r(x_0)}\,|du|^2\ dvol_{\p B_\rho(x_0)}\rg)^2
\end{array}
\]
Combining the two previous identities and denoting
\[
Y(r):=\int_{\B_r(x_0)}|d{u}|^2+|{\eta}|^2+\tau\,|d\eta-{u}^\ast\om_{\S^2}|^2\ dx^3
\]
using the minimality, whenever $Y'(r)<\delta$ there holds
\be
\label{Ydec}
Y(r)\le \lf(\max\lf\{1-\nu\,,\,\frac{4}{5}\rg\}+C\, \delta+\frac{4}{5}\,\frac{\tau}{r^2}\,\rg)\, r\,Y'(r)
\ee
Then, arguing as in the end of the proof  of  Lemma~\ref{lm-sca} we finally conclude the proof of Lemma~\ref{lm-montau}.\hfill $\Box$

\medskip

\begin{Lm}
\label{lm-mont}
Let $0<r_1<1$. There exists $\delta_0>0$ and $C>0$ independent of $\tau\in (0,1)$ but depending only on $r_1$ such that for any minimizer $(u,\eta)$ of ${\mathcal R}_\tau$ and any $\B_r(x_0)\subset\B_{r_1}(0)$ if
\be
\label{smuet-3}
\frac{1}{r}\int_{\B_r(x_0)} |du|^2+|\eta|^2+\tau\,|d\eta -u^\ast\om|^2\ dx^3<\delta
\ee
 for any $\delta<\delta_0$, then for any  $0<\rho<r$ there holds
\be
\label{smuet-4}
\frac{1}{\rho}\int_{\B_\rho(x_0)} |du|^2+|\eta|^2+\tau\,|d\eta -u^\ast\om|^2\ dx^3<C\,\delta
\ee
\hfill $\Box$
\end{Lm}
{\bf Proof of Lemma~\ref{lm-mont}.} Because of the previous lemma we can assume $\rho<\delta$ and
\[
\frac{1}{\sqrt{\tau}}\int_{\B_{\sqrt{\tau}}(x_0)}\tau\,|d\eta-u^\ast \om_{\S^2}|^2\ dx^3<\delta
\]
Then we have
\[
\|F_{A(u,\eta)}\|_{L^{3/2}(\B_{\sqrt{\tau}}(x_0)}\le C\, \tau^{1/4}\, \|F_{A(u,\eta)}\|_{L^{2}(\B_{\sqrt{\tau}}(x_0)}\le C\,\sqrt{\delta}\ .
\]
We keep in mind that $A(u,\eta)$ is a connection form associated to a pair $(u,\eta)$ and we  simply write $A$ for $A(u,\eta)$. We have
\[
F_A=2^{-1}\,G\, u
\]
where $G=u^\ast\om-d\eta$. Then we have the following system
\[
\lf\{
\begin{array}{l}
\ds dG=0\qquad\mbox{ in }\B_{\sqrt{\tau}}(x_0)\\[5mm]
\ds d^\ast G=\frac{1}{\tau}\eta\qquad\mbox{ in }\B_{\sqrt{\tau}}(x_0)\ .
\end{array}
\rg.
\]
Classical elliptic estimates give
\[
\begin{array}{l}
\ds\|G\|_{L^{(6,2)}(\B_{\sqrt{\tau}}(x_0))}+\|\nabla G\|_{L^2(\B_{\sqrt{\tau}}(x_0))}\le C\, \tau^{-1}\, \|\eta\|_{L^2(\B_{2\,\sqrt{\tau}}(x_0))}+C\, \sqrt{\tau}^{-1}\, \|G\|_{L^2(\B_{2\,\sqrt{\tau}}(x_0))}\\[5mm]
\ds\quad\le C\,{\tau}^{-3/4}\ \lf(\frac{1}{\sqrt{\tau}}\, \int_{\B_{2\,\sqrt{\tau}}(x_0)}|\eta|^2+\tau\,|u^\ast\om-d\eta|^2\ dx^3\rg)^{1/2}
\end{array}
\]
Interpolation estimates is giving
\[
\begin{array}{l}
\ds\|G\|^2_{L^{(3,2)}(\B_{\sqrt{\tau}}(x_0))}\ds\le \|G\|_{L^{2}(\B_{\sqrt{\tau}}(x_0))}\ \|G\|_{L^{(6,2)}(\B_{\sqrt{\tau}}(x_0))}\\[5mm]
\ds \quad\le C\, \tau^{-1}\, \frac{1}{\sqrt{\tau}}\, \int_{\B_{2\,\sqrt{\tau}}(x_0)}|\eta|^2+\tau\,|u^\ast\om-d\eta|^2\ dx^3
\end{array}
\]
Hence we have
\be
\label{estG}
\begin{array}{l}
\ds\int_{\B_{\sqrt{\tau}}(x_0))}\frac{1}{|x-x_0|}\,\tau\,|G|^2\ dx^3\le \lf\|\frac{1}{|x-x_0|}\rg\|_{L^{(3,\infty)}(\B_{\sqrt{\tau}}(x_0))}\, \tau\,\||G|^2\|_{L^{(3/2,1)}(\B_{\sqrt{\tau}}(x_0))}\\[5mm]
\ds\qquad\le C\,\tau \|G\|^2_{L^{(3,2)}(\B_{\sqrt{\tau}}(x_0))}\le \frac{C}{\sqrt{\tau}}\, \int_{\B_{2\,\sqrt{\tau}}(x_0)}|\eta|^2+\tau\,|u^\ast\om-d\eta|^2\ dx^3\le C\,\delta
\end{array}
\ee
The stationarity condition (\ref{moncalR}) consequence of the minimality of $(u,\eta)$ with respect to ${\mathcal R}_\tau$ is giving the identity
\[
\begin{array}{l}
\ds-\frac{1}{2}\int_{\B_\rho(x_0)}{|du|^2+|\eta|^2}\ dx^3+\frac{\rho}{2}\int_{\p \B_\rho(x_0)}{|\iota^\ast_{\p \B_\rho(x_0)}du|^2+|\iota^\ast_{\p \B_\rho(x_0)}\eta|^2}\ dvol_{\p \B_\rho(x_0)}\\[5mm]
\ds+\frac{\tau}{2}\int_{\B_\rho(x_0)}\,{|G|^2}\ dx^3+\frac{\rho}{2}\,\tau\,\int_{\p \B_\rho(x_0)}\,{|\iota^\ast_{\p \B_\rho(x_0)}G|^2}\ dvol_{\p B_\rho(x_0)}\\[5mm]
\ds=\frac{\rho}{2}\int_{\p \B_\rho(x_0)}{|\p_ru|^2+|\eta(\p_r)|^2}\ dvol_{\p \B_\rho(x_0)}+ \frac{\rho}{2}\,\tau\,\int_{\p \B_\rho(x_0)}\,{|G\res\p_r|^2}\ dvol_{\p \B_\rho(x_0)}
\end{array}
\]
from which we deduce
\be
\label{moncalR}
\frac{d}{d\rho}\lf( \frac{1}{\rho}\int_{\B_\rho(x_0)}  |du|^2+|\eta|^2 dx^3+\frac{1}{\rho}\int_{\B_\rho(x_0)}\lf(\frac{2\,\rho}{|x-x_0|}-1\rg)\,\tau\,|G|^2\ dx^3 \rg)\ge 0\ .
\ee
Combining (\ref{estG}) and (\ref{moncalR}) is giving
\[
\forall \rho\le \sqrt{\tau}\qquad  \frac{1}{r}\int_{\B_r(x_0)}  |du|^2+|\eta|^2 dx^3\le C\,\delta\ .
\]
This concludes the proof of Lemma~\ref{lm-mont}.\hfill $\Box$
\begin{Th}
\label{th-delregt}
Let $0<r_1<1$. There exists $\delta_0>0$ and $C>0$ independent of $\tau\in (0,1)$ but depending only on $r_1$ such that for any minimizer $(u,\eta)$ of ${\mathcal R}_\tau$ and any $\B_r(x_0)\subset\B_{r_1}(0)$ if
\be
\label{smuet-5}
\frac{1}{r}\int_{\B_r(x_0)} |du|^2+|\eta|^2+\tau\,|d\eta -u^\ast\om|^2\ dx^3<\delta
\ee
Then  
\be
\label{concth}
\|\nabla\eta\|_{L^\infty(\B_{r/2}(x_0))}+\|\nabla u\|_{L^\infty(\B_{r/2}(x_0))}\le C_{r}\ .
\ee
where $C_r>0$ only depends on $r$ and not on $\tau$.
\hfill $\Box$
\end{Th}
\noindent{\bf Proof of Theorem~\ref{th-delregt}.} Equation~\ref{antisy} is implying
\be
\label{sys-4}
\Delta u^i= du^{i+1}\wedge d B^{i-1}-du^{i-1}\wedge d B^{i+1}
\ee
This system can be recasted in a Schr\"odinger system with antisymmetric potential $-\Delta u=\Om\cdot\nabla u$ and the smallness assumption (\ref{smuet-5}) together with  Lemma~\ref{lm-mont} we obtain
\[
\sup_{x_1\in \B_{r/2}(x_0)}\ \sup_{\rho<r/4}\frac{1}{\rho}\int_{\B_\rho(x_1)}|du|^2+|\Om|^2\ dx^3\le C\,\delta\ .
\]
Applying the main theorem of \cite{RivSt} we obtain that there exists $C_r>0$ only depends on $r$ and not on $\tau$ such that for any $p<2$
\[
\|\nabla u\|_{W^{1,p}(\B_{r/2}(x_0))}\le C_r
\]
This implies in particular
\[
\|\nabla u\|_{L^4(\B_{r/2}(x_0))}\le C_r
\]
Let $\chi_r\in C^0_\infty(\B_{r/2}(x_0),{\R}_+)$ such that $\chi_r\equiv 1$ on $\B_{r/4}(x_0)$. One multiplies the London equation (\ref{lond}) by $\chi_r\, G$. This gives, integrating by parts and using the fact that
$dG=0$ which is implying $-\Delta G=dd^\ast G$
\[
\begin{array}{l}
\ds\int_{\B_{r/2}(x_0)}\chi_r\,\tau\,|d^\ast G|^2+\chi_r\, |G|^2\ dx^3= \int_{\B_{r/2}(x_0)}\chi_r\ G\cdot\,u^\ast\om-\tau \int_{\B_{r/2}(x_0)}\nabla \chi_r\, G\cdot\nabla G\ dx^3\\[5mm]
\ds\quad\le \frac{1}{2}\int_{\B_{r/2}(x_0)}\chi_r\,|G|^2\ dx^3+ \frac{1}{2}\, \int_{\B_{r/2}(x_0)}\chi_r\,|u^\ast\om|^2\ dx^3+ \frac{1}{2}\,\int_{\B_{r/2}(x_0)}\Delta \chi_r\, \tau\,|G|^2\ dx^3
\end{array}
\]
Since 
\[
\int_{\B_1^3(0)}\tau\,|G|^2\, dx^3\le C\, \inf\lf\{{\mathcal R}_\tau(u,\eta)\ ;\ u=\phi\ \mbox{ on }\p\B_1(0)\rg\}\ ,
\]
we obtain that there exists a constant $C_r$ independent of $\tau$ such that
\[
\int_{\B_{r/4}(x_0)}\tau\,|d^\ast G|^2+ |G|^2\ dx^3\le C_r
\]
Since
\[
\lf\{
\begin{array}{l}
\ds d(\tau d^\ast G)=-G+u^\ast\om\qquad\mbox{ in }\B_{r/4}(x_0)\\[5mm]
\ds d^\ast(\tau d^\ast G)=0
\end{array}
\rg.
\]
we have
\[
\|\tau d^\ast G\|_{L^6(\B_{r/8}(x_0)}\le C\,\|G\|_{L^2(\B_{r/4}(x_0))}+C\,\|u^\ast\om\|_{L^2(\B_{r/4}(x_0))}+C\, r^{-1}\|\tau d^\ast G\|_{L^2(\B_{r/8}(x_0)}\
\]
Recall that
\[
dB=\ast\lf(u\times d u\rg)-\tau \,d\ast G\ u \ .
\]
We deduce that the $L^4$ norm of $dB$ is uniformly bounded in $\B_{r/8}(x_0)$ which permits to bootstrap further in the equation (\ref{sys-4}) and after a finite number of steps we obtain (\ref{concth}) and Theorem~\ref{th-delregt} is proved.
\hfill $\Box$

\appendix
\section{Appendix}
\reset
\subsection{Controlled $W^{1,3}(\B^3,{\S}^3)$ harmonic map extensions under small Dirichlet Energy Assumption at the Boundary}
The goal of this subsection is to establish the following Lemma
\begin{Lma}
\label{lm-ext} {\bf [Harmonic Map extensions with Epiperimetric Inequality]}
There exists $\delta>0$ such that for any map $u\in W^{1,2}(\p \B^3,\S^2)$ satisfying
\be
\label{A-sma}
\int_{\p\B^3}|\nabla_Tu|^2\ dvol_{\p \B^3}<\delta
\ee
where $\nabla_T$ denotes the tangential derivative of $u$ along $\p \B^3$, then there exists $\ti{u}\in W^{1,3}(\B^3,{\S}^2)$ satisfying
\begin{itemize}
\item[i)]
\[
\ti{u}=u\qquad\mbox{ on }\p \B^3\ ,
\]
\item[ii)]
\be
\label{harm-1}
-\Delta\ti{u}=\ti{u}\,|\nabla\ti{u}|^2\qquad\mbox{ in }{\mathcal D}'(\B^3)\ ,
\ee
\item[iii)]
\be
\label{A-03}
\|\nabla\ti{u}\|_{L^3(\B^3)}\le C\, \|\nabla_T u\|_{L^2(\p \B^3)}\ ,
\ee
where $C>0$ is a universal constant and
\item[iv)]
\be
\label{A-04}
\int_{\B^3}|\nabla \ti{u}|^2\ dx^3\le (1-\nu)\ \int_{\p \B^3}|\nabla_T {u}|^2\ dvol_{\p \B^3}\ .
\ee
where $\nu\in(0,1)$ is universal.
\end{itemize}
\hfill $\Box$
\end{Lma}
We will first prove the existence of an $\S^3$ valued harmonic map extension of $u$ satisfying i)...iv). To that aim we shall be using a continuity method introduced in \cite{PeR}. First we establish the following lemma which is an a-priori estimate.
\begin{Lma}
\label{lm-ap}
There exists $\delta>0$  and $C>0$ such that for any harmonic map $g\in H^{3/2}(\B^3,\S^3)$ satisfying 
\begin{itemize}
\item[i)]
\[
d^\ast\lf(g^{-1}\,dg\rg)=0\quad\mbox{ in }{\mathcal D}'(\B^3)
\]
\item[ii)]
\[
 \|dg\|_{H^{1/2}(\B^3)}<\delta
\]
\end{itemize}
then
\be
\label{A-1}
\|dg\|_{H^{3/2}}\le C\, \|\iota_{\p \B^3}^\ast d g\|_{L^2(\p B^3)}\ .
\ee
Moreover, if $g\in W^{2,2}(\B^3)$, for the same $\delta>0$ there holds
\be
\label{A-01}
\|dg\|_{W^{1,2}}\le C\, \|\iota_{\p \B^3}^\ast d g\|_{H^{1/2}(\p B^3)}\ .
\ee
\hfill $\Box$
\end{Lma}
\noindent{\bf Proof of Lemma~\ref{lm-ap}} We first consider the minimizer of the following problem considered also in (\ref{min})
\[
\inf\lf\{  \int_{\B^3}|g^{-1}\,dg-d^\ast\zeta |^2\ dx^3 ;\ d\zeta=0\quad\mbox{ and }\quad \iota_{\p B^3}^\ast \ast\zeta=0 \rg\}
\]
The existence of a unique minimizer is explained in section III. Moreover in the same section we proved that
\[
d\lf(g^{-1}\,dg-d^\ast\zeta\rg)=0
\]
Hence there exists $\sigma\in W^{1,2}(\B^3,\frak{su}(2))$ such that
\[
g^{-1}\,dg=d\sigma+d^\ast\zeta
\]
$\zeta$ satisfies
\[
\lf\{
\begin{array}{l}
\ds d d^\ast\zeta=dg^{-1}\wedge d g\quad\mbox{ in }{\mathcal D}'(\B^3)\\[5mm]
\ds d\zeta=0\quad\mbox{ in }{\mathcal D}'(\B^3)\\[5mm]
\ds\iota_{\p B^3}^\ast\ast\zeta=0
\end{array}
\rg.
\]
Classical elliptic theory gives then using hypothesis ii)
\[
\|\nabla^2\zeta\|_{L^{3/2}(\B^3)}+\|\nabla \zeta\|_{L^{3}(\B^3)}\le C\, \|dg\|_{L^3(\B^3)}^2\le C\, \delta\,\|dg\|_{H^{1/2}(\B^3)}
\]
Then we have that $d^\ast\zeta$ is a $W^{1,3/2}(\B^3,\frak{su}(2))$. Using trace theorem $W^{1,3/2}(\B^3)\ \hookrightarrow\ W^{1-2/3,3/2}(\p\B^3)$ we have
\[
\|\iota_{\p B^3}^\ast d\zeta\|_{W^{1/3,3/2}(\p\B^3)}\le C\, \delta\,\|dg\|_{H^{1/2}(\B^3)}
\]
Using Sobolev embedding $W^{1/3,3/2}(\p\B^3)\ \hookrightarrow\ L^2(\B^3)$ we have
\[
\|\iota_{\p B^3}^\ast d\zeta\|_{ L^2(\B^3)}\le C\, \delta\,\|dg\|_{H^{1/2}(\B^3)}\ .
\]
The function $\sigma$ is harmonic in $\B^3$. Hence we have
\[
\begin{array}{l}
\ds\|d\sigma\|_{H^{1/2}(\B^3)}\le C\, \|\iota_{\p B^3}^\ast d\sigma\|_{L^2(\B^3)}\le C\, \|\iota_{\p B^3}^\ast  dg\|_{ L^2(\B^3)}+\|\iota_{\p B^3}^\ast d\zeta\|_{ L^2(\B^3)}\\[5mm]
\ds\quad\le C\, \|\iota_{\p B^3}^\ast  dg\|_{ L^2(\B^3)}+ C\, \delta\,\|dg\|_{H^{1/2}(\B^3)}\ .
\end{array}
\]
Using lemma~\ref{lm-in} to  $a=$ and $b=g^{-1}\p_{x_k}g$ for any $k=1,2,3$ we obtain using the hypothesis ii)
\[
\begin{array}{l}
\ds\|dg\|_{H^{1/2}(\B^3)}\le C\ (1+\|g\|_{\dot{W}^{1/2,6}(\B^3)}\ )\ \|g^{-1}dg\|_{H^{1/2}(\B^3)}\\[5mm]
\ds\qquad\le C\ (1+\|dg\|_{H^{1/2}(\B^3)})\ \lf(\|d\sigma\|_{H^{1/2}(\B^3)}+\|d^\ast\zeta\|_{H^{1/2}(\B^3)}\rg)\\[5mm]
\ds\qquad\le C\ (1+\|dg\|_{H^{1/2}(\B^3)})\ \lf(\|\iota_{\p B^3}^\ast  dg\|_{ L^2(\B^3)}+ C\, \delta\,\|dg\|_{H^{1/2}(\B^3)}\rg)
\end{array}
\]
Using the smallness assumption ii), for $\delta$ small enough we deduce (\ref{A-1}). The proof of (\ref{A-01}) is following the same lines as the proof of (\ref{A-1}) by taking the stronger norms 
$W^{1,2}(\B^3)$ for $dg$ and $H^{1/2}(\p B^3)$ for the restriction to $\p B^3$ of $dg$, still assuming the smallness of $\|dg\|_{H^{1/2}(\B^3)}$ and lemma~\ref{lm-ap} is proved.\hfill $\Box$

\medskip

\noindent{\bf Proof of Lemma~\ref{lm-ext}} First we prove the counterpart to the lemma~\ref{lm-ext} for $\S^3$ valued map. 

\medskip

For $\delta>0$ and $C>0$ we introduce the following two spaces 
\[
{\mathcal G}_\delta:=\lf\{ g\in H^{3/2}(\p B^3,SU(2))\ ;\ \int_{\p B^3}|\iota_{\p B^3}^\ast  dg|^2\ dvol_{\p B^3}<\delta\rg\}
\]
and
\[
{\mathcal F}_{\delta,C}=
\lf\{
\begin{array}{c}
\ds g\in {\mathcal G}_\delta\ ;\ \exists\,\ti{g}\in W^{2,2}(\B^3,SU(2))\quad\mbox{ s. t. }d^\ast(\ti{g}^{-1}\,d\ti{g})=0\\[5mm]
\ds \ti{g}=g\quad\mbox{ on }\p B^3\\[5mm]
\ds\|d\ti{g}\|_{H^{3/2}(\B^3)}\le C\, \|\iota_{\p \B^3}^\ast d g\|_{L^2(\p B^3)}\\[5mm]
\ds \|d\ti{g}\|_{H^1(\B^3)}\le C\, \|dg\|_{H^{1/2}(\p B^3)}
\end{array}
\rg\}
\]
The first goal is to show
\be
\label{A-2}
\exists \,\delta>0\quad\exists\, C>0\qquad\mbox{ s. t. }\qquad{\mathcal G}_\delta={\mathcal F}_{\delta,C}\ .
\ee
We first claim that for $\delta$ small enough ${\mathcal G}_\delta$ is path connected. To that aim we solve the harmonic map flow equation on $\p B^3$ with initial data
an arbitrary element $g\in {\mathcal G}_\delta$
\[
\lf\{
\begin{array}{l}
\ds\p_t\hat{g}+\Delta_{\p B^3}\hat{g}=\hat{g}\,|d\hat{g}|_{\p \B^3}^2\qquad\mbox{ on }\p \B^3\\[5mm]
\ds \hat{g}=g\qquad\mbox{ at }t=0
\end{array}
\rg.
\]
where $\Delta_{\p B^3}$ is the positive Laplace Beltrami operator on $\p \B^3$. Thanks to the main result of \cite{LinL}, there exists $\delta>0$ such that for any initial data $g$ satisfying 
\[
\int_{\p B^3}|\iota_{\p B^3}^\ast  dg|^2\ dvol_{\p B^3}<\delta
\]
then there exists a unique energy decreasing  solution to the flow and it uniformly converges (in any ``strong norm'') to a constant map in $SU(2)$. Since $SU(2)$ is itself path connected we deduce  the path connectedness of ${\mathcal G}_\delta$.

\medskip

In order to prove the claim (\ref{A-2}) we shall now prove that for some well chosen $C>0$ the subset ${\mathcal F}_{\delta,C}$ is both closed and open within ${\mathcal G}_\delta$ for the $H^{3/2}$ topology.

\medskip

First we prove that ${\mathcal F}_{\delta,C}$ is relatively closed in ${\mathcal G}_\delta$. Assume $u_k\in  {\mathcal F}_{\delta,C}$ strongly converges in $H^{3/2}$ to $u\in {\mathcal G}_\delta$. We consider a corresponding sequence of harmonic extensions $\ti{g}_k$. Because of the norm control given by the membership to ${\mathcal F}_{\delta,C}$ we deduce that we can extract a subsequence such that
\[
\ti{g}_{k'}\ \rightharpoonup\ \ti{g}_\infty\qquad\mbox{ weakly in }W^{2,2}(\B^3)\ .
\]
From Rellich Kondrachov we deduce the strong convergence of $d\ti{g}_{k'}$ to $d\ti{g}_{\infty}$in $L^p(\B^3)$ for any $p<6$ and then one can pass to the limit in the harmonic map equation as well as in the trace.
Using the lower semicontinuity
of the $H^{1/2}(\B^3)$ (resp. $H^1(\B^3)$) norms for the differentials of the sequence of extensions $d\ti{g}_{k'}$ we deduce that the pairs $(g_\infty,\ti{g}_\infty)$ satisfies the inequalities
\[
\ds\|d\ti{g}_\infty\|_{H^{3/2}(\B^3)}\le C\, \|\iota_{\p \B^3}^\ast d g_\infty\|_{L^2(\p B^3)}\qquad\mbox{ and }\qquad \|d\ti{g}_\infty\|_{H^1(\B^3)}\le C\, \|dg_\infty\|_{H^{1/2}(\p B^3)}\ .
\]
for the same constant $C>0$. Hence $g_\infty\in {\mathcal F}_{\delta,C}$ .

\medskip

Finally we prove that ${\mathcal F}_{\delta,C}$ is relatively open in ${\mathcal G}_\delta$. Let $g_0\in {\mathcal F}_{\delta,C}$. and $\ti{g}_0$ the harmonic map extensiongiven by the  We consider the following map
\[
L_{g_0}\circ\exp:\ v\in H^{3/2}(\p \B^3,\frak{su}(2))\ \longrightarrow\ g=g_0\,\exp(v)\ ,
\]
where $L_{g_0}$ is the left multiplication by $g_0$. Since $H^{3/2}(\p \B^3,SU(2))\ \hookrightarrow\ C^0(\p B^3, SU(2))$ we have the existence of $\sigma>0$ such that, if
\[
\|g-g_0\|_{H^{3/2}}<\sigma
\]
then $g_0^{-1}\,g$ takes values in a neighbourhood of the identity where $\exp^{-1}$ realizes a $C^\infty$ diffeomorphism and
\[
v: =\exp^{-1}(g_0^{-1}\,g)\in H^{3/2}(\p \B^3,\frak{su}(2))\ .
\]
Hence the image of an open neighborhood of $0$ in $H^{3/2}(\p \B^3,\frak{su}(2))$ is an open neighbourhood of the identity in $H^{3/2}(\p \B^3,SU(2))$.
To any $v\in H^{3/2}(\p \B^3,\frak{su}(2))$ we denote by $\ti{v}$ the harmonic extension of $v$
\[
\lf\{
\begin{array}{l}
\ds\Delta\ti{v}=0\qquad\mbox{ in }\B^3\\[5mm]
\ds\ti{v}=v\qquad\mbox{ on }\p \B^3
\end{array}
\rg.
\]
We now consider the map
\[
\begin{array}{l}
{\mathcal N}\ :\ (v, \ti{w})\in H^{3/2}(\p \B^3,\frak{su}(2))\times H^2_0(\B^3, \frak{su}(2))\\[5mm]
\qquad \longrightarrow\ d^\ast\lf(\exp(-\ti{w})\,\exp(-\ti{v})\,\ti{g}_0^{-1}\,\,d\lf(\ti{g}_0\,\exp(\ti{v})\,\exp(\ti{w})\,  \rg)\rg)\in L^2(\B^3,\frak{su}(2))\ .
\end{array}
\]
Using the fact that $H^2(\B^3)$ is an algebra, the nonlinear operator ${\mathcal N}$ is $C^1$ and the partial derivative at the origin along the $\ti{w}$ direction is equal to
\[
\frac{\p{\mathcal N}}{\p \ti{w}}(0,0)(0,\ti{W})=d^\ast d\ti{W}+d^\ast\lf([\ti{g}_0^{-1} d\ti{g}_0,\ti{W}]\rg)=-\Delta\ti{W}+[\ti{g}_0^{-1} d\ti{g}_0;d\ti{W}]
\]
where we have used that $d^\ast (\ti{g}_0^{-1}\,d\ti{g}_0)=0$. Observe that we have
\[
\begin{array}{l}
\ds\|[\ti{g}_0^{-1} d\ti{g}_0;d\ti{W}]\|_{L^2(\B^3)}\le \|d\ti{g}_0\|_{L^3(\B^3)}\ \|d\ti{W}]\|_{L^6(\B^3)}\le C\, \|d\ti{g}_0\|_{H^{1/2}(\B^3)}\ \|d\ti{W}\|_{W^{1,2}}\\[5mm]
\ds\qquad\le C\, \|dg_0\|_{L^2(\p \B^3)}\  \|\ti{W}\|_{H^2_0}\le C\,\sqrt{\delta}\  \|\ti{W}\|_{H^2_0}\ .
\end{array}
\]
Hence, having chosen $\delta>0$ sufficiently small we have that 
\[
\frac{\p{\mathcal N}}{\p \ti{w}}(0,0)\ W^{2,2}_0(\B^3)\ \longrightarrow\ L^2(\B^3)
\]
realizes an isomorphism. Using the implicit function theorem, there exists a $C^1$ map defined in the neighborhood of the origin in $H^{3/2}(\p \B^3,\frak{su}(2))$
\[
\Xi\ :\ H^{3/2}(\p \B^3,\frak{su}(2))\ \longrightarrow \ W^{2,2}_0(\B^3)
\]
such that
\[
{\mathcal N}(v,\Xi(v))=0
\]
Let
\[
\ti{g}_v:=\ti{g}_0\,\exp(\ti{v})\,\exp(\Xi(v))
\]
It satisfies $d^\ast(\ti{g}_v^{-1}\,d\ti{g}_v)=0$ and, choosing $\ti{v}$ small enough, since 
\[
\|d\ti{g}_0\|_{H^{1/2}(\B^3)}\le C\, \|\|dg_0\|_{L^2(\B^3)}\le C\, \delta
\]
by taking $\delta$ small enough we can ensure that the hypothesis ii) of Lemma~\ref{lm-ap} is fulfilled so that we have (\ref{A-1}) and (\ref{A-01}) holding for $\ti{g}_v$. If we then chose the constant $C$ to be the one given by lemma~\ref{lm-ap}, we have proved that ${\mathcal F}_{\delta,C}$ is open.

Thus for these choices of $\delta$ and $C>0$ we have that ${\mathcal F}_{\delta,C}$ is both open and closed in ${\mathcal G}_\delta$ which is path connected and we deduce
\[
{\mathcal F}_{\delta,C}={\mathcal G}_\delta\ .
\]
This proves the counterpart of Lemma~\ref{lm-ext} for $\S^3$ valued maps. Lets consider now $\S^2\subset\S^3$ as being the intersection of $\S^3$ with the hyperplane $z_4=0$ and assume
\[
u\in W^{1,2}(\p B^3,{\S}^2)\qquad\mbox{ s. t. }\qquad \int_{\p \B^3}|du|^2\ dvol_{\p \B^3}le \delta\ ,
\]
where $\delta>0$ is the small constant obtained in the first part of the proof which ensures ${\mathcal F}_{\delta,C}={\mathcal G}_\delta$. Let $\ti{g}$ the $\S^3$ harmonic map extension of $u$
given by this first part of the proof of lemma~\ref{lm-ext}. Then the fourth coordinate $\ti{g}_4$ of $\ti{g}$ satisfies
\[
\lf\{
\begin{array}{l}
\ds -\Delta \ti{g}_4=\ti{g}_4\, |\nabla\ti{g}|^2\qquad\mbox{ in }\B^3\\[5mm]
\ds\ti{g}_4=0\qquad\mbox{ on }\p\B^3
\end{array}
\rg.
\]
Multiplying by $\ti{g}_4$ the equation and integrating by parts is giving, using H\"older and then Sobolev embedding
\[
\int_{\B^3}|\nabla\ti{g}_4|^2\ dx^3=\int_{\B^3}|\ti{g}_4|^2\, |\nabla\ti{g}|^2\ dx^3\le \|\ti{g}_4\|^2_{L^6(\B^3)}\ \|\nabla\ti{g}\|_{L^3(\B^3)}^2\le C\, \delta \, \int_{\B^3}|\nabla\ti{g}_4|^2\ dx^3
\]
where we have used the embedding $H^{1/2}(\B^3)\ \hookrightarrow\ L^3(\B^3)$. Hence for $\delta$ small enough we have ensured $\ti{g}_4\equiv 0$ and $\ti{u}:=\ti{g}$ is an $\S^2$ valued harmonic map fulfilling
(\ref{A-03}). Recall how to derive the monotonicity formula for smooth harmonic maps : We have for any $r<1$
\[
\begin{array}{l}
\ds0=\int_{\B_r(0)}\sum_{j,k=1}^3x_j\,\p_{x_j}\ti{u}\cdot\p^2_{x_k^2}\ti{u}\ dx^3\\[5mm]
\ds=\int_{\p \B_r(0)} r\,|\p_r\ti{u}|^2\ dvol_{\p \B_r(0)}-\int_{\B_r(0)}|\nabla\ti{u}|^2\ dx^3-\frac{1}{2}\int_{\B_r}\sum_{j=1}^3x_j\,\p_{x_j}|\nabla\ti{u}|^2\ dx^3\\[5mm]
\ds=\int_{\p \B_r(0)} r\,|\p_r\ti{u}|^2\ dvol_{\p \B_r(0)}+\frac{1}{2}\int_{\B_r(0)}|\nabla\ti{u}|^2\ dx^3-\frac{r}{2}\,\int_{\p \B_r(0)}|\nabla u|^2\ dvol_{\p \B_r(0)}
\end{array}
\]
Observe $|\nabla u|^2=|\p_r u|^2+|\nabla_Tu|^2$. This implies then
\be
\label{monha}
\int_{\B_r(0)}|\nabla\ti{u}|^2\ dx^3+\int_{\p \B_r(0)} r\,|\p_r\ti{u}|^2\ dvol_{\p \B_r(0)}= {r}\,\int_{\p \B_r(0)}|\nabla_T \ti{u}|^2\ dvol_{\p \B_r(0)}\ .
\ee

\medskip

In order to establish (\ref{A-04}) we proceed as follows. Poincar\'e inequality is giving
\[
\dashint_{\B_1}\lf|\ti{u}(x)-\dashint_{\B_1} \ti{u}(y)\ dy^3 \rg|^3\ dx^3\le C\, \int_{\B_1}|\nabla \ti{u}|^3\ dx^3\le C\,\delta^{3/2}
\]
Multiplying the equation (\ref{harm-1}) by $\ti{u}(x)-\dashint_{\B_1(0)} \ti{u}(y)\ dy^3$ and integrating by parts is giving
\[
\begin{array}{l}
\ds\int_{\B_1}|\nabla \ti{u}|^2\ dx^3=\int_{\p\B_1}\lf( \ti{u}(x)-\dashint_{\B_1} \ti{u}(y)\ dy^3  \rg)\cdot\frac{\p\ti{u}}{\p r}\ dvol_{\p\B_1}\\[5mm]
\ds\quad+\int_{\B_r(0)}\ti{u}\cdot \lf( \ti{u}(x)-\dashint_{\B_r(0)} \ti{u}(y)\ dy^3  \rg)\,|\nabla\ti{u}|^2\ dx^3
\end{array}
\]
We have
\[
\begin{array}{l}
\ds\lf|\dashint_{\B_1)} \ti{u}(y)\ dy^3-\dashint_{\p\B_1} \ti{u}(y)\ dvol_{\p\B_1}\rg|=\lf|\frac{3}{4\pi}\int_{\B_1} \ti{u}(y)\ dy^3-\frac{1}{4\,\pi}\int_{\p\B_1} \ti{u}(y)\ dvol_{\p\B_1}\rg|\\[5mm]
\ds=\frac{1}{4\,\pi}\lf|\int_0^1\lf(\int_{\p\B_1} 3\,{\rho^2}\,\ti{u}(\rho\,y)\ dvol_{\p\B_1}-\int_{\p\B_1} \ti{u}(y)\ dvol_{\p\B_1} \rg)\ d\rho\rg|\\[5mm]
\ds=\frac{1}{4\,\pi}\lf|\int_0^1\lf(\int_{\p\B_1} 3\,{\rho^2}\,\ti{u}(\rho\,y)\ dvol_{\p\B_1}-\int_{\p\B_1}3\,{\rho^2}\, \ti{u}(y)\ dvol_{\p\B_1} \rg)\ d\rho\rg|\\[5mm]
\ds=\frac{1}{4\,\pi}\lf|\int_0^1 3\,{\rho^2}\, \lf(\int_{\p\B_1}  (\ti{u}(\rho\,y)-\ti{u}(y))\ dvol_{\p\B_1} \rg)\ d\rho\rg|\\[5mm]
\ds\le \frac{1}{4\,\pi}\, \int_{\p\B_1}\int_0^1 3\,{\rho^2}\, \int^1_\rho\lf|\frac{\p\ti{u}}{\p r}\rg|(t, y/|y|)\ dt\ dvol_{\p\B_1(0)} \ d\rho\\[5mm]
\ds\le \frac{3}{4\,\pi}\, \int_0^1 dt\int_{\p\B_t}\lf|\frac{\p\ti{u}}{\p r}\rg|\ dvol_{\p\B_t(0)}=\frac{3}{4\,\pi}\, \int_{\B_1(0)}\lf|\frac{\p\ti{u}}{\p r}\rg|\ dx^3\le \sqrt{\frac{3}{4\pi}}\, \lf(\int_{\B_1}\lf|\frac{\p\ti{u}}{\p r}\rg|^2\ dx^3\rg)^{1/2}
\end{array}
\]
We then deduce from the two previous identities using (\ref{A-03})
\[
\begin{array}{l}
\ds\int_{\B_1}|\nabla \ti{u}|^2\ dx^3=\int_{\p\B_1}\lf( \ti{u}(x)-\dashint_{\p\B_1} \ti{u}(y)\ dvol_{\p\B_1}  \rg)\cdot\frac{\p\ti{u}}{\p r}\ dvol_{\p\B_1}\\[5mm]
\ds\quad+\int_{\p\B_1}\lf( \dashint_{\p\B_1}\ti{u}(y)\ dvol_{\p\B_1}  -\dashint_{\B_1} \ti{u}(y)\ dy^3  \rg)\cdot\frac{\p\ti{u}}{\p r}\ dvol_{\p\B_1}+\int_{\B_1}\ti{u}\cdot \lf( \ti{u}(x)-\dashint_{\B_1} \ti{u}(y)\ dy^3  \rg)\,|\nabla\ti{u}|^2\ dx^3\\[5mm]
\ds\le\lf\|  \ti{u}(x)-\dashint_{\p\B_1} \ti{u}(y)\ dvol_{\p\B_1}   \rg\|_{L^2(\p\B_1)}\ \lf\| \frac{\p\ti{u}}{\p r}  \rg\|_{L^2(\p\B_1)}+ \sqrt{\frac{3}{4\pi}}\,\lf\| \frac{\p\ti{u}}{\p r}  \rg\|_{L^2(\B_1)}\ \ \lf\| \frac{\p\ti{u}}{\p r}  \rg\|_{L^2(\p\B_1)}\\[5mm]
\ds+\lf\|  \ti{u}(x)-\dashint_{\B_1} \ti{u}(y)\ dvol_{\p\B_1}   \rg\|_{L^3(\B_1)}\ \||\nabla\ti{u}|^2\|_{L^{3/2}(\B_1)}\\[5mm]
\ds\le\ C\, \lf\| \nabla_T\hat{u}  \rg\|_{L^2(\p\B_1)}\ \lf\| \frac{\p\ti{u}}{\p r}  \rg\|_{L^2(\p\B_1)}+\|\nabla \ti{u}\|_{L^2(\B_1)}\ \lf\| \frac{\p\ti{u}}{\p r}  \rg\|_{L^2(\p\B_1)}+C\, \sqrt{\delta}\, \|\nabla_T\hat{u}\|^2_{L^2(\p\B_1)}\\[5mm]
\ds\le \ C\, \lf\| \nabla_T\hat{u}  \rg\|_{L^2(\p\B_1)}\ \lf\| \frac{\p\ti{u}}{\p r}  \rg\|_{L^2(\p\B_1)}+\frac{1}{2}\int_{\B_1}|\nabla \ti{u}|^2\ dx^3+\frac{1}{2}\, \lf\| \frac{\p\ti{u}}{\p r}  \rg\|^2_{L^2(\p\B_1)}+C\, \sqrt{\delta}\, \|\nabla_T\hat{u}\|^2_{L^2(\p\B_1)}
\end{array}
\]
Combining this identity with (\ref{monha}) for $r=1$ we obtain
\[
\lf(\frac{1}{2}-C\,\sqrt{\delta}\rg)\, \int_{\p\B_1}\lf|\nabla_T \hat{u}\rg|^2\ dvol_{\p\B_1}\le C\, \lf\| \nabla_T\hat{u}  \rg\|_{L^2(\p\B_1)}\ \lf\| \frac{\p\ti{u}}{\p r}  \rg\|_{L^2(\p\B_1)}+\lf\| \frac{\p\ti{u}}{\p r}  \rg\|^2_{L^2(\p\B_1)}
\]
Choosing $C\,\sqrt{\delta}<1/4$ and using Cauchy-Schwartz we obtain the existence of $C_0>0$ such that
\[
\int_{\p\B_1}\lf|\nabla_T \hat{u}\rg|^2\ dvol_{\p\B_1}\le C_0\, \int_{\p\B_1}\lf| \frac{\p\ti{u}}{\p r} \rg|^2\ dvol_{\p\B_1}
\]
Inserting this identity in (\ref{monha}) for $r=1$ we obtain for this special constant $C_0$
\[
\begin{array}{l}
\ds C_0\,\int_{\p \B_1}|\nabla_T \ti{u}|^2\ dvol_{\p \B_1}=C_0\,\int_{\B_1}|\nabla\ti{u}|^2\ dx^3+C_0\int_{\p \B_1} \lf| \frac{\p\ti{u}}{\p r} \rg|^2\ dvol_{\p \B_1}\\[5mm]
\ds\ge C_0\,\int_{\B_1}|\nabla\ti{u}|^2\ dx^3+\int_{\p\B_1}\lf|\nabla_T \hat{u}\rg|^2\ dvol_{\p\B_1}
\end{array}
\]
This gives (\ref{A-04}) for $1-\nu= \frac{C_0}{C_0+1}$  and Lemma~\ref{A-04} is proved.\hfill $\Box$

\subsection{ A Product Inequality}
\begin{Lma}
\label{lm-in} There exists a constant $C>0$ such that for any $a\in H^{1/2}(\B^3)$ and $b\in L^\infty\cap W^{1/2,6}(\B^3)$ we have
\[
\|a\,b\|_{H^{1/2}(\B^3)}\le C\, \|a\|_{H^{1/2}(\B^3)}\, \lf(\|b\|_{L^\infty((\B^3)}+\|b\|_{W^{1/2,6}(\B^3)}\rg)
\]

\end{Lma}
\noindent{\bf Proof of Lemma~\ref{lm-in}} This result is a well known fact but we give a proof below for the convenience of the readers. Thanks to \cite{Zhu} we can replace $\B^3$ by ${\R}^3$. Let 
\[
a=\sum_{i\in {\N}}a_i\quad\mbox{ and }b=\sum_{i\in {\N}}b_i\
\]
in-homogeneous Littlewood Paley decompositions respectively of $a$ and $b$ that is Supp$(\hat{a}_0)\subset \B_2(0)$ and Supp$(\hat{a}_i)\subset \B_{2^{i+1}}(0)\setminus \B_{2^{i-1}}(0)$.
We also denote
\[
a^i:=\sum_{j\le i-3} a_j
\]
so that  Supp$(\hat{a}^i)\subset \B_{2^{i-2}}(0)$. We proceed to the classical high$\times $low+low$\times $high +high$\times $high decomposition
\be
\label{hllhhh}
a\,b=\sum_{i\ge 0} a^i\,b_i+\sum_{i\ge 0} a_i\,b^i+\sum_{|i-j|\le2}a_i\,b_j\ .
\ee
Observe that Supp$\,{\mathcal F}({a}^i\, b_i)\subset \B_{9\times\, 2^{i-2}}(0)\setminus \B_{3\times\, 2^{i-2}}(0)$. We have
\[
\lf\|\sum_{i\ge 0} a^i\,b_i\rg\|^2_{H^{1/2}({\R}^3)}\simeq\int_{\R^3}\sum_{i\ge 0} 2^{i}\,|a^i|^2\,|b_i|^2\ dx^3
\]
Since $H^{1/2}({\R}^3)\hookrightarrow L^3(\R^3)$ we have thanks to the Marcikiewicz Maximal function theorem
\[
\|M(a)\|_{L^3({\R}^3)}\le C\ \|a\|_{L^3({\R}^3)}\le C\ \|a\|_{H^{1/2}({\R}^3)}
\]
from which one deduces
\[
\lf(\int_{\R^3}\sup_{i\in I}|a^i|^3\ dx^3\rg)^{1/3}\le  C\ \|a\|_{H^{1/2}({\R}^3)}
\]
Hence we bound
\[
\begin{array}{l}
\ds\int_{\R^3}\sum_{i\ge 0} 2^{i}\,|a^i|^2\,|b_i|^2\ dx^3\le \lf(\int_{\R^3}\sup_{i\in I}|a^i|^3\ dx^3\rg)^{2/3}\ \lf(\int_{\R^3}\lf|\lf(\sum_{i\ge 0} 2^{i}\,|b_i|^2\rg)^{1/2}\rg|^6\ dx^3\rg)^{2/3}\\[5mm]
\ds\quad\le C\, \|a\|^2_{H^{1/2}(\B^3)}\,\|b\|^2_{W^{1/2,6}(\B^3)}
\end{array}
\]
For the second term in (\ref{hllhhh}) we have
\[
\lf\|\sum_{i\ge 0} a_i\,b^i\rg\|^2_{H^{1/2}({\R}^3)}\simeq\int_{\R^3}\sum_{i\ge 0} 2^{i}\,|a_i|^2\,|b^i|^2\ dx^3\le \|a\|^2_{H^{1/2}(\B^3)}\,\|b\|^2_{L^\infty(\B^3)}
\]
Finally, for the last term we have taking $j=i$
\[
\lf\|\sum_{i\ge 0} a_i\,b_i\rg\|^2_{H^{1/2}({\R}^3)}=\sup_{\|h\|_{H^{-1/2}({\R}^3)}\le 1}\int_{\R^3}\sum_{i\ge 0} a_i\,b_i\, h\ dx^3
\]
Observe that we have
\[
\begin{array}{l}
\ds\int_{\R^3} \sum_{i\ge 0}\,a_i\,b_i\, h\ dx^3= \int_{\R^3} \sum_{i\ge 0}\,a_i\,b_i\, h^{i+6}\ dx^3\le \|b\|_\infty\,\int_{\R^3} \sum_{i\ge 0}\,2^{i/2}\,|a_i|\, 2^{-i/2}|h^{i+6}|\ dx^3\\[5mm]
\ds\quad\le  \|b\|_\infty\,\lf(\int_{\R^3} \sum_{i\ge 0}\,2^{i}\,|a_i|^2\ dx^3\rg)^{1/2}\,\lf(\int_{\R^3} \sum_{i\ge 0}\,2^{-i}\,|h^{i+6}|^2\ dx^3\rg)^{1/2}\\[5mm]
\ds\quad\le C\,  \|b\|_{L^\infty(\R^3)} \,\|a\|_{H^{1/2}(\R^3)}\,\|h\|_{H^{-1/2}(\R^3)}\ .
\end{array}
\]
This concludes the proof of the lemma~\ref{lm-in}.\hfill $\Box$ 
 
\end{document}